\documentclass[11pt,letterpaper]{amsart}
\usepackage{graphicx} 
\input{style.sty}
\title{Real Link Floer Homology}
\author{Yonghan Xiao}
\address{Department of Mathematics, School of Mathematical Sciences\\Peking University}
\email{judy\_xyh0530@stu.pku.edu.cn}

\makeatletter
\let\@wraptoccontribs\wraptoccontribs
\makeatother
\contrib[with an appendix joint with]{Zhenkun Li}
\address{Academy of Mathematics and Systems Science\\ Chinese Academy of Sciences}
\email{zhenkun@amss.ac.cn}

\usepackage{fullpage, graphicx, enumitem}
\usepackage{url}
\usepackage{verbatim}
\usepackage{bbm}
\usepackage{amssymb,amsmath}
\usepackage{amsfonts,savesym,graphicx,bm,amsthm}
\usepackage{color}
\usepackage[mathscr]{eucal}
\usepackage[all]{xy}

\date{\today}
\begin{document}
\maketitle

\begin{abstract} 
In this paper, we define real link Floer homology for strongly invertible and doubly periodic links in closed real $3$-manifolds with connected fixed sets, which generalizes real Heegaard Floer homology defined in~\cite{guth2025real} and real sutured Heegaard Floer homology defined in~\cite{BGX}. We give a combinatorial description of the theory in $S^3$ via real grid diagrams and use it to investigate structural properties of the theory as well as properties of strongly invertible knots. A computer implementation computing real grid homology of knots was written by Zhenkun Li in \cite{ZhenkunLioythonprgram}. The appendix includes real grid homology of 50+ small knots, from which we observe several interesting phenomenon.

\end{abstract}
\tableofcontents

\section{Introduction}
In \cite{ozsvath2004holomorphic}, Ozsv\'ath and Szab\'o introduced various versions of Heegaard Floer homology $\HF^{\circ}$, which are useful invariants of closed oriented $3$-manifolds. Afterwards, knot Floer homology was introduced by Ozsv\'ath and Szab\'o in \cite{Holomorphicdisksandknotinvariants} and independently by Rasmussen in \cite{Rasmussen}. In \cite{SarkarWang2010}, Sarkar and Wang defined nice Heegaard diagram, from which $\widehat{\HF}$ can be computed combinatorially. Following that, a combinatorial version of knot Floer homology first appeared in \cite{MOS2009knot}. Later, Manolescu, Ozsv\'ath, Szab\'o, and Thurston \cite{MOST2007grid} generalized this to links. The book~\textit{Grid homology for knots and links} \cite{OSS2015grid} gives a detailed and complete description of this theory. Recently, Guth and Manolescu \cite{guth2025real} defined real Heegaard Floer homology for closed real $3$-manifolds (closed oriented $3$-manifolds equipped with an orientation-preserving involution whose fixed set is non-empty) and Hendricks \cite{hendricks2025noterealheegaardfloer} gave a localization spectral sequence relating it to classical Heegaard Floer homology. In \cite{BGX} and \cite{LOrealbordered}, real sutured Floer homology and real Bordered Floer homology were developed and the authors introduced the notion of a real nice Heegaard diagram. 

On the other hand, equivariant knots and links in real $3$-manifolds have always been interesting for low-dimensional topologists. The investigation of periodic \cite{Trotter1961PeriodicAO,Murasugi1971OnPK,Edmonds1983GroupAO} and strongly invertible knots \cite{Sakuma} dates back more than 60 years. More recently, in \cite{hirasawa2023equivariant} and \cite{hirasawainvariantseifertsurface}, Hirasawa, Huira and Sakuma investigated equivariant Seifert genera of strongly invertible knots in $S^3$. Equivariant slice genera of strongly invertible knots have also been estimated using various tools in recent years: Donaldson's theorem \cite{BIequiv4genusofsiandperidicknots}, Blanchfield form \cite{MillerPowellequivslicegenera}, knot Floer homology \cite{DMSequivariantknotandHFK}, Khovanov homology \cite{Lobb2021ArefinementofKhovanovhomology,Sano_2025, borodzik2025khovanovhomologyequivariantsurfaces}, etc. Other invariants like the equivariant unknotting number \cite{boyle2025equivariantunknottingnumbersstrongly} have also been defined and estimated. Moreover, strongly invertible knots can be used to construct exotic disks \cite{DMSequivariantknotandHFK,dai2026exoticdiskssingularinstanton}.

In previous works, real theoretic tools have been applied to double branched cover of knots \cite{Konno2024,li2024realmonopolesspectralsequence,guth2025real}, which have several amazing application, see \cite{miyazawa2023gaugetheoreticinvariantembedded,miyazawa2025satelliteformularealseibergwitten,DKMPS,Kang_2026}, but the investigation of symmetries of knots is still left open. An idea of real knot Floer homology has appeared in \cite{hendricks2025noterealheegaardfloer}, but its invariance was not proved and its properties were not investigated. In this paper, we introduce real link Floer homology for (generalized) strongly invertible and doubly periodic links (see Definition~\ref{def:strongly invertible links} and \ref{def:periodic links}) in closed real $3$-manifolds with connected fixed sets. We will see that this homology is algorithmically computable, provides bounds on various equivariant knot invariants, detects non-commutativity of the equivariant connected sum and also has interesting structural properties like or in contrast with existing knot homologies.

\begin{theorem}\label{thm:intro-real link Floer homology} Let $(Y,\tau)$ be closed real $3$-manifold with $\fix(\tau)$ connected.

\begin{enumerate}
    \item If $(L,\fra)$ is a generalized strongly invertible link with auxiliary data in $(Y,\tau)$(see Definition~\ref{def:strongly invertible knots with extra data} for the notion of auxiliary data), then there are real link Floer homology groups \[\HFLR^-(L,\fra) \quad \text{and}\quad \widehat{\HFLR}(L,\fra),\] which are invariants of the isotopy class of $(L,\fra)$ valued in the category of modules over a suitable polynomial ring and $\F$-vector spaces, respectively.
    \item If $(L,\fro)$ is an oriented generalized doubly periodic link in $(Y,\tau)$, then there is a real link Floer homology group \[ \widehat{\HFLR}(L,\fro),\] which is an invariant of the isotopy class of $(L,\fro)$ valued in the category of $\F$-vector spaces.
\end{enumerate}
Throughout this paper, $\F$ denotes the $2$-element field $\F_2$.
\end{theorem}
The real link Floer homology groups split along real $\spinc$ structures into \[\HFLR^\circ(L,\fra)=\bigoplus_{\s^R\in \rspinc(Y,\tau)} \HFLR^{\circ}(L,\fra,\s^R), \] similar to the real Heegaard Floer groups associated to closed or sutured real $3$-manifolds. It can also be equipped with additional grading structures (see Section~\ref{sub:Strongly invertible links}). More precisely, when $c_1(\s^R)$ is torsion, there is a real Maslov grading $M^R$ valued in $\Z$ on $\HFLR^{\circ}(L,\fra,\s^R)$. In the strongly invertible case, if we further assume that $L$ is null-homologous, then there is a real Alexander grading $A^R$ valued in $\frac{1}{2}\Z$. 

\begin{rem}\label{rmk:importance of auxiliary data}
For a strongly invertible knot, the choice of auxiliary data reduces to the choice of a direction, which is a more classical notion (see \cite{Sakuma}). One will see later that the choice of auxiliary data in this paper forces the Heegaard diagram to have one $O$ and one $X$-base point on each strongly invertible knot component. We will see through Example~\ref{ex:direction dependence} that the real link Floer group depends essentially on the this choice. Via the spectral sequence roughly from $\widehat{\HFK}$ to $\widehat{\HFKR}$ in Theorem~\ref{thm:spectral sequence from HFK,strongly invertible case}, this explains why a choice of direction is essential for defining $\tau$ action on $\fullCFK^-$ in \cite[Section~3]{DMSequivariantknotandHFK}(see Section~\ref{sub:calculatetauKonclassicalHFK}).

We can define another real link Floer theory for (generalized) strongly invertible links using real Heegaard diagrams that have one pair of $O$-base points and no $X$-base points along the fixed set on each strongly invertible knot component. It can be seen from examples that this theory is quite different from the one defined in this paper and we will come back to this theory and the comparison between two versions in a later work~\cite{YXHFLR2}.

We also want to remark that the connectivity assumption on the fixed set is not very restrictive, as it has been shown in \cite[Section~2]{Myers1981HomologyW} that every real structure on an integral homology $3$-sphere has connected fixed set.
\end{rem}

With the help of real nice Heegaard diagram, as defined in \cite{BGX} and \cite{LOrealbordered}, we give a combinatorial description for real link Floer homology in $(S^3,\tau_{\mathrm{std}})\subset (\C^2,\text{conjugation})$ via real grid diagrams (see Definition~\ref{def:real grid diagram}). 

\begin{theorem}\label{thm:intro-real grid homology}
Fix the background real $3$-manifold to be $(S^3,\tau_{\mathrm{std}})$.
\begin{enumerate}

\item If $(L,\fra)$ is a (generalized) strongly invertible link with auxiliary data, then there are real grid homology groups \[\GHR^-(L,\fra) \quad \text{and}\quad \widehat{\GHR}(L,\fra),\] which are invariants of the isotopy class of $(L,\fra)$ valued in the category of bigraded modules over a suitable polynomial ring and bigraded $\F$-vector spaces, respectively.

\item If $(L,\fro)$ is an oriented (generalized) doubly periodic link, then there is a real grid homology group \[ \widehat{\GHR}(L,\fro),\] which is an invariant of the isotopy class of $(L,\fro)$ valued in the category of $\F$-vector spaces.

\end{enumerate}
\end{theorem}

These are isomorphic to the real link Floer groups defined in Theorem~\ref{thm:intro-real link Floer homology} and have the advantage of being computable. A computer implementation by Zhenkun Li is given in \cite{ZhenkunLioythonprgram} and some calculation results will be given in Appendix~\ref{app:Calculation results for knots with small crossing numbers}. 

In classical knot Floer homology, torsion order $\ord_{U}$ and $\tau$-invariants provide useful bounds on knot invariants. We investigate their counterparts $\ord_{u}$ and $\tau^R$ in real knot Floer homology via the combinatorial description following the strategy from \cite{Alishahi_2020}, \cite{Sarkar2011tau} and \cite{OSS2015grid}. More precisely, we provide bounds on type $A$ and $B$ unknotting number (see Definition~\ref{def:equiv unknotting number}). 

\begin{prop}\label{prop:intro-bounding u_A and u_B}
Let $K$ be a strongly invertible knot in $(S^3,\tau)$. For any choice of auxiliary $\fra$, we have \[\ord_{u}(K,\fra)\le 2\widetilde{u}_A(K)\quad \text{and}\quad \ord_{u}(K,\fra)\le 2\widetilde{u}_B(K),\] as well as \[\vert \tau^R(K,\fra)\vert\le \widetilde{u}_A(K)\quad \text{and}\quad \vert \tau^R(K,\fra)\vert\le \widetilde{u}_B(K).\] 
\end{prop}

We also show that $\tau^R$ provides a lower bound on a variation of equivariant slice genus.
\begin{theorem}\label{thm:intro-bounding g+m/2}
For any strongly invertible knot $K$ in $S^3$, let $S$ be an equivariant smooth slice surface of $K$ inside $(B^4,\tau_c)$. Then for any choice of auxiliary data $\fra$, \[\vert \tau^R(K,\fra) \vert\le g(S)+(m-1)/2,\] where $m$ is the minimal number of fixed critical points of $r|_{S'}$ for $r$ the radius function on $B^4$ ranging over all generic $S'$ (in the sense that $r|_{S'}$ is a Morse function) that are equivariantly isotopic to $S$.
\end{theorem}

\begin{rem}
In this theorem, we restrict ourselves to considering equivariant slice surfaces in $(B^4,\tau_c)$, which is the unit ball of $(\C^2,\text{conjugation})$. The restriction comes from the combinatorial nature of our argument. This notion of slice genus is the same as the one appeared in \cite{borodzik2025khovanovhomologyequivariantsurfaces}, while the one considered in \cite{DMSequivariantknotandHFK} and some other papers is slightly more general. In their setup, $B^4$ is allowed to have any involution extending $(S^3,\tau_{\std})$. It was known from \cite{boyle2025involutionss4} that the extension of $\tau_{\std}$ is not unique, while it is unknown whether the two notions of equivariant slice genus are different. 
\end{rem}

\begin{rem}
It was pointed out by Alessio Di Prisa after the author posted this paper on arXiv that the bound in Theorem \ref{thm:intro-bounding g+m/2} is possibly related to the notion of ``slice complexity'' defined in \cite[Definition~4.14]{DiPrisaequivalgconcordance}. In particular, using the argument from Proposition 4.16 of that paper, one can get another bound  \[\vert \tau^R(K,\fra) \vert\le 2\widetilde{g}_4(K),\] where $\widetilde{g}_4(K)$ is the equivariant slice genus of $K$.
\end{rem}

The term $m$ appearing in Theorem~\ref{thm:intro-bounding g+m/2} is possibly unavoidable. In Section~\ref{sub: Proof of the main theorem for tauR}, we give an estimation of $m$ and propose a possible obstruction for $m=0$ in Remark~\ref{rmk:m is unavoidable}. In particular, we will see in Example~\ref{ex:bounding butterfly 4-genus} that for butterfly 4-genus (\cite{BIequiv4genusofsiandperidicknots}), $\tau^R$ provides a genuine lower bound. 

The knot Floer group $\widehat{\HFK}$ serves as a categorification of the Alexander polynomial. In particular, it satisfies an oriented skein relation (see \cite{Holomorphicdisksandknotinvariants}). We investigate the real version. 

\begin{thm}\label{thm:intro-minus oriented skein triple}
Let $(L_+, L_- ,L_0,\fra)$ be a real oriented skein triple defined as in the beginning of Section~\ref{sub:Oriented version}. Let $l_p$ ($l_p'$) be the number of pairs of components in $L_+$ ($L_0$).
\begin{itemize}
 \item If $l_p'=l_p+1$, we have a long exact sequence \begin{equation*}
    \to c\GHR^-_d(L_+,s) \xrightarrow{f^-} c\GHR^-_d(L_-,s) \xrightarrow{g^-} c\GHR^-_{d-1}(L_0,s) \xrightarrow{h^-} c\GHR^-_{d-1}(L_+,s) \to 
    \end{equation*}
 \item If $l_p'=l_p$, we have a long exact sequence \begin{equation*}
    \to c\GHR^-_d(L_+,s) \xrightarrow{f^-} c\GHR^-_d(L_-,s) \xrightarrow{g^-} (c\GHR^-(L_0)\otimes W)_{d-1,s} \xrightarrow{h^-} c\GHR^-_{d-1}(L_+,s) \to 
    \end{equation*}
    for $W=\F_{(0,-1/2)}\oplus \F_{(-1,-1/2)}$.
    \item If $l_p'=l_p-1$, we have a long exact sequence \begin{equation*}
    \to c\GHR^-_d(L_+,s) \xrightarrow{f^-} c\GHR^-_d(L_-,s) \xrightarrow{g^-} (c\GHR^-(L_0)\otimes J)_{d-1,s} \xrightarrow{h^-} c\GHR^-_{d-1}(L_+,s) \to 
    \end{equation*}
    for $J=\F_{(0,1)}\oplus \F_{(-1,0)}\oplus \F_{(-1,0)}\oplus \F_{(-2,-1)}$.
\end{itemize}
Here, $c\GHR^-$ is a variation of $\GHR^-$ that will be defined in Section~\ref{sub:Collapsed grid homology for generalized strongly invertible links} and the subscripts denote the bigrading $(M^R,A^R)$. We also have a corresponding version for $\widehat{\GHR}$.
\end{thm}

We can take Euler characteristic of $\widehat{\GHR}(L,\fra)$ to define a real Alexander polynomial $\Delta^R(L,\fra)$. As a corollary of Theorem~\ref{thm:intro-minus oriented skein triple}, we have 
\begin{cor}\label{cor:intro-skein relation for real Alexander polynomial}
Let $(L_+, L_- ,L_0,\fra)$ be a real oriented skein triple. Then we have \[\Delta^R(L_+,\fra)(t)-\Delta^R(L_-,\fra)(t)=(t-t^{-1})\cdot\Delta^R(L_0,\fra)(t).\]
\end{cor}

\begin{rem}\label{rmk:intro-polynomial}
After the first version of this paper was posted on arxiv, Collari and Lisca \cite{collari2026polynomialinvariantstronglyinvolutive} introduced a polynomial invariant $P^{e}(L)\in \Z[a^{\pm 1}, z^{\pm 1}]$ for strongly involutive links which is an equivariant counterpart of the HOMFLY--PT polynomial. It is interesting to observe that the skein relation in Corollary~\ref{cor:intro-skein relation for real Alexander polynomial} is the specialization of their Equation (2) by setting $a=1$ and $z=t-t^{-1}$. By comparing Appendix~\ref{app:Calculation results for knots with small crossing numbers} with their Appendix~B, one sees that these polynomial also coincide on strong inversions on all knots with crossing numbers less than or equal to $7$. Careful readers may notice a discrepancy at $7_7$. This is because to obtain the smallest real grid diagrams of the knots, the two strong inversions we computed for $7_7$ are actually for $7_7$ and $\bar{7}_7$, respectively. Applying the mirroring formula from Proposition~\ref{prop:mirroring the polynomial invariant}, we get the correct polynomial. The same remark needs to be taken into consideration for Proposition~\ref{prop:intro-distinguish involutions on <=7 crossing knots}. These observations may motivate the conjecture that our polynomial is a specialization of theirs. 

However, we then found that our polynomial is no longer a specialization of theirs when we move to links. After setting $a=1$ in their Equation (5), we see that $P^e|_{a=1}$ vanishes on split strongly invertible links, which is not true for the real Alexander polynomial. Also, if we let  $U_2$ denote the two component unlink in $S^3$ with two components interchanged by the involution, then $\Delta^R(U_2)\ne P^e(U_2)|_{a=1,z=t-t^{-1}}$. The former can be computed using a minimally pointed real Heegaard diagram or a real grid diagram and the later appears in \cite[Equation (4)]{collari2026polynomialinvariantstronglyinvolutive}. Thus, we may form the conjecture that for a strongly invertible knot $K$ in $S^3$, $\Delta^R(K)= P^e(K)|_{a=1,z=t-t^{-1}}$ and pose the question asking how $\Delta^R$ and $P^e$ are related in general. 
\end{rem}

Following \cite[Chapter~10]{OSS2015grid}, we also investigate the  real grid homology associated to an unoriented skein triple.

\begin{thm}\label{thm:intro-unoriented skein exact sequence}
Let $L$, $L_1$ and $L_2$ be a real unoriented skein triple defined in Section~\ref{sub:Unoriented version} and $l_p$, $l_p^1$, $l_p^2$ be the number of paired components in $L$, $L_1$ and $L_2$, respectively. By definition, they share the same number of strongly invertible components $l_f$. For sufficiently large $N\in \Z$, we have an exact triangle  
\[\begin{tikzcd}
    \widehat{\GHR}(L_1)\otimes W^{N-l_f-l_p^1} \arrow{rr} & & \widehat{\GHR}(L_2)\otimes W^{N-l_f-l_p^2}  \arrow{dl}\\
    & \widehat{\GHR}(L)\otimes W^{N-l_f-l_p} \arrow[swap]{ul}\\
\end{tikzcd}\]
Here, $W$ is a 2-dimensional $\F$-vector space. This exact triangle respects a $\delta^R$-grading which will be introduced in Section~\ref{sub:Unoriented version}.
\end{thm}

Via the python program \cite{ZhenkunLioythonprgram}, the author computed real grid invariants for various strongly invertible knots with transvergent diagrams of small crossing number. From the results, we get amphichiral knots with non-vanishing $\tau^R$, knots with $\tau^R> g_4$, $\HFK$-thin knots with interesting $\HFKR$ (asked for by \cite{hendricks2025noterealheegaardfloer}), etc. Most notably, we find that 
\begin{prop}\label{prop:intro-distinguish involutions on <=7 crossing knots}
The real Alexander polynomial (thus the real knot Floer group) is able to distinguish different strong inversions on all knots with crossing number $\le 7$.
\end{prop}

Actually, we can recover the classical knot Floer homology from real knot Floer homology.

\begin{prop}\label{prop:intro-recovering hat HFK from hat HFKR}
Let $K$ be any knot in $S^3$. Then the bigraded group $\widehat{\HFK}(K)$ can be recovered from $\widehat{\HFKR}(K\#\tau K^r,\fra)$ for any choice of auxiliary data. Here, $K\#\tau K^r$ is the strongly invertible knot canonically associated to $K$ in \cite[Section~2]{BIequiv4genusofsiandperidicknots}(see also \cite[Section~4.2]{DMSequivariantknotandHFK}).

Similarly, we can consider the doubly periodic knot $K\#K$ from \cite[Section~8]{JZstabilizationdistancefromHFL}. We have \[\rank \widehat{\HFK}(K)=\rank \widehat{\HFKR}(K\# K,\fro),\] for either choice of orientation. 
\end{prop}

In a previous version of this paper, this proposition was stated as a conjecture. It was then pointed out by Fraser Binns to the author that this can be proved using the real sutured decomposition developed in \cite[Section~4]{BGX}.

\begin{proof}
We first argue for the strongly invertible case. After removing an equivariant neighborhood $\nu(K\#\tau K^r)$ from $S^3$, the connected sum sphere becomes a separating $\{1\}$-reflection product annulus (see \cite[Definition~4.4]{BGX}) in the real sutured exterior of $K\#\tau K^r$ (see \cite[Example~2.5]{BGX}). A real product decomposition along this annulus provides us with the disjoint union of two copies of $S^3-\nu(K)$ interchanged by the involution, each equipping with two meridian sutures. Then a combined use of \cite[Proposition~4.6]{BGX} and the tautological isomorphisms $\SFH(S^3-\nu(K))\cong \widehat{\HFK}(S^3,K)$ and  $\RSFH(S^3-\nu(K\#\tau K^r),\tau)\cong \widehat{\HFKR}(S^3,K\#\tau K^r,\tau)$ provides us with the desired result. By keeping track of the Maslov index of disks and relative $\spinc$ structures carefully, the bigradings can also be recovered via $(M,A)=(M^R, 2A^R)$.

A similar argument works in the periodic case. More precisely, after removing an equivariant neighborhood $\nu(K\# K)$ from $S^3$, the connected sum sphere becomes a separating $\Z/2$-reflection product annulus (see \cite[Definition~4.4]{BGX}) in the real sutured exterior of $K\# K$. Decomposing this annulus again provides us with the disjoint union of two copies of $S^3-\nu(K)$ interchanged by the involution, each equipping with two meridian sutures. Then result follows in exactly the same way as above.  
\end{proof}

For more examples and applications, see Section~\ref{sec:Examples and application}, in which we also make comparison between different spectral sequences and try to recover $\tau_K$ action on the usual knot Floer complex.

\begin{remark}
Before ending the introduction, we would like to make some comparison between our theory and existing theories. \begin{itemize}
\item Comparing to equivariant Khovanov homology, as considered in \cite{Lobb2021ArefinementofKhovanovhomology}, \cite{Sano_2025} and \cite{borodzik2025khovanovhomologyequivariantsurfaces}, our theory is known to be sensitive to the choice of auxiliary data (i.e., direction of strongly invertible knots), but their invariants are not. On the other hand, a certain Kunneth principle is satisfied by equivariant Khovanov homology according to \cite[Section~6]{borodzik2025khovanovhomologyequivariantsurfaces}, which is not true for real knot Floer homology (see Section~\ref{sub:Kunneth principle fails}). A common feature of the two theories is that they are both sensitive to different strong inversions on the same knots on many small examples.

\item Comparing to the classical knot Floer homology, again the failure of Kunneth principle serves as an interesting new feature. Thus, $\tau$ is known to be a group homomorphism from the concordance group to $\Z$, while $\tau^R$ is only known to be a map from the equivariant concordance group. On the other hand, $\tau$ always vanishes on amphichiral knots, but $\tau^R$ can take non-zero values on them, as we shall see in Example~\ref{ex:amphichiral}. Of course, this paper only covers a small number of properties enjoyed by real knot Floer theory, more aspects will be developed or investigated in future works. 
\end{itemize}
\end{remark}

\subsection{Organization}
In Section~\ref{sec:Real knot Floer chain complex and real knot Floer homology}, we define real link Floer homology for equivariant links. In Section~\ref{sec:A combinatorial description}, we reformulate the theory in $S^3$ combinatorially and prove its basic properties in Section~\ref{sec:Basic properties and examples}. Real torsion order and $\tau^R$-invariant will be defined and investigated in Section~\ref{sec:Torsion order and equivariant unknotting number} and~\ref{sec:Real tau-invariant and equivariant slice genus}. Then, we prove various skein relations in Section~\ref{sec:Skein exact triangles}. Finally, we provide applications of real knot Floer homology and analyze examples from Appendix~\ref{app:Calculation results for knots with small crossing numbers} in Section~\ref{sec:Examples and application}.  

\begin{ack}
The author would like to thank Zhenkun Li for useful discussions during the preparation of this paper and his contribution on computing examples and making Appendix~\ref{app:Calculation results for knots with small crossing numbers} via python program.
The author is also grateful to Fraser Binns, Wenzhao Chen, Irving Dai, Gary Guth, Kristen Hendricks, Robert Lipshitz, Ciprian Manolescu and Lisa Piccirillo for helpful discussions. Moreover, the author thanks Alessio Di Prisa for his comments on an early draft. Many useful discussions happened during the author's visit to Stanford and Princeton University, she is thankful to Fraser Binns, Gary Guth and Ciprian Manolescu for providing her the chance to visit.
\end{ack}

\section{Real link Floer chain complex and real link Floer homology}\label{sec:Real knot Floer chain complex and real knot Floer homology}
\subsection{Strongly invertible links}\label{sub:Strongly invertible links}
\begin{defn}\label{def:strongly invertible links}
    Let $(Y,\tau)$ be a closed oriented real 3-manifold and $L$ be a link in $Y$.
    \begin{itemize}
        \item $L$ is called \emph{strongly invertible} if for any choice of orientation on $L$, $\tau$ restricts to an orientation-reversing involution on each component of $L$. 
        \item $L$ is called \emph{generalized strongly invertible} if there exists an orientation on $L$ such that $\tau$ restricts to an orientation-reversing involution on $L$. Here, we allow pairs of components to be interchanged in an orientation-reversing way.
    \end{itemize} 
    We add the adjective ``generalized'' in the second case to avoid confusion with the notion of strongly invertible link already defined in \cite{Montesinos+1975+227+260}.
\end{defn}

\begin{defn}\label{def:strongly invertible knots with extra data}

Let $K$ be a strongly invertible knot in $S^3$ and $C=\fix (\tau)$ be the axis of the involution, where $\tau$ is the unique involution on $S^3$ up to isotopy. Note that $C\cap K$ consists of two points, which divides $C$ into two segments. Each of these will be regarded as a half axis for the strongly invertible knot. Basing on this observation, a \emph{directed strongly invertible knot} in $S^3$ is defined as strongly invertible $K$ together with a choice of an orientation on $C$ and a half axis.  

We generalize this notion as follows. Let $(Y,\tau)$ be a closed real $3$-manifold  with $C=\fix(\tau)$ connected. A generalized strongly invertible links \emph{with auxiliary data} in $(Y,\tau)$ is a pair $(L,\fra)$, in which $L$ is a generalized strongly invertible link in $(Y,\tau)$ and $\fra$ is a data set consisting of \begin{enumerate}
    \item a choice of orientation $\fro$ on $L$ that fits $L$ into the definition of a generalized strongly invertible link;
    \item a labeling $\frl$ using $(O,X)$ on two fixed points for each strongly invertible knot component in $L$;
    \item a choice of orientation $\fro'$ on $C$;
    \item a choice of half-axis $A$ on $C$, i.e., a segment on $C$ containing $L\cap C$;
\end{enumerate}
 satisfying that \begin{itemize}
     \item for each strongly invertible knot component $K$ of $L$, the orientation of $A$ given by orienting from $O$ to $X$ is the same as that gotten from $\fro'$;
     \item or equivalently, for each strongly invertible knot component $K$ of $L$, moving $O$ and $X$ along $\fro|_K$ for a short distance, then at two intersection points $K\cap C$, the orientation of $A$ and $K_O$ agree. Here, $K_O$ means the segment of $K\setminus C$ containing $O$. 
 \end{itemize}
Here, the first description generalizes \cite{hendricks2025noterealheegaardfloer} while the second generalizes \cite{DMSequivariantknotandHFK}.
\end{defn}

One may have noticed that the compatibility requirement in the definition above is quite strong. Actually, the four choices above are closely related as observed in \cite[Section~3]{DMSequivariantknotandHFK}. For a strongly invertible link,
\begin{itemize}
    \item after we fix a choice of (1), (3) and (4), then (2) is automatically determined; Similarly, if we fix (2), (3) and (4), then (1) is determined. More precisely, once we fix (3) and (4), for each component $K\subset L$, there are exactly two choices of $(\fro|_K,\frl|_K)$ that are compatible with the given information on axis related by changing $\fro|_K$ and the labeling simultaneously.
    \item On the other hand, if we fix choices of (1) and (2) at first, the choices for (3) and (4) are also nearly uniquely determined if a compatible one exists.
\end{itemize}

When pairs of components appear in $L$, we only need to choose an orientation on each of them and we have two choices for each pair, since the compatibility condition has nothing to do with them. Specializing to strongly invertible knots, this recovers the discussion in \cite[Section~3.5]{DMSequivariantknotandHFK}.

Motivated by this, we say two auxiliary data $\fra$ and $\fra'$ are \emph{equivalent} if they share the same data (1) and (2). One will see later that fixing (1) and (2) is enough for the bigraded real link Floer group to be well-defined (uniquely determined up to isomorphism). We will abuse the notation $\fra$ for an equivalent class of auxiliary data.

\begin{defn}\label{def:real Heegaard diagram for strongly invertible links}
Let $(Y,\tau)$ be a closed real $3$-manifold with $\fix(\tau)$ connected and $L$ be a (generalized) strongly invertible link in $(Y,\tau)$. Fix an auxiliary data $\fra$ for $L$. We say a real Heegaard diagram $\cH=(\Sigma,\bm\alpha,\bm\beta,\bfO,\bfX, R)$ represents $(L,\fra)$ if as usual: 
    \begin{itemize}
        \item $(\Sigma,\bm\alpha,\bm\beta,\bfO, R)$ and $(\Sigma,\bm\alpha,\bm\beta,\bfX, R)$ are both multi-based real Heegaard diagram for $(Y,\tau)$.  Here, we only require $\bfO$ and $\bfX$ to be fixed as sets by $R$, which is more general than the notion of real Heegaard diagram defined in \cite{guth2025real}.
        \item Each connected component of $\Sigma\setminus \bm\alpha$ contains exactly one $O$-mark and one $X$-mark. The same holds for $\Sigma\setminus \bm\beta$.
        \item Connect arcs from $\bfO$ to $\bfX$ in the complement of $\bm\beta$ and push their interior into the $\beta$-handlebody. Then, connect arcs from $\bfX$ to $\bfO$ in the complement of $\bm\alpha$ and push their interior into the $\alpha$-handlebody. Performing these simultaneously and equivariantly, we get an oriented equivariant link in $(Y,\tau)$. We require it to be equivariantly isotopic to $(L,\fro)$.
        \item If $l_f$ is the number of components in $L$ that are fixed setwise by $\tau$, then $\vert \fix (R)\cap \bfO\vert=\vert \fix (R)\cap \bfX\vert= l_f$.
    \end{itemize}
    In addition to these ``classical'' requirements, we require that $(O,X)$ labelings on $C\cap L$ coming from $\cH$ coincides with that chosen in $\fra$.
    Furthermore, $\cH$ is said to be \emph{minimal} if there is exactly one pair of $(O,X)$ on each component of $L$.
\end{defn}

\begin{convention}\label{conv:setup of holomorphic structure}
Since our theory follows closely from those in \cite{guth2025real} and \cite{BGX}, we decide not to include basic definitions in real Heegaard Floer theory. We will setup some notations here and refer readers to these papers for details. 
\begin{itemize}
    \item $\mathrm{Sym}^m(\Sigma)$ will denote the $m$-fold symmetric product of $\Sigma$, on which we have an induced involution which we still denote by $R$. For a point $p\in \Sigma$ and a map $\phi$ from $D^2$ to $\mathrm{Sym}^m(\Sigma)$, we use $n_{p}(\phi)$ to denote the intersection number $\# (\im(\phi)\cap \{p\} \times \mathrm{Sym}^{m-1}(\Sigma))$, when the intersection is transverse. 
    
    \item $\T_\alpha$ and $\T_\beta$ are two Lagrangian tori (this can be achieved by using a special class of symplectic forms on $\mathrm{Sym}^m(\Sigma)$, see~\cite{guth2025real}) in $\mathrm{Sym}^m(\Sigma)$ coming from the products $\alpha_1\times \alpha_2\times \ldots\times \alpha_m$, $\beta_1\times \beta_2\times\ldots\times \beta_m$, in which $m=\vert\bm\alpha\vert=\vert\bm\beta\vert$. By definition, $R$ interchanges $\T_\alpha$ and $\T_\beta$ diffeomorphically.

    \item We always fix a generic choice of a real almost complex structure on $\mathrm{Sym}^m(\Sigma)$, so that all the index $1$ moduli spaces (see below) are cut out transversely.
   
    \item We will assume $\T_\alpha$ and $\T_\beta$ intersect transversely, so that $(\T_\alpha\cap \T_\beta)^R$ is a finite set of points. These points will be refer to as generators.
    
    \item For $\xv, \yv \in (\T_\alpha\cap \T_\beta)^R$, we use $\pi_2^R(\xv,\yv)$ to denote the set of homotopy classes of disks connecting $\xv$ to $\yv$. For $\phi\in \pi_2^R(\xv,\yv) $, $\ind_R(\phi)$ denotes the Fredholm index of $\bar{\partial}$-operator acting on the space of real sections after suitable Sobolev completion and $\mu_R(\phi)$ denotes its real Maslov index. These two indices coincide on most classes that we will consider. Moreover, $\widehat{\cM}_R(\phi)$ denotes the moduli space of real holomorphic representatives of $\phi$ quotienting the natural $\R$-translation on $\C \supset D^2-\{\pm i\}\cong [0,1]\times \R$ and $\#\widehat{\cM}_R(\phi)$ denotes the modulo $2$ count of points in it when it is of dimension zero.

    \item We will use $\rspinc(Y,\tau)$ to denote the set of real $\spinc$ structures on the real manifold $(Y,\tau)$. There is a complete characterization of this in \cite[Section~3.5-3.8]{guth2025real} and a classification theorem was proved in \cite{li2022monopolefloerhomologyreal}. As in \cite[Section~3.7]{guth2025real}, we can assign to each generator $\xv$, a real $\spinc$ structure $\s^R(\xv)$. When we do this assignment, we always use $\bfO$ as the base points of the underlying real $3$-manifold.

    \item We need some real admissibility assumptions on the diagram for the homology theories to be well-defined. In \cite[Section~3.9]{guth2025real}, the authors introduced two notions of admissibility on real Heegaard diagram. The weak admissibility works once for all, while the strong admissibility is defined for each real $\spinc$ structure.
    One can generalize the argument from \cite[Lemma~3.33-3.34]{guth2025real} to show that \begin{itemize}
        \item Fix a real $\spinc$ structure $\s^R$. Any real Heegaard diagram is equivariantly isotopic to a $\s^R$-strongly admissible one. 
        \item Any two strongly $\s^R$-admissible (weakly admissible) real Heegaard diagrams can be connected by a sequence of real Heegaard moves such that each intermediate diagram is strongly $\s^R$-admissible (weakly admissible).
    \end{itemize}
    
\end{itemize}
\end{convention}

\begin{prop}\label{prop:existence of real HD for strongly invertible knots and real H moves}
For each generalized strongly invertible link with auxiliary data $(L,\fra)$ in $(Y,\tau)$, there exists a (minimal) real Heegaard diagram representing it. Moreover, if $\cH$ and $\cH'$ both represent $(L,\fra)$, then they can be connected by a sequence of real Heegaard moves (see \cite[Definition~2.13]{BGX}) :\begin{itemize}
    \item $\{1\}$-(de)stabilization along the fixed set away from the base points.
    \item $\Z/2$-(de)stabilizations of indices (1,2) or (0,3) away from the fixed point set;
    \item real handleslide away from the base points;
    \item equivariant isotopy of $\alpha$ and $\beta$-curves away from the base points.
\end{itemize}
Moreover, for any real $\spinc$ structure $\s^R$ on $(Y,\tau)$, we can choose a strongly $\s^R$-admissible real Heegaard diagram $\cH$ representing $(L,\fra)$. If two real Heegaard diagrams representing 
$(L,\fra)$ are both strongly $\s^R$-admissible, then they can be connected by a sequence of real Heegaard moves such that each intermediate real Heegaard diagram is also strongly $\s^R$-admissible.
\end{prop}
\begin{proof}
Note that a minimal diagram of $(L,\fra)$ is exactly a real sutured Heegaard diagram for the real sutured manifold $(Y-\nu(L),\gamma,\tau)$ characterized in \cite[Example~2.5]{BGX} with sutures suitably decorated by $O$ and $X$. So the existence follows from \cite[Proposition~2.12]{BGX}.

When the manifold is $(S^3,\tau_{\text{std}})\subset (\C^2,\text{conjugation})$, there is a more constructible proof: Fixing a generic point $\infty$ on $C-A$, we can project $L$ to an equivariant copy of $\R^2\subset \R^3\subset S^3$ and obtain a transvergent diagram for $L$. From this, a multi-pointed real Heegaard diagram can be constructed easily then we can trade base points with genus to eliminate unwanted pairs of $(O,X)$ to get a minimal diagram, as Hendricks did in \cite{hendricks2025noterealheegaardfloer}. 

For the real Heegaard moves, it is easy to see that they do not affect the fact that $\cH$ represents $(L,\fra)$. It follows from \cite[Proposition~2.15]{BGX} that if two real Heegaard diagrams representing $(L,\fra)$ have the same number of base points on each component, then they can be connected by a finite sequence of real Heegaard moves without using $\Z/2$-(de)stabilization of indices $(0,3)$. On the other hand, using a finite sequence of such (de)stabilizations, we can make any two real Heegaard diagrams representing $(L,\fra)$ to share the same number of base points on each component. The discussion on admissibility follows from that in \cite[Section~3.9]{guth2025real} and the fact that $\Z/2$-(de)stabilizations of indices $(0,3)$ do not break admissibility. These together complete the proof.
\end{proof}

Fix a real $3$-manifold $(Y,\tau)$ with $\fix(\tau)$ connected and also a real $\spin^c$ structure $\s^R\in \rspinc(Y,\tau)$. 

Let $\cH=(\Sigma,\bm\alpha,\bm\beta,\bfO,\bfX, R)$ be a real Heegaard diagram representing a generalized strongly invertible link $(L,\fra)$ in $(Y,\tau)$. Assume it is strongly $\s^R$-admissible. We name the $O$ and $X$-base points so that $O^f_1\ldots O^f_{l_f}$ and $X^f_1\ldots X^f_{l_f}$ lie on $\fix(R)=\fix (\tau)$ and $(O_1,O_1'),\ldots,(O_k,O_k')$ and $(X_1,X_1'),\ldots,(X_k,X_k')$ are pairs of base points interchanged by $R$. We introduce a graded polynomial ring \[\cR(\cH)=\F[u_1,\ldots u_{l_f}, U_1,\ldots,U_k,v_1,\ldots,v_{l_f},V_1,\ldots V_k].\] The grading is given by  \[M_{\bfO}^R(u_i)=-1, \quad M_{\bfX}^R(u_i)=0, \quad \forall 1\le i\le l_f;\quad M_{\bfO}^R(U_j)=-2, \quad M_{\bfX}^R(U_j)=0;\quad \forall 1\le j\le k; \]
\[M_{\bfO}^R(v_i)=0, \quad M_{\bfX}^R(v_i)=-1, \quad \forall 1\le i\le l_f;\quad M_{\bfO}^R(V)=0, \quad M_{\bfX}^R(V)=-2, \quad \forall 1\le j\le k. \] Then let \[A^R(-)=\frac{1}{2}(M_{\bfO}^R(-)-M_{\bfX}^R(-)).\]
Throughout this paper, when we talk about bigrading, we mean $(M^R_{\bfO},A^R)$ and we will often omit $\bfO$ from the notation.

We define the \emph{(minus) full real link Floer complex} $\fullCFLR^-(\cH,\s^R)$ to be the module over $\cR(\cH)$ with generators \[\{\xv\in (\T_\alpha\cap \T_\beta)^R| \s^R(\xv)=\s^R \}.\] We define an endomorphism on it by \[\partial \xv =\sum_{\yv} \sum_{\phi\in \pi_2^R(\xv,\yv),\mu_R(\phi)=1}\# \widehat{\cM}_R(\phi) \prod_{1\le i\le l_f} v_i^{n_{X_i^f}(\phi)} u_i^{n_{O_i^f}(\phi)}\prod_{1\le j\le k} V_j^{n_{X_j}(\phi)} U_j^{n_{O_j}(\phi)}\yv\]
on generators and extend linearly over the base ring.
A prior, the summation above is not necessarily finite. Under the strongly $\s^R$-admissible assumption, the sum is indeed finite using the argument from \cite[Lemma~3.32]{guth2025real}.

Note also that $\partial^2$ is in general not zero, but we can calculate the curvature explicitly as in \cite{guth2025real}. To make the expression clear, we need to introduce some new notations. If $K\subset L$ is a strongly invertible knot, then we label the base points so that the two fixed base points on $K$ are $O_1$ and $X_n$ and along the arc from $O_1$ to $X_n$, the base points are \[O_1, X_1, \ldots, X_{n-1}, O_{n}, X_{n}\] while along the arc from $X_n$ to $O_1$, the base points are \[X_n,O_n',X_{n-1}', \ldots,X_1',O_1.\]  Now, we define \[\omega_{K}=u_1^2\cdot V_1+ U_2\cdot V_1+ U_2\cdot V_2+\ldots U_{n}\cdot V_{n-1}+ U_{n}\cdot v_n^2.\] If $(K,K')\subset L$ is a pair of components interchanged by $\tau$, then we label the base points on $K$ so that it reads \[O_1, X_1,\ldots,O_n,X_n\] in it's orientation. Define  \[\omega_{K,K'}=U_1\cdot V_1+ U_2\cdot V_1+ U_2\cdot V_2+\ldots U_{n}\cdot V_{n-1}+ U_{n}\cdot V_n+ V_n\cdot U_1.\] Using these, we can write down the curvature as \[\partial^2= (\sum_{K\subset L, K \text{ strongly invertible}} \omega_K+\sum_{(K,K')\subset L, \text{ }\tau(K)=K'}\omega_{K,K'})\cdot \id\] if $L$ is a genuine link or $L$ is a knot and have more than one pair of base points. On the other hand, if $K$ is a strongly invertible knot and the real Heegaard diagram is minimal, then $\partial^2=0$.


When the diagram is minimal, the homotopy type of $\fullCFLR^-(\cH,\s^R)$ is an invariant of $(L,\fra)$ and $\s^R$, which we shall denote by $\fullCFLR^-(L,\fra,\s^R)$. This is a generalization of the full link Floer complex $\mathcal{CFL}^\circ$ from \cite{zemke2019link}. As seen in Proposition~\ref{prop:existence of real HD for strongly invertible knots and real H moves}, any multi-based diagram is related to a minimal one by some real Heegaard moves. When $\cH$ is any diagram representing $(L,\fra)$ that is strongly $\s^R$-admissible, by setting all the $U$ and $V$ variables belonging to same component of $L$ to be equal, $\fullCFLR^-(\cH,\s^R)$ is up to homotopy an algebraic stabilization of $\fullCFLR^-(L,\fra,\s^R)$. Here, by an algebraic stabilization, we mean taking tensor product with $(\F^2)^{\otimes s}$, where $s$ is roughly the number of extra pairs of $O$ and $X$-base points. 

For a strongly invertible link with auxiliary data $(L,\fra)$ in $(Y,\tau)$, fix a real $\spinc$ structure $\s^R$ as above and pick a real Heegaard diagram $\cH$ representing $(L,\fra)$ that is strongly $\s^R$-admissible. Then, we define \[\CFLR^-(\cH,\s^R)=\fullCFLR^-(\cH,\s^R)/ (v_i,V_j)_{1\le i\le l_f,1\le j\le k}\] which is a genuine chain complex and let $\partial^-$ denote the induced differential on it. Then, the homology 
\[\HFLR^-(\cH,\s^R)=H_*(\CFLR^-(\cH,\s^R),\partial^-)\] 
is an invariant of $(L,\fra)$ and $\s^R$ when it is regarded as an $\F[u_1,\ldots,u_{l_f}]$-module. We will denote it by $\HFLR^-(L,\fra,\s^R)$ and call it \emph{minus version of real link Floer homology} associated to $(L,\fra)$ in the real $\spinc$ structure $\s^R$. We can further introduce the chain complex \[\widehat{\CFLR}(\cH,\s^R)=\CFLR^-(\cH,\s^R)/(u_i)_{1\le i\le l_f}\] with induced differential $\widehat{\partial}$. Then 
\[\widehat{\HFLR}(\cH,\s^R)=H_{*}(\widehat{\CFLR}(\cH,\s^R),\widehat{\partial})\] 
is an invariant of $(L,\fra)$ and $\s^R$ as a vector space over $\F$. We will denote it by $\widehat{\HFLR}(L,\fra,\s^R)$ and call it \emph{hat version of real link Floer homology} associated to $(L,\fra)$ in the real $\spinc$ structure $\s^R$. A further variation is \[\widetilde{\CFLR}(\cH,\s^R)=\CFLR^-(\cH,\s^R)/(u_i,U_j)_{1\le i\le l_f,1\le j\le k},\] the homology $\widetilde{\HFLR}(\cH,\s^R)$ is not an invariant of $(L,\fra)$ and $\s^R$, but it appears as a suitable stabilization of $\widehat{\HFLR}(L,\fra,\s^R)$ according to the size of the real Heegaard diagram.

More generally, when $L$ is a generalized strongly invertible link, we let $l_p$ be the number of pairs of components interchanged by $\tau$. (We then have $\vert L\vert= 2l_p+l_f$.) Without loss of generality, we assume that $O_1, \ldots, O_{l_p}$ lie on distinct pairs of components. With these notations prepared, we define 
\[\CFLR^-(\cH,\s^R)=\fullCFLR^-(\cH,\s^R)/ (v_i,V_j)_{1\le i\le l_f,1\le j\le k},\] 
\[\widehat{\CFLR}(\cH,\s^R)=\CFLR^-(\cH,\s^R)/(u_i,U_j)_{1\le i\le l_f,1\le j\le l_p},\] and 
\[\widetilde{\CFLR}(\cH,\s^R)=\CFLR^-(\cH,\s^R)/(u_i,U_j)_{1\le i\le l_f,1\le j\le k}.\] 
The homology 
\[\HFLR^-(\cH,\s^R)=H_*(\CFLR^-(\cH,\s^R),\partial^-)\] is an invariant of $(L,\fra)$ and $\s^R$, when it is regarded as an $\F[u_1,\ldots,u_{l_f},U_1,\ldots,U_{l_p}]$-module, while \[\widehat{\HFLR}(\cH,\s^R)=H_{*}(\widehat{\CFLR}(\cH,\s^R),\widehat{\partial})\] is an invariant of $(L,\fra)$ and $\s^R$ as a vector space over $\F$. We name them in the same fashion as for genuine strongly invertible links. Again, the homology $\widetilde{\HFLR}(\cH,\s^R)$ is not an invariant of $(L,\fra)$ and $\s^R$, but it appears as a suitable stabilization of $\widehat{\HFLR}(L,\fra,\s^R)$ according to $k-l_p$.

When $\s^R$ has torsion first Chern class, we can define a relative real Maslov grading $M^R$ on the generators by \[M^R_{\bfO}(\xv,\yv)=\mu_R(\phi)-n_{\bfO}(\phi)\]
for any $\phi\in \pi^R_2(\xv,\yv)$. The torsion assumption makes sure that this is well-defined independent of the choice of $\phi$. Further assume $L$ is null-homologous, then we can define \[M^R_{\bfX}(\xv,\yv)=\mu_R(\phi)-n_{\bfX}(\phi)\] and introduce the relative real Alexander grading as their difference \[A^R(\xv,\yv)= \frac{1}{2} (M^R_{\bfO}(\xv,\yv)-M^R_{\bfX}(\xv,\yv)).\] The null-homologous and  torsion assumptions make sure that this is also well-defined. In general, if $c_1(\s^R)$ is non-torsion with divisibility $2n$, we can define a $\Z/n\Z$-valued real Maslov grading on $\CFLR^\circ$ using the same formula.

Taking direct sum over all real $\spinc$ structures, we define \[\HFLR^-(L,\fra)=\bigoplus_{\s^R\in \rspinc(Y,\tau)} \HFLR^-(L,\fra,\s^R) \quad\text{ and } \quad\widehat{\HFLR}(L,\fra)=\bigoplus_{\s^R\in \rspinc(Y,\tau)} \widehat{\HFLR}(L,\fra,\s^R), \] which are invariants of $(L,\fra)$. 

We have now finished the construction. To prove (1) of Theorem~\ref{thm:intro-real link Floer homology}, it remains to check the invariance.

The invariance almost follows from the complete set of real Heegaard moves for sutured manifold introduced in \cite{BGX} and the proof of invariance for closed 3-manifolds in \cite{guth2025real}. The only move that hasn't been dealt with in previous works is that real $(0,3)$ stabilization, which introduces two new pairs of base points. This is exactly two  ``quasi-stabilization'' operations defined in \cite{Zemkequasistabandbasepointmoving} being performed in a symmetrical way. The argument of Proposition~5.3 in that paper applies directly to our case and shows that using suitable real Heegaard diagrams $(\cH,\cH^+)$ that are strongly $\s^R$-admissible and a sufficiently stretched real almost complex structure on the Heegaard surface, the endomorphisms on the full real link Floer complexes are related by \[\partial_{\cH^+}=
\begin{bmatrix}\partial_{\cH}& U_{O}+U_{O_a} \\
V_{X}+V_{X_a} &\partial_{\cH}\\
\end{bmatrix},\] in which $U_O$ and $V_X$ are the variables associated to the new pairs of base points $(O,O')$ and $(X,X')$ and $U_{O_a}$ and $V_{X_a}$ are the variables associated to the pairs of base points adjacent to the new base points. $U_{O_a}$ denotes $u_{o_a}^2$ if $O_a$ lies on the fixed, and similarly for $V_{X_a}$. Set all $V$ variables to zero, we see that \[\CFLR^-(\cH^+,\s^R)\cong \Cone(U_{O}+U_{O_a}\colon \CFLR^-(\cH,\s^R)[U_O]\to \CFLR^-(\cH,\s^R)[U_O]).\] Then the invariance of minus theory follows from the following lemma. 

\begin{lem} (A generalization of \cite[Lemma~5.2.16]{OSS2015grid})
Let $C$ be a (bigraded) chain complex over $\F[U_2,\ldots,U_k,u_1,\ldots,u_f]$. Then there is a quasi-isomorphism between (bigraded) $\F[U_2,\ldots,U_k,u_1,\ldots,u_f]$ modules \[H(\Cone(U_1-U_2\colon C[U_1]\to C[U_1]))\cong H(C).\]
\end{lem}

The invariance of hat theory can be proved similarly. One can refer to \cite[Proposition~2.3]{MOS2009knot} for a detailed argument, which asserts $(0,3)$-stabilization invariance for the usual knot Floer homology and still works in our case. 

Currently, we can only assert that the isomorphism class of the group or module is a link invariant. These invariants can be made natural once the naturality of real Heegaard theory is established. This issue has been addressed in more recent works, see \cite{GM_real_naturality} and \cite{YXHFLR2}.

\begin{convention}
Above, we mainly focused on minus version of full real link Floer complex and its derivatives and used the notations like $\fullCFLR^-$, $\CFLR^-$, etc. In fact, we can consider an infinity counterpart by inverting variables in the base ring. In the following, we will use the terminology \emph{real knot Floer homology} and notations like $\CFKR^-$ when we want to stress the object is a strongly invertible knot.
\end{convention}

\subsection{Doubly periodic link}
\begin{defn}\label{def:periodic links}
    Let $(Y,\tau)$ be a closed oriented real 3-manifold and $L$ be a link in $Y$.
    \begin{itemize}
        \item $L$ is called \emph{doubly periodic} if for any choice of orientation on $L$, $\tau$ restricts to an orientation-preserving involution on each component of $L$.
        \item $L$ is called \emph{generalized doubly periodic} if there is a choice of orientation on $L$, so that $\tau$ restricts to an orientation-preserving involution on $L$. Here, we allow pairs of components to be interchanged in an orientation-preserving way.
    \end{itemize} 
\end{defn}

\begin{defn}\label{def:real Heegaard diagram for periodic links}
Let $(Y,\tau)$ be a closed real $3$-manifold with connected fixed set and $L$ be a (generalized) doubly periodic link in $(Y,\tau)$. We fix a compatible orientation $\fro$ on $L$. We say a real Heegaard diagram $\cH=(\Sigma,\bm\alpha,\bm\beta,\bfO,\bfX, R)$ represents $(L,\fro)$ if the following holds. 
    \begin{itemize}
        \item $\cH=(\Sigma-\nu(\bfO)-\nu(\bfX),\bm\alpha,\bm\beta, R)$ is a real sutured Heegaard diagram for $Y-\nu(L)$ with meridian sutures (one positively oriented meridian for each $O$ and one negatively oriented meridian for each $X$) as defined in \cite{BGX}.
        \item $R$ interchanges $\bfO$ and $\bfX$.
        \item Each connected component of $\Sigma\setminus \bm\alpha$ contains exactly one $O$-mark and one $X$-mark. The same holds for $\Sigma\setminus \bm\beta$.
        \item Connect arcs from $\bfO$ to $\bfX$ in the complement of $\bm\beta$ and push their interior into the $\beta$-handlebody. Then, connect arcs from $\bfX$ to $\bfO$ in the complement of $\bm\alpha$ and push their interior into the $\alpha$-handlebody. Performing these simultaneously and equivariantly, we get an oriented equivariant link in $(Y,\tau)$. We require it to be equivariantly isotopic to $(L,\fro)$.
    \end{itemize}
\end{defn}

The existence of real Heegaard diagrams and the complete set of real Heegaard moves are characterized and proved in the same way as in the strongly invertible case. We still setup notations and holomorphic structures as in Convention~\ref{conv:setup of holomorphic structure}. Now, we can proceed to define real link Floer homology for doubly periodic links.

Again, we fix a closed real $3$-manifold $(Y,\tau)$ with connected fixed set and a real $\spinc$ structure $\s^R\in \rspinc(Y,\tau)$. Let $(L,\fro)$ be a (generalized) doubly periodic link in $(Y,\tau)$ and $\cH$ be a strongly $\s^R$-admissible real Heegaard diagram for $(L,\fro)$. We label the base points in pairs as $(X_1,O_1),\ldots (X_k,O_k)$ so that $R(O_i)=X_i$ for all $i$ and define $\cR(\cH)=\F[V_1,\ldots,V_k]$ with $M^R_{\bfO}(V_i)= M^R_{\bfX}(V_i)=-1$. Then let the \emph{full real link Floer complex} $\fullCFLR^-(\cH,\s^R)$ be the module over $\cR(\cH)$ generated by \[\{\xv\in (\T_\alpha\cap \T_\beta)^R| \s^R(\xv)=\s^R \},\]   with an endomorphism \[\partial \xv =\sum_{\yv} \sum_{\phi\in \pi_2^R(\xv,\yv),\mu_R(\phi)=1}\# \widehat{\cM}_R(\phi) \prod_{1\le i\le k} V_i^{n_{X_i}(\phi)} \yv=\sum_{\yv} \sum_{\phi\in \pi_2^R(\xv,\yv),\mu_R(\phi)=1}\# \widehat{\cM}_R(\phi) \prod_{1\le i\le k} V_i^{n_{O_i}(\phi)} \yv,\] in which the second equality follows from the symmetry assumption on $\bfO$ and $\bfX$. The strongly admissible assumption ensures that the sum above is finite. Again, $\partial^2$ may not be zero, but we can calculate it explicitly. For a periodic component $K\subset L$, we assume the labels read \[O_1, X_2, O_3, X_4,\ldots O_n, X_1, O_{2}, X_{3},\ldots O_{n-1}, X_n\] along the orientation of $K$ and $\tau(O_i)=X_i$ for each $i$. Let \[\omega_{K}^p= V_1V_2+ V_2V_3+\ldots V_{n-1}V_{n}+V_1V_n.\]  

An interesting observation is the following, which implies that killing the curvature $\omega_{K}^p$ is difficult even after modulo $2$.
\begin{lem}\label{lem:odd number of pairs of base points on a periodic knot}
If $\cH$ is a real Heegaard diagram representing a doubly periodic knot in $(Y,\tau)$, then the number $n$ of base point pairs must be odd.  
\end{lem} 
\begin{proof}
Suppose for contradiction that we have a real Heegaard diagram $\cH$ for a doubly period knot with an even number of pairs of $(O,X)$-base points. Then we use genus to trade handles as in \cite{hendricks2025noterealheegaardfloer} until we obtain a diagram $\cH'$ for the same knot with exactly two pairs of base points. This is possible since each time, we eliminate two pairs of base points. Label the two pairs by $(O_1,X_1)$, $(O_2,X_2)$ and let $(A_1,A_2)$ ($(B_1,B_2)$) be the two $\alpha$ ($\beta$)-periodic domains on $\cH'$, so that $\tau(A_i)=B_i$. Since $\cH$ represents a knot, $O_1$ and $X_1$ must share one of the four domains. Without loss of generality, we may assume $A_1$ contains $O_1$, $X_1$, but now they also share $B_1$, which contradicts the assumption that $\cH'$ represents a periodic knot.
\end{proof}
For a pair of components $(K,K') \subset L$, we assume the labels read \[O_1, X_2, O_3, X_4,\ldots O_{n-1}, X_{n}\] along the orientation of $K$ and $\tau(O_i)=X_i$ for each $i$. Let \[\omega_{K,K'}^p= V_1V_2+ V_2V_3+\ldots V_{n-1}V_{n}+V_1V_n.\] Then,  \[\partial^2= (\sum_{K\subset L, K \text{ periodic}} \omega^p_K+\sum_{(K,K')\subset L, \text{ }\tau(K)=K'}\omega^p_{K,K'}) \cdot\id\] if $L$ is a genuine link or $L$ is a knot and has more than one pair of base points. On the other hand, if $K$ is a doubly periodic knot and the diagram is minimal, then $\partial^2=0$.

When diagram is minimal, we assume that $L$ has $l_f$ components fixed setwise by $\tau$ and $l_p$ pairs of components interchanged by $\tau$. 
Then, we consider $\widehat{\CFLR}(\cH,\s^R)=\fullCFLR^-(\cH,\s^R)/(V_i)_{1\le i\le l_p+l_f}$, which is a $\F$-vector space. Of course, now the curvature becomes zero, so we can take homology \[\widehat{\HFLR}(\cH,\s^R)=H_{*}(\widehat{\CFLR}(\cH,\s^R),\widehat{\partial}),\] which is an invariant of $(L,\fro)$ and $\s^R$ as a vector space over $\F$.  We denote it by $\widehat{\HFLR}(L,\fro,\s^R)$ and call it the \emph{hat version of real link Floer homology for the generalized doubly periodic link} in the real $\spinc$ structure $\s^R$. When the diagram is not minimal, we can define a tilde version of homology, which appears as a stabilization of the hat version. 

When $\s^R$ is torsion, we can define a real Maslov grading in exactly the same way for strongly invertible links. But the real Alexander grading no longer makes sense. 

Taking direct sum over all real $\spinc$ structures, we define \[\widehat{\HFLR}(L,\fro)=\bigoplus_{\s^R\in \rspinc(Y,\tau)} \widehat{\HFLR}(L,\fro,\s^R), \] which is the invariant of $(L,\fro)$ promised in (2) of Theorem~\ref{thm:intro-real link Floer homology}. The invariance can be proved in exactly the same way as for strongly invertible links.


\section{A combinatorial description}\label{sec:A combinatorial description}
In this section, we give a combinatorial description for the real link Floer homologies introduced in Section~\ref{sec:Real knot Floer chain complex and real knot Floer homology} following the strategy in \cite{OSS2015grid}. We mainly focus on strongly invertible case, the periodic theory can be reformulated in exactly the same way.

It is well-known that $S^3$ admits a unique real structure up to isotopy and any strongly invertible knot in $S^3$ is ``Sakuma equivalent'' (see~\cite{Sakuma}) to a strongly invertible knot in $S^3$ equipped with the standard involution. So from now on, we will only consider strongly invertible knots in $(S^3,\tau)\subset (\C^2, \text{conjugation})$. For notational simplicity, we omit the subscript $\std$ from $\tau$. This involution admits a standard genus $1$ real Heegaard diagram as shown in Figure~\ref{fig:A real Heegaard diagram for S3}.   

\begin{figure}
    \centering
    \includegraphics[width=0.5\linewidth]{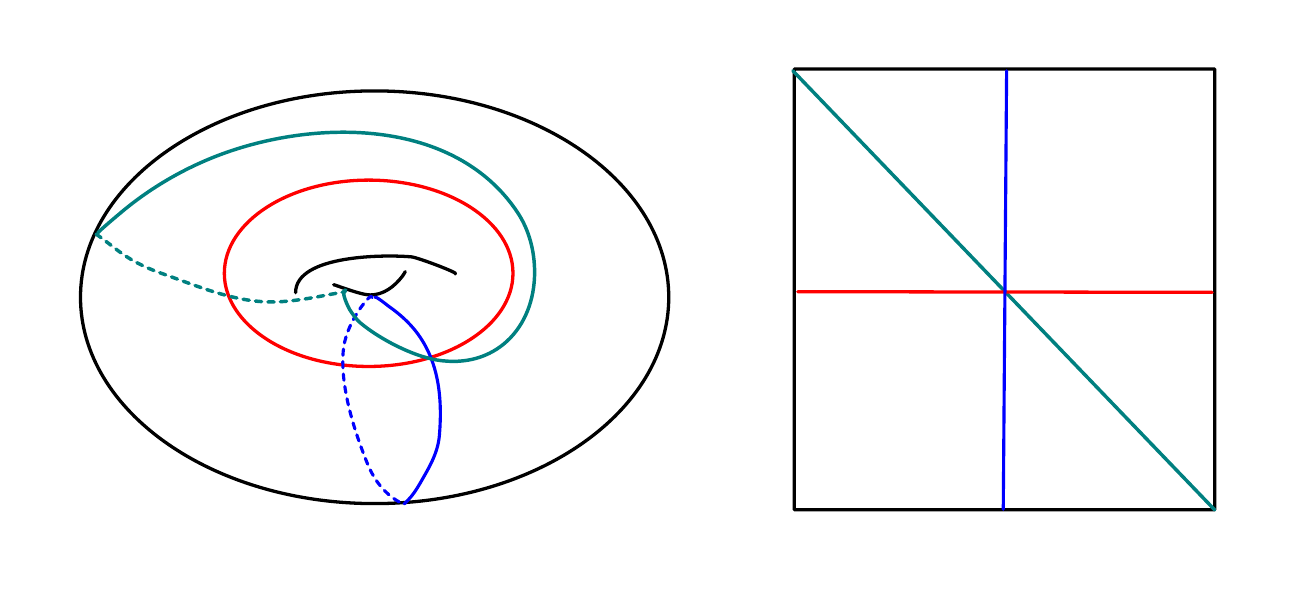}
    \caption{A genus one real Heegaard splitting of $S^3$.}
    \label{fig:A real Heegaard diagram for S3}
\end{figure}

\subsection{Real grid homology}\label{sub:Real grid homology}
\begin{defn}\label{def:real grid diagram}
Let $(L,\fra)$ be an equivariant link with auxiliary in $(S^3,\tau)$. A \emph{real grid diagram} for $(L,\fra)$ is a real Heegaard diagram $\cH=(\Sigma,\bm\alpha,\bm\beta,\bfO,\bfX,R)$ such that \begin{itemize}
    \item $\Sigma=T^2$ is a torus;
    \item $\bm\alpha$ and $\bm\beta$ are parallel copies of meridians and longitudes in $T^2$;
    \item $R$ is the reflection of $T^2$ across the circle of slope $-1$.
\end{itemize}
Such a real grid diagram is of \emph{size $n$} if it has $n$ pairs of $\alpha$ and $\beta$-curves.
\end{defn}

Choose any pair of $\alpha$, $\beta$-curves interchanged by $R$. Then we can cut $T^2$ open and get a \emph{real planar grid diagram} for $L$. Connecting $O$, $X$-base points as we described in Definition~\ref{def:real Heegaard diagram for strongly invertible links} and~\ref{def:real Heegaard diagram for periodic links}, we get a transvergent diagram for $L$. For a generalized strongly invertible link with auxiliary data, this choice of cut really matters. More precisely, when projecting a link to a plane, we are making a choice of $\infty$ point which is closely related to choice (4) in Definition~\ref{def:strongly invertible knots with extra data}. A shifting of $\infty$ leads to a change of $A$, such $(L,\fra)$ and $(L,\fra')$ share the same real grid diagram, but have different preferred cuts. In terms of real grid moves, they are related by a real cyclic permutation that happens to move the row and column with an $O$ or $X$ on $C$. See Figure~\ref{fig:shifting A} for an intuition and see Figure~\ref{fig:grid diagrams for 8-19,8-20 and 9-42} for a concrete example. In Figure~\ref{fig:shifting A}, $A$ is marked in solid line while $C-A$ is drawn in dashed line. These two choices of $\fra$ are equivalent in the sense we remarked after Definition~\ref{def:strongly invertible knots with extra data}.

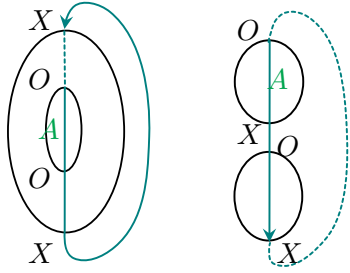
\begin{figure}
\begin{overpic}[width=0.4\textwidth]{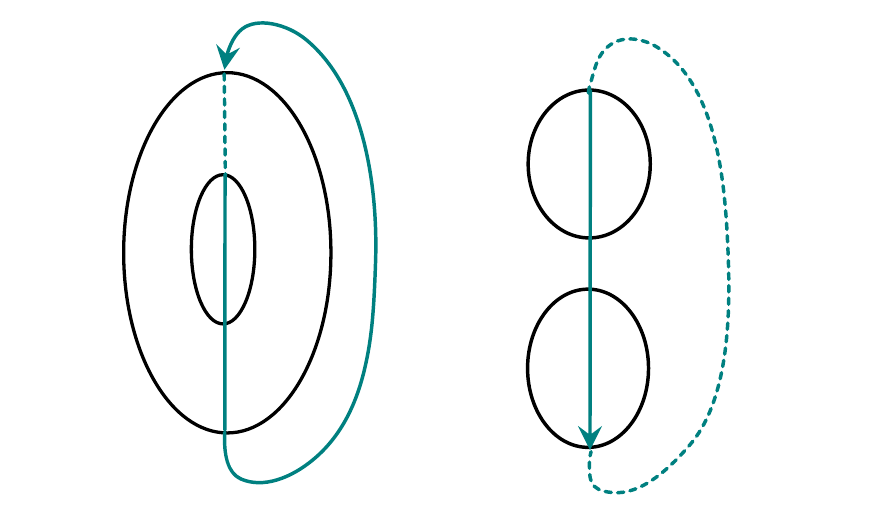}
			\put(18,5) {$X$}
			\put(18,20) {$O$}
            \put(18,52) {$X$}
			\put(18,40) {$O$}
            \put(68,5) {$X$}
			\put(68,26.5) {$O$}
            \put(60,50) {$O$}
			\put(60,29) {$X$}
            \put(20,30) {\textcolor{Green}{$A$}}
            \put(66,40) {\textcolor{Green}{$A$}}
		\end{overpic}
    \centering
    \caption{Choice of axis: in each frame, $A$ is marked by a solid arrow and $C-A$ is marked by a dotted arc.}
    \label{fig:shifting A}
\end{figure}

It was shown in \cite[Theorem~2.10]{borodzik2025khovanovhomologyequivariantsurfaces} that (see also \cite[Theorem~2.3]{Lobb2021ArefinementofKhovanovhomology} and see \cite{BDMS26} for a proof.)
\begin{prop} 
If $D_1$ and $D_2$ are transvergent diagrams for equivariantly isotopic links in $S^3$, then they are related by a finite sequence of involutive Reidemeister moves shown in Figure~\ref{fig:equivariant R moves} and the $I$-move shown in Figure~\ref{fig:I-move}.
\end{prop}

\begin{figure}
    \centering
    \includegraphics[width=1.0\linewidth]{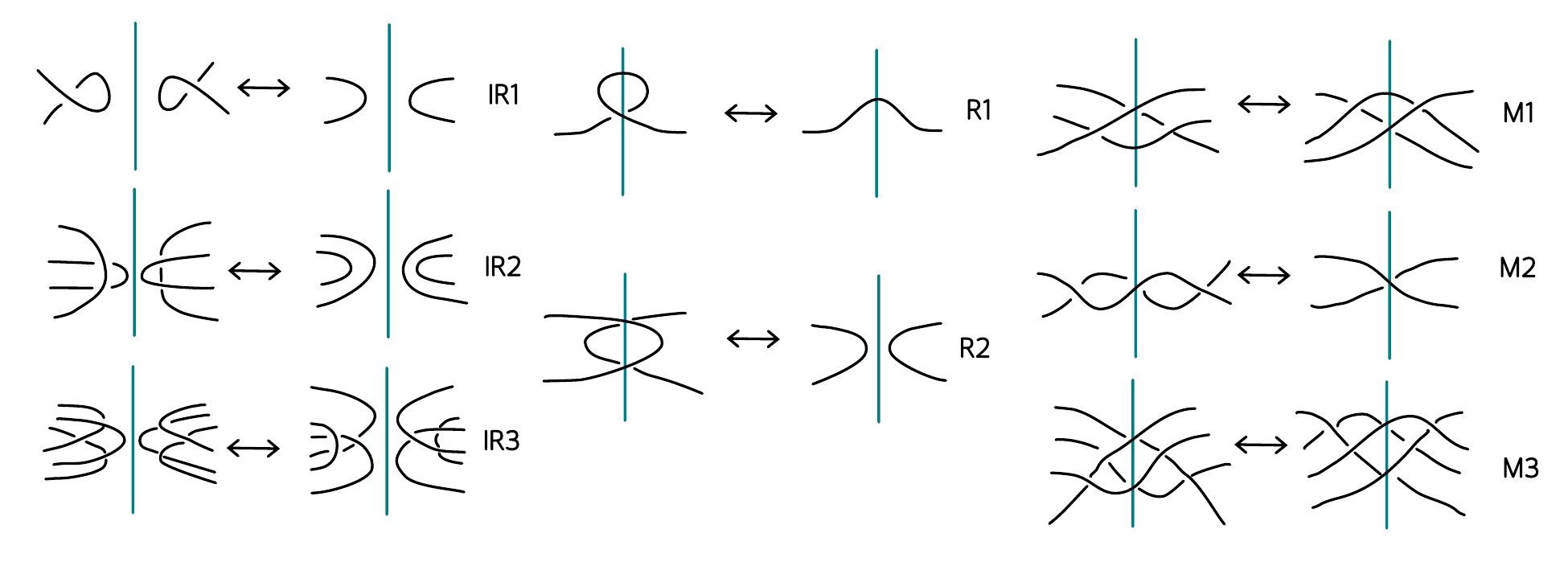}
    \caption{Involutive Reidemeister moves.}
    \label{fig:equivariant R moves}
\end{figure}

\begin{figure}
    \centering
    \includegraphics[width=0.4\linewidth]{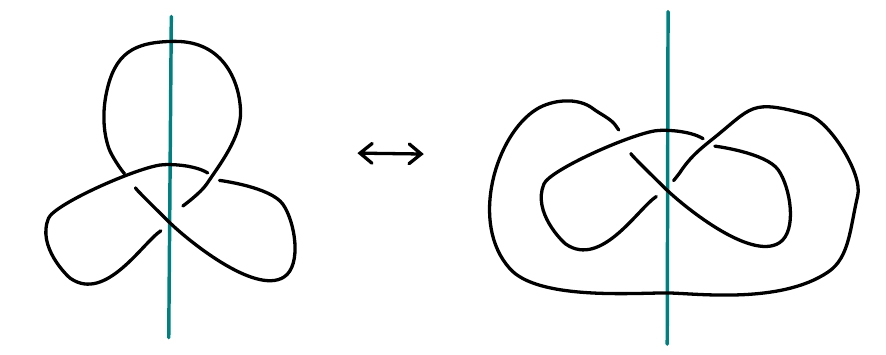}
    \caption{An example of an I-move.}
    \label{fig:I-move}
\end{figure}

Following this, we can show that 
\begin{prop}\label{prop:real grid moves}
If $\cH_1$ and $\cH_2$ are two real grid diagrams representing the same link with auxiliary data $(L,\fra)$, then they can be connected by a finite sequence of the following \emph{real grid moves}. 
\begin{itemize}
    \item real cyclic permutation;
    \item real stabilization;
    \item real destabilization;
    \item real commutation.
\end{itemize}
\end{prop}
The definition of real grid moves follows immediately from the usual ones. More precisely, we perform a usual move with its reflection under $R$ simultaneously. This is easy to imagine, so we decide not to provide details on this. The only case we need to remark is that we do not allow the stabilization to happen at any mark point on $C$. Instead, we realize a $RI$ move as a real commutation (Here, we include ``switch'' (see \cite[Definition~3.1.10]{OSS2015grid}) as a special commutation move), see Figure~\ref{fig:RI move as real commutation} for an example. 

\begin{figure}
    \centering
    \includegraphics[width=0.45\linewidth]{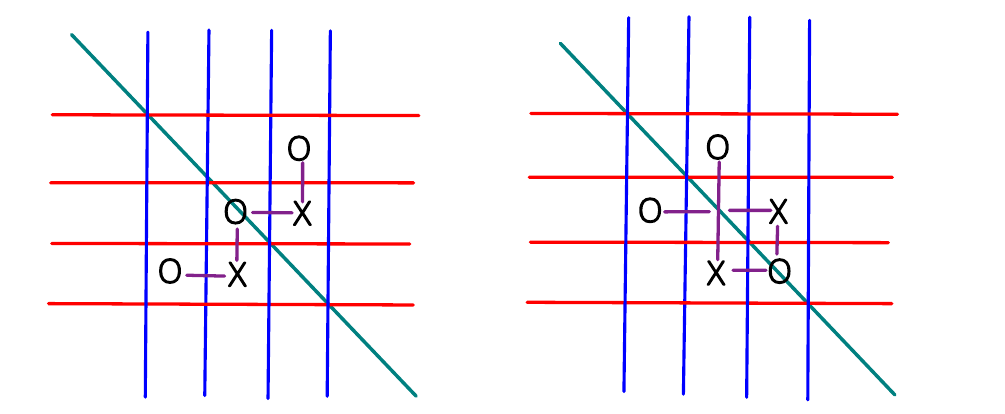}
    \caption{Realizing RI move as a real commutation.}
    \label{fig:RI move as real commutation}
\end{figure}

\begin{proof}
Note that planar diagrams obtained from different choices of cut on $T^2$ are related to real cyclic permutations on planar grid diagram. Then one can check that each of the moves listed in Figure~\ref{fig:equivariant R moves} can be realized by a finite sequence of real (de)stabilizations and real commutations as in \cite[Appendix~B.4]{OSS2015grid}.
\end{proof}

\begin{rem}
It seems unreasonable to prohibit (de)stabilizations at the first glance, but actually, it is important: such a move will break the assumption of having a single pair of $(O,X)$-base points on the fixed set for each strongly invertible knot component. The diagram still represents the equivariant link $L$, but no longer represents $(L,\fra)$. As we mentioned in Remark~\ref{rmk:importance of auxiliary data}, this type of diagrams will lead to a different theory which is also interesting.
\end{rem}

By definition, real grid diagrams are special real Heegaard diagrams, so they are valid for computation of real link Floer homologies. They have a crucial and useful property that they are \emph{real nice Heeggard diagrams} defined in \cite{BGX} and \cite{LOrealbordered}, so the differentials can be computed combinatorially. It follows from \cite[Section~8]{BGX} or \cite[Section~5]{LOrealbordered} that

\begin{prop}\label{prop:domains of real index 1}
Let $\cH=(T^2,\bm\alpha,\bm\beta,\bfO,\bfX,R)$ be a real grid diagram for some equivariant link $L$ in $(S^3,\tau)$. Any real domain $\cD$ of real Maslov index 1 takes one of the following forms. 
\begin{enumerate}
    \item $\cD$ is an embedded square that is fixed by $R$ as a set. $R|_{\cD}$ is the reflection across its diagonal.
    \item $\cD$ is an embedded $L$-shaped hexagon that is fixed by $R$ as a set. $R|_{\cD}$ is a reflection that interchanges two arms of $L$.
    \item $\cD$ is an embedded $X$-shaped octagon that is fixed by $\tau$ as a set. $R|_{\cD}$ is a reflection across one of the obvious symmetry axis of $X$.
    \item $\cD$ is the union of two (not necessarily disjoint) rectangles with no vertex on $C$ and $R$ interchanges them.
\end{enumerate}

Examples of such domains are shown in Figure~\ref{fig:real rectangles}. We call a real domain in the list above a \emph{real rectangle}, since they play the roles of rectangles in the usual grid homology. For a pair of generators (in a grid diagram, we also call them \emph{real grid states}) $\xv$ and $\yv$ in $(\T_\alpha\cap \T_\beta)^R$, we denote the set of real rectangles from $\xv$ to $\yv$ by $\RRect(\xv,\yv)$. For a real rectangle $r$ in $\RRect(\xv,\yv)$ to have real Maslov index $1$, it needs to be \emph{empty}, i.e., $\mathrm{int}(r)\cap \xv=\mathrm{int}(r)\cap \yv=\emptyset$. We refer to $\RRecto(\xv,\yv)$ as the set of empty real rectangles from $\xv$ to $\yv$.
\begin{figure}
    \centering
    \includegraphics[width=0.6\linewidth]{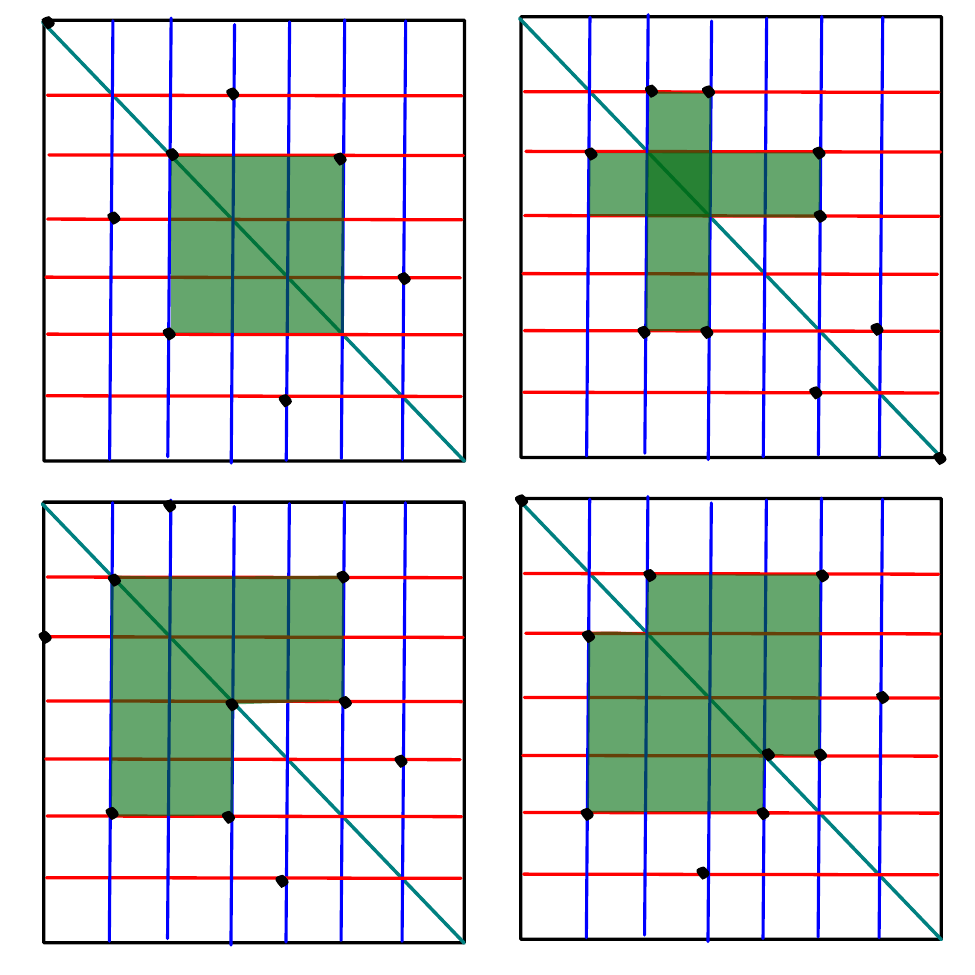}
    \caption{Example of real rectangles.}
    \label{fig:real rectangles}
\end{figure}
\end{prop} 

\begin{rem}
The usual grid homology allows direct generalization to lens spaces, which are the only spaces admitting toroidal Heegaard diagrams other than $S^3$ and $S^1\times S^2$. However, one must be careful when considering generalization of real homology to real lens spaces, due to the following reasons. \begin{itemize}
    \item Involutions on lens spaces have been classified in \cite{Involutionsandisotopiesoflensspaces}. Firstly, the involution is far from unique. Moreover, it was mentioned in \cite{Realopenbooksandrealcontactstructures} that some of these real structures do not admit genus $1$ real Heegaard splitting. 
    \item Even after fixing a real lens spaces $(L(p,q),\tau_0)$ that admits a genus $1$ real Heegaard splitting, the domains listed in are not enough to defined the differential- there are nontrivial real domains with nontrivial topology (annular and toroidal) and real index $1$ (See~\cite[Section~8.2]{BGX}).
\end{itemize} 
\end{rem}

Using this proposition, we can redefine the real link Floer homologies combinatorially as follows. Let $\cH$ be a real grid diagram representing  $(L,\fra)$, a strongly invertible link with auxiliary data. The $O$ and $X$-base points are labeled as in Section~\ref{sub:Strongly invertible links}. Let $\cR^{G}(\cH)$ be the ring $\F[u_1,\ldots,u_{l_f}, U_1,\ldots, U_k]$, where as usual, $k$ is the number of pairs of $O$-base points and the $u_i$'s are of degree $(-1,-1/2)$ while $U_j$'s are of degree $(-2,-1)$. We define 
\[\GCR^-(\cH)=\cR^G(\cH) \left \langle (\T_\alpha\cap \T_\beta)^R \right \rangle,\] equipped with differential 
\[\partial^{G,-}\xv=\sum_{\yv}\sum_{r\in \RRecto(\xv,\yv),r\cap \bfX=\emptyset} \prod_{1\le i\le l_f} u_{i}^{n_{O_i^f}(r)}\prod_{1\le j\le k} U_{j}^{n_{O_j}(r)} \yv.\]  
Taking homology, we get \[\GHR^-(\cH)=H_*(\GCR^-(\cH),\partial^{G,-}).\]
Regarding it as a module over $\F[u_1,\ldots,u_{l_f}]$, we get an invariant of $(L,\fra)$, denoted by $\GHR^-(L,\fra)$. Similarly, we can define 
\[\widehat{\GCR}(\cH)=\GCR^-(\cH)/(u_i)_{1\le i\le l_f},\] and take homology \[\widehat{\GHR}(\cH)=H_*(\widehat{\GCR}(\cH),\widehat{\partial^G}),\] which is a $\F$-vector space valued invariants associated to $(L,\fra)$. $\widetilde{\GCR}$ can be characterized similarly and $\widetilde{\GHR}(\cH)$ appears as a stabilization of $\widehat{\GHR}(L,\fra)$ as usual.

Actually, $\widetilde{\GHR}$ admits a more straightforward description, which is useful for computing $\widehat{\GHR}(L,\fra)$ in practice. In fact, we can consider \[\widetilde{\GCR}(\cH)=\F\left \langle (\T_\alpha\cap \T_\beta)^R \right \rangle\] equipped with differential \[\widetilde{\partial^{G}}\xv =\sum_{\yv}\sum_{r\in \RRecto(\xv,\yv),r\cap \bfX=\emptyset, r\cap\bfO=\emptyset} \yv.\] Then $H_*(\widetilde{\GCR}(\cH),\widetilde{\partial^{G}})$ is the same as $H_*(\GCR^-(\cH)/(u_i,U_j),\text{induced differential})$.

For the convenience of computer computation, we introduce a notion of minus-tilde real grid homology $\widetilde{\GHR^-}(\cH)$ when $\cH$ describes a strongly invertible knot. (A similar description works for strongly invertible links, but we restrict ourselves to knots for notational convenience.) As a module, $\widetilde{\GCR^-}(\cH)$ is generated by $(\T_\alpha\cap \T_\beta)^R$ over $\F[u]$, the polynomial ring with a single variable. The differential is given by \[\widetilde{\partial^{G,-}}\xv =\sum_{\yv}\sum_{r\in \RRecto(\xv,\yv),r\cap \bfX=\emptyset} u^{\vert\bfO\cap r\vert}\yv.\]

Following \cite[Lemma~4.6.9]{OSS2015grid}, we have the following identification of base point actions.

\begin{prop}\label{prop:U,v action identification}
Let $\cH$ be a real grid diagram representing $(K,\fra)$, a strongly invertible knot  with auxiliary data in $S^3$. Then there is a single $u$-variable in $\cR^G(\cH)$, and for any $1\le j\le k$ ($k$ is the number of pairs of $O$-base points in $\cH$), the $U_j$-action on $\GCR^-(\cH)$ is homotopic to the $u^2$-action.    

More generally, let $\cH$ be any real diagram representing a generalized strongly invertible link $(L,\fra)$. If $(O_i,O_i')$ and $O^f_j$ lie on the same component of $L$, then the $U_i$-action on $\GCR^-(\cH)$ is homotopic to the $u_j^2$-action. If $(O_i,O_i')$ and $(O_j,O_j')$ lie on the same component or same pair of components of $L$, then the $U_i$ and $U_j$-actions on $\GCR^-(\cH)$ are homotopic.    
\end{prop}

Using this proposition together with some straightforward homological algebra, we know that for $\cH$ representing a strongly invertible knot \[\widetilde{\GHR^-}(\cH)=\GHR^-(\cH)\otimes (\F^2)^{\otimes k},\] where $2k+1$ is the size of the real grid diagram.

By manipulating the $u_i$ and $U_j$ variables carefully, we can define $\GHR^-(L,\fra)$, $\widehat{\GHR}(L,\fra)$ and $\widetilde{\GHR}(\cH)$ for a generalized strongly invertible link. By blocking one pair of $(O,X)$ on each component of $L$ or blocking all $O$ and $X$-base points, we can also define hat or tilde version of combinatorial real knot Floer homology for generalized doubly periodic links. 

To stress the combinatorial aspect of this theory and its relationship with grid homologies defined in \cite{OSS2015grid}, we use notations $\GCR^\circ$ and $\GHR^\circ$ instead of $\CFLR^\circ$ and $\HFLR^\circ$ in the discussion above.

\begin{rem}[On the well-definedness and invariance of combinatorial real link Floer homology]
On one hand, real grid diagrams form a subset of real Heegaard diagram for links and the real rectangles we counted in real grid homology are exactly domains of real Maslov index 1 as noted in Proposition~\ref{prop:domains of real index 1}. In \cite{BGX} and \cite{LOrealbordered}, it was shown that for such a real homotopy class, the modulo $2$ count of holomorphic moduli space is always one. Therefore, $\GCR^{\circ}$ is indeed a combinatorial reintepretation of $\CFLR^{\circ}$, its well-definedness and invariance follow from this. On the other hand, one can mimic the proof of \cite[Chapter~4\&5]{OSS2015grid} to check $(\partial^{G,\circ})^2=0$ directly and construct chain homotopy equivalence between $\GCR^\circ$ chain complexes associated to real grid diagrams related  by real grid moves using ``real pentagon'' and ``real hexagon'' counts (see Section~\ref{sub:Equivariant crossing changes and maps associated to them}). The first approach is more concise, while the later one shows that $\GHR^\circ$-homology theories can be defined without reference to holomorphic theory. 

In this way, we have proved Theorem~\ref{thm:intro-real grid homology} modulo the assertion on bigrading, which will be defined in next subsection.
\end{rem}

\subsection{Relative gradings}\label{sub:Relative gradings}

The usual link Floer homology theory in $S^3$ is a bigraded homology theory. More precisely, we have Maslov grading and Alexander grading, among which the former acts as a homological grading while the second one acts as a counterpart of the quantum grading in TQFTs. The relative gradings between generators were defined using Maslov index and multiplicities of base points of the domain connecting them. For the absolute grading, $M$ is pinned down using the fact that forgetting one of the family of base points, we get back to a Heegaard diagram of $S^3$, which has $\widehat{\HF}(S^3)=\F$ that we assume living in level $M=0$, $A$ is pinned down using the symmetric property of Alexander polynomial. In \cite[Section~4.3]{OSS2015grid}, the authors provided a combinatorial description of these. We will follow their approach to define two real relative gradings on $\GHR^\circ$ for generalized strongly invertible links, absolute lifts will be pinned down later. This agrees with the bigrading in holomorphic theory defined in previous section (upto a shift).

\begin{prop}
On any real grid diagram $\cH$, there exist two functions 
\[M^R_{\bfO}:(\T_\alpha\cap\T_\beta)^R\to \Z\]
\[M^R_{\bfX}:(\T_\alpha\cap\T_\beta)^R\to \Z\] satisfies the following property. If $\xv$ and $\yv$ are two grid states and $r\in \RRect(\xv,\yv)$, then
\[M^R_{\bfO}(\xv)-M^R_{\bfO}(\yv)=1-\vert(r\cap \bfO)\vert+\vert(\xv\cap r)\vert,\]
\[M^R_{\bfX}(\xv)-M^R_{\bfX}(\yv)=1-\vert(r\cap \bfX)\vert+\vert(\xv\cap r)\vert.\]
They will be called \emph{real $\bfO$-Maslov function} and \emph{real $\bfX$-Maslov function} on real grid states, respectively. Furthermore, they are uniquely determined up to over all shift by $\Z$.
\end{prop}

\begin{proof}
    The uniqueness part is easy using the fact that any two real grid states in a real grid diagram can be connected by a finite sequence of real rectangles. 
    
    For the existence part, we will only construct $M^R_{\bfO}$, since the same argument works for $M^R_{\bfX}$ with $X$ in place of $O$. Forgetting the real structure for a moment, recall that there is a function $M_{\bfO}$ in \cite[Section~4.3]{OSS2015grid} that satisfies similar axiom with genuine rectangles in place of real rectangles. We claim that $M_{\bfO}^R$ can be defined by the formula \[M_{\bfO}^R(\xv)=\frac{1}{2}M_{\bfO}(\xv)-\frac{1}{4}\vert \xv\cap C\vert +\frac{1}{4} l_f. \] 
    
    By decomposing real rectangles into usual rectangles, one can check case by case that $M_{\bfO}^R$ satisfies the axiom. We use $X$-shaped octagon as an example, the other cases are similar and easier.
    
 	When $r$ is a $X$-shaped octagon, it can be decomposed into a composition of three usual rectangles $r_1\in \Rect(\xv,\zv_1)$, $r_2\in \Rect(\zv_1,\zv_2)$ and $r_3\in \Rect(\zv_2,\yv)$, for some non-real grid states $\zv_1$, $\zv_2$. Then one has \[\vert \xv\cap C\vert=\vert \yv\cap C\vert+2,\] \[M_{\bfO}(\xv)-M_{\bfO}(\yv)=3-2\vert r_1\cap \bfO \vert-2\vert r_2\cap \bfO \vert-2\vert r_3\cap \bfO \vert+2\vert \xv\cap r_1 \vert+2\vert \xv\cap r_2 \vert +2\vert \xv\cap r_3 \vert.\]
 	Since $\vert \xv\cap r \vert=\vert \xv\cap r_1 \vert+\vert \xv\cap r_2 \vert+\vert \xv\cap r_3 \vert$, while $\vert r\cap \bfO \vert=\vert r_1\cap \bfO \vert+\vert r_2\cap \bfO \vert+\vert r_3\cap \bfO \vert$, the desired formula for $M_{\bfO}^R(\xv)-M_{\bfO}^R(\yv)$ follows. 

    To show that $M^R_{\bfO}$ is integer valued, we recall that $M_{\bfO}$ was normalized by $M_{\bfO}(\xv^{\mathrm{NW}})=0$. Then since $\vert\xv^{\mathrm{NW}}\cap C\vert=l_f$, we know that \[M^R_{\bfO}(\xv^{\mathrm{NW}})=\frac{1}{2}M_{\bfO}(\xv^{\mathrm{NW}})-\frac{1}{4}\vert\xv^{\mathrm{NW}}\cap C\vert+\frac{1}{4}l_f=0\in \Z.\] Since any other real grid state can be connected to $\xv^{\mathrm{NW}}$ be a sequence of real rectangles and the relative grading is obviously $\Z$-valued. Now we have shown the claim and the proposition follows.
\end{proof}

As in the usual case, we define Alexander grading as the difference of $\bfO$- and $\bfX$-Maslov gradings. More precisely, the (relative) real Alexander grading is characterized by \[A^{R}(\xv)-A^R(\yv)= \frac{1}{2}(M^R_{\bfO}(\xv)-M^R_{\bfX}(\xv)-M^R_{\bfO}(\yv)+M^R_{\bfX}(\yv)).\] This is just one half of $A(\xv)-A(\yv)$, which can be seen from the proof of the proposition above. Actually, we will see later that after normalization, the absolute $A^R$ grading is exactly $\frac{1}{2} A$ for strongly invertible knots.


In the following, $M^R$ means $M_{\bfO}^R$ as we refer to $O$-base points as base points for the underlying $3$-manifold.

\begin{rem}
\begin{itemize}

    \item In case of links, the method above is also capable of define multi-real Maslov or Alexander gradings with one component for each strongly invertible knot component or pair of components as in \cite[Chapter~11]{OSS2015grid}. More generally, one can consider colored grading system as in \cite{Zemke2019absolutegradinginHFL}. Since we won't make use of these in this paper, we postpone details to later papers.

    \item Above we choose to work with generalized strongly invertible links, a similar combinatorial formula is also valid on real grid diagrams for periodic links. There are two subtleties: $\xv^{\text{NW}}$ is not a real grid states in such a diagram, so we have to choose some arbitrary states to define $M^R_{\bfO}$ and in particular, the normalization in Section~\ref{sub:Basic properties on strongly invertible knots} no longer works. Also, in a periodic diagram, each real domain/rectangle contains same number of $O$ and $X$-base points, all generators live in a single $A^R$-level, so we cannot extract any interesting information from it.
\end{itemize}
\end{rem}

\begin{rem}
    In practice, it is more convenient to calculate $\widetilde{\GHR}$ or $\widetilde{\GHR^-}$ and then ``destabilize'' to get $\widehat{\GHR}$ or $\GHR^-$. For the graded version, when $L$ is a (generalized) strongly invertible link, we have an extra tensor factor $\F^2$ supported in $(M^R,A^R)=(-1,-1),(0,0)$ for each pair of extra $O$-base points. When $L$ is a (generalized) periodic link, we have a tensor factor $\F^2$ supported in $(M^R,A^R)=(0,0)$ for each  extra pair of $(O,X)$-base points. 
\end{rem}

\section{Basic properties and examples}\label{sec:Basic properties and examples}
\subsection{Basic properties on strongly invertible knots}\label{sub:Basic properties on strongly invertible knots}

When $L$ is a strongly invertible knot, or more generally a generalized strongly invertible link, we can define an absolute lift of the bi-grading as follows using the construction and remarks in the previous section. The formulae are given by \[M_{\bfO}^R(\xv)=\frac{1}{2} M_{\bfO}(\xv)-\frac{1}{4} \vert \xv\cap C \vert+\frac{1}{4}l_{f}\] 
\[M_{\bfX}^R(\xv)=\frac{1}{2} M_{\bfX}(\xv)-\frac{1}{4} \vert \xv\cap C\vert +\frac{1}{4}l_f \] 
\[A^R(\xv)=\frac{1}{2}(M_{\bfO}^R(\xv) -M_{\bfX}^R(\xv))-\frac{(n-l_f-2l_p)}{4}.\]  

Following notations in~\cite{OSS2015grid}, we use $\GHR^{\circ}_{d}(L,\fra,s)$ to denote the summand of $\GHR^{\circ}(L,\fra)$ in real Maslov grading $d$ and real Alexander grading $s$.

Note that modulo the equivalence relation of $\fra$, there are exactly four choices of $\fra$ for a strongly invertible knot $K\subset (S^3,\tau)$- determined only by the choices (1), (2) in Definition~\ref{def:strongly invertible knots with extra data}. Fix any choice of $\fra$ as a based data, the other three will be denoted by $\fra^r$, $\fra^i$ and $\fra^{i,r}$, in which $r$ means orientation reversed while $i$ means the roles of $\bfX$ and $\bfO$ are interchanged. If $(T^2,\bm\alpha,\bm\beta,\bfO,\bfX)$ represents $(K,\fra)$, then \begin{itemize}
    \item $(-T^2,\bm\beta,\bm\alpha,\bfO,\bfX)$ represents $(K,\fra^r)$;
    \item $(-T^2,\bm\beta,\bm\alpha,\bfX,\bfO)$ represents $(K,\fra^{i})$;
    \item $(T^2,\bm\alpha,\bm\beta,\bfX,\bfO)$ represents $(K,\fra^{i,r})$.
\end{itemize}
See Figure~\ref{fig:auxiliary data for strongly invertible knots} for an example using the strongly invertible unknot. 
\begin{figure}
    \centering
    \begin{overpic}[width=0.8\textwidth]{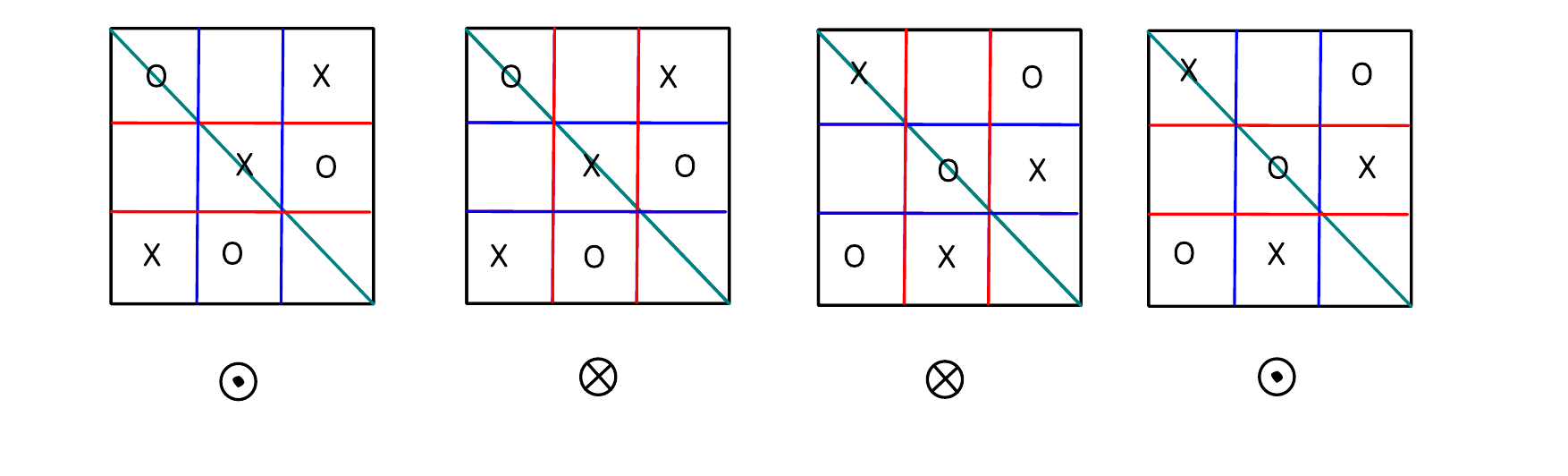}
			\put(12,0) {$(U,\fra)$}
			\put(34.5,0) {$(U,\fra^r)$}
			\put(56.5,0) {$(U,\fra^{i})$}
			\put(78,0) {$(U,\fra^{i,r})$}
		\end{overpic}
    \caption{Choices of auxiliary data for the unknot.}
    \label{fig:auxiliary data for strongly invertible knots}
\end{figure}
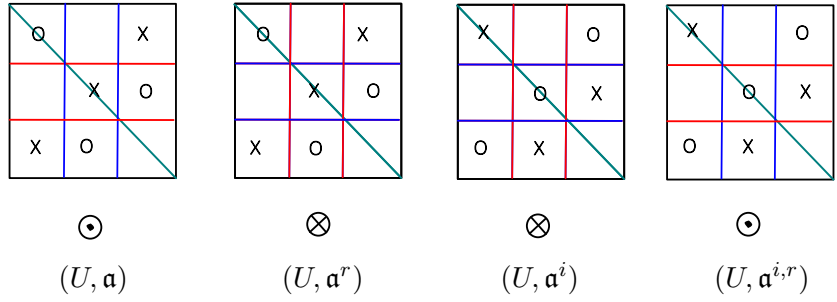
The proof of~\cite[Proposition~5.3.2]{OSS2015grid} applies to show that 
\begin{prop}\label{prop:fra v.s.fra^r}
If $K$ is a strongly invertible knot in $S^3$, with respect to the normalization of grading above, we have grading-preserving isomorphisms \[\widehat{\GHR}(K,\fra)\cong \widehat{\GHR}(K,\fra^r),\]  \[\GHR^-(K,\fra)\cong \GHR^-(K,\fra^{r}). \] 
And of course, there are similar isomorphisms connecting real grid homology of $(K,\fra^{i})$ and $(K,\fra^{i,r})$. 
\end{prop}

The proof of~\cite[Proposition~7.1.1]{OSS2015grid} applies to show that 
\begin{prop}\label{prop:fra v.s. fra^i,r}
If $K$ is a strongly invertible knot, with respect to the normalization of grading above, we have \[\widehat{\GHR}_{d}(K,\fra,s)\cong \widehat{\GHR}_{d-2s}(K,\fra^{i,r},-s),\]  for all $d\in \Z$ and $s\in 1/2\Z$. Furthermore, there are similar isomorphisms connecting real grid homology of $(K,\fra^{r})$ and $(K,\fra^{i})$. 
\end{prop}

\begin{rem}
Careful readers should have noticed that Proposition~5.3.2 and 7.1.1 in \cite{OSS2015grid} lead to symmetric property on $\widehat{\GH}(K)$, while for us, Proposition~\ref{prop:fra v.s.fra^r} and~\ref{prop:fra v.s. fra^i,r} are stated as isomorphism between real grid homology groups of the same knot with different auxiliary data. This difference is essential: the definition of usual knot Floer homology relies on a choice of orientation and base points but the resulting groups do not depend on this choice. While for real knot (link) Floer homology, the choice of $\fra$ is essential in the sense that it affects the resulting bi-graded groups. Concrete examples will be given in Section~\ref{sec:Examples and application}.  
\end{rem}

The proof of~\cite[Proposition 7.1.2 and Proposition 7.4.3]{OSS2015grid} applies to show that 
\begin{prop}\label{prop:hat version mirror sym}
If $(K,\fra)$ is a strongly invertible knot in $S^3$ with auxiliary data chosen and $(m(K),m(\fra))$ is its mirror with inherit symmetry and auxiliary data. Then, with respect to the normalization of grading above, we have \[\widehat{\GHR}_{d}(K,\fra, s)\cong \widehat{\GHR}_{2s-d}(m(K),m(\fra),s),\]  for all $d\in \Z$ and $s\in 1/2\Z$.

For the minus version, we have an isomorphism of bigraded modules
\[\GHR^-(m(K),m(\fra))\cong \mathrm{Hom}(\GHR^-(K,\fra),\F[u])\oplus \mathrm{Ext}(\GHR^-(K,\fra),\F[u])[1,0].\] More explicitly, if \[\GHR^-(K,\fra)\cong \F[u]_{(-\tau^R,-\frac{1}{2}\tau^R)}\oplus (\bigoplus_{i=1}^k \F[u]/u^{n_i}_{(d_i,s_i)}),\] then 
\[\GHR^-(m(K),m(\fra))\cong \F[u]_{(\tau^R,\frac{1}{2}\tau^R)}\oplus (\bigoplus_{i=1}^k \F[u]/u^{n_i}_{(n_i-1-d_i,n_i-s_i)}),\] where $\tau^R=\tau^R(K)$ is the real $\tau$-invariant that will be defined in Section~\ref{sub:Definition of invariant and statement of main theorem}.

\end{prop}

By construction, we have the following long exact sequence.

\begin{prop}\label{prop:exact sequence relating minus and hat}
Let $(K,\fra)$ be a strongly invertible knot in $S^3$ with auxiliary data chosen. Then we have the following long exact sequence relating its minus and hat versions of real grid homology.
\[\ldots\to \GHR^-_{d+1}(K,\fra,s+1/2)\xrightarrow{u}\GHR^-_{d}(K,\fra,s)\to \widehat{\GHR}_d(K,\fra, s)\to \GHR^-_{d}(K,\fra,s)\to \ldots\]
Since $\GHR^-(K,\fra)$ is a module over a polynomial ring with one variable by construction, we omit the subscript of $u$.
\end{prop}
\begin{proof}
This follows from the short exact sequence of chain complexes \[0\to \GCR^-(\cH)\xrightarrow{u} \GCR^-(\cH)\to \widehat{\GCR}(\cH)\to 0,\] for any compatible choice of real grid diagram $\cH$.
\end{proof}

\subsection{A spectral sequence to $\widehat{\HFR}(S^3,\tau)$}
In this subsection, we consider the relationship between the real knot Floer group and the real Heegaard Floer homology of the underlying real $3$-manifold. As in the usual case, they can be connected by a spectral sequence. 
\begin{thm}\label{thm:spectral sequence relating HFKR and HFR}
Let $(K,\fra)$ be a strongly invertible knot in $S^3$ with auxiliary data chosen. Then there is a spectral sequence \[\widehat{\HFKR}(K,\fra) \Longrightarrow \widehat{\HFR}(S^3,\tau).\]
\end{thm} 

To see this theorem in a quickest way, we step back to the holomorphic theory. From Proposition~\ref{prop:existence of real HD for strongly invertible knots and real H moves}, we know that any strongly invertible knot with auxiliary data $(K,\fra)$ admits a minimal real Heegaard diagram $\cH=(\Sigma,\bm\alpha,\bm\beta,\bfO=\{O\},\bfX=\{X\},R)$. In holomorphic theory, we have well-defined relative gradings characterized by \[M^R(\xv,\yv)=\mu_R(\cD)-n_{\bfO}(\cD),\] \[A^R(\xv,\yv)=\frac{1}{2}(n_{\bfX}(\cD)-n_{\bfO}(\cD)),\]
when $\xv$, $\yv$ is any pair of real generators and $\cD\in \pi_2^R(\xv,\yv)$. One can check that this definition coincides with the combinatorial one on multi-based diagrams. By forgetting the $X$-base points, we get a singly-pointed real Heegaard diagram $\underline{\cH}$ for $(S^3,\tau)$. $M^R$ acts as a homological grading on $\widehat{\CFR}(\underline{\cH})$ and $A^R$ induces a filtration on this chain complex. Here, the unique real $\spinc$ structure on $S^3$ is torsion, so $M^R$ is well-defined with value in $\Z$ and the null-homologous property of $K$ makes $A^R$ well-defined. Then the standard construction in algebraic topology leads to the desired spectral sequence.

\subsection{Some basic examples}
In this subsection, we illustrate the property of $\GHR^\circ$ using some basic examples. For these small knots, the auxiliary data $\fra$ is irrelevant for the real knot Floer groups, so we shall omit them from the notation for simplicity.

\begin{example}\label{ex:unknot}
A real grid diagram of size $3$ for the strongly invertible unknot $U$ in $S^3$ is shown in the left in Figure~\ref{fig:basic examples}. It can be calculated easily that \[\GHR^-(U)\cong \F[u]_{(0,0)} \quad \widehat{\GHR}(U)\cong \F_{(0,0)}.\]
\end{example}

\begin{figure}
    \centering
    \begin{overpic}[width=0.8\textwidth]{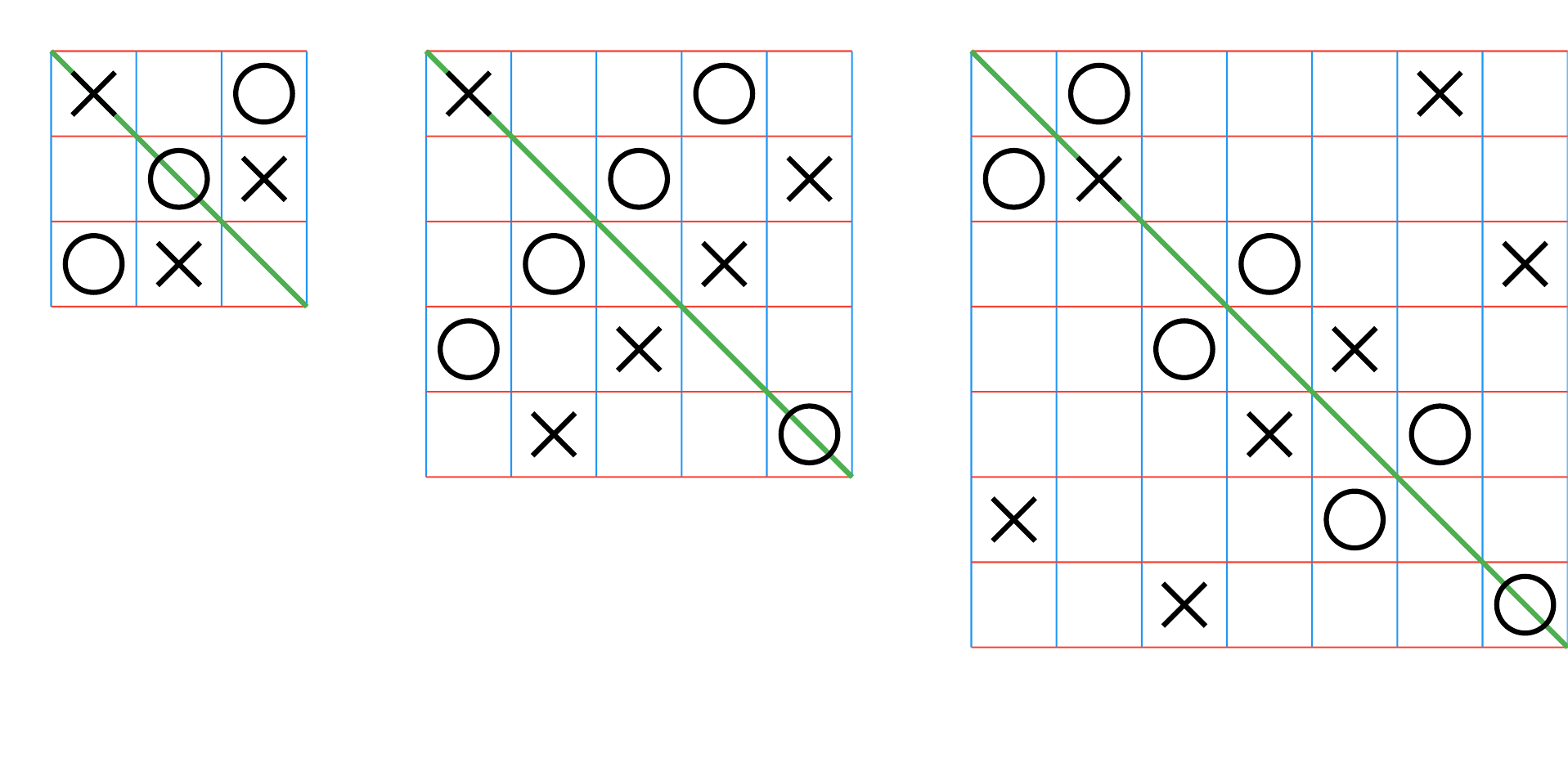}
			\put(7,25) {unknot}
			\put(40,15) {$\overline{3}_1$}
			\put(80,3) {$4_1$}
		\end{overpic}

    \caption{Grid diagrams for strongly invertible unknot, $\overline{3}_1$ and $4_1$.}
    \label{fig:basic examples}
\end{figure}

\begin{example}\label{ex:trefoil and figure 8}
Using computer program, we calculated various versions of real grid homology for the trefoil and figure eight knots using the real grid diagrams shown in the middle and right of Figure~\ref{fig:basic examples}. The result is the following: \[\widehat{\GHR}(\Bar{3}_1)=\F_{(0,-1/2)}\oplus \F_{(0,0)}\oplus \F_{(1,1/2)}\quad \GHR^-(\Bar{3}_1)=\F[u]_{(1,1/2)}\oplus \F_{(0,0)}. \]
\[\widehat{\GHR}(4_1)=\F_{(-1,-1/2)}\oplus \F_{(0,0)}\oplus \F_{(0,1/2)} \quad \GHR^-(4_1)=\F[u]_{(-1,-1/2)}\oplus \F_{(0,1/2)}. \]
For $\bar{3}_1$, the rank of hat version coincides with the rank of usual knot Floer homology as predicted for L-space knots in \cite{hendricks2025noterealheegaardfloer}. In \cite{hendricks2025noterealheegaardfloer}, Hendricks also calculated $\widehat{\HFKR}(4_1)$ using spectral sequence from $\widehat{\HFK}$ to $\widehat{\HFKR}$ (see Theorem~\ref{thm:spectral sequence from HFK,strongly invertible case}), our result agrees with hers. Except the part agrees with previous works, we also determine the real Maslov grading. This extra information together with Proposition~\ref{prop:hat version mirror sym}, tell us that $\widehat{\HFKR}$ actually distinguishes two strong inversions on $4_1$. (Recall that $4_1$ admits two distinct involutions related by mirroring, which can be distinguished by the induced action on $\fullCFK^{\infty}(4_1)$.)
\end{example}

\subsection{Relationship with $\widehat{HFK}$}

As we mentioned in Example~\ref{ex:trefoil and figure 8}, we have the following spectral sequence. 
\begin{thm}\label{thm:spectral sequence from HFK,strongly invertible case}(\cite[Theorem~1.4]{hendricks2025noterealheegaardfloer})
For a strongly invertible knot $(K,\fra)\subset (S^3,\tau)$, there is a spectral sequence starting from $\widehat{\HFK}(K)\otimes \F[\theta,\theta^{-1}]$ and converging to $\widehat{\HFKR}(K,\fra)\otimes \F[\theta,\theta^{-1}]$. Moreover, this spectral sequence respects the splitting of $\widehat{\HFK}$ and $\widehat{\HFKR}$ along the (real) Alexander grading. (In that paper, she used $\HFKR(\cH)$, since she did not show it is an invariant of the pair $(K,\fra)$.) 
\end{thm}

This follows from the more general theorem concerning real Lagrangian Floer homology of symplectic manifolds by taking $(M,L_0,L_1)=(\mathrm{Sym}^g({\Sigma}-O-X),\T_\alpha,\T_\beta)$ and $R$ to be the induced involution on the symmetric product for a minimal real Heegaard diagram $(\Sigma,\bm\alpha,\bm\beta,\bfO,\bfX, R)$ of $(K,\fra)$.

\begin{thm}\label{thm:spectral sequence on general Lagrangian Floer homology} (\cite[Theorem~1.1]{hendricks2025noterealheegaardfloer})
Let $(M,L_0,L_1)$ be a triple such that $M$ is an exact symplectic manifold which has the structure of a symplectization near infinity, and $L_0$ and $L_1$ are compact exact Lagrangians. Suppose that $R$ is an anti-symplectic involution on $M$ which interchanges $L_0$ and $L_1$. There is a spectral sequence whose $E_1$ page is \[\HF(M,L_0,L_1)\otimes \F[\theta,\theta^{-1}]\] and whose $E_{\infty}$ page is isomorphic to \[\HFR(M,L_0,L_1)\otimes \F[\theta,\theta^{-1}].\]
Here, $\HF$ and $\HFR$ denote the usual and real Lagrangian Floer homology on symplectic manifolds, respectively. (See~\cite[Section~2]{guth2025real})
\end{thm}

Note that the conditions in Theorem~\ref{thm:spectral sequence on general Lagrangian Floer homology} are still true for $(M,L_0,L_1)=(\mathrm{Sym}^g({\Sigma}-O-X),\T_\alpha,\T_\beta)$, when $O$, $X$-base points on $\Sigma$ are interchanged by $R$ instead of lying on $\fix(R)$. Thus, we have an analogue for doubly periodic knots.

\begin{thm}\label{thm:spectral sequence from HFK,doubly periodic case}
For a doubly periodic knot $(K,\fro)\subset (S^3,\tau)$, there is a spectral sequence starting from $\widehat{\HFK}(K)\otimes \F[\theta,\theta^{-1}]$ and converging to $\widehat{\HFKR}(K,\fra)\otimes \F[\theta,\theta^{-1}]$.
\end{thm}

\section{Torsion order and equivariant unknotting number}\label{sec:Torsion order and equivariant unknotting number}
\subsection{Equivariant crossing changes and maps associated to them}\label{sub:Equivariant crossing changes and maps associated to them}
In \cite{boyle2025equivariantunknottingnumbersstrongly}, Boyle and Chen introduced three types of equivariant unknotting operations and defined an equivariant unknotting number for strongly invertible knots. We first recall some basics from their work.
\begin{defn}\label{def:equiv unknotting number}
Let $K$ be a strongly invertible knot in $(S^3,\tau)$ and consider any of its transvergent planar diagram. On such a diagram, there are three types of unknotting operations as listed in Figure~\ref{fig:equivariant unknotting operations}. The \emph{equivariant unknotting number} $\widetilde{u}(K)$ was defined in  \cite{boyle2025equivariantunknottingnumbersstrongly} as the minimal number of crossing changes. In their convention, a type $A$ operation accounts for two crossing changes while a type $B$ or $C$ operation accounts for one. More generally, the \emph{type $X$ unknotting number} $\widetilde{u}_X(K)$ was defined to be the minimal number of type $X$ operations needed to change $K$ into the strongly invertible unknot, for $X=A,B$ or $C$. (They might be $\infty$, if $K$ cannot be unknotted only use that kind of crossing changes.) 
\end{defn}

\begin{figure}
    \centering
    \begin{overpic}[width=0.5\textwidth]{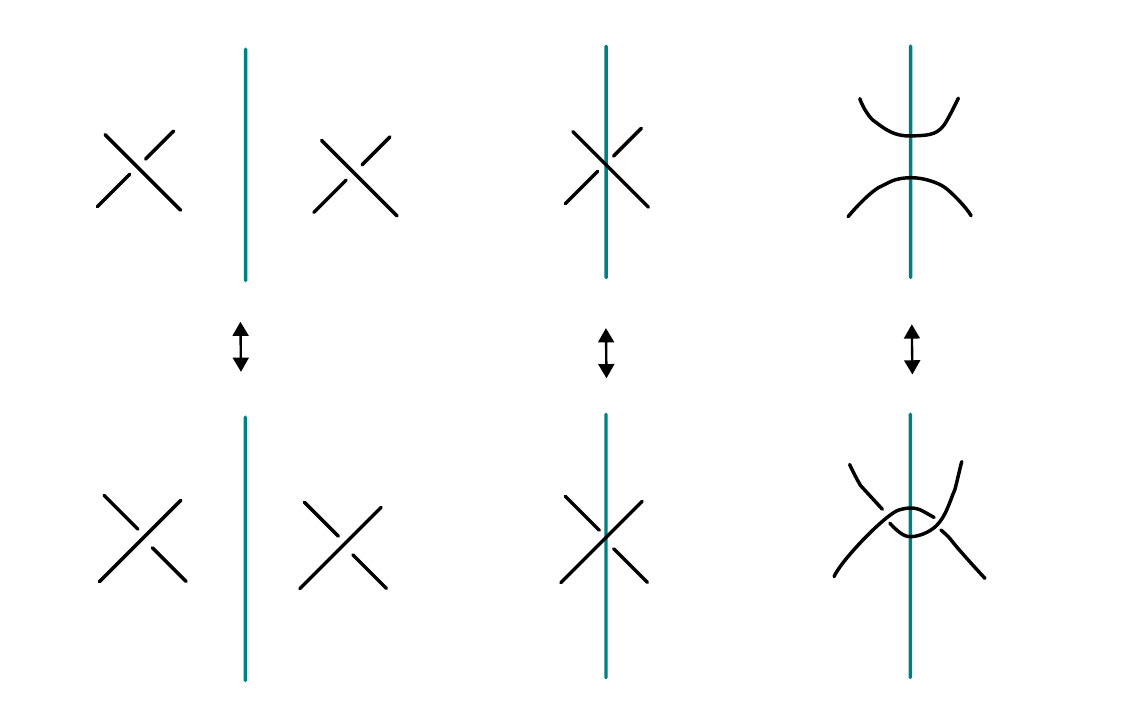}
			\put(15,-3) {type $A$}
			\put(45,-3) {type $B$}
			\put(75,-3) {type $C$}
			
		\end{overpic}
    \caption{Equivariant crossing changes.}
    \label{fig:equivariant unknotting operations}
\end{figure}

The authors of \cite{boyle2025equivariantunknottingnumbersstrongly} have shown that any strongly invertible knot can be unknotted using only type $A$ crossing changes and given an explicit criterion for when a strongly invertible knot can be unknotted using only type $C$ crossing changes. It remains open that whether all strongly invertible knots can be unknotted using only type $B$ crossing changes. 

To a type $A$ or $B$ crossing change, we associate a map on real grid homology mimicking the one considered in \cite[Section 6]{OSS2015grid}. This will lead to bounds on $\widetilde{u}(K)$ and also help us to see some algebraic property of $\GHR^{\circ}$.

\begin{prop} \label{prop:unknotting maps}
Let $(K_+,\fra_+),(K_-,\fra_-)$ be a pair of strongly invertible knots with auxiliary data. When $(K_-,\fra_-)$ is obtained from $(K_+,\fra_+)$ by a $+$ to $-$ type $A$ crossing change, we have $\F[u]$-module maps 
\[C_-^A\colon \GHR^-(K_+,\fra_+)\to \GHR^-(K_-,\fra_-),\quad C_+^A\colon \GHR^-(K_-,\fra_-)\to \GHR^-(K_+,\fra_+)\] of bigrading $(M^R,A^R)=(0,0)$, $(-2,-1)$, respectively, so that \[C_+^A\circ C_-^A=u^2,\quad C_-^A\circ C_+^A=u^2,\] where $u^2$ means that module action of multiplication by $u^2$.

When $K_+$ and $K_-$ are related by a type $B$ crossing change, then we have $\F[u]$-module maps \[C_-^B\colon \GHR^-(K_+,\fra_+)\to \GHR^-(K_-,\fra_-),\quad C_+^B\colon \GHR^-(K_-,\fra_-)\to \GHR^-(K_+,\fra_+)\] of bigrading $(M^R,A^R)=(-1,-1/2)$, $(-1,-1/2)$, respectively, so that \[C_+^B\circ C_-^B=u^2,\quad C_-^B\circ C_+^B=u^2.\] 
\end{prop}

\begin{figure}
    \centering
    \begin{overpic}[width=0.7\textwidth]{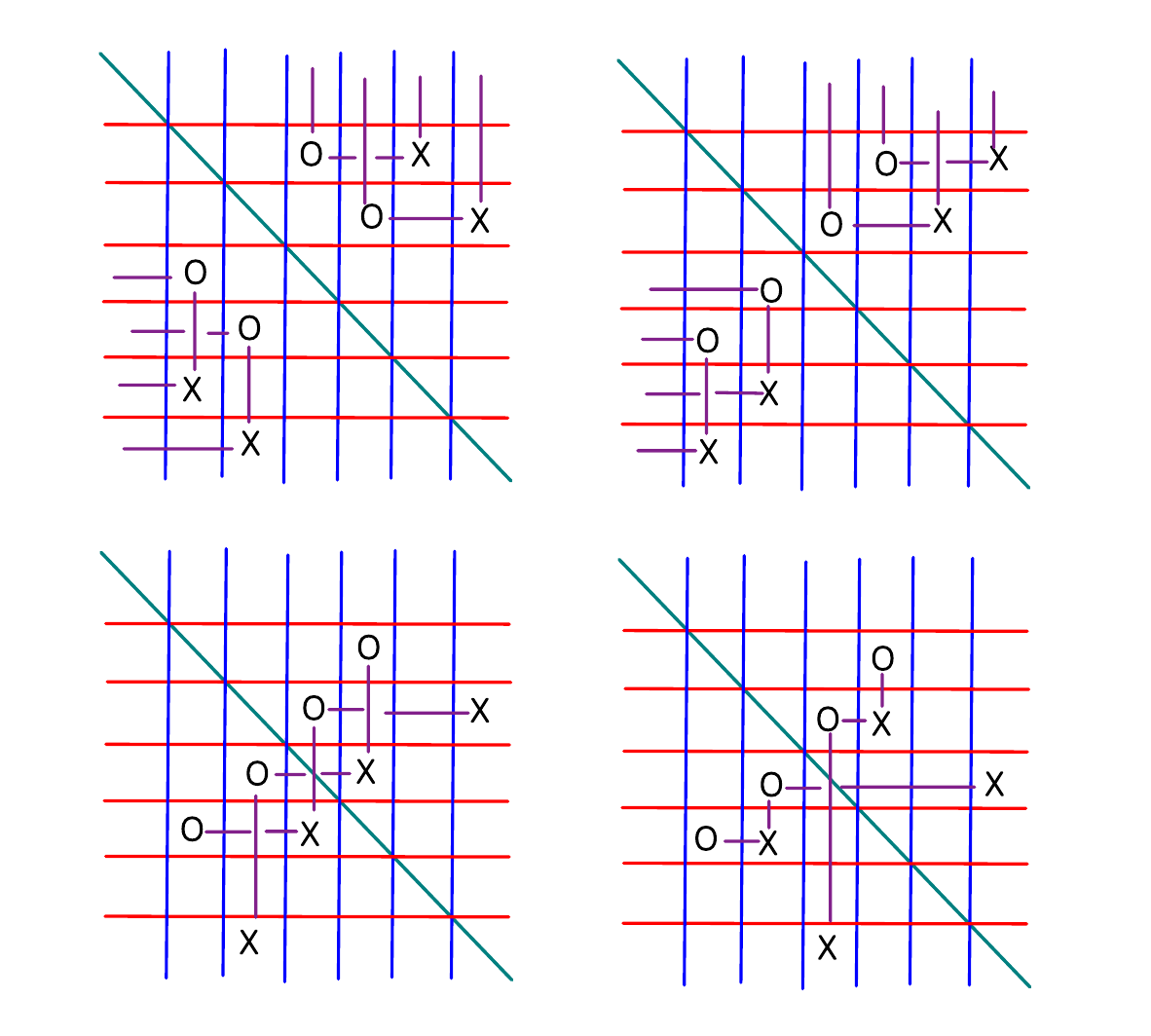}
			\put(35,44.5) {type $A$ crossing change}
			\put(35,0) {type $B$ crossing change}
			
		\end{overpic}
    \caption{Grid diagrams for type $A$ and $B$ crossing change.}
    \label{fig:type A crossing change}
\end{figure}

In Figure~\ref{fig:type A crossing change}, we realize type $A$ and $B$ crossing changes as local operations in real grid diagrams. We first observe that a type $B$ crossing change can be realized as a paired merge-split move that we shall introduce in Definition~\ref{def:real saddle moves}, so $C_+^B$ and $C_-^B$ can be defined as $\sigma$ and $\mu$ from Proposition~\ref{prop:band maps}. For a type $A$ crossing change, its realization on a real grid diagram appears as a pair of equivariant cross-commutations, so we can closely follow \cite[Section~6.2]{OSS2015grid} to define the crossing change maps. The only complication in real case is that we have four types of real rectangles in contrast to genuine rectangles in a usual grid diagram. Nevertheless, their construction works for us. To be concise, we will only sketch the construction and point out differences.

Up to real cyclic permutations, a type $A$ crossing change can be realized locally in a real grid diagram as in Figure~\ref{fig:combine diagram for crossing changes}. In that figure, we use $\alpha_\pm$, $\beta_\pm$ to denote different circles for diagrams $\cH_{\pm}$ associated to $K_{\pm}$, respectively. All the $\alpha$ and $\beta$-circles as well as $O$, $X$-base points that are omitted from the picture are the same in $\cH_+$ and $\cH_-$ by default. 

In $(\alpha_+\cap \alpha_-)\cup (\beta_+\cap \beta_-)$, we have two distinguished pairs of intersection points $(s,s')$ and $(t,t')$. Using them together with usual intersection points between $\alpha$ and $\beta$-circles, we can introduce the notion of real pentagons and real hexagons. By counting this, we construct the crossing change maps and verify the desired relations.

\begin{figure}
    \centering
    \begin{overpic}[width=0.8\textwidth]{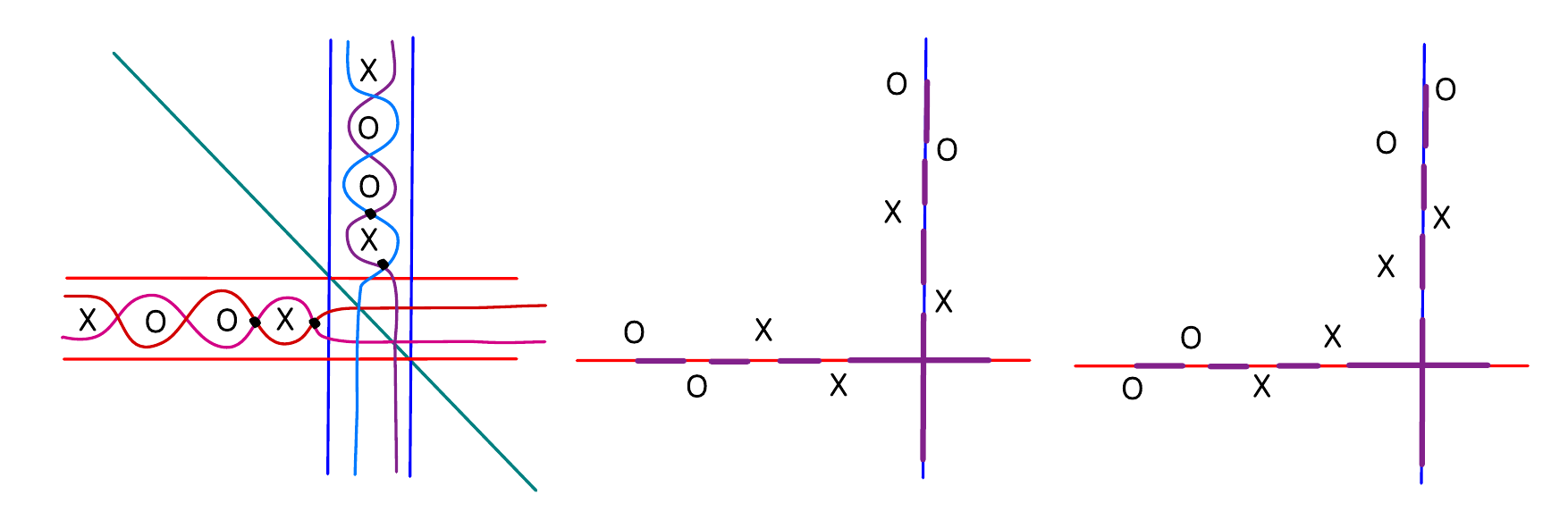}
			\put(0,15) {\textcolor{Mahogany}{$\alpha_+$}}
			\put(0,12) {\textcolor{Magenta}{$\alpha_-$}}
			\put(21,32) {\textcolor{Cyan}{$\beta_+$}}
			\put(24,32) {\textcolor{Purple}{$\beta_-$}}
            
			\put(40,8) {\textcolor{Purple}{$A'$}}
            \put(45,8) {\textcolor{Purple}{$B'$}}
            \put(49,8) {\textcolor{Purple}{$C'$}}
            \put(55,8) {\textcolor{Purple}{$D'$}}
            
            \put(73,8) {\textcolor{Purple}{$A'$}}
            \put(76.5,8) {\textcolor{Purple}{$B'$}}
            \put(82,8) {\textcolor{Purple}{$C'$}}
            \put(86,8) {\textcolor{Purple}{$D'$}}

            \put(59,26) {\textcolor{Purple}{$A$}}
            \put(59,21) {\textcolor{Purple}{$B$}}
            \put(59,17) {\textcolor{Purple}{$C$}}
            \put(59,12) {\textcolor{Purple}{$D$}}

            \put(88,26) {\textcolor{Purple}{$A$}}
            \put(88,21) {\textcolor{Purple}{$B$}}
            \put(91,17) {\textcolor{Purple}{$C$}}
            \put(91,12) {\textcolor{Purple}{$D$}}

            \put(14,10) {$s$}
            \put(19,10) {$t$}
            \put(26.5,20) {$s'$}
            \put(26.5,16) {$t'$}

		\end{overpic}
		
    \caption{Combined diagram for a type $A$ crossing change.}
    \label{fig:combine diagram for crossing changes}
\end{figure}
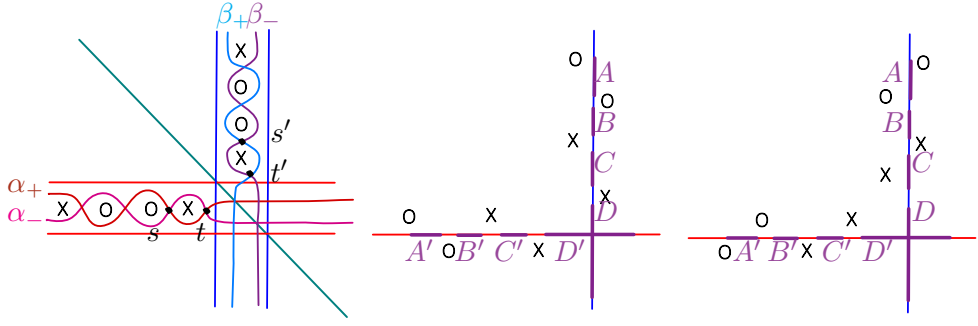

Recall from \cite[Section~5.1]{OSS2015grid}, a \emph{pentagon} in a grid diagram can be regarded as a rectangle with one edge replaced by a union of two consecutive segments in $\alpha_+\cup\alpha_-$ (or $\beta_+\cup \beta_-$) intersecting in a chosen distinguished vertex in $\alpha_+\cap\alpha_-$ (or $\beta_+\cap \beta_-$). Using this point of view, we can introduce a \emph{real pentagon} as a real domain in a combined real grid diagram that is obtained from any type of real rectangle by replacing a pair of edges with a pair of two consecutive segments in $\alpha_\pm$ then $\alpha_\mp$ and $\beta_\pm$ then $\beta_\mp$ intersecting in a chosen pair of vertices $(s,s')$ or $(t,t')$. Note that we already have four types of real rectangles and for each kind of them, the choice of edges for replacement is far from unique. This makes it hard to describe real pentagons systematically as in \cite[Definition~5.1.1]{OSS2015grid}. So instead of pursuing a rigorous definition, we illustrate real pentagons by examples: In Figure~\ref{fig:hexagon and pentagon}, the first four figures show examples of real pentagons with distinguished vertices at $(s,s')$. The underlying real rectangles of them are all different, so these have covered all kinds of real rectangles defined in Proposition~\ref{prop:domains of real index 1}. Similarly, we can define a \emph{real hexagon} as a real rectangle with a pair of edges replaced by a pair of three consecutive segments in $\alpha_\pm$, $\alpha_\mp$, $\alpha_\pm$ and $\beta_\pm$, $\beta_\mp$, $\beta_\pm$, where the corners are at $(s,s')$ and $(t,t')$ (or their order reversed). Two examples are shown in the last two frames in Figure~\ref{fig:hexagon and pentagon}, the lower-middle one models on a $X$-shaped octagon, the lower right one has a pair of rectangles as its underlying real rectangle. 

\begin{figure}
    \centering
    \begin{overpic}[width=0.8\textwidth]{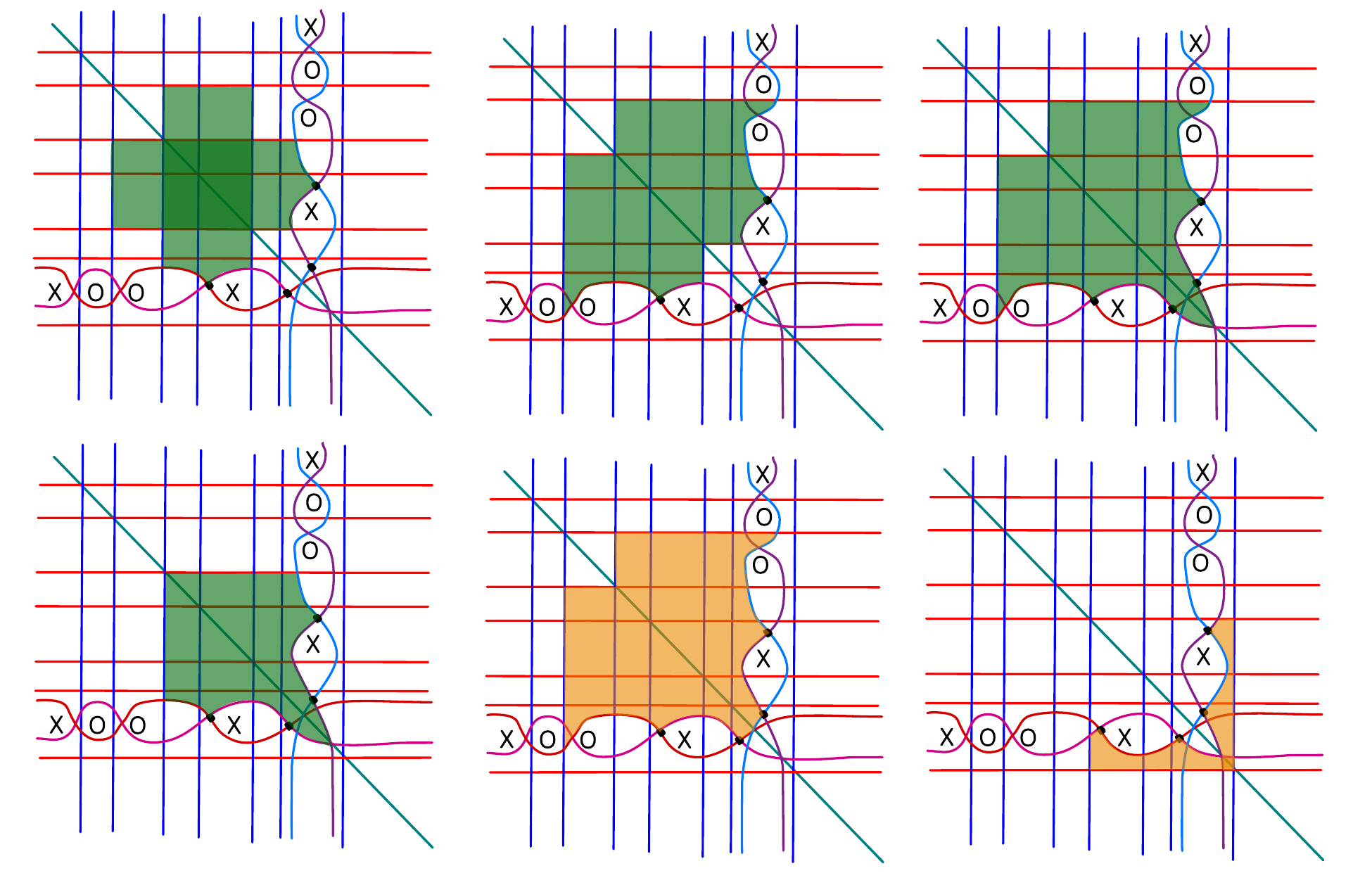}
			\put(-1,12) {\textcolor{Mahogany}{$\alpha_+$}}
			\put(-1,9) {\textcolor{Magenta}{$\alpha_-$}}
			\put(20,32) {\textcolor{Cyan}{$\beta_+$}}
			\put(23,32) {\textcolor{Purple}{$\beta_-$}}

            \put(-1,43) {\textcolor{Mahogany}{$\alpha_+$}}
			\put(-1,40) {\textcolor{Magenta}{$\alpha_-$}}
			\put(20,64) {\textcolor{Cyan}{$\beta_+$}}
			\put(23,64) {\textcolor{Purple}{$\beta_-$}}

            \put(32,11.5) {\textcolor{Mahogany}{$\alpha_+$}}
			\put(32,8.5) {\textcolor{Magenta}{$\alpha_-$}}
			\put(52,31) {\textcolor{Cyan}{$\beta_+$}}
			\put(55,31) {\textcolor{Purple}{$\beta_-$}}

            \put(32,43) {\textcolor{Mahogany}{$\alpha_+$}}
			\put(32,40) {\textcolor{Magenta}{$\alpha_-$}}
			\put(52,63) {\textcolor{Cyan}{$\beta_+$}}
			\put(55,63) {\textcolor{Purple}{$\beta_-$}}

            \put(64,11.5) {\textcolor{Mahogany}{$\alpha_+$}}
			\put(64,8.5) {\textcolor{Magenta}{$\alpha_-$}}
			\put(84,31) {\textcolor{Cyan}{$\beta_+$}}
			\put(87,31) {\textcolor{Purple}{$\beta_-$}}

            \put(64,43) {\textcolor{Mahogany}{$\alpha_+$}}
			\put(64,40) {\textcolor{Magenta}{$\alpha_-$}}
			\put(84,63) {\textcolor{Cyan}{$\beta_+$}}
			\put(87,63) {\textcolor{Purple}{$\beta_-$}}

            \put(14.5,8) {$s'$}
			\put(19,7) {$t'$}
			\put(25,18) {$s$}
			\put(25,12) {$t$}

            \put(47.5,7) {$s'$}
			\put(51.5,6.5) {$t'$}
			\put(57,17) {$s$}
			\put(57,12) {$t$}

            \put(79.5,7) {$s'$}
			\put(83.5,6.5) {$t'$}
			\put(90,17) {$s$}
			\put(90,12) {$t$}

            \put(14.5,40) {$s'$}
			\put(19,39) {$t'$}
			\put(25,50) {$s$}
			\put(25,44) {$t$}

            \put(47.5,39) {$s'$}
			\put(51.5,38.5) {$t'$}
			\put(57,49) {$s$}
			\put(57,44) {$t$}

            \put(79,39) {$s'$}
			\put(83.5,38.5) {$t'$}
			\put(90,49) {$s$}
			\put(90,44) {$t$}
		\end{overpic}

    \caption{Examples of pentagon and hexagon in a real grid diagram.}
    \label{fig:hexagon and pentagon}
\end{figure}
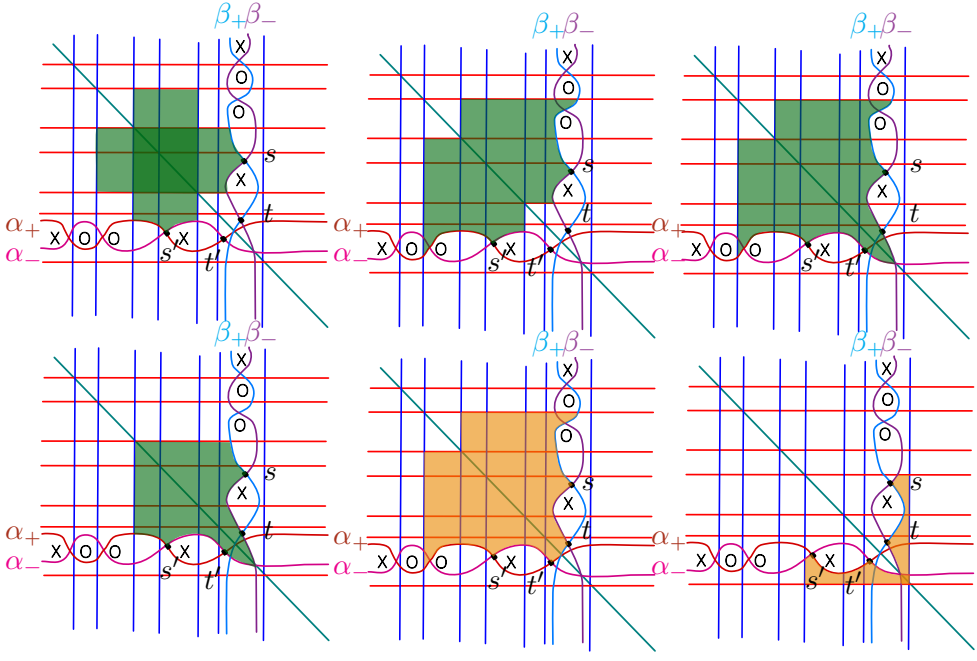

We will denote generators in $(\T_{\alpha_\pm}\cap\T_{\beta_\pm})^R$ using $\xv_\pm$, $\yv_\pm$, etc. Also, we will use $\xv_+$ and $\xv_-$ to denote two generators that are only different on components on $\alpha_\pm$ and $\beta_\pm$. Note that we have real pentagons with corners at $(s,s')$ ($(t,t')$) connecting $\xv_+$ to $\yv_-$ ($\xv_-$ to $\yv_+$). For real hexagons, following the boundary orientation, if the distinguished vertex $t$ appears before (after) $s$ up to cyclic permutation, then it connects some $\xv_+$ to $\yv_+$ ($\xv_-$ to $\yv_-$). We will use $\RPent(\xv_\pm,\yv_\mp)$, $\RHex(\xv_\pm,\yv_\pm)$ to denote the set of real pentagons or hexagons between a fixed pair of real generators. As for real rectangles, a real hexagon or pentagon is \emph{empty} if it contains no points of $\xv$ or $\yv$ in its interior, and the subset of empty real pentagon and hexagon will be denoted by $\RPento(\xv_\pm,\yv_\mp)$, $\RHexo(\xv_\pm,\yv_\pm)$ 

With all these preparations in hand, we can define maps 
\[c_-^A:\GCR^-(\cH_+)\to \GCR^-(\cH_-) \quad c_+^A:\GCR^-(\cH_-)\to \GCR^-(\cH_+)\] 
\[c_-^A(\xv_+)=\sum_{\yv_-\in (\T_{\alpha_-}\cap\T_{\beta_-})^R} \sum_{p\in \RPento(\xv_+,\yv_-),p\cap \bfX=\emptyset} u^{n_{O^{f}}(p)}\prod_{1\le i\le k} U^{n_{O_i}(p)} \yv_-;\]
\[c_+^A(\xv_-)=\sum_{\yv_+\in (\T_{\alpha_+}\cap\T_{\beta_+})^R} \sum_{p\in \RPento(\xv_-,\yv_+), p\cap \bfX=\emptyset} u^{n_{O^{f}}(p)}\prod_{1\le i\le k} U^{n_{O_i}(p)} \yv_+;\]
\[H_+:\GCR^-(\cH_+)\to \GCR^-(\cH_+)\quad H_-:\GCR^-(\cH_-)\to \GCR^-(\cH_-)\] 
\[H_+(\xv_+)= \sum_{\yv_+\in (\T_{\alpha_+}\cap\T_{\beta_+})^R} \sum_{h\in \RHexo(\xv_+,\yv_+), h\cap \bfX=\emptyset} u^{n_{O^{f}}(h)}\prod_{1\le i\le k} U^{n_{O_i}(h)} \yv_+;\]
\[H_-(\xv_-)= \sum_{\yv_-\in (\T_{\alpha_-}\cap\T_{\beta_-})^R} \sum_{h\in \RHexo(\xv_-,\yv_-), h\cap \bfX=\emptyset} u^{n_{O^{f}}(h)}\prod_{1\le i\le k} U^{n_{O_i}(h)} \yv_-;\]

The calculation in~\cite[Lemma~6.2.1]{OSS2015grid} works in real case without change, which shows that $c^A_\pm$ have the desired grading shifts. Note that the value of grading shift is actually the same- although real grading is roughly half of the original one, we perform a pair of cross commutation simultaneously, so the $1/2$ and $2$ factors cancel. 

Then, by considering decomposition of real domains (of real Maslov index $2$, if readers prefer a holomorphic theory point of view) into a real rectangle union a real pentagon, one can show that $c_\pm^A$ are chain maps. Further considering decomposition of real domains into two real pentagons or a real rectangle union a real hexagon, we can show that 
\[H_+\circ\partial^-+\partial^- \circ H_+=c_+^A \circ c_-^A +U;\]
\[H_-\circ\partial^-+\partial^- \circ H_-=c_-^A \circ c_+^A +U.\]
Here, $U$ is the variable associated to the pair of $O$-base points that is the top-most and left-most in Figure~\ref{fig:combine diagram for crossing changes}. Since we are working with a knot, Proposition~\ref{prop:U,v action identification}  tells us that the action of $u^2$ and $U$ agree on homology level. This concludes a sketch proof of Proposition~\ref{prop:unknotting maps} for type $A$ crossing changes.

\begin{defn}\label{def:torsion submodule and rank}(\cite[Definition~6.1.2]{OSS2015grid})
Let $M$ be a module over $\F[u]$. The \emph{torsion submodule} $\Tors(M)$ of $M$ is 
\[\Tors(M)=\{m\in M| \text{there is a $0\ne p\in \F[u]$ such that $p\cdot m=0$}\}.\] When $M$ is bigraded, then $M/\Tors(M)$ inherits a bigrading as well. When $M$ is finitely generated, there exists an $r\in \Z_{\ge 0}$ such that $M/\Tors(M)=\F[u]^r$, such $r$ will be called the \emph{rank} of $M$.
\end{defn}
\begin{lem}\label{lem:injectivity on torsion submodule}
Let $M$ and $N$ be two modules over $\F[u]$. If $\phi\colon M\to N$ and $\psi\colon N\to M$ are a pair of module maps so that $\psi\circ\phi=u^2$, then $\phi$ induces an injective map from $M/\Tors(M)$ to $N/\Tors(N)$.
\end{lem}
\begin{proof}
It is obvious that if $m\in \Tors(M)$, then $\phi(m)\in \Tors(N)$. Conversely, if $\phi(m)\in \Tors(N)$, then there exists $k\ge0$ so that $u^k\phi(m)=0$. The assumption implies that $0=\psi(u^k\phi(m))=u^{k+2}\cdot m$, so $m\in \Tors(M)$. Thus, there is a well-defined map $\widetilde{\phi}\colon  M/\Tors(M)\to N/\Tors(N)$ whose injectivity is implied by the argument above. 
\end{proof}

\begin{prop}\label{prop:a single tower}
For any strongly invertible knot $K$ with auxiliary data $\fra$, $\GHR^-(K,\fra)$ has rank $1$; in fact \[\GHR^-(K,\fra)/\Tors\cong\F[u]\] is supported in bigradings $(d,s)$ with $d-2s=0$.
\end{prop}
\begin{proof}
From Example~\ref{ex:unknot}, we know that $\GHR^-(U)=\F[u]$, supporting in bigrading $(0,0)$. So the conclusion is true for the unknot. Since any strongly invertible knot can be connected to the strongly invertible unknot by a sequence of type $A$ crossing changes. This is still true after we add auxiliary data to $K$. The proposition follows from Proposition~\ref{prop:unknotting maps} and Lemma~\ref{lem:injectivity on torsion submodule} as \cite[Proposition~6.1.4]{OSS2015grid}.
\end{proof}

\subsection{Torsion order}
We begin this subsection with an algebraic definition.
\begin{defn}\label{def:algebraic torsion order}
Let $M$ be an $\F[u]$-module, we define its \emph{torsion order} to be \[\ord_{u}(M)=\min \{k\in \N| u^k\cdot \Tors(M)=0\}\in \N\cup \{\infty\}.\]   
\end{defn}

Apply this to $\GHR^-$, we have the following strongly invertible knot invariant.
\begin{defn}\label{def:torsion order of a knot}
Let $(K,\fra)$ be a strongly invertible knot with auxiliary data in $(S^3,\tau)$, we define the \emph{torsion order} $\ord_{u}(K,\fra)$ to be $\ord_{u}(\GHR^-(K,\fra))$. This is always finite since $\GHR^-(K,\fra)$ is finitely generated as a $\F[u]$-module. 
\end{defn}

\begin{prop}
If $(K_+,\fra_+)$ and $(K_-,\fra_-)$ are a pair of knots related by a type $A$ or $B$ crossing change, then \[\vert \ord_{u}(K_+,\fra_+) -\ord_{u}(K_-,\fra_-)\vert \le 2. \]   
\end{prop}
\begin{proof}

For either type of crossing change, Proposition~\ref{prop:unknotting maps} provides us with a pair of maps $C_\pm^\circ$ between their real knot Floer groups such that composition in either direction is $u^2$. For simplicity, we omit the choice of auxiliary data from the notation during this proof. Let $\xi$ be an element in $\GHR^-(K_+)$ realizing the torsion order $m=\ord_{u}(K_+)$, then $u^m\cdot \xi=0$, while $u^{m-1}\cdot \xi\ne 0$. Since $C_+\circ C_-(\xi)= u^2 \cdot \xi$, \[C_-(u^{m-3}\cdot (C_+(\xi)))= u^{m-1}\cdot \xi\ne 0.\] This shows that $\ord_{u}(K_-)\ge\ord_{u}(K_+) -2$. The roles of $K_+$ and $K_-$ are symmetric here, so we can conclude that $\vert \ord_{u}(K_+) -\ord_{u}(K_-)\vert \le 2$.

\end{proof}

\begin{cor}\label{cor:torsion order bounds u'}
Let $K$ be a strongly invertible knot in $(S^3,\tau)$, for any choice of $\fra$, we have \[\ord_{u}(K,\fra)\le 2\widetilde{u}_A(K), \] \[\ord_{u}(K,\fra)\le 2\widetilde{u}_B(K). \] 
\end{cor}

\begin{example}
The knot $6_1$ is known as the Stevedore knot or the twist knot with four half twists.  With the help of computer program, we computed that \[\GHR^-(6_1,\fra)=\F[u]_{(-1,-1/2)}\oplus\F_{(0,1/2)}\oplus \F[u]/(u^2)_{(0,1/2)},\] so that $\ord_{u}(6_1,\fra)=2$ for some choice of $\fra$. (See Appendix~\ref{app:Calculation results for knots with small crossing numbers} for its real grid diagram.) On the other hand, it is relatively easy to see that a single type $B$ move is enough to unknot $6_1$, so the factor $2$ in the second inequality in Corollary~\ref{cor:torsion order bounds u'} is essential.
\end{example}

\section{Real $\tau$-invariant and equivariant slice genus}\label{sec:Real tau-invariant and equivariant slice genus}
\subsection{Definition of invariant and statement of main theorem}\label{sub:Definition of invariant and statement of main theorem}
\begin{defn}
For any strongly invertible knot $K$ with a choice of auxiliary data $\fra$, the \emph{real $\tau$-invariant} $\tau^R(K,\fra)$ is $-1$ times the maximal integer $i$ for which there is a homogeneous, non-torsion element in $\GHR^-(K,\fra)$ with real Alexander grading equals to $i/2$. 
\end{defn}
\begin{prop}
If $(K_+,\fra_+)$ and $(K_-,\fra_-)$ are related by a pair of $+$ to $-$ crossing changes of type $A$, then \[0\le \tau^R(K_+,\fra_+)-\tau^R(K_-,\fra_-)\le 1.\] If $(K,\fra)$ and $(K',\fra')$ are related by a type $B$ crossing change, then \[\vert \tau^R(K,\fra)-\tau^R(K',\fra')\vert\le 1.\]
\end{prop}
\begin{proof}
This proposition shares the same proof strategy with Theorem~\ref{thm:comparison of real tau sets}. We will prove that theorem in details, so we leave this as an exercise for readers. Careful readers should note that the part for type $B$ crossing change follows immediately from second item in Theorem~\ref{thm:comparison of real tau sets}.
\end{proof}

\begin{cor}
Let $(K,\fra)$ be a strongly invertible knot with auxiliary data, then \[\vert \tau^R(K,\fra)\vert\le \widetilde{u}_A(K),\] 
\[\vert \tau^R(K,\fra)\vert\le \widetilde{u}_B(K).\] 
\end{cor}
 
\begin{rem}
The bounds from real knot Floer homology seems to tell us that a type $B$ crossing change share the same weight as a type $A$ crossing change. This is different from the point of view of \cite{boyle2025equivariantunknottingnumbersstrongly}, in which Boyle and Chen counted the total number of crossing changes, so a type $A$ move weights twice as a type $B$ move. 
\end{rem}

Our $\tau^R$ models on the $\tau$-invariant introduced by Ozsvath and Szabó in~\cite{OSHFKand4ballgenus}. It was shown that $\vert\tau\vert$ provides a lower bound for the usual unknotting number and smooth slice genus. (For a combinatorial proof, see~\cite[Section~6\&8]{OSS2015grid}.) In the real setting, we have an interesting counterpart. To state it precisely, we need to set up some conventions on equivariant cobordism first. 

In \cite{borodzik2025khovanovhomologyequivariantsurfaces} and \cite{BDMS26}, the authors have characterized a complete set of elementary cobordism for an equivariant cobordism in $(S^3\times I,\tau\times \id)$, which we shall rephrase in the following. Our situation is much simpler than the one in~\cite{BDMS26}, since we have no need to deal with the projection to a generic plane.

\begin{defn}\label{def:real saddle moves}
On an equivariant cobordism between generalized strongly invertible links, there are two kinds of equivariant saddle moves which appear in the level set change of equivariant Morse function at index $1$ critical points. \begin{itemize}
    \item If we think of a saddle move as attaching an unknotted band (1-handle) $B$ to a Seifert surface of $L$ and taking the new boundary, then the $1$-handle associated to a \emph{fixed point saddle move} is fixed as a set by $\tau$, see the first row of Figure~\ref{fig:equivariant saddle moves} for examples. In this case, $B$ is a copy of $[0,1]\times [0,1]$ with involution $\iota(x,y)=(x,1-y)$. Two kinds of attachment are allowed. \begin{enumerate}
        \item The attaching region is $\{0,1\}\times [0,1]$ and the equivariance forces $(0,1/2)$, $(1,1/2)$ to be identified with distinct fixed points on $L$. 
        \item The attaching region is $[0,1]\times \{0,1\}$, two new fixed points are introduced to the link.
    \end{enumerate} 
    \begin{itemize}
        \item For (1), it may happen that $l_f-1$ while $l_p+1$ or $l_f-1$ while $l_p$ keeps unchanged.
        \item For (2), it may happen that $l_f+1$ while $l_p-1$ or $l_f+1$ while $l_p$ keeps unchanged.
    \end{itemize} We call the move a \emph{fixed point split move} or a \emph{fixed point merge move} according to whether $l_f+2l_p$ increases by $1$ or decreases by $1$. 
    \item A \emph{pair of equivariant saddle moves} can be thought of as attaching two bands to an equivariant Seifert surface of $L$ equivariantly and taking the new boundary. In this case, the attaching region of the bands does not intersect the fixed points on $L$. $l_f$ is unaffected during this operation, while several different cases may happen to $l_p$.
    \begin{enumerate}
        \item $l_p$ increases by $1$;
        \item $l_p$ decreases by $1$;
        \item $l_p$ also kept unchanged.
    \end{enumerate}
    Some model pictures are shown in the second row of Figure~\ref{fig:equivariant saddle moves}. We call them \emph{paired merge move}, \emph{paired split move} and \emph{paired merge-split move}, respectively.
\end{itemize} 
\begin{figure}
    \centering
    \includegraphics[width=1.0\linewidth]{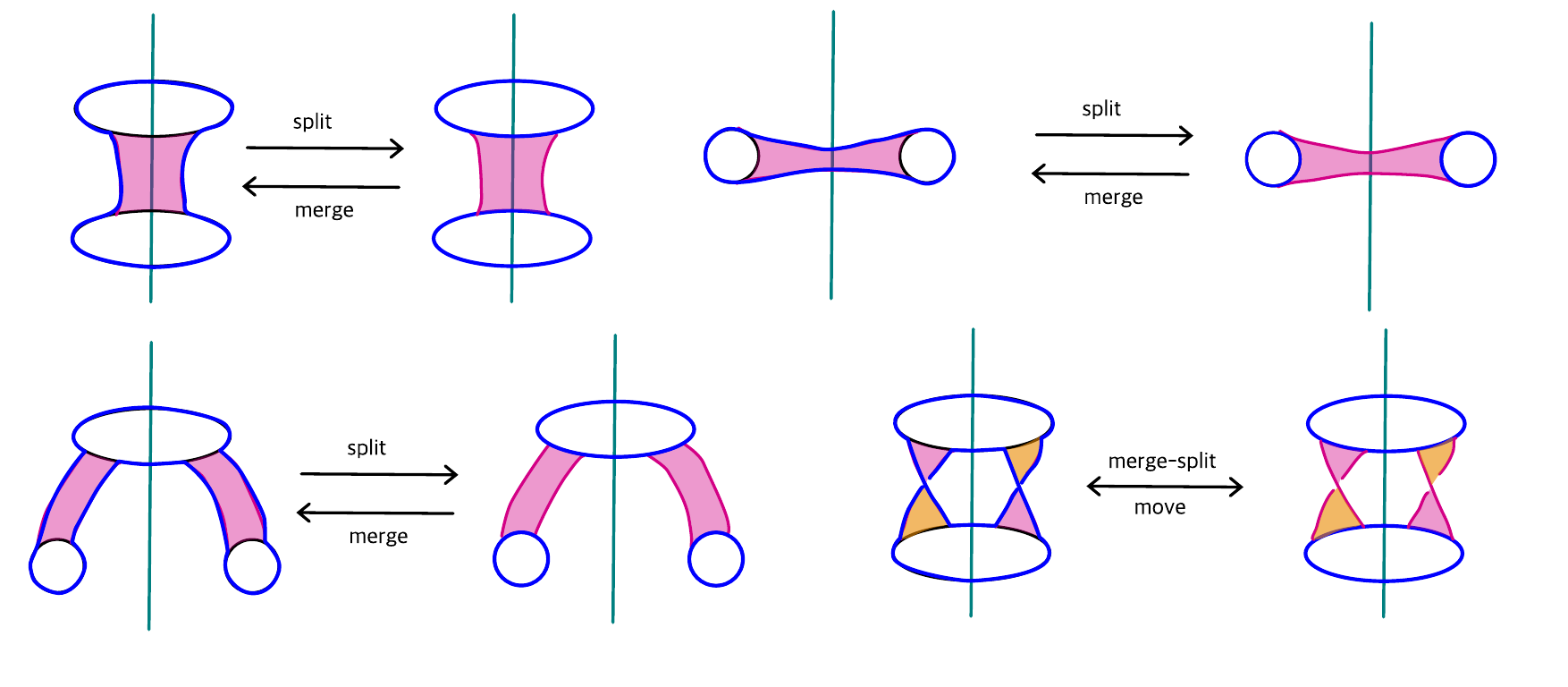}
    \caption{Equivariant saddle moves.}
    \label{fig:equivariant saddle moves}
\end{figure}
Among all the moves defined above, merge and split moves appear as a pair of inverse moves, while the inverse operation of a paired merge-split move is also a paired merge-split move.
\end{defn}

As in the usual case, we also have death and birth appearing in the change of level set of an equivariant Morse function.
\begin{defn}\label{def:equivariant birth and death}
There are two kinds of \emph{equivariant birth} that can happen on a generalized strongly invertible link:\begin{itemize}
    \item an index $0$ critical point along the fixed set introduces a new strongly invertible unknot to the link which is unlinked with the existing link;
    \item a pair of index $0$ critical points interchanged by the involution induces a new two component unlink which is unlinked with the existing link.
\end{itemize}
We shall call the first a \emph{fixed point birth}, the second a \emph{paired birth}. Dually, we have \emph{fixed point death} and \emph{paired death}. 
\end{defn}

The equivariant saddle moves together with the equivariant birth and death will be refer to \emph{equivariant Morse moves}.

\begin{defn}\label{def:compatible auxiliary data on cobordism}
If a generalized strongly invertible link $L_1$ is obtained from another generalized strongly invertible link $L_0$ by an equivariant saddle move (see Definition~\ref{def:real saddle moves}), an equivariant birth or death (see Definition~\ref{def:equivariant birth and death}), then for any auxiliary data on $L_0$, there is an auxiliary data on $L_1$ such that \begin{itemize}
    \item they are the same on components unaffected by the move;
    \item the choice of labeling and orientation on the new component is compatible with existing choice on $C$ if the move is a birth or fixed point split move;
    \item the choice of labeling and orientation on the ``new'' component is just the induced one if the move is fixed point merge move.
\end{itemize} 
When $L_0$ and $L_1$ are equipped with such a pair of auxiliary data $\fra_0$ and $\fra_1$, we say their auxiliary choices are \emph{compatible with the Morse move}.

Let $S$ be an equivariant connected cobordism $S$ in $(S^3\times [0,1],\tau\times \id)$ between two strongly invertible knots $K_0$ and $K_1$. After a small equivariant isotopy, $S$ can be decomposed into a sequence of elementary cobordisms with each stage an equivariant Morse move. We say two auxiliary data $\fra_0$ and $\fra_1$ on $K_0$ and $K_1$ are \emph{compatible with the cobordism $S$} if for any choice of decomposition, there are auxiliary data on intermediate links so that for each single move, the data on two sides are compatible in the sense above. It is easy to see that given any cobordism $S:K_0\to K_1$, there exists at least one pair of auxiliary data $\fra_0$ and $\fra_1$ compatible with it.
\end{defn}

\begin{thm}\label{thm:real taut bounds genus of cobordism}
Consider an equivariant connected cobordism $S$ of genus $g$ in $(S^3\times [0,1],\tau\times \id)$ between two strongly invertible knots $K_0$ and $K_1$. Fix compatible auxiliary data $\fra_0$ and $\fra_1$ on $K_0$ and $K_1$.  By perturbing $S$ if necessary, we may assume that the projection $p\colon S^3\times [0,1]$ defines an equivariant Morse function $S\to [0,1]$. Two types of critical points could appear on $S$: it may lie on the fixed set or appear in pairs, we call the first kind a \emph{fixed critical point}. (Note that a fixed critical point of index $1$ precisely corresponds to a fixed point saddle move defined in Definition~\ref{def:real saddle moves}, while a fixed critical point of index $0$ or $2$ corresponds to a fixed point birth or death defined in Definition~\ref{def:equivariant birth and death}.) Then \[\vert \tau^R(K_0,\fra_0)-\tau^R(K_1,\fra_1) \vert\le g+m/2,\] where $m$ is the minimal number of fixed critical points of $p|_{S'}$ ranging over all generic $S'$ (in the sense that $p|_{S'}$ is a Morse function) that are equivariantly isotopic to $S$.  
\end{thm}

Among all the $4$-manifolds bounded by $S^3$, $B^4$ is the simplest one. Regarding it as the unit ball $\C^2$, it has a canonical involution induced by complex conjugation which we will denote by $\tau_c$.

\begin{cor}\label{cor:real tau bounds tilde g4+m}
For any strongly invertible knot $K$ in $S^3$, let $S$ be an equivariant smooth slice surface of $K$ inside $(B^4,\tau_c)$. Then for any choice of auxiliary data $\fra$, \[\vert \tau^R(K,\fra) \vert\le g(S)+(m-1)/2,\] where $m$ is the minimal number of fixed critical points of $r|_{S'}$ for $r$ the radius function on $B^4$ ranging over all generic $S'$ (in the sense that $r|_{S'}$ is a Morse function) that are equivariantly isotopic to $S$.
\end{cor}

\subsection{Collapsed grid homology for generalized strongly invertible links}\label{sub:Collapsed grid homology for generalized strongly invertible links}
Since a generic level set of an equivariant cobordism in $(S^3\times [0,1],\tau\times \id)$ might be a generalized strongly invertible link with more than one components, it is necessary for us to investigate their grid homologies before proceeding to prove Theorem~\ref{thm:real taut bounds genus of cobordism}.

Slightly generalizing the definition in Section~\ref{sub:Real grid homology}, we can define real grid homologies for a generalized strongly invertible link with auxiliary data $(L,\fra)$. Assume that $L$ has $l_f$ components fixed by $\tau$ as sets and $l_p$ pairs of components interchanged by $\tau$. Let $\cH$ be a real grid diagram for $(L,\fra)$, with $O$-base points labeled so that $(O_1,O_1')$, $\ldots$, $(O_{l_p}, O_{l_p}')$ belong to distinct pairs of components. Then we have \[\GCR^-(\cH)=\F[u_1,\ldots,u_{l_f}, U_1,\ldots,U_k] \left \langle (\T_\alpha\cap\T_\beta)^R \right \rangle\] with differential \[\partial^{G,-}\xv=\sum_{\yv}\sum_{r\in \RRecto(\xv,\yv),r\cap \bfX=\emptyset} \prod_{1\le i\le l_f} u_{i}^{n_{O_i^f}(r)}\prod_{1\le j\le k} U_{j}^{n_{O_j}(r)} \yv.\] Here $k$ is the number of pairs of $O$-base points in $\cH$. 
Taking homology, we get $\GHR^-(L,\fra)$, which is an invariant of $(L,\fra)$ as an $\F[u_1,\ldots,u_{l_f},U_1,\ldots,U_{l_p}]$-module. As usual, $\widehat{\GCR}(\cH)=\GCR^-(\cH)/(u_i,U_j)_{1\le i\le l_f, 1\le j\le l_p}$ and $\widehat{\GHR}(\cH)$ is defined as its homology, which is an invariant of  $(L,\fra)$ as a $\F$-vector space. Similarly, we can define $\widetilde{\GCR}(\cH)$ by quotienting out all $u$ and $U$ variables from $\GCR^-$ and resulting in $\widetilde{\GHR}(\cH)$, which is a stabilization of $\widehat{\GHR}(L,\fra)$.

In the link case, we have a further variation, called the \emph{collapsed real grid homology}, naming after the construction in~\cite[Section~8.2]{OSS2015grid}. The chain complex is \[c\GCR^-(\cH)=\frac{\GCR^-(\cH)}{u_1=u_2=\ldots=u_{l_f}, u_1^2=U_1=\ldots=U_{l_p}},\] with induced differential $c\partial^{G,-}$. The homology \[c\GHR^-(L,\fra)=H_*(c\GCR^-(\cH),c\partial^{G,-})\] is an invariant of $(L,\fra)$ as an $\F[u]$-module. To make the $u$-action well-defined, we shall assume $l_f\ge 1$.

On real grid homology for generalized strongly invertible links, the absolute lift of gradings has been defined at the beginning of Section~\ref{sec:Basic properties and examples}. 

\begin{prop}\label{prop:band maps}
Let $W=\F_{(0,0)}\oplus\F_{(-1,-1)}$. 
\begin{enumerate}
    \item If $L'$ is obtained from $L$ by a paired split move, then for any choice of compatible auxiliary data $\fra$ and $\fra'$, there are $\F[u]$-module maps 
    \[\sigma_p\colon c\GHR^-(L,\fra)\otimes W\to c\GHR^-(L',\fra')\]
    \[\mu_p\colon c\GHR^-(L',\fra')\to c\GHR^-(L,\fra)\otimes W \] with the following properties: \begin{itemize}
        \item $\sigma_p$ is homogeneous of degree $(-1,0)$;
        \item $\mu_p$ is homogeneous of degree $(-1,-1)$;
        \item $\mu_p\circ\sigma_p$ is the multiplication by $u^2$;
        \item $\sigma_p\circ\mu_p$ is the multiplication by $u^2$.
    \end{itemize}
    \item If $L'$ is obtained from $L$ by a paired merge-split move, then there are $\F[u]$-module maps 
    \[\sigma\colon c\GHR^-(L,\fra)\otimes W\to c\GHR^-(L',\fra')\otimes W\]
    \[\mu\colon c\GHR^-(L',\fra')\otimes W\to c\GHR^-(L,\fra)\otimes W\] with the following properties: \begin{itemize}
        \item $\sigma$ is homogeneous of degree $(-1,-1/2)$;
        \item $\mu$ is homogeneous of degree $(-1,-1/2)$;
        \item $\mu\circ\sigma$ is the multiplication by $u^2$;
        \item $\sigma\circ\mu$ is the multiplication by $u^2$.
    \end{itemize}
\end{enumerate}
\end{prop}

Similar to Proposition~\ref{prop:a single tower}, we have a rank computation of $c\GHR^-$ for a  generalized strongly invertible link.

\begin{prop}\label{prop:number of tower for strongly invertible links}
Let $(L,\fra)$ be a generalized strongly invertible link with auxiliary data. Assume $L$ has $l_f\ge 1$ strongly invertible knot component and $l_p$ components interchanged by $\tau$. Then we have an isomorphism \[c\GHR^-(L,\fra)/\Tors\cong \F[u]^{r},\] for $r=2^{l_p+l_f-1}$.
\end{prop}

\begin{defn}\label{def:real tau set for links}
Let $(L,\fra)$ be a generalized strongly invertible link with auxiliary data. As above, assume it has $l_f+2l_p$ components such that $l_f\ge 1$. Then we define its \emph{real $\tau$-set ($\tau^R$-set)} to be the set of integers $\tau^R_{\min}(L,\fra)=\tau^R_1\le\ldots\le \tau^R_{2^{l_f+l_p-1}}=\tau^R_{\max}(L,\fra)$, defined as follows. Choose a generating set of $2^{l_f+l_p-1}$ elements for $c\GHR^-(L,\fra)/\Tors$ such that each element is homogeneous with respect to the $A^R$-grading. Then $\tau^R$-set is defined to be $-2$ times the real Alexander grading of the generators. 
\end{defn}

\begin{cor}\label{cor:real tau set is an oriented link invariant}
The real $\tau$-set is an invariant of the generalized strongly invertible link $(L,\fra)$.  
\end{cor}
\begin{proof}
Since $c\GHR^-$ is an invariant of $(L,\fra)$, we only need to check the basis-independence. The proof of \cite[Corollary~8.3.4]{OSS2015grid} works here without change.
\end{proof}

\begin{thm}\label{thm:comparison of real tau sets}
Let $L$ and $L'$ be generalized strongly invertible links differ by an equivariant saddle move. Fix compatible auxiliary data $\fra$ and $\fra'$ on them. \begin{itemize}
    \item If $L'$ is obtained from $L$ by a paired split move, then \[\tau^R_{\min}(L,\fra)-2\le \tau^R_{min}(L',\fra')\le \tau^R_{\min}(L,\fra),\]
    \[\tau_{\max}(L,\fra)\le \tau_{\max}(L',\fra')\le \tau_{\max}(L,\fra)+2.\]
    \item If $L'$ is obtained from $L$ by a paired merge-split move, then \[\tau^R_{\min}(L,\fra)-1\le \tau^R_{min}(L',\fra')\le \tau^R_{\min}(L,\fra)+1,\]
    \[\tau^R_{\max}(L,\fra)-1\le \tau^R_{\max}(L',\fra')\le \tau^R_{\max}(L,\fra)+1.\]
\end{itemize}
\end{thm}
\begin{proof}
These can be shown using the same argument as that for \cite[Theorem~8.3.6]{OSS2015grid}. One only needs to be careful that for paired merge-split moves, both $\sigma$ and $\mu$ shift the real Alexander grading. 

\end{proof}

\begin{lem}\label{lem:U=U' leads to otimes W}
Let $\cH$ be a real grid diagram representing $(L,\fra)$ and $W=\F_{(0,0)}\oplus\F_{(-1,-1)}$. If $(O_i,O_i')$ and $(O_j,O_j')$ ($i,j\ge l_p+1$) belong to the same strongly invertible knot component of $L$ or a pair of components interchanged by $\tau$, then \[H_*(\frac{c\GCR^-(\cH)}{U_i=U_j})\cong c\GHR^-(\cH)\otimes W.\]
\end{lem}
\begin{proof}
This follows from the proof of \cite[Lemma~8.3.8]{OSS2015grid}.
\end{proof}

\begin{proof}[Proof of Proposition~\ref{prop:band maps}]
For the paired merge-split move or the paired split and merge move, $\sigma$ and $\mu$ maps are defined using same local picture shown in Figure~\ref{fig:saddle map}, which is just the usual one considered in~\cite[Figure~8.1]{OSS2015grid} union its reflection image. 
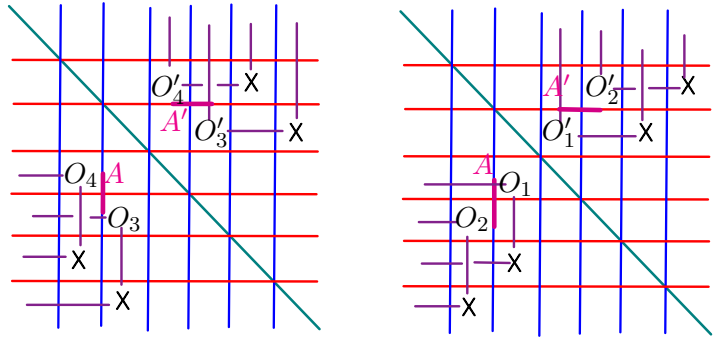
\begin{figure}
    \centering
    \begin{overpic}[width=0.7\textwidth]{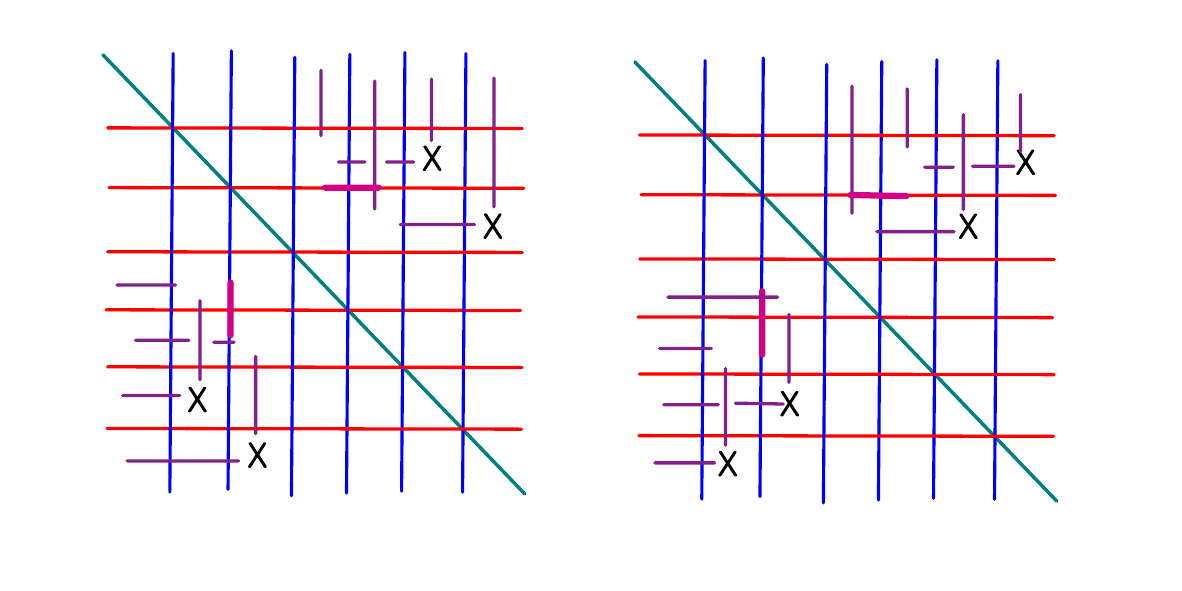}
			\put(15,25) {$O_4$}
			\put(20,20) {$O_3$}
			\put(25,35) {$O_4'$}
			\put(30,30) {$O_3'$}

            \put(65,24) {$O_1$}
			\put(60,20) {$O_2$}
			\put(70,30) {$O_1'$}
			\put(75,35) {$O_2'$}

            \put(19.5,25) {\textcolor{Magenta}{$A$}}
			\put(26,31) {\textcolor{Magenta}{$A'$}}

            \put(62,26) {\textcolor{Magenta}{$A$}}
			\put(70,35) {\textcolor{Magenta}{$A'$}}
		\end{overpic}
	
    \caption{Grid diagram for a paired saddle move.}
    \label{fig:saddle map}
\end{figure}

In Figure~\ref{fig:saddle map}, we mark a pair of distinguished segment $(\bfA,\bfA')$ on the pair of $\alpha$, $\beta$-circles where the saddle moves happen. We call their complements in the circles $(\bfB,\bfB')$. Then we let $\cA$ and $\cB$ be the set of real grid states such that $\xv\cap \bfA\ne\emptyset$ and  $\xv\cap \bfB\ne\emptyset$, respectively. For $L$ and $L'$ related by the paired merge and split moves, we can find grid diagram $\cH$ and $\cH'$ with the standard local pictures in Figure~\ref{fig:saddle map} and define maps \[\sigma_p:\frac{c\GCR^{-}(\cH)}{U_1=U_2}\to c\GCR^-(\cH'),\quad \text{and } \quad \mu_p\colon c\GCR^-(\cH')\to  \frac{c\GCR^{-}(\cH)}{U_1=U_2}\] by \[\sigma_p(\xv)=\begin{cases}
    u^2\cdot \xv & \text{if } \xv\in \cA\\
    \xv  & \text{if } \xv\in \cB
\end{cases}, \quad \text{and } \quad \mu_p(\xv)=\begin{cases}
    \xv & \text{if } \xv\in \cA\\
    u^2\cdot \xv  & \text{if } \xv\in \cB
\end{cases}.\]  
Similarly, for the paired merge-split move, we have 
\[\sigma\colon\frac{c\GCR^{-}(\cH)}{U_1=U_2}\to \frac{c\GCR^{-}(\cH')}{U_3=U_4},\quad \text{and } \quad \mu\colon \frac{c\GCR^{-}(\cH')}{U_3=U_4}\to  \frac{c\GCR^{-}(\cH)}{U_1=U_2}\] by \[\sigma(\xv)=\begin{cases}
    u^2\cdot \xv & \text{if } \xv\in \cA\\
    \xv  & \text{if } \xv\in \cB
\end{cases}, \quad \text{and } \quad \mu(\xv)=\begin{cases}
    \xv & \text{if } \xv\in \cA\\
    u^2\cdot \xv  & \text{if } \xv\in \cB
\end{cases}.\] 

It is obvious that various compositions are all equal to the multiplication by $u^2$. Now we check that they are chain maps: If $\xv$ and $\yv$ both belong to $\cA$ or $\cB$, then it is obvious. Then the claim follows by noting that any real rectangle from $\xv\in \cA$ to $\yv\in \cB$ contains exactly one pair in $(O_1,O_1')$ and $(O_2,O_2')$ and none of $(O_3,O_3')$ or $(O_4,O_4')$, while any real rectangle from $\xv\in \cB$ to $\yv\in \cA$ contains exactly one pair in $(O_3,O_3')$ or $(O_4,O_4')$ and none of $(O_1,O_1')$, $(O_2,O_2')$.

Lastly, the grading shift property follows from the calculation in the proof of \cite[Proposition~8.3.1]{OSS2015grid} and the facts that during both paired split move and paired merge-split move, $l_f$ and $\vert\xv\cap C\vert$ keep unchanged.
\end{proof}

\subsection{Adding unknots to a link}
\begin{defn}
  Let $L$ be a (generalized) strongly invertible link, then we denote by $\cU_{d_1,d_2}(L)$ the link obtained from $L$ by adding $d_1$ strongly invertible unknots and $d_2$ two-component unlinks with their two components interchanged by $\tau$.  
\end{defn}
\begin{lem}
If $L$ is a link of the form $\cU_{d_1,d_2}(K)$ for a strongly invertible knot $K$, then for any compatible choice of auxiliary data $\fra_K$ and $\fra_L$ on $K$ and $L$, \[\tau^R_{\min}(L,\fra_L)=\tau^R_{\max}(L,\fra_L)=\tau^R(K,\fra_K).\]
\end{lem}
\begin{proof}
This follows from the lemma below by induction.
\end{proof}
\begin{lem}
Let $L$ be a (generalized) strongly invertible link with $l_f\ge 1$, and let $L'$ be $\cU_{1,0}(L)$ or $\cU_{0,1}(L)$. Then for any compatible choice of auxiliary data $\fra$ and $\fra'$, there is a bigraded isomorphism of $\F[u]$-modules \[c\GHR^-(L',\fra')=c\GHR^-(L,\fra)\oplus c\GHR^-(L,\fra).\] 
\end{lem}
\begin{proof}
This follows from a neck-stretching argument on a minimal real Heegaard diagram and the invariance of real link Floer homology. 
\end{proof}

\subsection{Proof of the main theorem}\label{sub: Proof of the main theorem for tauR}
\begin{prop}\label{prop:change of tau is bounded by number of saddle moves}
Suppose that $L_1$ and $L_2$ are two strongly invertible links (not generalized!) with same number of components so that $L_2$ is constructed from $L_1$ by $g$ paired saddle moves. Fix any compatible auxiliary data $\fra_i$ $(i=1,2)$ on them. If $\tau^R_{\min}(L_i,\fra_i)=\tau^R_{\max}(L_i,\fra_i)$ for $i=1,2$, then\[\vert \tau^R_{\min}(L_1,\fra_1)-\tau^R_{\min}(L_2,\fra_2)\vert =\vert \tau^R_{\max}(L_1,\fra_1)-\tau^R_{\max}(L_2,\fra_2)\vert\le g.\] 
\end{prop}
\begin{proof}
Since $l_p=0$ at the beginning and the end, the number of paired merge moves and paired split moves must be equal, say $k$. And the remaining $g-2k$ are paired merge-split moves. Then the proposition follows from iterated use of Theorem~\ref{thm:comparison of real tau sets}.
\end{proof}

\begin{prop}\label{prop:real tau is kept by merging unknots}
Suppose that $K_1$ is a strongly invertible knot and $L_1$ is a strongly invertible link obtained from $\cU_{d_1,d_2}(K_1)$ by $d_2$ paired merge moves, then for any compatible auxiliary data on $\fra_K$ and $\fra_L$ on them,  $\tau^R(K_1,\fra_K)=\tau^R_{\min}(L_1,\fra_L)=\tau^R_{\max}(L_1,\fra_L)$. 
\end{prop}
The proposition is proved by induction on $d_2$ and the following lemma.
\begin{lem}\label{lem: min=max is kept by a single merge move}
Suppose that $(L',\fra')$ is strongly invertible link with auxiliary data with the property that $\tau^R_{\min}(L',\fra')=\tau^R_{\max}(L',\fra')$ and $(L,\fra)$ is obtained from $(L',\fra')$ by a paired merge move, then $\tau^R_{\min}(L,\fra)=\tau^R_{\max}(L,\fra)$.
\end{lem}
\begin{proof}
This follows from Theorem~\ref{thm:comparison of real tau sets}.
\end{proof}
Before proving Theorem~\ref{thm:real taut bounds genus of cobordism}, we need an extra topological input. 

\begin{prop}\label{prop:rearrangement of saddlemoves in an equivariant cobordism}
Suppose that two strongly invertible knots $K_1$ and $K_2$ can be connected by an equivariant smooth, oriented genus $g$ cobordism $W$ inside $(S^3\times [-1,3], \tau\times\id)$. Then there are strongly invertible links $L_1$ and $L_2$ and pairs of non-negative integers $(b_1,b_2)$ and $(d_1,d_2)$ with the following properties. \begin{itemize}
    \item $\cU_{b_1,b_2}(K_1)$ can be obtained from $L_1$ by simultaneously performing $b_2$ paired saddle moves.
    \item $L_1$ and $L_2$ can be connected by a sequence of $g+d_1$ paired saddle moves.
    \item  $\cU_{d_1,d_2}(K_2)$ can be obtained from $L_2$ by simultaneously performing $d_2$ paired saddle moves.
\end{itemize} 
\end{prop}
We sketch a proof of this proposition following the strategy from \cite[Appendix B]{OSS2015grid}. Here, we need an extra trick in handle trading, which was inspired by a discussion with Gary Guth.

\begin{proof}
Let $p$ be the projection $S^3\times [-1,3]\to[-1,3]$. By isotoping $W$ equivariantly, we may assume that $f_W=p|_{W}:W\to [-1,3]$ is an equivariant Morse function with critical values in $[0,2]$. This induces a handle decomposition of $W$. For $i=0,1,2$, and any $i$-handle $h$, it may be itself invariant under the involution or there is another $i$-handle $h'$ such that $\tau$ interchanges $h$ and $h'$ by a diffeomorphism. We call the first kind a \emph{fixed handle} and the second kind a \emph{pair of handles}. Our claim is that we can eliminate fixed $1$-handles at the cost of increasing the total number of handles. From the Morse function point of view, it is equivalent to that by isotoping $W$ equivariantly, we can achieve that $f_W$ has its index $1$ critical points appear in pairs. This is done as follows: If a fixed 1-handle introduces more fixed points to the level set, then we can trade it with a fixed $0$-handle and a pair of $1$-handles; if a fixed 1-handle decreases the number of fixed points in the level set, then we can trade it with a fixed $2$-handle and a pair of $1$-handles. The procedure is illustrated in Figure~\ref{fig:trading handles}.

\begin{figure}
    \centering
    \includegraphics[width=1.0\linewidth]{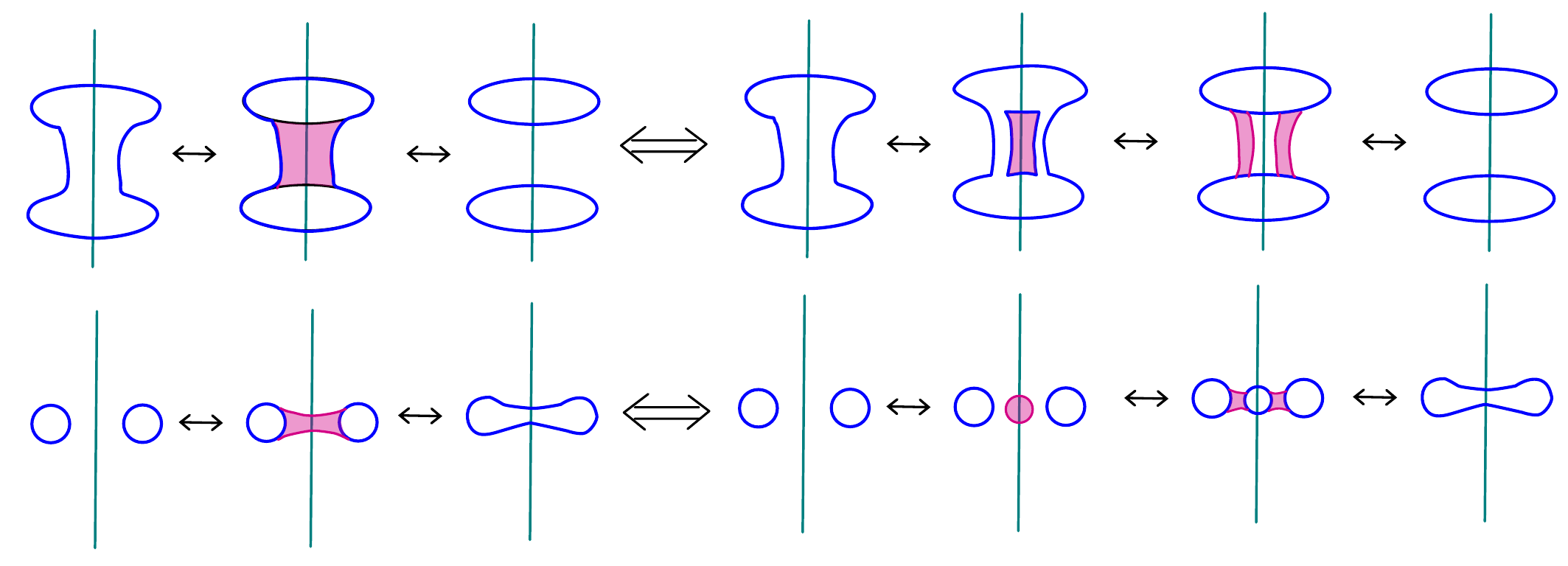}
    \caption{Eliminate fixed $1$-handles.}
    \label{fig:trading handles}
\end{figure}

After this, we can assume that $f_W$ has all its index 1 critical points appear in pairs. Then as in the prove of \cite[Lemma B.5.3]{OSS2015grid}, we can connect fixed index $0$ critical points to $S^3\times \{0\}$ using disjoint arcs inside $\fix(\tau)\times [0,2]$ and paired index $0$ critical points to $S^3\times \{0\}$ using disjoint and equivariant pairs of arcs. We may assume all these arcs are disjoint by choosing them generically. In this way, we can push all index $0$ critical points down to $S^3\times \{0\}$. In a similar fashion, we can push all index $2$ critical points up to $S^3\times \{2\}$.

Now, we deal with $1$-handles. The key observation is that since all these handles appear in pairs, they must be away from the fixed set. We can just perform moves in \cite[Figure B.25]{OSS2015grid} in pairs to arrange all these index $1$ critical points to appear simultaneously at level $f_W^{-1}(1)$.

After this, we take $\cU_{b_1,b_2}(K_1)$ as $f^{-1}_{W}(0.1)$ and $\cU_{d_1,d_2}(K_2)$ as $f^{-1}_{W}(1.9)$. Since the cobordism is connected, for each of the $b_2$ pairs of 2-component unlinks, up to some band slides along other pair of $1$-handles, there must be a pair of bands (a pair of equivariant $1$-handles) joining them to $K_1$ or one of the $d_1$ strongly invertible unknots. For each of the $b_2$ pairs, we pick such a pair of $1$-handles, attaching them to $\cU_{b_1,b_2}(K_1)$ to get a strongly invertible link $L_1$. Playing this trick upside down, we obtain $L_2$. Now the cobordism appear as a composition $K_1\to\cU_{b_1,b_2}(K_1)\to L_1\to L_2\to\cU_{d_1,d_2}(K_2)\to K_2$. Since paired $1$-handles do not change the number of fixed points, so $b_1=d_1$ and $\vert L_1\vert =\vert L_2\vert$. Finally, an Euler characteristic calculation shows that the middle cobordism $L_1\to L_2$ has $\chi=-2g-2d_1$. Since it has only paired index $1$ critical points, it follows that the cobordism consists of $g+d_1$ pairs of paired saddle moves.
\end{proof}

\begin{proof}[Proof of Theorem~\ref{thm:real taut bounds genus of cobordism}]
Without loss of generality, we assume that $p|_{S}$ is an equivariant Morse function and that $p|_{S}$ has the minimal number of fixed critical points for all $S'$ isotopic to $S$. Again, to simplify the notation, we omit auxiliary choices from it during the proof. Apply the construction in Proposition~\ref{prop:rearrangement of saddlemoves in an equivariant cobordism} to $S$, we get a decomposition of $S$ into three stages: $K_1\to L_1\to L_2\to K_2$. Note that rearranging heights of critical points does not increase the number of equivariant critical points and the trading process in Figure~\ref{fig:trading handles} also keeps it unchanged. In the final decomposition, we only have fixed index $0$ and $2$ critical points and the numbers of them are both $d_1$. Thus, we have $2d_1=m$. By Proposition~\ref{prop:real tau is kept by merging unknots}, $\tau^R(K_1)=\tau^R_{\min}(L_1)=\tau^R_{\max}(L_1)$,  $\tau^R(K_2)=\tau^R_{\min}(L_2)=\tau^R_{\max}(L_2)$. By Proposition~\ref{prop:change of tau is bounded by number of saddle moves}, $\vert \tau^R_{\min}(L_1)-\tau^R_{\min}(L_2)\vert =\vert \tau^R_{\max}(L_1)-\tau^R_{\max}(L_2)\vert\le g+d_1$. It follows that $\vert \tau^R(K_1)-\tau^R(K_2)\vert \le g+d_1$. This concludes the proof.
\end{proof}

We end this subsection with an estimation of $m$. 

\begin{thm}\label{thm:minimizing m on a slice surface}
Let $K$ be a strongly invertible knot in $S^3$. If $S\subset (B^4,\tau_c)$ is an equivariant slice surface of $K$ with a single arc as fixed set, then $m=1$.
\end{thm}
\begin{proof}
By isotoping $S$ equivariantly to avoid the origin of $B^4$, we can regard $S$ as a slice surface inside $S^3\times I$ with $\partial S=K\subset S^3\times \{1\}$ and $S\cap S^3\times \{0\}=\emptyset$. After a suitable perturbation if needed, we may assume that $h$ restricts to a Morse function on $S$. Then using similar tricks as shown in Figure~\ref{fig:trading handles}, we can 
\begin{itemize}
    \item replace each fixed $1$-handle whose attaching region is away from the fixed set with a pair of $1$-handles and a fixed $0$-handle;
    \item replace each fixed $2$-handle with a pair of $2$-handles and a fixed $1$-handle. In this way, we can make sure that $S$ has a handle decomposition with no fixed $2$-handle.
\end{itemize}


Let $A$ denote the fixed set of $\tau|_S$. By the equivariant assumption, a fixed critical point of $h|_S$ on $A$ must also be a critical point of $h|_A$ and vice versa. Then our goal is to show that we can isotope $S$ relative to its boundary to make $h|_A$ has a single local minimum, or equivalently, $h|_S$ has a single fixed critical point, which is of index $0$.

Suppose that we have an unwanted local maximum on $A$, then there must be more than one local minimum on $A$. After our replacement above, for any $x\in \crit (h|_A)$, 
\begin{itemize}
    \item if $\ind_A(x)=0$ then $\ind_S(x)=0$, which gives rise to a disk $D_x$ in the decomposition of $S$;
    \item if $\ind_A(x)=1$, then $\ind_S(x)=1$, which gives rise to a band $B_x$ in the decomposition of $S$.
\end{itemize} 
Choose a pair of local minimum and maximum $(a,b)$ such that $a$ is the local minimum that nearest to one of the end of $A$, $a$ and $b$ are adjacent on $A$. This is achievable since critical points with $\ind_A=0$ or $1$ must appear alternatingly along $A$ and the number of critical points with $\ind_A=0$ is one bigger than that of $\ind_A=1$. Forgetting $S$, $(a,b)$ are of course cancelable. The local pictures of $D_a$ and $B_b$ are shown in the left of Figure~\ref{fig:cancelling 0-1 pair}.
\begin{figure}
    \centering
    \begin{overpic}[width=0.5\textwidth]{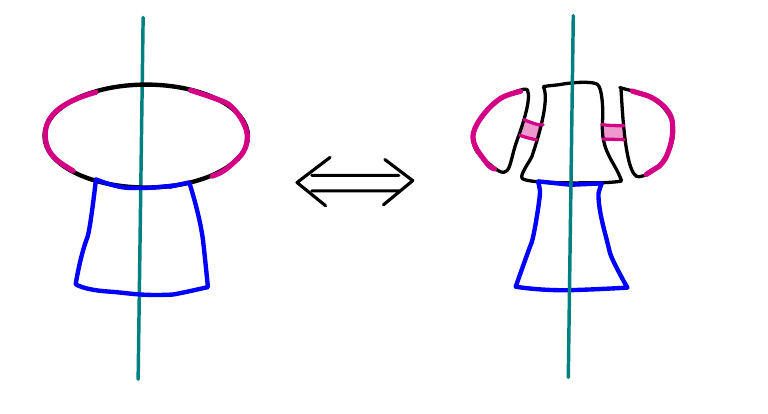}
			\put(28,10) {\textcolor{Blue}{$B_b$}}
			\put(30,40) {$D_a$}
			
			\put(62,35) {$D$}
			\put(82,35) {$D'$}

            \put(75,45) {$D'_a$}
            \put(80,10) {\textcolor{Blue}{$B_b$}}

		\end{overpic}
    \caption{Canceling a fixed $0$-$1$ handle pair.}
    \label{fig:cancelling 0-1 pair}
\end{figure}
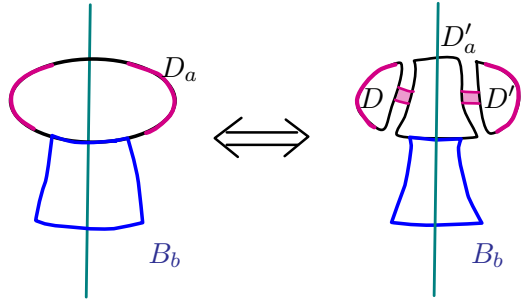

By assumption, there is no fixed $1$-handle attaching to the upper fixed point of $D_a$ and since $A$ is an arc, the bottom edge of $B_b$ lies on another fixed index $0$ disk. Thus, $a$ and $b$ are geometrically cancelable. Before performing the cancellation, some handleslides are needed to free $D_a$ from other $1$-handles. These $1$-handles must appear in pairs. We replace $D_a$ with a smaller fixed disk $D_a'$, a pair of $1$-handles $(B,B')$ (shown in pink) and a pair of $0$-handles $(D,D')$, so that $D_a'$ is actually a neighborhood of $A$ in $D_a$ and $B,B'$ are the only bands with ends on $D_a'$. See the right of Figure~\ref{fig:cancelling 0-1 pair}. After this modification, we only need to slide $B$ and $B'$ over $B_b$ and then we can perform the cancellation of $D_a'$ and $B_b$. A concrete example is shown in Figure~\ref{fig:Surface of T23 as example}. On the left-hand-side, we see an equivariant slice surface of the trefoil with its standard handle decomposition given by Seifert circles. On the right-hand-side, we show how we can perform the handleslides and cancellation to eliminate a fixed $(0,1)$-handle pair at the cost of introducing pairs of interchanged handles.
 \begin{figure}
    \centering
    \includegraphics[width=1.0\linewidth]{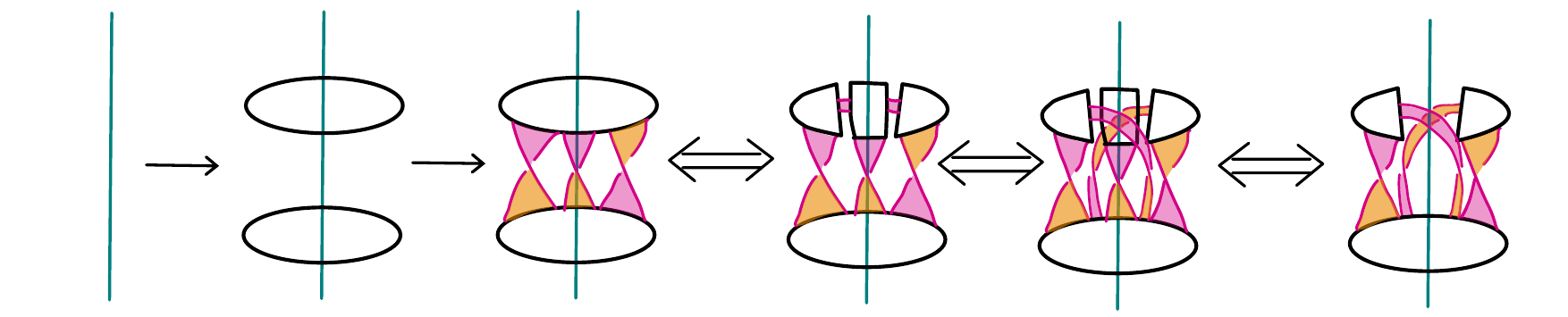}
    \caption{Eliminate fixed $1$-handles from the handle decomposition of an equivariant slice surface of $T_{2,3}$.}
    \label{fig:Surface of T23 as example}
\end{figure}

Apply the procedure above repeatedly until $h|_A$ has a single critical point. This concludes the proof of our theorem.
\end{proof}

\begin{cor}\label{cor:tau^R bounds equiv slice genus+m/2 enhanced}
Let $K$ be a strongly invertible knot in $S^3$ and $S$ be a connected equivariant slice surface of $K$ in $(B^4,\tau_c)$ with a connected fixed set. Then the fixed set of $\tau|_S$ is necessarily an arc and \[\vert \tau^R(K,\fra)\vert\le g(S).\] In particular, if $K$ is equivariantly slice, then $\tau^R(K,\fra)=0$ for any choice of auxiliary data.

\end{cor}
\begin{proof}
The first statement follows from Corollary~\ref{cor:real tau bounds tilde g4+m} and Theorem~\ref{thm:minimizing m on a slice surface}. The second one follows from the first and the fact that any involution on $D^2$ has connected fixed set.
\end{proof}

The proof of Theorem~\ref{thm:minimizing m on a slice surface} also gives a control on $m$ on more general knot cobordisms.

\begin{prop}\label{prop:minimize m on cobordism}
Let $S$ be an equivariant cobordism between two strongly invertible knots $K_1$ and $K_2$ inside $(S^3\times I, \tau\times \id)$. \begin{enumerate}
    \item If $\fix(\tau\times\id)\cap S$ is a pair of arcs, each has two ends on $S^3\times \{0\}$ and $S^3\times \{1\}$, respectively, then $m=0$.
    \item If $\fix(\tau\times\id)\cap S$ is a pair of arcs among which one has two ends on $S^3\times \{0\}$ and the other has two ends on $S^3\times \{1\}$, then $m=2$. 
\end{enumerate} 
\end{prop}

\begin{cor}
Let $S$ be an equivariant cobordism between two strongly invertible knots $K_1$ and $K_2$ inside $(S^3\times I, \tau\times \id)$. 
\begin{enumerate}
    \item If $\fix(\tau\times\id)\cap S$ is a pair of arcs, each has two ends on $S^3\times \{0\}$ and $S^3\times \{1\}$, respectively, then $\vert\tau^R(K_1,\fra_1)-\tau^R(K_2,\fra_2)\vert \le g(S)$, for any compatible choice of auxiliary data.
    \item If $\fix(\tau\times\id)\cap S$ is a pair of arcs in which one has two ends on $S^3\times \{0\}$ and the other has two ends on $\R^3\times \{1\}$, then $\vert\tau^R(K_1,\fra_1)- \tau^R(K_2,\fra_2)\vert \le g(S)+1$, for any compatible choice of auxiliary data.
\end{enumerate}
\end{cor}
\begin{proof}
The estimation follows from Theorem~\ref{thm:real taut bounds genus of cobordism} and Proposition~\ref{prop:minimize m on cobordism}.
\end{proof}

The equivariant concordance group $\widetilde{\cC}$ of directed strongly invertible knots first appeared in \cite{Sakuma}. Many recent papers devoted to the study of its properties, for example \cite{MillerPowellequivslicegenera}, \cite{DiPrisaequivalgconcordance} and \cite{DMSequivariantknotandHFK}. Following them, we introduce the notion of concordance between strongly invertible knots with auxiliary data in $(S^3,\tau)$.
\begin{defn}
Let $(K_1,\fra_1)$ and $(K_2,\fra_2)$ be two strongly invertible knots with auxiliary data in $(S^3,\tau)$. We say they are \emph{smoothly (topologically) equivariantly concordant} if there exists a involution $\rho$ on $S^3\times I$ restrict to $\tau_{\std}$ on the two ends and an equivariantly properly embedded annulus $A\subset S^3\times I$ which is smooth (locally flat) such that \begin{itemize}
	\item $A\cap S^3\times \{0\}=K_1$ while $A\cap S^3\times \{1\}=K_2$; 
	\item the two fixed arcs on $A$ connects the $O$, $X$-base points on $K_1$ to the $O$, $X$-base points on $K_2$, respectively. 
\end{itemize}
\end{defn}

We define the connected sum of two strongly invertible knots with auxiliary data to be the connected sum along a fixed band which identifies an $O$-base point with an $X$-base point. In this way, we obtain the equivariant concordance group of strongly invertible knots with auxiliary data. This is similar to the equivariant concordance group of directed strongly invertible knots (see \cite[Definition~2.11]{MillerPowellequivslicegenera} or \cite[Definition~2.7]{DiPrisaequivalgconcordance}), as we remarked after Definition~\ref{def:strongly invertible knots with extra data}.

\begin{cor}\label{cor:tauR is a concordanc invariant}
If $(K_1,\fra_1)$ and $(K_2,\fra_2)$ are equivariantly concordant in $(S^3\times I,\tau\times \id)$, then $\tau^R(K_1,\fra_1)= \tau^R(K_2,\fra_2)$. In particular, if we consider smooth equivariant concordance group $\widetilde{\cC}_{\std}$ of strongly invertible knot with auxiliary data modulo equivariant concordance in $(S^3\times I,\tau\times \id)$, then $\tau^R$ is a well-defined and nontrivial map $\widetilde{\cC}_{\std}\to \Z$. 
\end{cor}

Here, we add a subscript ``$\std$'' to note that we only allow the standard involution on $S^3\times I$ and require the concordance to be smooth in contrast to the existing notion of concordance group, which allows any involution on $S^3\times I$ extending the standard one on ends and locally flat embedded annuli. We have to add restrictions due to the combinatorial and Morse theoretic nature of our argument. But the author conjectures that the more general result is also true.

\begin{remark}\label{rmk:m is unavoidable}
Originally, the author expected that we can always reduce $m$ to $1$ when concerning equivariant slice surface in $(B^4,\tau_c)$ using the trick above and the fact that $\fix(\tau_c)\subset B^4$ is simply connected, which may allow us to do  ``compression'' along a fixed disk whenever a circle component appear in the fixed set. However, the author found that there is a framing obstruction lying in $\pi_1
(\mathrm{SO}(2))\cong \Z$. It was then pointed out by Lisa Piccirillo that such obstruction is theoretically essential as shown in \cite[Construction~2.3 \& Remark~2.4]{conway2025simplyslicingknots}. The author currently do not have a concrete example of a strongly invertible $K$ with an equivariant slice surface $S$ such that \[\tau^R(K,\fra)> g(S), \text{ but } \tau^R(K,\fra)\le g(S)+\frac{m-1}{2},\] but a potential example will be given in Section~\ref{sec:Examples and application}.

After the author posted the first draft of this paper, Alessio Di Prisa pointed out that this framing obstruction also appeared in his bound of equivariant slice genus via algebraic concordance in \cite[Section~4]{DiPrisaequivalgconcordance}. 
\end{remark}

\begin{example}\label{ex:equivariant ribbon surface}
Recall that there is a notion of \emph{ribbon surface} bounded by a knot which is an immersed surface in $S^3$ with $K$ as boundary with only ribbon singularities at those immersed points. They shall become slice surface when push into $B^4$. (It is also well-known that not every slice surface is ribbon.) Now, we consider an \emph{equivariant ribbon surface} $S$ for a strongly invertible knot $K$ in $S^3$. More precisely, $S$ is a ribbon surface of $K$ in $S^3$ so that $\tau$ restricts an automorphism of $S$. We claim that $g(S)\le \tau^R(K,\fra)$ for any choice of $\fra$. It follows from Corollary~\ref{cor:tau^R bounds equiv slice genus+m/2 enhanced} and the fact that $\fix(\tau)\subset S^3$ is just a copy of $S^1$. The later implies that the fixed set of an equivariant ribbon surface must be an arc, since we have no singularity at boundary.  
\end{example}

\begin{example}\label{ex:bounding butterfly 4-genus}
In \cite[Section~3.1]{BIequiv4genusofsiandperidicknots}, the authors introduced the notion of \emph{butterfly 4-genus} of a strongly invertible knot $K$ as \[\widetilde{bg}_4(K)= \min\{g(S)| \text{$S$ is a smooth invariant slice surface of $K$ with a separating arc in the fixed point set}\}.\] By assumption, the fixed set of a butterfly surface of $K$ is an arc, so Corollary~\ref{cor:tau^R bounds equiv slice genus+m/2 enhanced} implies that \[\vert \tau^R(K,\fra) \vert \le g(S) \]for any butterfly surface $S$ of $K$ inside $(B^4,\tau_c)$ and any choice of auxiliary data $\fra$ on $K$.
\end{example}

\subsection{Interlude: Failure of Kunneth principle}\label{sub:Kunneth principle fails}

Careful readers may have noticed that in Corollary~\ref{cor:tauR is a concordanc invariant}, we said that $\tau^R$ induces a well-defined map from $\widetilde{\cC}_{\std}$ to $\Z$ instead of a group homomorphism.  The subtlety comes from the facts that \begin{itemize}
	\item the inverse of $[(K,\fra)]$ in $\widetilde{C}$ is provided by $[(m(K),m(\fra)^{i})]$ and currently, the author do not know how the $\tau^R$-invariants of these two knots are related.
	\item the proof of $\tau$ is additive under connected sum fails for $\tau^R$- more precisely, the connected sum formula given in \cite[Section~7]{Holomorphicdisksandknotinvariants} is not true for $\HFKR^-$ or even $\widehat{\HFKR}$. 
	\end{itemize}  
Among all the examples we have computed, a tensor product formula is rarely true, but \[\tau^R(K_1\#K_2, \fra_1\#\fra_2)= \tau^R(K_1,\fra_1)+ \tau^R(K_2,\fra_2)\] always holds, so the author still expect $\tau^R$ to be a group homomorphism.

In the rest of this subsection, we provide some examples to illustrate real knot Floer homology of connected sums. A few more examples will be provided in Appendix~\ref{app:Calculation results for knots with small crossing numbers}.

In Figure~\ref{fig:grid diagram of connected sum}, we show real grid diagrams for equivariant connected sum $3_1\#4_1$, $3_1\#5_1$ and $3_1\#5_2$. Here, we abuse the notation from Rolfsen table for the knots and their mirror images.
	
\begin{figure}
		\centering
		\begin{overpic}[width=0.9\textwidth]{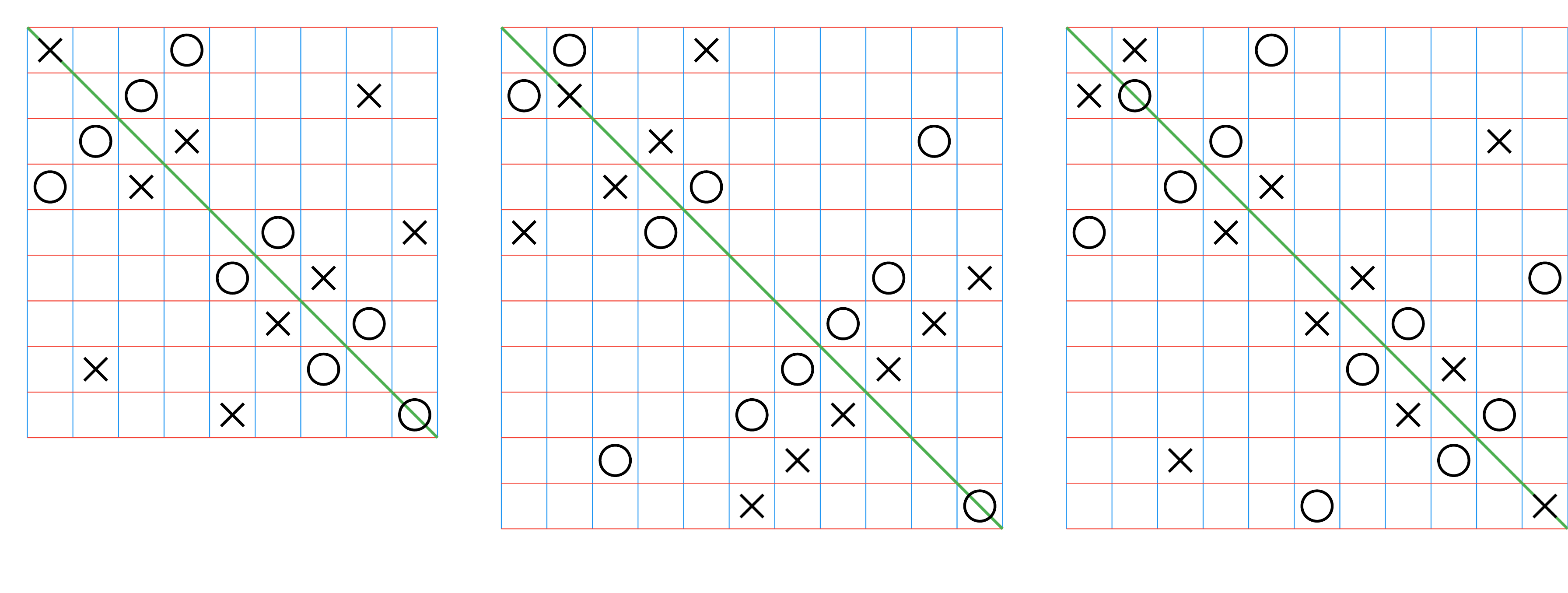}
			\put(9,0) {$(3_1\#4_1,\fra_1)$}
			\put(42,0) {$(3_1\#5_1,\fra_2)$}
			\put(78,0) {$(3_1\#5_2,\fra_3)$}
		\end{overpic}
		\caption{Real grid diagram for equivariant connected sums.}
		\label{fig:grid diagram of connected sum}
\end{figure}
	
Using \cite{ZhenkunLioythonprgram}, we calculate that \[\widehat{\GHR}(3_1\#4_1, \fra_1) \cong\F_{(-1,-1)}\oplus \F_{(0,0)}^3 \oplus \F_{(0,1/2)}\oplus \F_{(1,1/2)}\oplus  \F_{(1,1)};\]
\[\GHR^-(3_1\#4_1, \fra_1) \cong\F[u]_{(0,0)}\oplus \F_{(0,1/2)} \oplus \F_{(1,1)}\oplus \F[u]/(u^2)_{(0,0)};\]
	
\[\widehat{\GHR}(3_1\#5_1, \fra_2) \cong\F_{(0,-3/2)}\oplus \F_{(0,-1)}^2 \oplus \F_{(1,-1/2)}\oplus \F_{(1,0)}^3\oplus  \F_{(2,1/2)} \oplus  \F_{(2,1)}^2 \oplus  \F_{(3,3/2)};\]
\[\GHR^-(3_1\#5_1, \fra_2) \cong\F[u]_{(3,3/2)}\oplus \F_{(0,-1)} \oplus \F_{(1,0)}\oplus  \F_{(2,1)}\oplus \F[u]/(u^2)_{(1,0)} \oplus \F[u]/(u^2)_{(2,1)};\]
	
\[\widehat{\GHR}(3_1\#5_2,\fra_3)=\F_{(0,-1/2)}\oplus \F_{(0,0)}\oplus \F_{(1,1/2)},\]
\[\GHR^-(3_1\#5_2,\fra_3)=\F[u]_{(1,1/2)}\oplus \F_{(0,0)}. \]

By comparing these with the results from Appendix~\ref{app:Calculation results for knots with small crossing numbers}, we see that the obvious connected sum formula fails on the first two and holds for the last one trivially. Here, by trivially, we mean that $5_2$ with the indicated involution and auxiliary data has trivial real knot Floer groups.
	
It is well-known that the connected sum of strongly invertible knots depends on the choice of axes (see~\cite{Sakuma}) and it was shown in \cite{DiprisaTheequivariantconcordancegroupisnotabelian} that the equivariant concordance group of directed strongly invertible knots is not abelian, in contrast to the usual concordance group. Actually, the real knot Floer group detects this. We will illustrate this by an example. If we write $(3_1\#4_1,\fra_1)$ as $(4_1,\fra)\#(3_1,\frb)$, then the first diagram in Figure~\ref{fig:grid diagram of connected sum} with the roles of $\bfO$ and $\bfX$ reversed represents $(3_1,\frb^{i,r})\#(4_1,\fra^{i,r})$. Here, we use the oriented half axis from $O$ to $X$ as the choice of direction and follow the convention of equivariant connected sum in~\cite{Sakuma}. For this second composite knot, we have \[\widehat{\GHR}\cong\F_{(-1,-1)}\oplus \F_{(-1,-1/2)}\oplus \F_{(0,-1/2)} \F_{(0,0)}^3 \oplus \F_{(1,1)};\]
\[\GHR^- \cong\F[u]_{(0,0)}\oplus \F_{(-1,-1/2)} \oplus \F_{(0,0)}\oplus \F[u]/(u^2)_{(1,1)},\] which are different from $\GHR^\circ(3_1\# 4_1,\fra_1)$ computed above. However, as we have remarked in Example~\ref{ex:trefoil and figure 8}, for strongly inversions on simple knots as $3_1$ and $4_1$, real knot Floer groups do not depends on a choice of auxiliary data. In particular, we have \[\GHR^\circ(3_1,\frb)\cong \GHR^\circ(3_1,\frb^{i,r}) \quad \GHR^\circ(4_1,\fra)\cong \GHR^\circ(4_1,\fra^{i,r}) \] for $\circ=\widehat{}, -$. On one hand, this tells us that real knot Floer groups detect the order of connected sum. On the other hand, it also tells us that if one tries to develop a connected sum formula for this theory, one must take the order of connected sum into account. 

Based on the first observation, the author expects that the real knot Floer homology is able to detect the non-commutativity of $\widetilde{\cC}$. It was pointed out by Alessio Di Prisa that if we can find a concordance invariant originated from $\HFKR^\circ$ such that it takes different value on the standard unknot and some commutator $[(K,\fra),(K',\frb)]$, then it will provide a Floer theoretic reproof of results in~\cite{DiprisaTheequivariantconcordancegroupisnotabelian}. 

\subsection{An alternative characterization of $\tau^R$}
In this subsection, we give another characterization of $\tau^R$ using only $\widehat{\HFLR}$ and the filtration considered in Theorem~\ref{thm:spectral sequence relating HFKR and HFR}. Let 
$(K,\fra)$ be a strongly invertible knot with auxiliary data in $(S^3,\tau)$ and $\cH$ be a minimally-pointed real Heegaard diagram for $K$. Let $\underline{\cH}$ be $\cH$ with $\bfX$ forgotten which is a singly-pointed real Heegaard diagram for $(S^3,\tau)$. Then $\widehat{\CFLR}(\cH)$ can be regarded as $\widehat{\CFR}(\underline{\cH})$ with an extra filtration given by $A^R$. We can write this as \[\{0\}=\cF(K,\fra,-N)\subset \ldots \cF(K,\fra,m)\subset \cF(K,\fra,m+\frac{1}{2})\subset \ldots \cF(K,\fra,N)=\widehat{\CFR}(S^3),\] in which $N$ is any sufficiently large integer or half integer and we use the finitely generated property of $ \widehat{\CFR}(S^3)$ implicitly.

For each $m\in \frac{1}{2}\Z$, we have an inclusion map 
\[\iota_{K}^m: H_*(\cF(K,\fra,m))\to H_*(\widehat{\CFR}(S^3))=\widehat{\HFR}(S^3,\tau)\cong \F. \] Then we can define \[\tau^R(K)=2\min\{m\in \frac{1}{2}\Z| \iota_{K}^m\text{ is nontrivial}\}.\] This definition can be identified with the original one using the relationship between $\widehat{\HFLR}$ and $\HFLR^-$, as we have seen from their definition and Proposition~\ref{prop:exact sequence relating minus and hat}.

\section{Skein exact triangles}\label{sec:Skein exact triangles}
In this section, we consider two versions of skein exact triangle originating from real grid homology. As a corollary, we define real Alexander polynomial for generalized strongly invertible links and deduce that it satisfies certain skein relation.

\subsection{Oriented version}\label{sub:Oriented version}

Let $(L_+,\fro_+)$, $(L_-,\fro_-)$ and $(L_0,\fro_0)$ be three generalized strongly invertible links related by local moves shown in Figure~\ref{fig:oriented skein triple}. Such three links are said to form a \emph{real oriented skein triple}. It is obvious that $l_f$ keeps unchanged during these local moves and $l_p(L_-)=l_p(L_+)$ (define this to be $l_p$), while $l_p'=l_p(L_0)$ may equal to $l_p(L_+)+1$, $l_p(L_+)$ or $l_p(L_+)-1$. 
\begin{figure}
    \centering
    \begin{overpic}[width=0.7\textwidth]{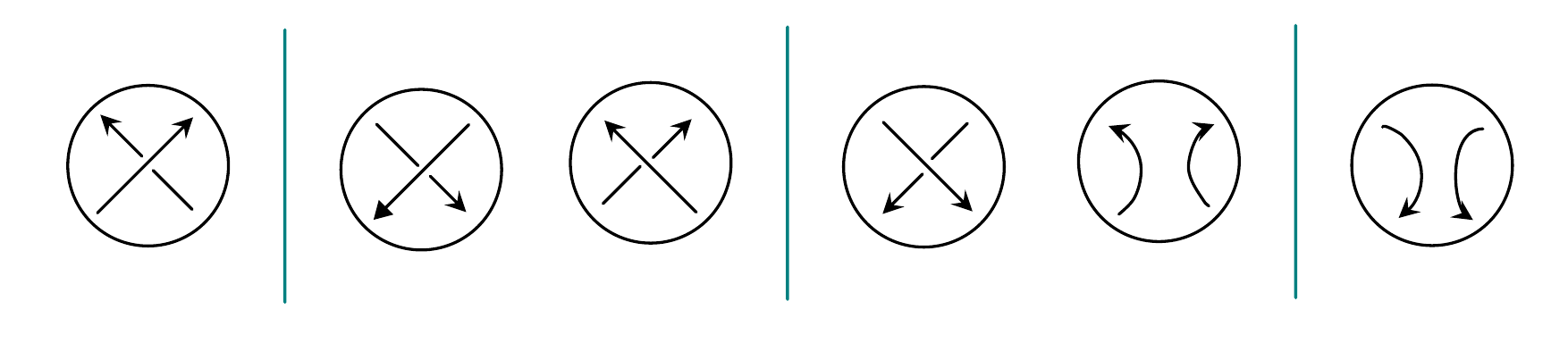}
			\put(16,-2) {$L_+$}
			\put(49,-2) {$L_-$}
            \put(81,-2) {$L_0$}
		\end{overpic}
    \caption{Real oriented skein triple.}
    \label{fig:oriented skein triple}
\end{figure}

Note that for any auxiliary data $\fra$ on one of $L_+$, $L_-$ and $L_0$, there are unique auxiliary data on the other two determined by $\fra$, since the local changes happen away from the fixed set. Throughout this section, whenever we talk about a real oriented skein triple, we fix a triple of auxiliary data constructed in this way, and abuse the same notation $\fra$ for all of them. When there is no ambiguity, this shall be omitted from the notation.

\begin{thm}\label{thm:minus oriented skein triple}
Let $(L_+, L_- ,L_0,\fra)$ be a real oriented skein triple defined as above.
\begin{itemize}
 \item If $l_p'=l_p+1$, we have a long exact sequence \begin{equation}\label{eq:l_p'=l_p+1 minus}
    \to c\GHR^-_d(L_+,s) \xrightarrow{f^-} c\GHR^-_d(L_-,s) \xrightarrow{g^-} c\GHR^-_{d-1}(L_0,s) \xrightarrow{h^-} c\GHR^-_{d-1}(L_+,s) \to 
    \end{equation}
 \item If $l_p'=l_p$, we have a long exact sequence \begin{equation}\label{eq:l_p'=l_p minus}
    \to c\GHR^-_d(L_+,s) \xrightarrow{f^-} c\GHR^-_d(L_-,s) \xrightarrow{g^-} (c\GHR^-(L_0)\otimes W)_{d-1,s} \xrightarrow{h^-} c\GHR^-_{d-1}(L_+,s) \to 
    \end{equation}
    for $W=\F_{(0,1/2)}\oplus \F_{(-1,-1/2)}$.
    \item If $l_p'=l_p-1$, we have a long exact sequence \begin{equation}\label{eq:l_p'=l_p-1 minus}
    \to c\GHR^-_d(L_+,s) \xrightarrow{f^-} c\GHR^-_d(L_-,s) \xrightarrow{g^-} (c\GHR^-(L_0)\otimes J)_{d-1,s} \xrightarrow{h^-} c\GHR^-_{d-1}(L_+,s) \to 
    \end{equation}
    for $J=\F_{(0,1)}\oplus \F_{(-1,0)}\oplus \F_{(-1,0)}\oplus \F_{(-2,-1)}$.
\end{itemize}
\end{thm}

We also have similar exact sequences for the hat flavor.

\begin{thm}\label{thm:hat oriented skein triple}
Let $(L_+, L_- ,L_0,\fra)$ be a real oriented skein triple defined as above.
\begin{itemize}
 \item If $l_p'=l_p+1$, we have a long exact sequence \begin{equation}
    \to \widehat{\GHR}_d(L_+,s) \xrightarrow{\hat{f}} \widehat{\GHR}_d(L_-,s) \xrightarrow{\hat{g}} \widehat{\GHR}_{d-1}(L_0,s) \xrightarrow{\hat{h}} \widehat{\GHR}_{d-1}(L_+,s) \to 
    \end{equation}
 \item If $l_p'=l_p$, we have a long exact sequence \begin{equation}
    \to \widehat{\GHR}_d(L_+,s) \xrightarrow{\hat{f}} \widehat{\GHR}_d(L_-,s) \xrightarrow{\hat{g}} (\widehat{\GHR}(L_0)\otimes W)_{d-1,s} \xrightarrow{\hat{h}} \widehat{\GHR}_{d-1}(L_+,s) \to 
    \end{equation}
    for $W=\F_{(0,1/2)}\oplus \F_{(-1,-1/2)}$.
    \item If $l_p'=l_p-1$, we have a long exact sequence \begin{equation}
    \to \widehat{\GHR}_d(L_+,s) \xrightarrow{\hat{f}} \widehat{\GHR}_d(L_-,s) \xrightarrow{\hat{g}} (\widehat{\GHR}(L_0)\otimes J)_{d-1,s} \xrightarrow{\hat{h}} \widehat{\GHR}_{d-1}(L_+,s) \to 
    \end{equation}
    for $J=\F_{(0,1)}\oplus \F_{(-1,0)}\oplus \F_{(-1,0)}\oplus \F_{(-2,-1)}$.
\end{itemize}
\end{thm}

We follow the strategy of \cite[Chapter~9]{OSS2015grid} to provide a proof. As for the construction of crossing change map and saddle map, the first step is realizing the local moves on real grid diagrams and constructing a combined diagram that includes information from all these diagrams. For any choice of $\fra$, we can find diagrams $\cH_+$, $\cH_-$ representing $L_+$, $L_-$ as well as two diagrams $\cH_0$ and $\cH'_0$ both representing $L_0$ with local picture showing in Figure~\ref{fig:grid diagram for oriented skein triple}.
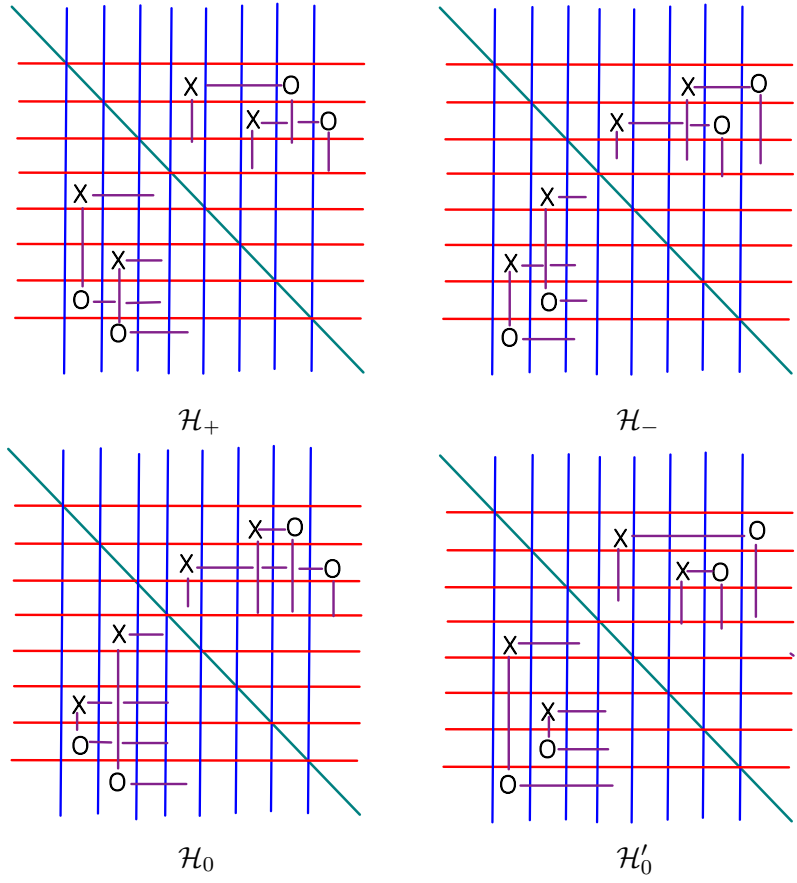
\begin{figure}
    \centering
    \begin{overpic}[width=0.7\textwidth]{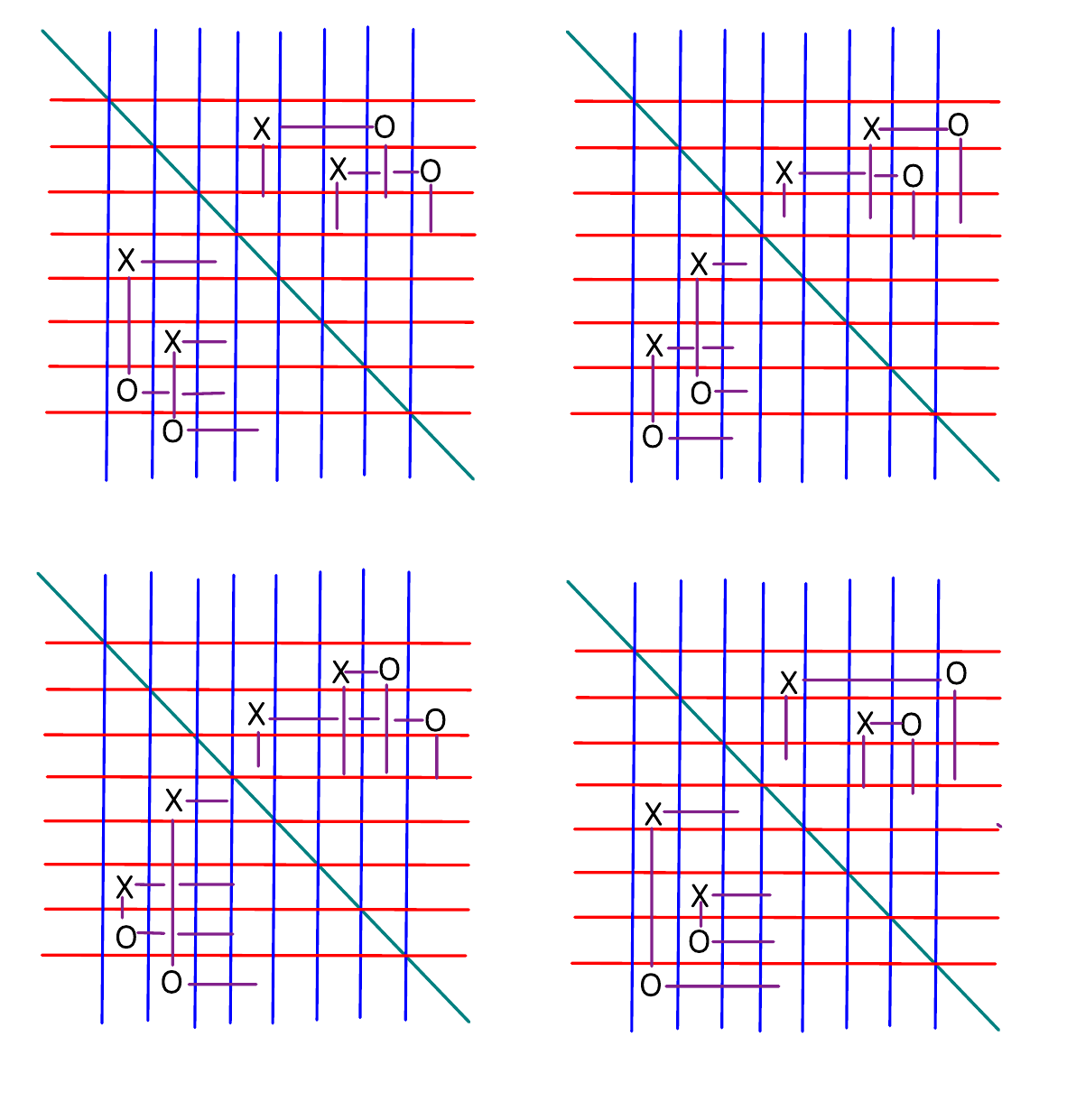}
			\put(22,51) {$\cH_+$}
			\put(70,51) {$\cH_-$}
            \put(22,3) {$\cH_0$}
            \put(70,3) {$\cH_0'$}
		\end{overpic}
    \caption{Grid diagrams for an real oriented skein triple.}
    \label{fig:grid diagram for oriented skein triple}
\end{figure}

In Figure~\ref{fig:combined diagram for oriented skein triple}, we show a combined real grid diagram adapted to the real oriented skein triple. The four diagrams share the same $\bfO$ as well as $X$-base points outside this region. They also have same set of $\alpha$, $\beta$-circles except those pairs drawn in distinguished colors. For the distinguished $X$-base points and curves: \begin{itemize}
    \item In $\cH_+$, we use $\bfX_+=\{X_1,X_1',X_2,X_2',\ldots\}$ and the pair of $\alpha$ and $\beta$-circles colored in light blue and dark red. The corresponding families will be denoted $\bm\alpha_+$, $\bm\beta_+$.
     \item In $\cH_-$, we use $\bfX_-=\{\widetilde{X}_1,\widetilde{X}_1',X_2,X_2',\ldots\}$ and the pair of $\alpha$ and $\beta$-circles colored in purple and pink. The corresponding families will be denoted $\bm\alpha_-$, $\bm\beta_-$.
    \item In $\cH_0$, we use $\bfX_0=\{Y_1,Y_1',Y_2,Y_2',\ldots\}$, $\bm\alpha_+$ and $\bm\beta_+$.
     \item In $\cH_0'$, we use $\bfX_0=\{Y_1,Y_1',Y_2,Y_2',\ldots\}$, $\bm\alpha_-$ and $\bm\beta_-$.
\end{itemize}    

\begin{figure}
    \centering
    \begin{overpic}[width=0.6\textwidth]{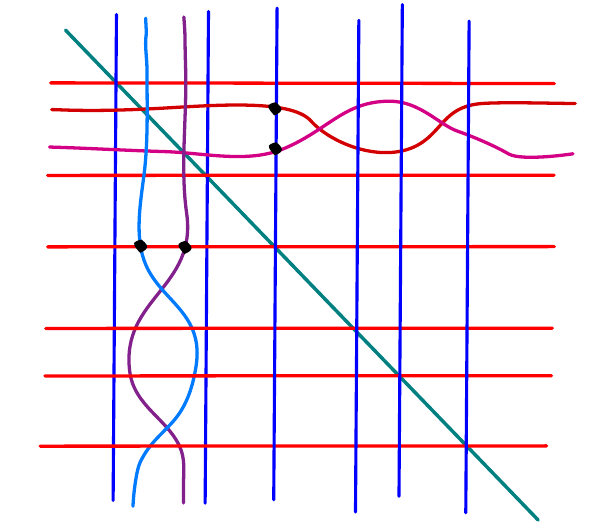}
			\put(20,49) {$c$}
			\put(27,49) {$\Tilde{c}$}
            \put(48,72) {$c'$}
            \put(48,61) {$\Tilde{c}'$}

            \put(24.5,43.5) {$X_2$}
			\put(25.5,54) {$Y_1$}
            \put(18.5,54) {$X_1$}
            \put(31.5,54) {$\widetilde{X_1}$}
            \put(24.5,35) {$Y_2$}
            
			\put(37,49.5) {$O_1$}
            \put(37,38) {$O_2$}
            \put(24.5,22) {$O_3$}
            \put(24.5,10) {$O_4$}

            \put(47,66.5) {$X_2'$}
			\put(40,66.5) {$Y_1'$}
            \put(36,72) {$X_1'$}
            \put(36,61) {$\widetilde{X_1}'$}
            \put(56,66.5) {$Y_2'$}
            
			\put(40.5,55) {$O_1'$}
            \put(50,55) {$O_2'$}
            \put(68,67) {$O_3'$}
            \put(81,67) {$O_4'$}
		\end{overpic}

    \caption{Combined diagram for oriented skein relation.}
    \label{fig:combined diagram for oriented skein triple}
\end{figure}
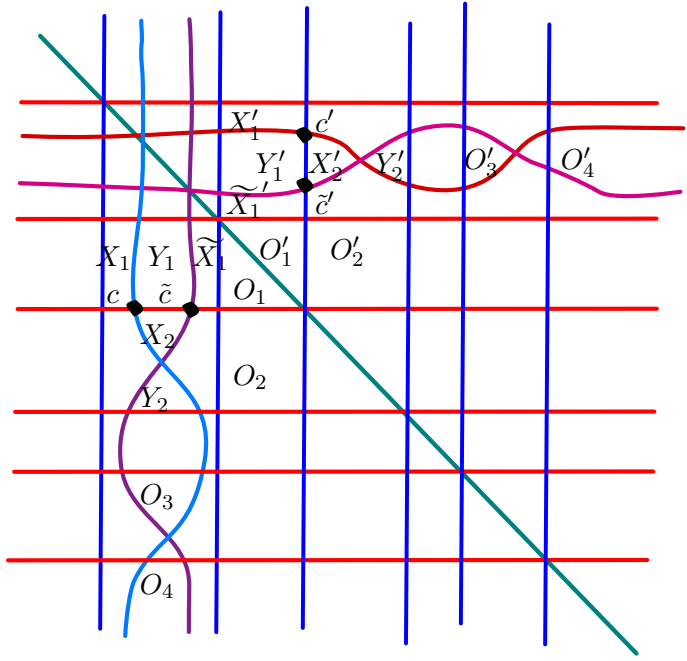
The main gadget is the following theorem.

\begin{thm}\label{thm:mapping cone isom for oriented skein triple}

Take any real oriented skein triple $(L_+, L_- ,L_0,\fra)$ as above and consider the associated real grid diagrams $\cH_+$, $\cH_-$, $\cH_0$, $\cH_0'$ that only differ locally as in Figure~\ref{fig:grid diagram for oriented skein triple}. Then there is a bigraded chain map $P_{+,-}:\GCR^-(\cH_+)\to \GCR^-(\cH_-)$ of $\F[u_{1},\ldots u_{l_f}, U_{1},\ldots, U_{k}]$-modules ($k$ is the number of pairs of $O$-base points in $\cH_+$), between the uncollapsed minus version of real grid chain complexes and a bigraded quasi-isomorphism from $\Cone (P_{+,-})$  to \[\Cone (U_1-U_2: \GCR^-(\cH_0)\to  \GCR^-(\cH_0))[-1,\frac{l_p'-l_p-1}{2}].\]
\end{thm}

There are two pairs of intersection points $(c,c')$ and $(\widetilde{c},\widetilde{c}')$ on $\alpha_+\cap \beta_+$ and $ \alpha_-\cap \beta_-$, respectively. We split the real grid states in $(\T_{\bm\alpha_+}\cap \T_{\bm\beta_+})^{R}$ ($(\T_{\bm\alpha_-}\cap \T_{\bm\beta_-})^{R}$) into $\bfI_+\cup \bfN_+$ ($\bfI_-\cup \bfN_-$) according to whether the real grid state contains the intersection points $(c,c')$ ($(\widetilde{c},\widetilde{c}')$). After this, we have the following mapping cone expressions for various real grid chain complexes. 
\[\GHR^-(\cH_+)=\Cone(\partial^{\bfN_+}_{\bfI_+} \colon \bfI_+\to \bfN_+); \quad \GHR^-(\cH_0)=\Cone(\partial^{\bfI_+}_{\bfN_+}\colon \bfN_+\to \bfI_+);\]
\[\GHR^-(\cH_-)=\Cone(\partial^{\bfI_-}_{\bfN_-} \colon \bfN_-\to \bfI_-); \quad \GHR^-(\cH_0')=\Cone(\partial^{\bfN_-}_{\bfI_-}\colon \bfI_-\to \bfN_-).\]

The real grid states in $\cH_+$ and $\cH_0$ can be identified naturally, and of course, the same is true for $\cH_-$ and $\cH_0'$. However, the real Alexander grading are not the same- due to the change of $\bfX$. We also have a natural one-to-one correspondence between $\bfI_+$ and $\bfI_-$, given by replacing $(c,c')$ with $(\widetilde{c},\widetilde{c}')$. As above, let $l_p=l_p(L_+)=l_p(L_-)$ and $l_p'=l_p(L_0)$, then we can calculate that
 
\begin{lem}\label{lem:grading shift in oriented skein triple}
	
For $\xv\in \bfN_+$, \[M^R_+(\xv)=M^R_0(\xv), \quad A_+^R(\xv)-\frac{1}{2}-\frac{l_p}{2}=  A_0^R(\xv)-\frac{l_p'}{2}.\]

For $\xv\in \bfN_-$, \[M^R_-(\xv)=M^{R'}_0(\xv), \quad A_-^R(\xv)+\frac{1}{2}-\frac{l_p}{2}=  A_0^{R'}(\xv)-\frac{l_p'}{2}.\]

For $\xv\in \bfI_-$ and $T(\xv)\in \bfI_+$, \[M^R_+(T(\xv))=M^{R'}_0(\xv)-1, \quad A_+^R(T(\xv))+\frac{1}{2}-\frac{l_p}{2}= A_0^{R'}(\xv)-\frac{l_p'}{2};\]

\[M^R_0(T(\xv))=M^{R}_-(\xv)-1, \quad A_0^R(T(\xv))+\frac{1}{2}-\frac{l_p'}{2}= A_-^{R}(\xv)-\frac{l_p}{2}.\]   
\end{lem}
\begin{proof}
This follows directly from \cite[Lemma~9.2.2]{OSS2015grid}, since $l_f$ and $\vert \xv\cap C\vert$ keep unchanged.
\end{proof}

\begin{lem}\label{lem:triangle map T}
The identification $T$ extends to an isomorphism of chain complexes \[T\colon (\bfI_-, \partial_{\bfI_-}^{\bfI_-}) \to (\bfI_+, \partial_{\bfI_+}^{\bfI_+}).\]
\end{lem}

Using this identification, we can consider the following chain complex \begin{equation}\label{eq:combined chain complex}
    \xymatrix{\bfI_- \ar[d]_{\partial^{\bfN_+}_{\bfI_+}\circ T} \ar[r]^{\partial^{\bfN_+}_{\bfI_+}} & \bfN_- \ar[d]^{T\circ \partial^{\bfI_-}_{\bfN_-}}\\
		\bfN_+ \ar[r]_{\partial^{\bfI_+}_{\bfN_+}} & \bfI_+\\}.
\end{equation}

Using Lemma~\ref{lem:grading shift in oriented skein triple} and \ref{lem:triangle map T}, this chain complex has the following nice properties.
\begin{lem}\label{lem:combined chain cplx}
Endow $\bfI_+$, $\bfN_+$ with bigrading $(M_0^R+1, A_0^R+\frac{l_p-l_p'+1}{2})$ and  $\bfI_-$, $\bfN_-$ with bigrading $(M_0^{R'}+1, A_0^{R'}+\frac{l_p-l_p'-1}{2})$, then the following is true for Diagram~\eqref{eq:combined chain complex}.
\begin{itemize}
    \item Each edge map is homogeneous of bigrading $(-1,0)$;
    \item The left column, as a bigraded module over $\F[u_1,\ldots,u_{l_f},U_1,\ldots,U_k]$ is identified with $\GCR^-(\cH_+)[-1,0]$;
    \item The right column as a bigraded module over $\F[u_1,\ldots,u_{l_f},U_1,\ldots,U_k]$ is identified with $\GCR^-(\cH_-)$;
    \item The top row is identified with $\GCR^-(\cH_0')[0,\frac{l_p'-l_p+1}{2}]$;
    \item The bottom row is identified with $\GCR^-(\cH_0)[-1,\frac{l_p'-l_p-1}{2}]$.
\end{itemize}
\end{lem}

Then we define $\F[u_1,\ldots,u_{l_f},U_1,\ldots,U_k]$-module maps $h_X: \GCR^-(\cH_0')\to \GCR^-(\cH_0')$ for $X=X_2, Y_1, Y_2$ using the formula \[h_X(\xv)=\sum_{\yv}\sum_{r\in \RRecto(\xv,\yv),r\cap (\bfX_+\cup \bfX_0)=\{X,X'\}} \prod_{1\le i\le l_f} u_{i}^{n_{O_i^f}(r)}\prod_{1\le j\le k} U_{j}^{n_{O_j}(r)} \yv.\]  We also set $h_Y=h_{Y_1}+h_{Y_2}$.

\begin{lem}\label{lem:properties of h_X}
The map $h_{X_2}$ is homogeneous of degree $(-1,0)$ and $h_{Y_i}$ is homogeneous of degree $(-1,-1)$. Furthermore, $h_{X_2}$ vanishes on $\bfN_-$ and takes $\bfI_-$ to $\bfN_-$ while $h_{Y_i}$ vanishes on $\bfI_-$. 
\end{lem}
\begin{proof}
    This follows from the definition of bigradings (see~Section \ref{sub:Relative gradings}) and consider the local picture around $(\widetilde{c},\widetilde{c}')$. 
\end{proof}

Forgetting $X_i,X_i'$'s in Figure~\ref{fig:combined diagram for oriented skein triple}, the diagram then serves as a combined grid diagram for the real commutation relating $\cH_0$ and $\cH_0'$. There are two pairs of intersection points in $(\alpha_+\cap \beta_+)\cup (\alpha_-\cap \beta_-)$. Counting real pentagons with corners at them leads to a chain map $P\colon\GCR^-(\cH_0) \to \GCR^-(\cH_0')$ which is a quasi-isomorphism and its quasi-inverse.

\begin{lem}\label{lem:property of commutation invariance map P}
The map $P$ satisfies \[P\circ (\partial^{\bfN_+}_{\bfI_+}\circ T+ T\circ \partial^{\bfI_-}_{\bfN_-})=h_{X_2}\circ h_{Y}+h_{Y}\circ h_{X_2}\] showing that $h_{X_2}\circ h_{Y}+h_{Y}\circ h_{X_2}$ is a chain map. Thus, $P$ induces a quasi-isomorphism between \[\Cone(\partial^{\bfN_+}_{\bfI_+}\circ T+ T\circ \partial^{\bfI_-}_{\bfN_-}\colon \GCR^-(\cH_0')\to \GCR^-(\cH_0'))\] and \[\Cone( h_{X_2}\circ h_{Y}+h_{Y}\circ h_{X_2}\colon \GCR^-(\cH_0')\to \GCR^-(\cH_0')).\]
\end{lem}
\begin{proof}
The equation is shown by considering decomposition of real domains into union of two real rectangles or a real rectangle, a real pentagon together with a pair of small triangles. The domains may be complicated in this case due to the complexity of real rectangles and pentagons, but the property is standard, so we won't provide details. The quasi-isomorphism between mapping cones follows from it using formal property of mapping cones.
\end{proof}

\begin{lem}\label{lem:composition of h=U_1-U_2}
The map $h_{X_2}\circ h_{Y}+h_{Y}\circ h_{X_2} \colon \GCR^-(\cH_0')\to \GCR^-(\cH_0')$ is chain homotopic to the multiplication by $U_1-U_2$ as homogeneous chain maps of degree $(-2,-1)$. 
\end{lem}
\begin{proof}
This is proved by constructing chain homotopies $h_{X_2,Y_i}$ which shares similar formula as $h_X$'s and apply Proposition~\ref{prop:U,v action identification} to $(O_1,O_1')$ and $(O_4,O_4')$.
\end{proof}

\begin{proof}[Proof of Theorem~\ref{thm:mapping cone isom for oriented skein triple}]
Using Lemma~\ref{lem:combined chain cplx}, the two horizontal maps in Diagram~\eqref{eq:combined chain complex} induce a map $P_{+,-}\colon \GCR^-(\cH_+)\to \GCR^-(\cH_-)$. Then combining Lemma~\ref{lem:combined chain cplx}, \ref{lem:property of commutation invariance map P} and \ref{lem:composition of h=U_1-U_2}, we get a quasi-isomorphism from \[\Cone(P_{+,-}\colon \GCR^-(\cH_+)\to \GCR^-(\cH_-))\] to \[\Cone(U_2-U_1\colon \GCR^-(\cH_0')\to \GCR^-(\cH_0'))[-1,\frac{l_p'-l_p-1}{2}].\] Note that $P$ is a quasi-isomorphism between $\GCR^-(\cH_0)$ and $\GCR^-(\cH_0')$ interwining $u_i$ and $U_j$ actions, so we have the claimed quasi-isomorphism as desired.
\end{proof}

Before the proof, we need an analogue of~\cite[Lemma~9.3.1]{OSS2015grid}.
\begin{lem}\label{lem:induced quasi-isom on quotient}
Let $C$ and $C'$ be two free bigraded chain complexes over $\F[u_1,\ldots,u_{l_f}, U_1,\ldots,U_k]$ ($\deg(u_i)=(-1,-1/2)$, $\deg(U_j)=(-2,-1)$). Fix any nonzero, homogeneous polynomial $r\in \F[u_1,\ldots,u_{l_f}, U_1,\ldots,U_k]$. A quasi-isomorphism $\phi\colon C\to C'$ induces a quasi-isomorphism $\bar{\phi}:\frac{C}{rC}\to \frac{C'}{rC'}$.    
\end{lem}

\begin{proof}[Proof of Theorem~\ref{thm:minus oriented skein triple}]
Setup notations $l_f$, $l_p$ and $l_p'$ as above. We label the $O$-base points so that $(O_1^*,O_1^{*'}), \ldots, (O^*_{l_p},O^{*'}_{l_p})$ belong to different paired components of $L_+$. Here, we add a superscript $*$ to avoid confusion with labels in the combined grid diagram. By definition, \[c\GCR^-(\cH_+)=\frac{\GCR^-(\cH_+)}{u_1=u_2=\ldots=u_{l_f}, u_1^2=U^*_1=\ldots=U^*_{l_p}},\]
\[c\GCR^-(\cH_-)=\frac{\GCR^-(\cH_-)}{u_1=u_2=\ldots=u_{l_f}, u_1^2=U^*_1=\ldots=U^*_{l_p}}.\] Then using Lemma~\ref{lem:induced quasi-isom on quotient} repeatedly, the quasi-isomorphism from Theorem~\ref{thm:mapping cone isom for oriented skein triple} leads to a quasi-isomorphism 
\begin{equation}\label{eq:mapping cones in oriented skein tripple}\begin{aligned}
    &\Cone(\bar{P}_{+,-}\colon c\GCR^-(\cH_+) \to c\GCR^-(\cH_-))\\
    \to &\Cone (U_1-U_2\colon\frac{\GCR^-(\cH_0)}{u_1=u_2=\ldots=u_{l_f}, u_1^2=U^*_1=\ldots=U^*_{l_p}} \to \frac{\GCR^-(\cH_0)}{u_1=u_2=\ldots=u_{l_f}, u_1^2=U^*_1=\ldots=U^*_{l_p}}).
\end{aligned}
\end{equation}

The case $l_p'=l_p\pm 1$ follows from previous propositions in the same as \cite[Theorem~9.1.1]{OSS2015grid}, so we only illustrate the case $l_p'=l_p$ into details. In this case, the two strands involving in the crossing change in $L_+$ must belong to a single or two strongly invertible components. In either case, \[c\GCR^-(\cH_0)=\frac{\GCR^-(\cH_0)}{u_1=u_2=\ldots=u_{l_f}, u_1^2=U^*_1=\ldots=U^*_{l_p}},\] and $(O_1,O_1')$, $(O_2,O_2')$ does not appear in $O^*$-base points. Note also that $U_1$, $U_2$ actions are homotopic using the collapsing relation and Proposition~\ref{prop:U,v action identification}. Then, the second mapping cone in Equation~\eqref{eq:mapping cones in oriented skein tripple} is isomorphic to $c\GHR^-(\cH_0)\otimes W$, since the map is null-homotopic. Then, the long exact sequence associated to $\Cone(\bar{P}_{+,-})$ leads to Equation~\eqref{eq:l_p'=l_p minus}.

\end{proof}

\begin{proof}[Proof of Theorem~\ref{thm:hat oriented skein triple}]
The proof is almost identical to the previous one. We just need to  note that in the hat chain complex, we just further let the variables that were set equal in the collapsed minus version to be all zero.
\end{proof}
 
\subsection{Decategorification}
\begin{defn}
Let $(L,\fra)$ be a generalized strongly invertible link with auxiliary data in $S^3$. We define its real Alexander polynomial to be the Euler characteristic of $\widehat{\GHR}$. More precisely, \[\Delta^R(L,\fra)(t)=(t-t^{-1})^{-l_p}\sum_{s\in \frac{1}{2}\Z} \sum_{d\in \Z} (-1)^d t^{2s} \dim \widehat{\GHR}_d(L,\fra,s).\]
\end{defn}

\begin{prop}\label{prop:mirroring the polynomial invariant}
When $L$ is a strongly invertible knot, its real Alexander polynomial satisfies \[\Delta^R(m(L),m(\fra))(t)= \Delta^R(L,\fra)(-t).\]
\end{prop}
\begin{proof}
This follows from Proposition~\ref{prop:hat version mirror sym}.
\end{proof}

\begin{thm}\label{thm:skein relation for real Alexander polynomial}
Let $(L_+, L_- ,L_0,\fra)$ be a real oriented skein triple. Then we have \[\Delta^R(L_+,\fra)(t)-\Delta^R(L_-,\fra)(t)=(t-t^{-1})\cdot\Delta^R(L_0,\fra)(t).\]
\end{thm}
\begin{proof}
This follows from Theorem~\ref{thm:hat oriented skein triple}.
\end{proof}
\subsection{Unoriented version}\label{sub:Unoriented version}
In Section~\ref{sub:Oriented version}, we introduced the real oriented skein triple and show that the associated real link Floer groups fit into a long exact sequence. In this subsection, we consider an unoriented counterpart. If three generalized strongly invertible links $L$, $L_1$ and $L_2$ have transvergent planar diagrams differ only in a pair of disks as shown in Figure~\ref{fig:unoriented skein triple}, then we say that they form a \emph{real unoriented skein triple}. Unfortunately, this triple does not admit a consistent choice of orientations, thus no consistent choice of auxiliary data and the groups $\GHR^\circ$ are sensitive to it. Thus, we cannot expect the original bi-graded groups to fit into a long exact sequence. Instead, we will consider a single $\delta^R$-grading on their homology groups and show that suitable algebraic stabilization of the singly-graded groups fit into a long exact sequence. 

\begin{figure}
    \centering
     \begin{overpic}[width=0.7\textwidth]{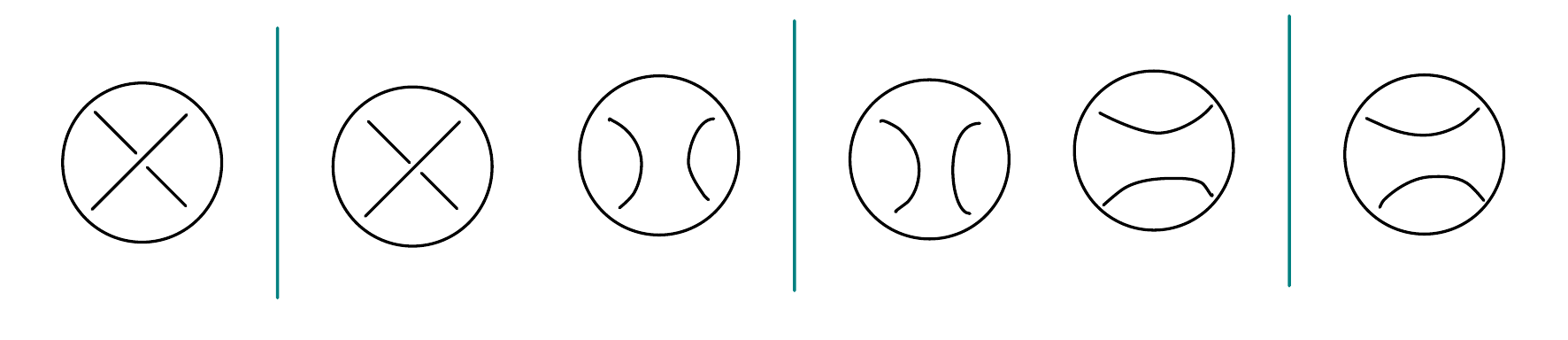}
            
			\put(16,-2) {$L$}
            \put(48,-2) {$L_1$}
            \put(80,-2) {$L_2$}
		\end{overpic}
        
    \caption{Real unoriented skein triple.}
    \label{fig:unoriented skein triple}
\end{figure}

Before investigating the unoriented triangle, we first loose our notion of auxiliary data. Careful readers may have noted that for a strongly invertible link, not all choices of $\fro$ and $\frl$ admits a compatible choice of $\fro'$ and $A$. A potential example is shown in Figure~\ref{fig:loosing fra}, for which the author can not find an obvious choice of $\fro'$ and $A$ that fits the quadruple into Definition~\ref{def:strongly invertible knots with extra data}. Since we may encounter such links during our manipulation of unoriented skein triple, in this subsection, $\fra$ will be referred to as only as a choice of $\fro$ and $\frl$, we no longer require (3) and (4) in Definition~\ref{def:strongly invertible knots with extra data} to exists. It can be seen easily that for such a choice of $\fra$, Proposition~\ref{prop:existence of real HD for strongly invertible knots and real H moves} still holds and for a real oriented skein triple $(L_+,L_-, L_0,\fra)$ with this loosen choice of $\fra$, the discussion from Section~\ref{sub:Oriented version} remains valid.

\begin{figure}
    \centering
    \includegraphics[width=0.5\linewidth]{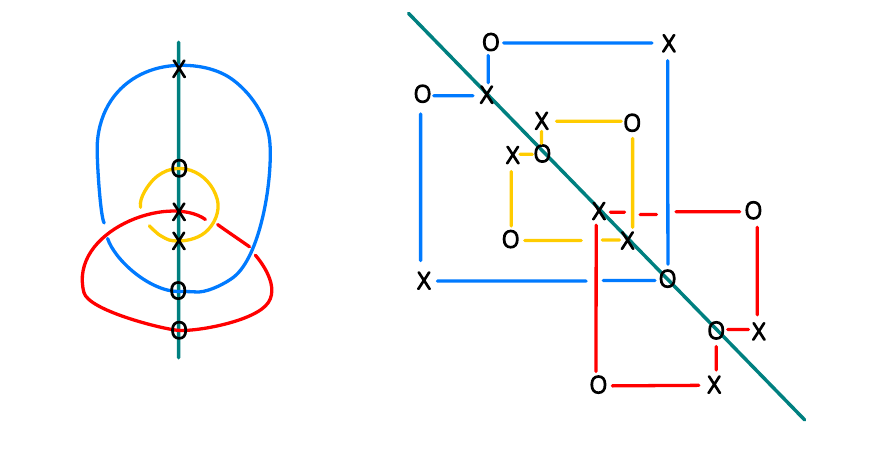}
    \caption{Loosing the restriction on auxiliary data.}
    \label{fig:loosing fra}
\end{figure}

Let $(L,\fra)$ be a generalized strongly invertible link with (loosed) auxiliary data in $S^3$. Take any of its real grid diagram $\cH$ and consider the associated chain complex $\GCR^-(\cH)$. We ``merge'' the two gradings on $\cH$ into a single \emph{$\delta^R$-grading} defined by \[\delta^R(\xv)=M^R(\xv)-A^R(\xv), \quad \delta^R(u_i)=-\frac{1}{2},\quad \delta^R(U_j)=-1.\]
Then, one see that \begin{align*}
    \delta^R(\xv)&=\frac{1}{2}(M^R_{\bfO}(\xv)+M^R_{\bfX}(\xv))+\frac{(n-l_f-2l_p)}{4}\\
    &= \frac{1}{4}(M_{\bfO}(\xv)+M_{\bfX}(\xv))-\frac{1}{4}\vert \xv\cap C\vert+\frac{1}{4}l_f+\frac{(n-l_f-2l_p)}{4}
\end{align*} Recall that the classical $\delta$-grading was defined by \begin{align*}
    \delta(\xv)&=M_{\bfO}(\xv)-A(\xv)\\
    &=\frac{1}{2}(M_{\bfO}(\xv)+M_{\bfX}(\xv))+\frac{(n-l_f-2l_p)}{2},\\
\end{align*} so $\delta^R(\xv)$ is not half of $\delta(\xv)$, instead, \[\delta^R(\xv)=\frac{1}{2}\delta(\xv)-\frac{1}{4}\vert \xv\cap C\vert+\frac{1}{4}l_f.\] Nevertheless, the extra correction terms are independent of position of $O$ and $X$-base points, so we can deduce the following lemma from~\cite[Lemma~10.1.7]{OSS2015grid}.

\begin{lem}\label{lem:change of deltaR under partial O,X switch}
Let $\cH$ be a real grid diagram associated to any generalized strongly invertible link $(L,\fra)$. Next, let $\cH'$ be a real grid diagram obtained from $\cH$ by interchanging some $O$ and $X$-base points. Then $\cH$ and $\cH'$ represent the same underlying  generalized strongly invertible with different choice of auxiliary data, say $\fra$ and $\fra'$. Then, we have two gradings $\delta^R_{\cH}$ and $\delta^R_{\cH'}$ associated to the same set of real grid states. For any real grid states $\xv$, we have \[\delta^R_{\cH}(\xv)+\frac{1}{8}\wri(L,\fra)=\delta^R_{\cH'}(\xv)+\frac{1}{8}\wri(L,\fra').\]
Here, $\wri$ denotes writhe of underlying planar diagram and we abuse $(L,\fra)$ to denote the underlying oriented link diagram. For a more diagram-independent description, we split $L=L_1\cup L_2$ so that orientations on $L_1$ are reversed when we go from $\cH$ to $\cH'$, while that on $L_2$ keeps unchanged. Then we have \[\delta^R_{\cH}(\xv)-\delta^R_{\cH'}(\xv)=-\frac{1}{2}\lk(L_1, L_2),\] when we induce orientations on $L_i$ from $\cH$.
\end{lem}

It follows from the lemma and the fact that as a vector space, $\widetilde{\GHR}$ does not distinguish $\bfO$ and $\bfX$ that

\begin{prop}\label{prop:isomorphism of tilde from OX rearrangement}
Let $\cH$, $\cH'$, $(L,\fra)$ and $(L,\fra')$ be set-up as in Lemma~\ref{lem:change of deltaR under partial O,X switch}. Then we have an isomorphism between $\delta^R$-graded vector spaces \[\widetilde{\GHR}_{\delta^R}(\cH)\cong\widetilde{\GHR}_{\delta^R+c}(\cH'), \] in which $c=\frac{1}{8}(\wri(L,\fra)-\wri(L,\fra'))$.
\end{prop}

\begin{thm}\label{thm:unoriented skein exact sequence}
Let $L$, $L_1$ and $L_2$ be a real unoriented skein triple and $l_p$, $l_p^1$, $l_p^2$ be the number of paired components in them, respectively. By definition, they share the same number of strongly invertible components. Then for sufficiently large $N\in \Z$, we have an exact triangle  
\[\begin{tikzcd}
    \widehat{\GHR}(L_1)\otimes W^{N-l_f-l_p^1} \arrow{rr} & & \widehat{\GHR}(L_2)\otimes W^{N-l_f-l_p^2}  \arrow{dl}\\
    & \widehat{\GHR}(L)\otimes W^{N-l_f-l_p} \arrow[swap]{ul}\\
\end{tikzcd}\]
Here, $W$ is a 2-dimensional $\frac{1}{2}\Z$-graded $\F$-vector space supported in $\delta^R$-grading zero, and $\widehat{\GHR}$ groups are thought of as relatively $\frac{1}{2}\Z$-graded vector spaces, using $\delta^R$-grading.
\end{thm}

The proof of this theorem relies on a corollary of Theorem~\ref{thm:mapping cone isom for oriented skein triple}. 

\begin{thm}\label{thm:mapping cone isom for oriented skein triple-reduction to tilde}
Take any real oriented skein triple $(L_+,L_-,L_0,\fra)$ and consider the associated real grid diagrams $\cH_+$, $\cH_-$, $\cH_0$, $\cH_0'$. Then there is a $\delta^R$-grading preserving chain map $\widetilde{P}_{+,-}:\widetilde{\GCR}(\cH_+)\to \widetilde{\GCR}(\cH_-)$ a quasi-isomorphism \[\Cone (\widetilde{P}_{+,-}) \to \widetilde{\GCR}(\cH_0)\oplus\widetilde{\GCR}(\cH_0)[\frac{l_p'-l_p-1}{2}].\]
\end{thm}

Using the trick shown in \cite[Figure~10.5]{OSS2015grid}, we only need to prove the theorem in the case that strands appear in a single disk in Figure~\ref{fig:unoriented skein triple} belong to different components of $L$. As in \cite[Section~10.2]{OSS2015grid}, we prove the theorem by manipulating two closely related real oriented skein triples. The local figures and associated real grid diagrams are given in Figure~\ref{fig:grid diagram for unoriented skein triple}. 

\begin{figure}
    \centering
    \begin{overpic}[width=0.7\textwidth]{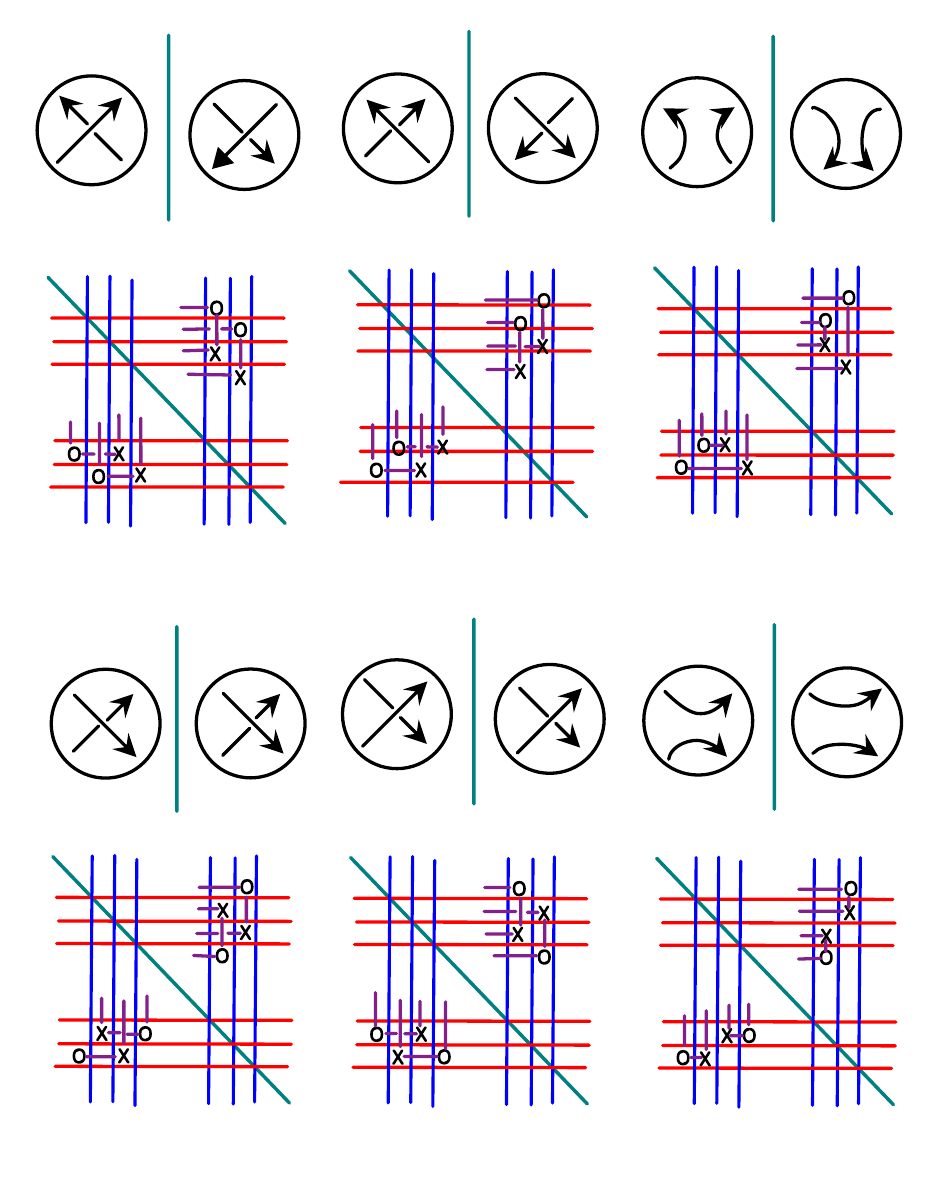}
			\put(14,78) {$L_+$}
            \put(14,52) {$\cH_+$}
            \put(38,78) {$L_-$}
            \put(38,52) {$\cH_-$}
            \put(64,78) {$L_0$}
            \put(64,52) {$\cH_0$}

            \put(14,28) {$R_+$}
            \put(14,3) {$\cH_+'$}
            \put(38,28) {$R_-$}
            \put(38,3) {$\cH_-'$}
            \put(64,2853) {$R_0$}
            \put(64,3) {$\cH_0'$}
            
		\end{overpic}
    \caption{Grid diagrams for a real unoriented skein triple.}
    \label{fig:grid diagram for unoriented skein triple}
\end{figure}
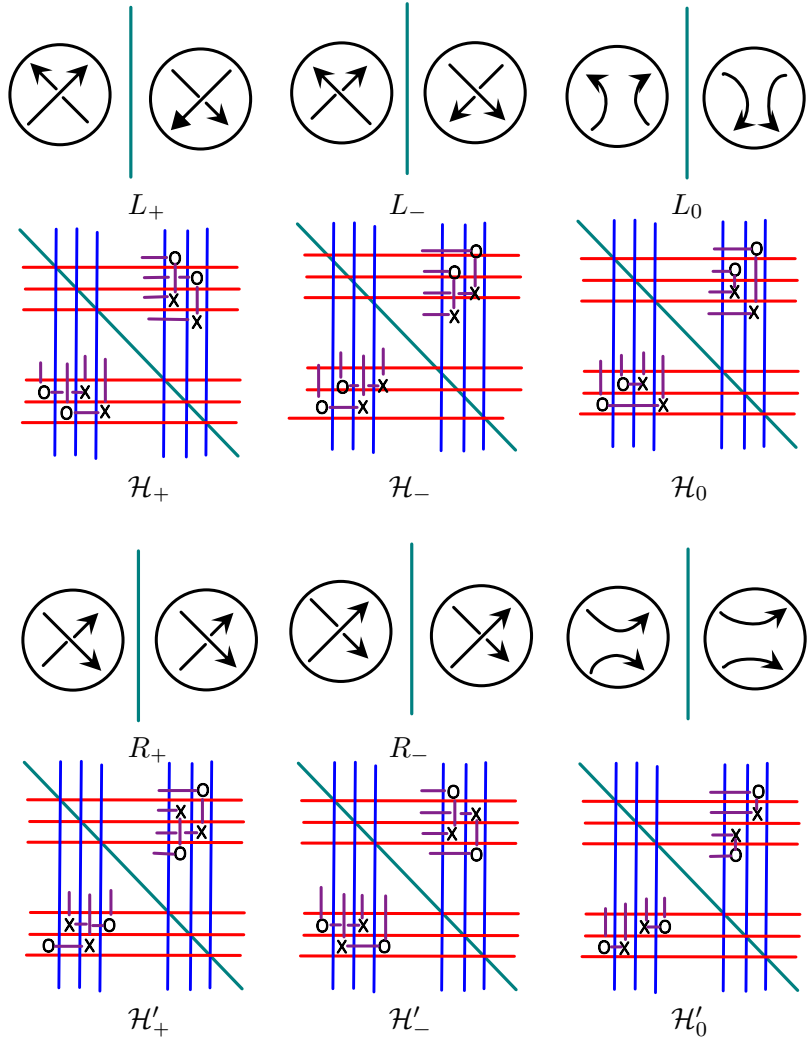

As shown in the figure, we assume that first three $O$, $X$-base points above $C$ are in consecutive rows while left most three base points below $C$ are in consecutive columns. The three links in the first and the second row form real oriented skein triples, respectively. The diagram $\cH_+$ for $L_+$ and the diagram $\cH_-'$ for $R_-$ differ by interchanging $O$, $X$-base points on a strongly invertible component or a pair of interchanged components. Similar observation holds for $\cH_-$ and $\cH_+'$. The consecutive assumption on base points leads to maps \[\widetilde{P}_{+,-}:\widetilde{\GCR}(\cH_+)\to \widetilde{\GCR}(\cH_-), \quad \widetilde{P}_{+,-}':\widetilde{\GCR}(\cH_+')\to \widetilde{\GCR}(\cH_-'),\] coming from Theorem~\ref{thm:mapping cone isom for oriented skein triple-reduction to tilde}. 

Using this, we define a map \[\widetilde{P}_{-,+}:\widetilde{\GCR}(\cH_-)\to \widetilde{\GCR}(\cH_+)\] as the composition \[\widetilde{\GCR}(\cH_-)\to \widetilde{\GCR}(\cH_+')\xrightarrow{\widetilde{P}_{+,-}'}  \widetilde{\GCR}(\cH_-') \to \widetilde{\GCR}(\cH_+),\] where the first and third arrow comes from the observation on $O$, $X$-base points rearrangement and Proposition~\ref{prop:isomorphism of tilde from OX rearrangement}. This is by construction a chain map and lower the $\delta^R$-grading by $1$.

\begin{lem}\label{lem:Pwilde has null composition}
Let $(L_+,L_-,L_0,\fra)$ be a real oriented skein triple and take associated real grid diagrams $\cH_+$, $\cH_-$, $\cH_0$ satisfying the consecutive rows(columns) assumption. Applying the previous discussion, we obtain maps $\widetilde{P}_{+,-}$ and $\widetilde{P}_{-,+}$. Then both compositions $\widetilde{P}_{+,-}\circ \widetilde{P}_{-,+}$ and $\widetilde{P}_{-,+}\circ \widetilde{P}_{+,-}$ are null-homotopic.
\end{lem}

An adaptation of proof of \cite[Proposition~9.6.1]{OSS2015grid} shows that \[\widetilde{P}_{-,+}=\widetilde{c}_++\widetilde{c}_+',\quad \widetilde{P}_{+,-}=\widetilde{c}_-+\widetilde{c}_-'.\] 
Here, $\widetilde{c}_\pm$ and $\widetilde{c}_{\pm}'$ are crossing change maps (associated to $L_{\pm}$ and $R_{\pm}$) constructed in Section~\ref{sub:Equivariant crossing changes and maps associated to them} by counting pentagons. Though we worked with knots previously, the construction works for links without essential change. We have seen that $c_\pm\circ c_{\mp}\simeq u^2$ on $\GHR^-$ for knots and this is still true for links. Since all the variables are set to be zero in $\widetilde{\GCR}$, we can actually show various compositions between $\widetilde{c}_+$, $\widetilde{c}_+'$ and  $\widetilde{c}_-$, $\widetilde{c}_-'$ are null-homotopic. 

\begin{lem}\label{lem:mapping cone of composition}
Let $f\colon C\to C'$ and $g\colon C'\to C$ be a pair of chain maps between $\frac{1}{2}\Z$-graded chain complexes. Suppose that $f$ is homogeneous of degree $a$ and $g$ is homogeneous of degree $b$. Then there is map $\Phi:\Cone (f)\to \Cone (g)$ homogeneous of degree $-a-1$ whose induced map on homology fit into the following exact triangle: 
\[\begin{tikzcd}
    H(\Cone (f)) \arrow{rr}{\Phi_*}  & & H(\Cone (g))\arrow{dl} \\
    & H(\Cone (f\circ g)) \arrow[swap]{ul}
\end{tikzcd}.\]
\end{lem}

\begin{proof}[Proof of Theorem~\ref{thm:unoriented skein exact sequence} in the case that crossings happen between different components]

We have noted in Proposition~\ref{prop:isomorphism of tilde from OX rearrangement} that as a (relatively) $\delta^R$-graded vector space, $\widetilde{\GHR}$ is independent of the choice of orientation on $L$, so without loss of generality, we assume that the local crossing in $L$ is positive. We take $L=L_+$ and form triples $(L_+,L_-,L_0)$ and $(R_+,R_-,R_0)$. Then, we can construct triples of real grid diagrams $(\cH_+,\cH_-,\cH_0)$, $(\cH_+',\cH_-',\cH_0')$ and introduce various $\widetilde{P}_{\pm,\mp}$ maps. Apply Lemma~\ref{lem:mapping cone of composition} to the null-homotopic composition considered in Lemma~\ref{lem:Pwilde has null composition}, we get 
\[\begin{tikzcd}
    H(\Cone (\widetilde{P}_{-,+})) \arrow{rr}  & & H(\Cone (\widetilde{P}_{+,-})) \arrow{dl} \\
    & H(\Cone (\widetilde{P}_{-,+}\circ\widetilde{P}_{+,-})) \arrow[swap]{ul}
\end{tikzcd}.\]

The desired triangle will be derived from this by identifying the mapping cones with suitable real grid homology groups. The strongly invertible links in the triple appear in $\cH_+$, $\cH_0$ and $\cH_0'$. Theorem~\ref{thm:mapping cone isom for oriented skein triple-reduction to tilde} provides us with quasi-isomorphisms \[\Cone (\widetilde{P}_{+,-})\to \widetilde{\GCR}(\cH_0)\oplus\widetilde{\GCR}(\cH_0), \quad \Cone (\widetilde{P}_{+,-}')\to \widetilde{\GCR}(\cH_0')\oplus\widetilde{\GCR}(\cH_0'),\] up to an overall grading shift. Apply Proposition~\ref{prop:isomorphism of tilde from OX rearrangement} to the composition formula of $\widetilde{P}_{-,+}$, we know that \[\Cone (\widetilde{P}_{+,-}')\cong \Cone (\widetilde{P}_{-,+}),\] up to a grading shift. Since $\widetilde{P}_{-,+}\circ\widetilde{P}_{+,-}$ is null-homotopic by Lemma~\ref{lem:Pwilde has null composition}, the homology of the mapping cone is just the direct sum of that of its domain and codomain, i.e., \[H(\Cone(\widetilde{P}_{-,+}\circ\widetilde{P}_{+,-}))\cong \widetilde{\GHR}(\cH_+)\oplus \widetilde{\GHR}(\cH_+).\]

Substituting this identifications into the triangle of mapping cones, we obtain \[\begin{tikzcd}
    \widetilde{\GHR}(\cH_0')\otimes W \arrow{rr}  & & \widetilde{\GHR}(\cH_0)\otimes W \arrow{dl} \\
    & \widetilde{\GHR}(\cH_+)\otimes W \arrow[swap]{ul}
\end{tikzcd}.\] 
Then the proof is completed by recalling that for any real grid diagram $\cH$ representing $L$ with $l_f$ strongly invertible components and $l_p$ pairs of components, we have an isomorphism of $\delta^R$-graded vector spaces \[\widetilde{\GHR}(\cH)=\widehat{\GHR}(\cH)\otimes W^{\frac{n-l_f-2l_p}{2}},\] where $n$ is the size of the real grid diagram. 
\end{proof}

\section{Examples and applications}\label{sec:Examples and application}
With the help of a python program \cite{ZhenkunLioythonprgram} by Zhenkun Li, we compute real grid homology for strongly invertible knots with small crossing numbers. In this section, we record some interesting phenomena that have been observed. A full list of calculation results will be given in Appendix~\ref{app:Calculation results for knots with small crossing numbers}.

\subsection{Observations from calculation results}
\begin{example}\label{ex:amphichiral}
In the classical knot Floer theory, we have $\tau(K)=-\tau(m(K))$, so an amphichiral knot must have zero $\tau$-invariant. This is not true when a strong inversion is added. The two strong inversions on $4_1$ have $\tau^R$-invariant $\pm 1$, respectively, independent of the choice of the auxiliary data. 
Similarly, amphichiral knots $8_{12}$ and $8_{18}$ have non-zero $\tau^R$-invariants when they are equipped with suitable strong inversions and auxiliary data. See Appendix~\ref{app:Calculation results for knots with small crossing numbers}.
\end{example}

\begin{example}\label{ex:slice with nonzero tauR}
Using computer computation, we also find some strongly invertible knots whose underlying knots are slice, but have non-zero $\tau^R$-invariants when equipped with suitable auxiliary data. They are $6_1$, $9_{46}$, $10_{35}$ and $4_1\#4_1$. See Appendix~\ref{app:Calculation results for knots with small crossing numbers} for the choice of involution and auxiliary data.

It can be deduced from Corollary~\ref{cor:tau^R bounds equiv slice genus+m/2 enhanced} that these strongly invertible knots are not equivariantly slice, i.e., they do not bound equivariant slice disk in $(B^4,\tau_c)$, since an equivariant slice disk must have connected fixed set.

We also note that for some choice of involution and auxiliary data, the amphichiral knot $8_{12}$ has $\tau^R$ equals to $-2$ whose absolute value its strictly bigger than its usual smooth slice genus $1$. Though Corollary~\ref{cor:tau^R bounds equiv slice genus+m/2 enhanced} is not enough to conclude that $8_{12}$ has equivariant slice genus bigger than $1$, it still serves as a potential example for that.
\end{example}

\begin{example}\label{ex:twist knots}
Twist knots form an interesting family of two bridges knots. It is known that they have two distinct families of involutions. See Figure~\ref{fig:involution on Twisted knots} for an illustration.  
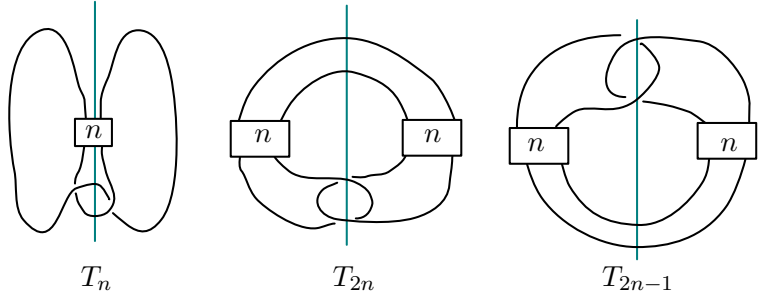
\begin{figure}
    \centering
    \begin{overpic}[width=0.7\textwidth]{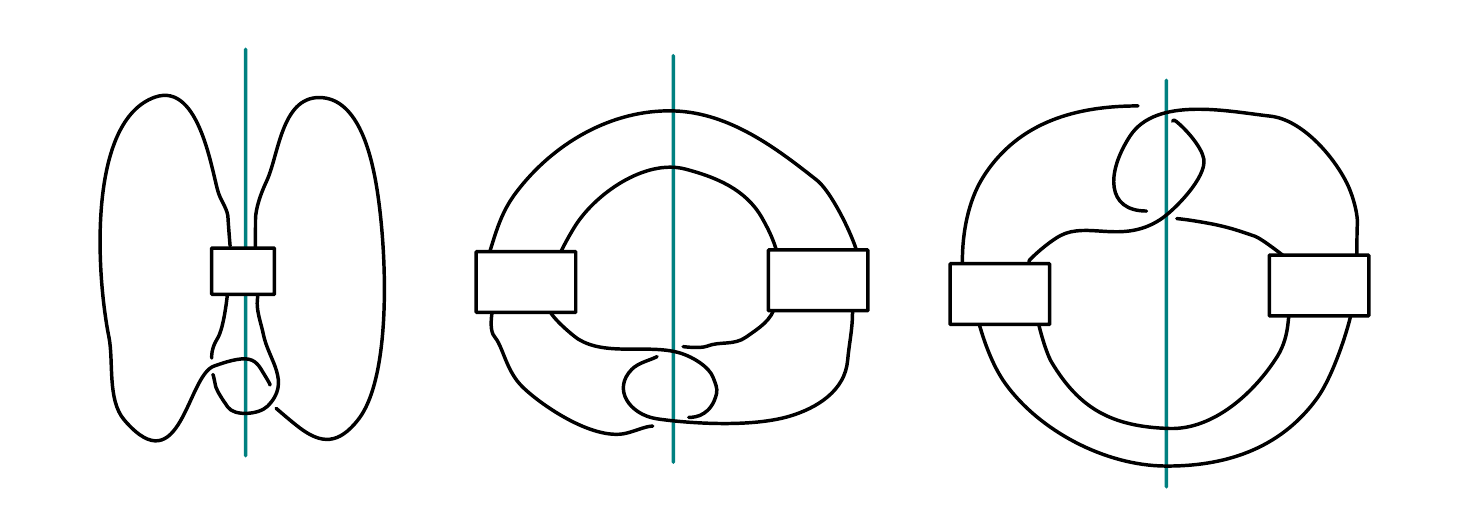}
			\put(15,0) {$T_n$}
            \put(44,0) {$T_{2n}$}
            \put(75,0) {$T_{2n-1}$}

            \put(15.5,17) {$n$}
            \put(35,16.5) {$n$}
            \put(54.5,16.5) {$n$}
            \put(66.4,15.5) {$n$}
            \put(88.6,15.5) {$n$}
	\end{overpic}

    \caption{Involution on twist knots: The left picture shows the first family of involutions, the middle and right figures show the second family of involutions on even and odd twist knots, respectively.} 
    \label{fig:involution on Twisted knots}
\end{figure}

For those with less than nine half-twists, we calculated their $\widehat{\GHR}$ for both involutions and find that for the left family, we have \[\Delta^R(T_n)= \begin{cases}
    1 & \text{for } n\equiv 0,3 \text{ mod } 4\\
    -t^{-1}+1+t & \text{for } n\equiv 1,2 \text{ mod } 4
\end{cases};\] 
while for the right one, \[\Delta^R(T_n)=-kt^{-1}+1+kt,\] when $n=2k$ or $2k-1$. In both cases, their real knot Floer group are determined by the polynomial up to a grading shift. The results are independent of $\fra$, so we omit it from our notation. For the right family, one can show that this is true for any $n$ using Theorem~\ref{thm:skein relation for real Alexander polynomial} and an induction and the author expects the formula for the left family is also true for general $n$.
\end{example}

\begin{example}\label{ex:nontrivial knots with trivial homology}
Besides some interesting phenomena, there are also some strongly invertible knots that do not enjoy interesting real knot Floer groups, i.e., $\HFKR^\circ$ fails to tell them apart from the standard strongly invertible unknot. Some of them are $5_2$, $6_1$, $8_3$ equipped with one of their involutions.
\end{example}

\begin{example}
In \cite{hendricks2025noterealheegaardfloer}, Hendricks asked for strongly invertible knots with $\dim \widehat{\HFKR} (K,0)>1$. For $\HFK$-thin knots, we have \begin{itemize}
    \item $6_2$, $6_3$, $7_3$, $7_5$, $7_6$ with one of their two strong inversions;
    \item $7_7$ with either strong inversion;
    \item $8_9$, $8_{18}$, $9_6$, $10_{75}$ with some strong inversion.
\end{itemize}
See Appendix~\ref{app:Calculation results for knots with small crossing numbers} for their diagrams and detailed real knot Floer groups.
\end{example}

\begin{example}\label{ex:direction dependence}
We calculated real knot Floer homology for knots $8_{19}$, $8_{20}$ and $9_{42}$ equipped with suitable involutions and auxiliary as shown in Figure~\ref{fig:grid diagrams for 8-19,8-20 and 9-42}. These examples originated from~\cite[Section~6]{Sarkar_2015movingbasept}, which pointed out that $8_{19}$ and $9_{42}$ are the only two prime knots with $\le 9$ crossings and non-thin knot Floer homology, while $8_{20}$ and $9_{42}$ admit no periodic involution. The top left grid diagram represents $(8_{19}, \fra_0)$ for some choices of auxiliary data $\fra_0$, the top middle and right diagrams represent a pair $(8_{20},\fra_1)$ and $(8_{20},\fra_1^{i,r})$. All three grid diagrams in the bottom row represent $9_{42}$ with the same involution, but different choices of $\fra$, say $\fra_a$, $\fra_b$ and $\fra_c$. The middle and right diagrams are related a real cyclic permutation, so $\fra_b$ and $\fra_c$ are actually equivalent (i.e., they share the same choices of (1) and (2) in Definition~\ref{def:strongly invertible knots with extra data}), while it is obvious that $\fra_b=\fra_a^{i,r}$.
\begin{figure}
    \centering
    \begin{overpic}[width=0.8\textwidth]{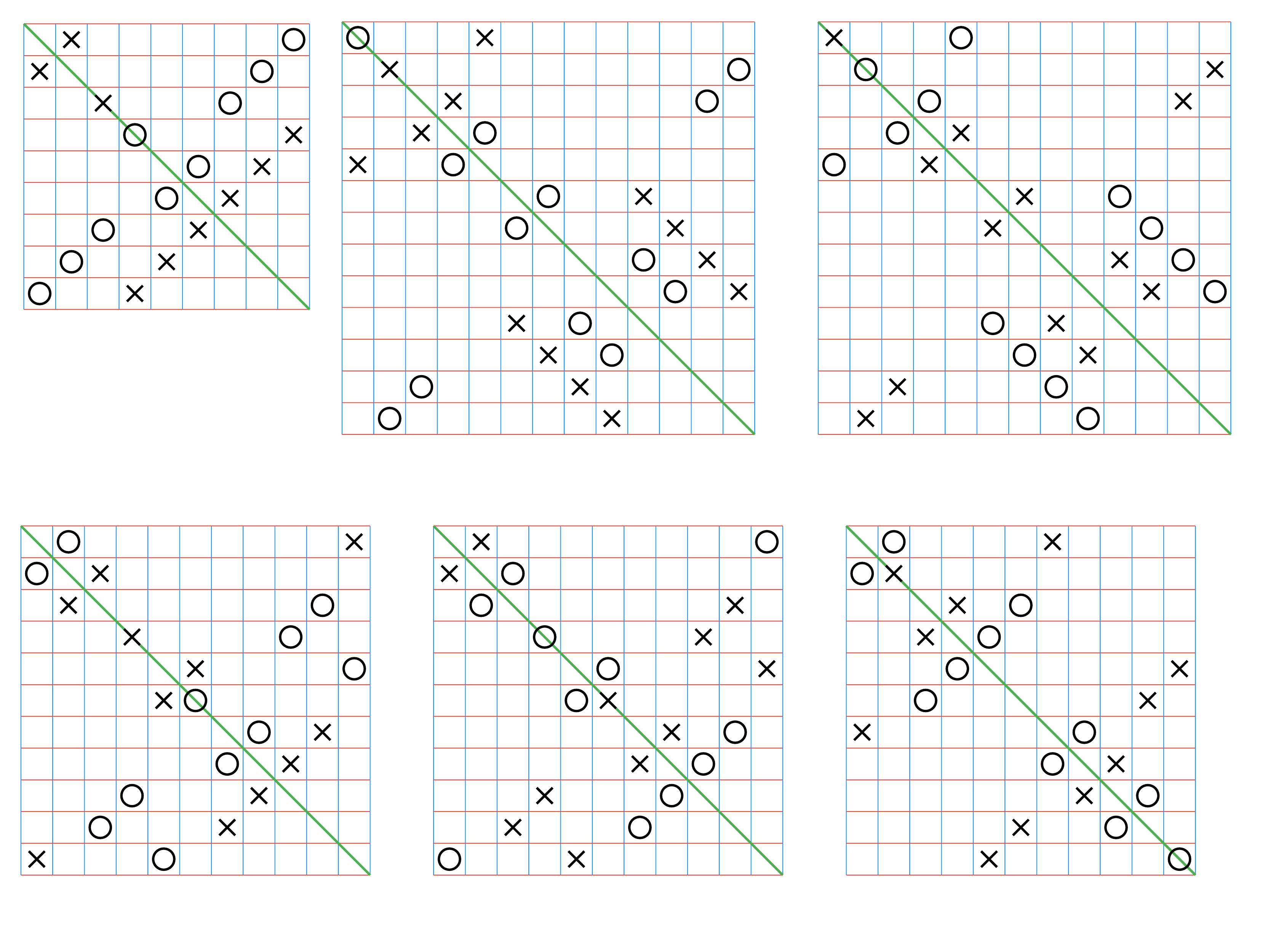}
			\put(11.5,0) {$(9_{42},\fra_a)$}
            \put(43,0) {$(9_{42},\fra_b)$}
            \put(75,0) {$(9_{42},\fra_c)$}

            \put(9,37) {$(8_{19},\fra_0)$}
            \put(40,37) {$(8_{20},\fra_1)$}
            \put(75,37) {$(8_{20},\fra_1^{i,r})$} 
	\end{overpic}
    
    \caption{}
    \label{fig:grid diagrams for 8-19,8-20 and 9-42}
\end{figure}
The results of calculation are \begin{itemize}
\item \[\widehat{\GHR}(8_{19},\fra_0)\cong \F_{(0,-3/2)}\oplus \F_{(0,-1)}\oplus \F_{(1,0)}\oplus \F_{(2,1)}\oplus \F_{(3,3/2)},\] \[\Delta^{R}(8_{19},\fra_0)=t^{-3}+t^{-2}-1+t^2-t^3.\]
\item \[\widehat{\GHR}(8_{20},\fra_1)\cong \F_{(-1,-1)}\oplus \F_{(-1,-1/2)}\oplus \F_{(0,-1/2)}\oplus \F_{(0,0)}^3\oplus \F_{(1,1)},\] \[\Delta^{R}(8_{20},\fra_0)=-t^{-2}+3-t^2.\]
\item \[\widehat{\GHR}(8_{20},\fra_1^{i,r})\cong\F_{(-1,-1)}\oplus  \F_{(0,0)}^3\oplus \F_{(0,1/2)}\oplus \F_{(1,1/2)}\oplus\F_{(1,1)},\]\[\Delta^{R}(8_{20},\fra^{i,r}_1)=-t^{-2}+3-t^2.\]
\item \[\widehat{\GHR}(9_{42},\fra_a)\cong\F_{(-1,-1)}\oplus \F_{(-1,-1/2)}\oplus \F_{(0,-1/2)}\oplus \F_{(0,0)}^3\oplus \F_{(1,1)},\] \[\Delta^{R}(9_{42},\fra_{a})=-t^{-2}+3-t^2.\]
\item \[\widehat{\GHR}(9_{42},\fra_a^{i,r})\cong\widehat{\GHR}(9_{42},\fra_c)\cong \F_{(-1,-1)}\oplus  \F_{(0,0)}^3\oplus \F_{(0,1/2)}\oplus \F_{(1,1/2)}\oplus\F_{(1,1)},\] \[\Delta^{R}(9_{42},\fra_{a}^{i,r})=\Delta^{R}(9_{42},\fra_c)=-t^{-2}+3-t^2.\]
\end{itemize}

From these, one can see easily that the choice of auxiliary data $\fra$ indeed affects the real Floer homology of $8_{20}$ and $9_{42}$ in the way predicted in Proposition~\ref{prop:fra v.s. fra^i,r}. 
Recall from~\cite{BGCOMPUTATIONSOFHEEGAARD-FLOERKNOTHOMOLOGY}, 
\[\widehat{\HFK}(8_{19})\cong\F_{(0,-3)}\oplus \F_{(1,-2)}\oplus \F_{(2,0)}\oplus \F_{(5,2)}\oplus \F_{(6,3)},\] 
\[\widehat{\HFK}(9_{42})\cong\F_{(-1,-2)}\oplus \F_{(0,-1)}^2\oplus \F_{(0,0)}\oplus \F_{(1,0)}^2\oplus \F_{(2,1)}^2\oplus \F_{(3,2)}.\]  Taking Euler characteristic, we get 
\[\Delta(8_{19})=t^{-3}+t^{-2}-1+t^2-t^3,\]
\[\Delta(9_{42})=-t^{-2}+2t^{-1}-1+2t-t^2.\]

Though $8_{19}$ and $9_{42}$ are both non-thin, so there is no predictable collapse of the spectral sequence from \cite{hendricks2025noterealheegaardfloer}, the behaviors of their real Floer groups are in great contrast. For $8_{19}$, the spectral sequence from $\widehat{\HFK}(8_{19})$ to $\widehat{\HFKR}(8_{19},\fra_0)$ collapses immediately at the first page leading to $\Delta(8_{19})=\Delta^R(8_{19},\fra_0)$. For $9_{42}$, the spectral sequence does not degenerate immediately and the constant term of $\Delta^R(9_{42},\fra_a)$ has bigger absolute value than that of $\Delta(9_{42})$. 

It is also worth noting that the real grid homology of $3_1\#4_1$ (see Section~\ref{sub:Kunneth principle fails}) equipped with one of the connected sum involutions shares the same real grid homology with $(9_{42},\fra_c)$. Combining this with Proposition~\ref{prop:fra v.s.fra^r}, we know that $3_1\#4_1$ provides us with another example of a strongly invertible knot whose real grid homology depends on the choice of auxiliary data, while whose underlying knot is $\HFK$-thin (since  $3_1\#4_1$ is alternating).
\end{example}

Combining the example and the above mentioned result from \cite{Sarkar_2015movingbasept}, we make the following conjecture. \begin{conjecture}
Let $K$ be a strongly invertible knot in $(S^3,\tau)$. The auxiliary data $\fra$ does not affect the real knot Floer groups when there is a periodic involution $\iota$ on the same underlying knot $K$ such that $\iota\circ \tau=\tau\circ\iota$. 
\end{conjecture} 

\begin{example}\label{ex: knots with <8 crossings}
With the help of planar diagrams from \cite{Lobb2021ArefinementofKhovanovhomology}, we constructed real grid diagrams for all strongly invertible knots with $\le 7$ crossings. Then, we used the python program to compute various versions of real knot Floer groups associated to these knots. The results are shown in Appendix~\ref{app:Calculation results for knots with small crossing numbers}. It turns out that although real knot Floer groups may be trivial on non-trivial knots, the real Alexander polynomial is able to distinguish different strong inversions on the same knot for all knots with crossing number $\le 7$. This needs the remark on mirroring we mentioned in Remark~\ref{rmk:intro-polynomial} and then proves Proposition~\ref{prop:intro-distinguish involutions on <=7 crossing knots}.
\end{example}
The example above motivates the following conjecture.
\begin{conjecture}
Let $K$ be any knot in $S^3$. If $K$ admits a strong inversion, then there is a strong inversion on it with nontrivial real knot Floer group.  
\end{conjecture}

\begin{example}
In Appendix~\ref{app:Calculation results for knots with small crossing numbers}, we calculated real knot Floer homology of $K\#\tau K^r$ equipped with involution $\widetilde{\tau}_{sw}$ from \cite{BIequiv4genusofsiandperidicknots,DMSequivariantknotandHFK,MallickHFKandsurgeryonequivaiantknot} for $K=3_1, 4_1, 5_1$ and $5_2$. All these examples satisfy  \[\widehat{\HFK}(K)\cong \widehat{\HFKR}(K\#\tau K^r,\fra)\] as bigraded groups (under suitable normalization) for any choice of auxiliary data. 
Motivated by this, we argued in Proposition~\ref{prop:intro-recovering hat HFK from hat HFKR} that for any knot $K$ in $S^3$, we have \[\widehat{\HFK}(K)\cong \widehat{\HFKR}(K\#\tau K^r,\fra)\] as bigraded group (under the correspondence $(M,A)=(M^R,2A^R)$) for any choice of auxiliary data. We now further conjecture that 

\begin{conjecture}
For any knot $K$ in $S^3$, we have \[\HFK^-(K)\cong \HFKR^-(K\#\tau K^r,\fra)\] as bigraded group for any choice of auxiliary data.
\end{conjecture}
\end{example}

\begin{remark}
Using \cite{ZhenkunLioythonprgram}, we can also calculate $\widehat{\GHR}(K,\fro)$ for doubly periodic knots $(K,\fro)$ which admit real grid diagrams of size $\le 13$. However, the author haven't found any example with $\widehat{\GHR}$ supporting in more than one real Maslov gradings, which makes this theory kind of uninteresting. Thus, we mainly focused on strongly invertible knots and links in this paper, though properties like skein relations are expected to be true also for doubly periodic links. 
\end{remark}

\subsection{Relationship between spectral sequences}
For a strongly invertible knot with auxiliary data $(K,
\fra)$ in $(S^3,\tau)$, we have the following diagram of Floer homology groups, in which each arrow represents a spectral sequence.

\[\begin{tikzcd}
    \widehat{\HFK}(K) \arrow{r} \arrow{d} &  \widehat{\HFKR}(K,\fra)\arrow{d} \\
      \widehat{\HF}(S^3) \arrow{r} & \widehat{\HFR}(S^3,\tau)\\ 
\end{tikzcd}\]

The two horizontal spectral sequences come from \cite{hendricks2025noterealheegaardfloer}, the right one comes from Theorem~\ref{thm:spectral sequence relating HFKR and HFR} and the left one is its classical counterpart. Then one can ask that whether the square commutes. We give a negative answer to this using $4_1$. The left and right spectral sequence are shown schematically in Diagram~\eqref{eq:HFK and HFKR of 4_1}. Each row shows a $A$ or $A^R=\frac{1}{2} A$ grading level.

\begin{equation}\label{eq:HFK and HFKR of 4_1}
\begin{tikzcd}
    & \bullet c \arrow{dr}  \\
    \bullet a \arrow{dr} & \bullet x& \bullet e\\
     & \bullet Ub \\ 
\end{tikzcd}\quad \quad \begin{tikzcd}
     \bullet  \arrow{d}  \\
    \bullet  \\
      \bullet \\ 
\end{tikzcd}
\end{equation}

Using the calculation from \cite{DMSequivariantknotandHFK}, Hendricks computed the differential for the top horizontal spectral sequence: $a+x$, $b$ and $c$ survives to $\widehat{\HFKR}$. Combining this, one can see that the generator survives to $\widehat{\HF}(S^3)$ and $\widehat{\HFR}(S^3,\tau)$ are not the ``same'', while we know the bottom horizontal arrow in just the identity map. This is also reflected by the fact that $\tau(4_1)=0$ while $\tau^R(4_1)\ne 0$ (for either involution, see Example~\ref{ex:trefoil and figure 8}). 

\subsection{Calculating $\tau_K$ action on $\widehat{\CFK}$}\label{sub:calculatetauKonclassicalHFK}

In \cite{DMSequivariantknotandHFK}, the authors showed that when $K$ is a strongly invertible knot in $(S^3,\tau)$, there is a well-defined $\tau_K$ action on $\mathcal{CFK}^-$ (the full knot Floer complex) induced by the symmetry of knot. In \cite{hendricks2025noterealheegaardfloer}, Hendricks noted that when the knot admits a real Heegaard diagram on which there is a $R$-symmetric family of almost complex structure such that $\mathrm{Sym}^m(J)$ achieves transversality for non-real moduli spaces, then the differential $d_1$ of her spectral sequence (see Theorem~\ref{thm:spectral sequence from HFK,strongly invertible case}) can be calculated as \[d_1=(1+(\iota_K\tau_K)_*),\]
in which $\iota_K$ is the involution on $\mathcal{CFK}^-$ coming from involutive knot Floer homology (see~\cite{HMinvolutive2017}) and $(\iota_K\tau_K)_*$ denote the induced action of their composition on $\widehat{\HFK}$. It was pointed out by Lipshitz and Ozsv\'{a}th in \cite{LOrealbordered} that the construction in \cite{BGX} is sufficient for this use. By noting that for a $\HFK$-thin knot, \begin{itemize}
    \item $\iota_K$ has been described in \cite[Section~8]{HMinvolutive2017} using some standard pieces;
    \item the spectral sequence $\widehat{\HFK}(K)\Longrightarrow \widehat{\HFKR}(K,\fra)$ collapses after the first page;
\end{itemize}
we expect that we can recover $\tau_K$-action on $\widehat{\HFK}(K)$ or even on $\mathcal{CFK}^-(K)$ via calculation of $\widehat{\HFKR}(K,\fra)$ for nice pairs $(K,\fra)$.

As a concrete example, we note that the involution on $\widehat{\HFK}(6_2)$ equipped with involution shown in Figure~\ref{fig:6_2 diagram} can be determined by $\widehat{\HFKR}(6_2,\fra)$ with $\fra$ implicitly given in the figure. The calculation was done by Kristen Hendricks after chatting with the author. 

\begin{figure}
    \centering
    \includegraphics[width=0.35\linewidth]{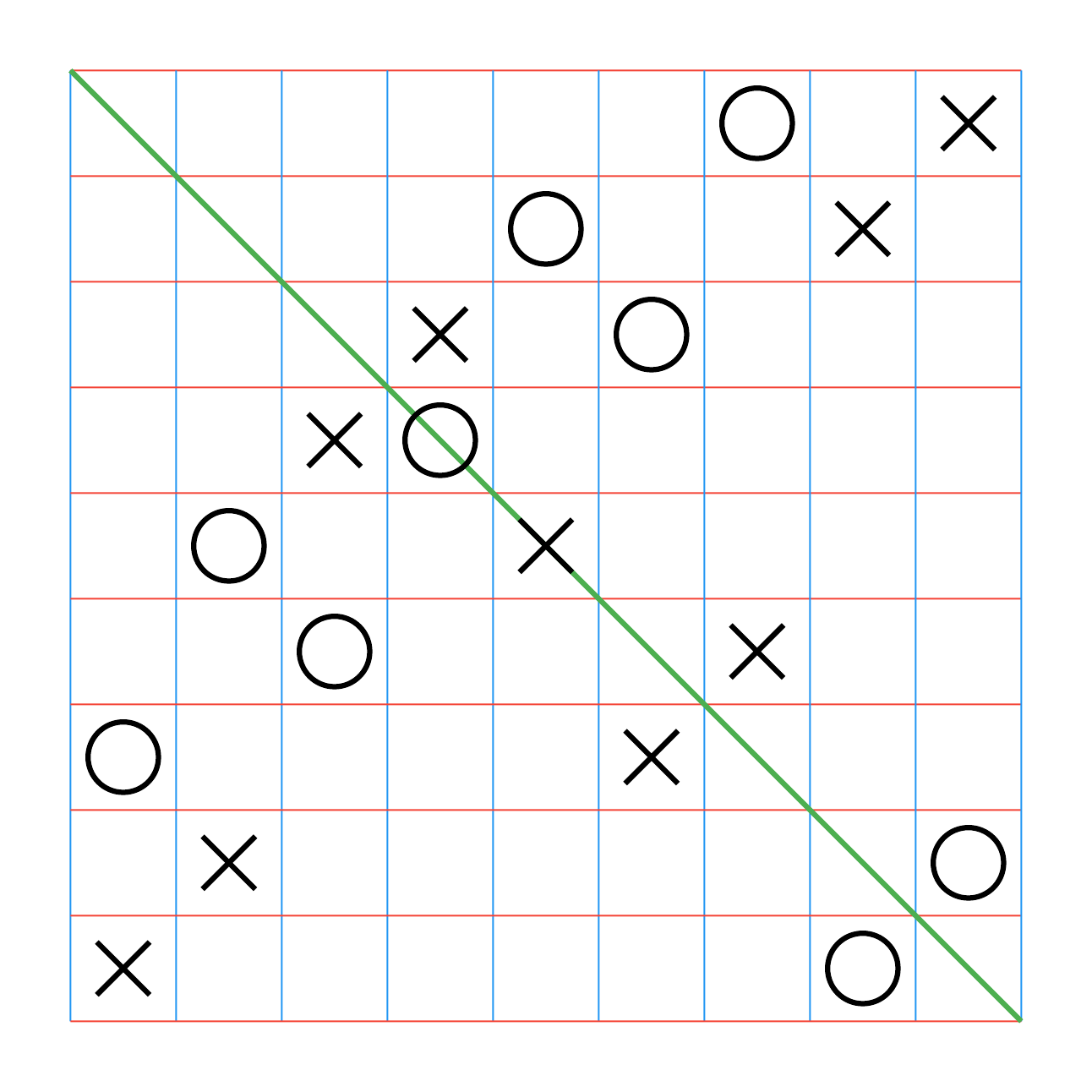}
    \caption{A grid diagram for $6_2$ with one of its a strongly inversion.}
    \label{fig:6_2 diagram}
\end{figure}

$\mathcal{CFK}(6_2)$ looks like in Diagram~\eqref{eq:HFK of 6_2}.

\begin{equation}\label{eq:HFK of 6_2}
\begin{tikzcd}
\bullet b_1 \arrow{d} &\bullet a_1\arrow{d} \arrow{l} \\
\bullet e_1  & \bullet c_1\arrow{l}\\
\bullet x_1  & \bullet x_0\arrow{d}\arrow{l} & \bullet b_2\arrow{d} & \bullet a_2 \arrow{d} \arrow{l} \\
     & \bullet x_2 & \bullet e_2 & \bullet c_2\arrow{l}\\ 
\end{tikzcd} 
\quad
\end{equation}

The $\iota_K$ action is given by \[x_0\mapsto x_0, \quad x_1 \leftrightarrow x_2,\quad a_1\mapsto a_2+e_2,\quad a_2\mapsto a_1,\] \[c_1 \leftrightarrow b_2,\quad b_1 \leftrightarrow c_2, \quad e_1 \leftrightarrow e_2,\] which follows immediately from \cite[Section~8]{HMinvolutive2017}.

From Appendix~\ref{app:Calculation results for knots with small crossing numbers}, we can see that $\widehat{\HFKR}(6_2,\fra)$ has rank 1, 1, 3, 1, 1 in $A^R$ grading -1, -1/2, 0, 1/2, 1, respectively. Using these and the formal property of $\iota_K$ and $\tau_K$, one can deduce that up to basis changes, $\tau_K$ is given by 
\[x_0\mapsto x_0, \quad x_1 \leftrightarrow x_2,\quad a_1\mapsto a_2+x_2,\quad a_2\mapsto a_1+x_1,\] \[c_1 \leftrightarrow b_2,\quad b_1 \leftrightarrow c_2, \quad e_1 \leftrightarrow e_2.\]

In general, it is not enough to determine $\tau_K$ up to basis change from $\widehat{\HFK}$, $\widehat{\HFKR}$ and $\iota_K$, even when the knot is thin. In fact, $6_1$ with one of its involutions provides such an example. But the author expect that for some specific examples like $6_2$, $\tau_K$ can be determined and have further interesting applications.

Before ending this subsection, we state a similar result for doubly periodic knots. This was communicated by Hendricks to us, which follows from the argument of \cite[Lemma~2.9]{hendricks2025noterealheegaardfloer} (see Section~6 there for details) and the construction of real nice diagrams from \cite{BGX}.

\begin{prop}
Any doubly periodic knot $(K,\fra)$ in $S^3$ admits a real Heegaard diagram on which there is a $R$-symmetric family of almost complex structure such that $\mathrm{Sym}^m(J)$ achieves transversality for non-real moduli spaces, so the differential $d_1$ in the spectral sequence from Theorem~\ref{thm:spectral sequence from HFK,doubly periodic case} can be calculated as \[d_1=(1+(\iota_K\tau_K\varsigma_K)_*),\]
in which $\iota_K$ is same as above, $\tau_K$ is the induced action of $\tau$ on $\widehat{\HFK}(K)$ coming from \cite[Section~8]{DMSequivariantknotandHFK} and $\varsigma_K$ is the Sarkar map from \cite{Sarkar_2015movingbasept}.  
\end{prop}

\appendix
\section{Calculation results for strongly invertible knots with small crossing numbers}\label{app:Calculation results for knots with small crossing numbers}





{\bf Introduction}. This is made jointly by Zhenkun Li and Yonghan Xiao. The first author developed the computer program and made the table shown below, while the second author collected examples of strongly invertible knots and constructed real grid diagrams for them.

The appendix includes a list of small crossing strongly invertible knots and their real grid homology. We cover all involutions on knots with less than or equal to $7$ crossings, using the planar diagrams from \cite[Appendix]{Lobb2021ArefinementofKhovanovhomology} and some other strongly invertible knots that admit a real grid diagram with size less than $17$. The computer program \cite{ZhenkunLioythonprgram} is able to calculate minus version of real grid homology for strongly invertible knots with diagrams of size $\le 11$, hat version of real grid homology for strongly invertible knots with diagrams of size $\le 13$ and the real Alexander polynomial for strongly invertible knots with diagrams of size $\le 17$.

{\bf AI usage}. Here we explain how the appendix was jointly made by human and a AI coding agent.
\begin{itemize}
	\item The decomposition and translation from mathematical construction to algorithm realizable by computer were done by human.
	\item The python code was written by a AI coding agent and then read through and verified by human.
	\item Knot data was collected by human and grid diagrams were also constructed and input into the python program by hand.
	\item The knot homologies and polynomials were output by the python program via json files.
	\item The AI coding agent wrote a python program with json files as input and latex code as output. This was used to generate the table below. 
	\item It has been checked by human that data in the table agree with the original output from json files.
\end{itemize}

{\bf Code trustworthiness}.
The scripts that compute real invariants from grid diagrams were developed by Zhenkun Li with assistance from the coding agent, while the grid diagrams of various knots were collected by Yonghan Xiao. We would like to be transparent about the extent to which these scripts should be trusted.
\begin{itemize}
	\item The core functionality---building the set of generators for the real chain complex, finding all domains between generators, and computing the bi-grading of each generator---was written by hand and tested against one example: a size-5 grid diagram of the trefoil, which has 26 generators and 61 domains.
	\item The computations for the polynomial and the real homology (hat and minus versions) were generated by the coding agent. The following validations have been applied:
	\begin{itemize}
		\item Direct comparison on the trefoil knot, the only non-trivial example computable by hand.
		\item Verification that $d^2 = 0$ for all examples.
		\item An intermediate step of the computation produces the real homology of the tilde version, whose dimension is divisible by a high power of $2$ (e.g., $128$ for any strongly invertible knot on a size-15 grid). This divisibility holds for all examples we tested.
		\item Symmetry of the polynomials and homologies for small-crossing knots.
	\end{itemize}
\end{itemize}
{\bf Notation}.
\begin{itemize}
	\item For the name of knots, we essentially use the name from Rolfsen table following \cite{knotinfo}, but we won't distinguish a knot with its mirror.
	\item For hat version of the homology, each summand is of the form $(2a,m)^d$ where $a$ is the real Alexander grading, $m$ is the real Maslov grading, and $d$ is the dimension of the corresponding graded space.
	\item For minus version, each summand is of the form $\bigg(U^o_{(2a,m)}\bigg)^d$. Here again $(a,m)$ encodes the grading information. The "power" of $U$ represents the $U$-torsion order, and $d$ represents the number of generators on the corresponding bi-grading and has the corresponding order. The first factor is always of the form $U^{\infty}_{(a,m)}$ representing the generator of the (unique) infinite $U$-tower.
\end{itemize}
We remind readers that the notation here is a little bit different from the one in the main part of our paper, which we write the grading as $(M^R,A^R)$ instead of $(2A^R,M^R)$.
\vspace{0.2in}

\setlength{\tabcolsep}{3pt}
\begin{longtable}{m{0.25\textwidth} m{0.3\textwidth} m{0.45\textwidth}}
  \hline
  \centering \textbf{Knot} & \centering \textbf{Grid Diagram} & \centering \textbf{Real Invariants} \tabularnewline
  \hline
  \endfirsthead
  \hline
  \centering \textbf{Knot} & \centering \textbf{Grid Diagram} & \centering \textbf{Real Invariants} \tabularnewline
  \hline
  \endhead
  \hline
  \endfoot
  \parbox[c]{\linewidth}{\centering $3_1$\\{\tiny Also twisted${}_1$}} & \centering \includegraphics[width=0.25\textwidth]{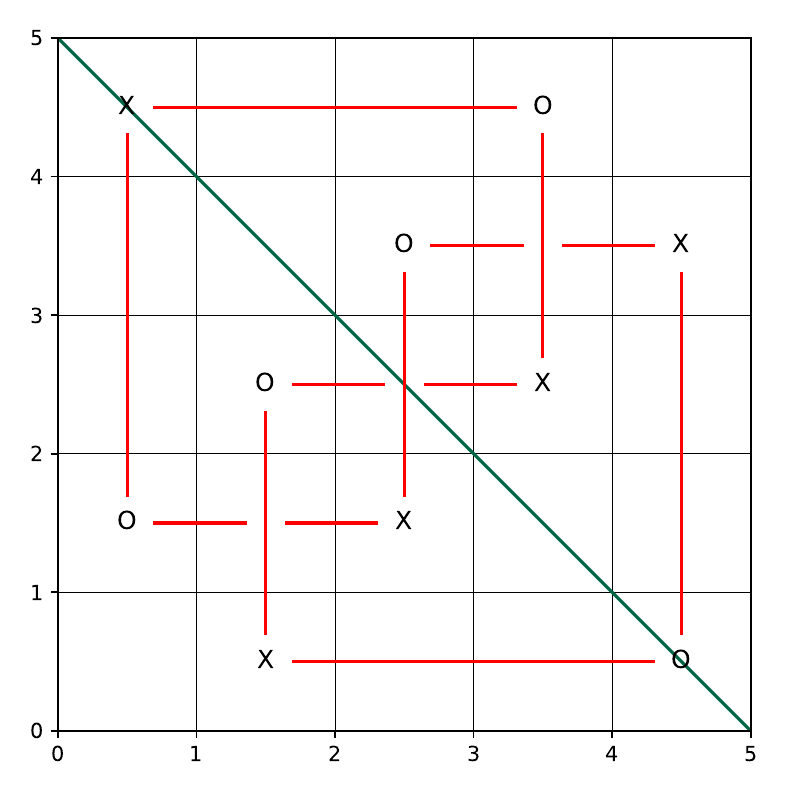} & \parbox[c]{\linewidth}{\centering\tiny
\textcolor{blue}{Polynomial Invariant}\\
$t^{-1} + 1 - t$\\[0.1in]
\textcolor{blue}{Real grid homology - hat version}\\
$(-1,0)\oplus (0,0)\oplus (1,1)$\\[0.1in]
\textcolor{blue}{Real grid homology - minus version}\\
$U^{\infty}_{(1,1)}\oplus U_{(0,0)}$
} \tabularnewline
  \hline
  \parbox[c]{\linewidth}{\centering $4_1$\\{\tiny Also twisted${}_2$}} & \centering \includegraphics[width=0.25\textwidth]{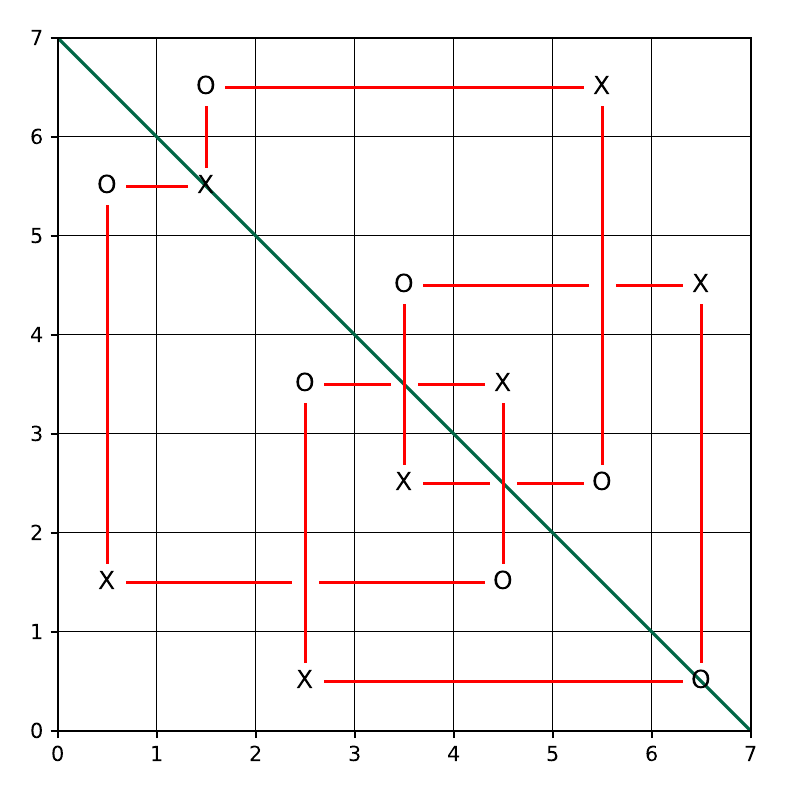} & \parbox[c]{\linewidth}{\centering\tiny
\textcolor{blue}{Polynomial Invariant}\\
$-t^{-1} + 1 + t$\\[0.1in]
\textcolor{blue}{Real grid homology - hat version}\\
$(-1,-1)\oplus (0,0)\oplus (1,0)$\\[0.1in]
\textcolor{blue}{Real grid homology - minus version}\\
$U^{\infty}_{(-1,-1)}\oplus U_{(1,0)}$
} \tabularnewline
  \hline
  \centering $5_1$ & \centering \includegraphics[width=0.25\textwidth]{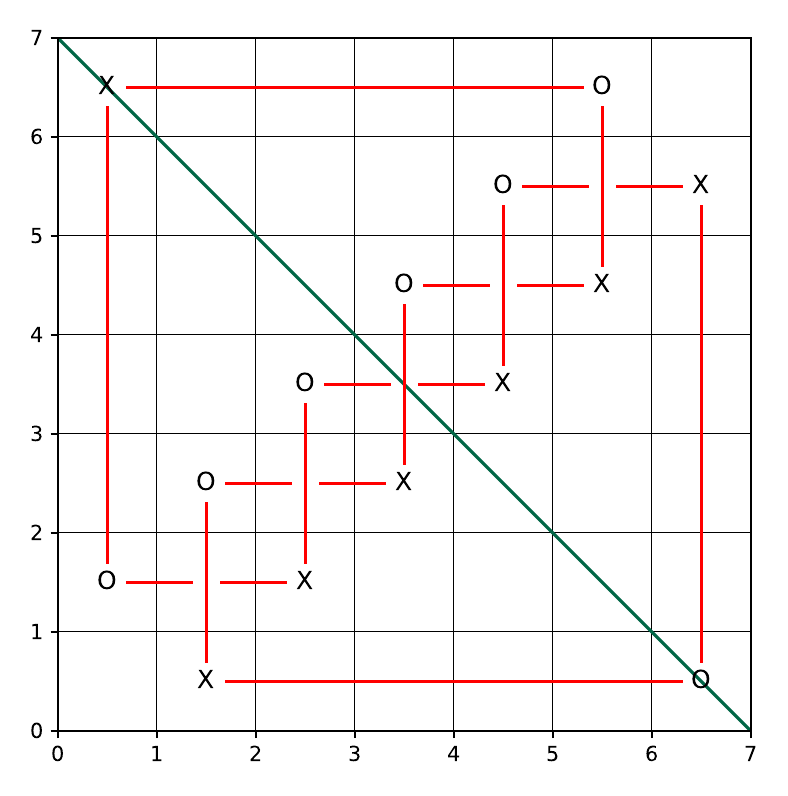} & \parbox[c]{\linewidth}{\centering\tiny
\textcolor{blue}{Polynomial Invariant}\\
$t^{-2} + t^{-1} - 1 - t + t^{2}$\\[0.1in]
\textcolor{blue}{Real grid homology - hat version}\\
$(-2,0)\oplus (-1,0)\oplus (0,1)\oplus (1,1)\oplus (2,2)$\\[0.1in]
\textcolor{blue}{Real grid homology - minus version}\\
$U^{\infty}_{(2,2)}\oplus U_{(-1,0)}\oplus U_{(1,1)}$
} \tabularnewline
  \hline
  \parbox[c]{\linewidth}{\centering $5_2$\\{\tiny Also twisted${}_3$}} & \centering \includegraphics[width=0.25\textwidth]{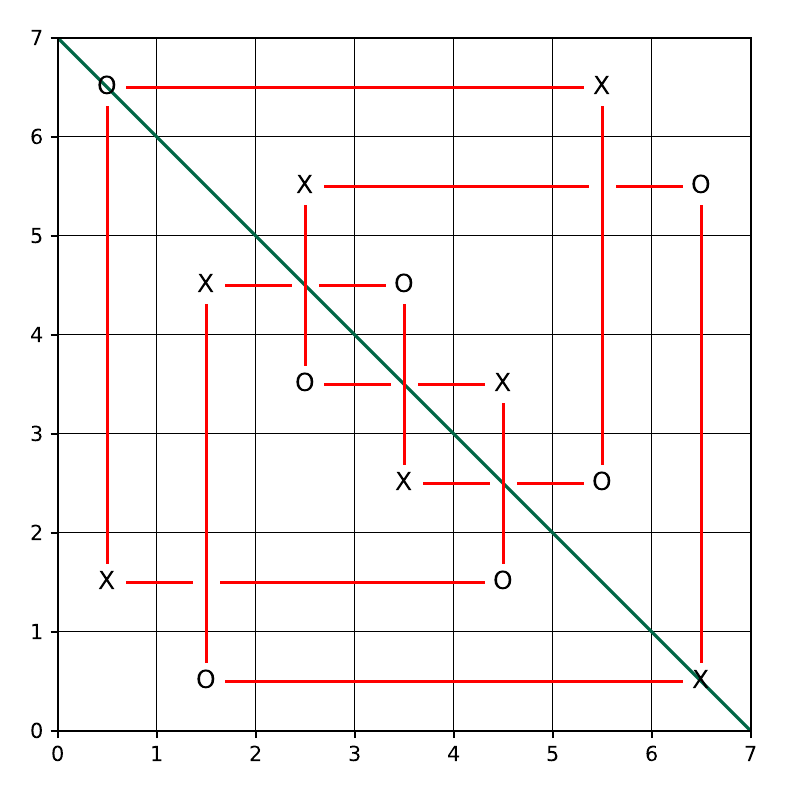} & \parbox[c]{\linewidth}{\centering\tiny
\textcolor{blue}{Polynomial Invariant}\\
$1$\\[0.1in]
\textcolor{blue}{Real grid homology - hat version}\\
$(0,0)$\\[0.1in]
\textcolor{blue}{Real grid homology - minus version}\\
$U^{\infty}_{(0,0)}$
} \tabularnewline
  \hline
  \parbox[c]{\linewidth}{\centering $(5_2)'$\\{\tiny Second symmetry}} & \centering \includegraphics[width=0.25\textwidth]{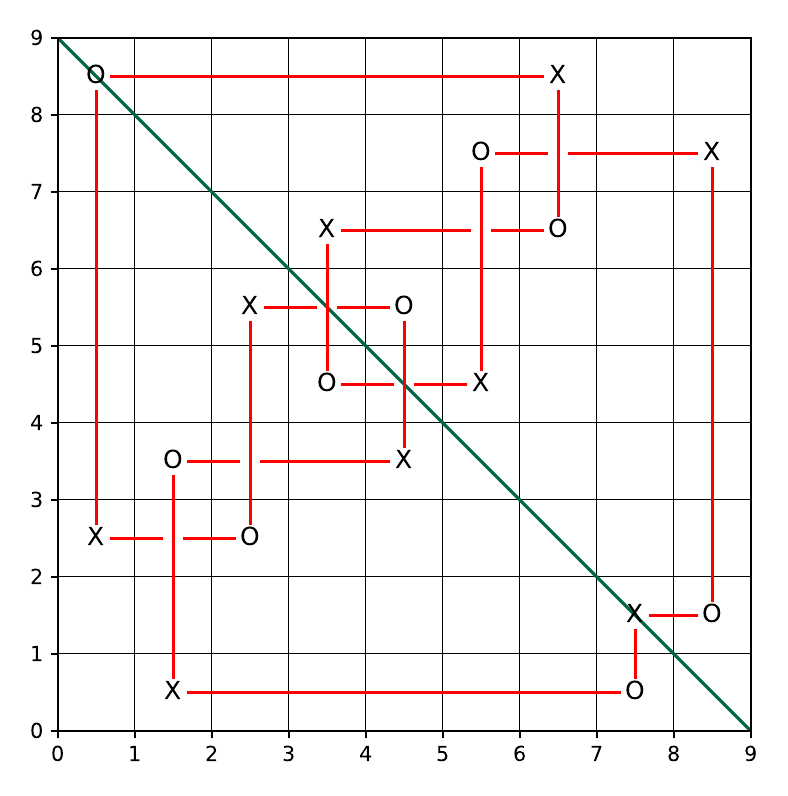} & \parbox[c]{\linewidth}{\centering\tiny
\textcolor{blue}{Polynomial Invariant}\\
$-2t^{-1} + 1 + 2t$\\[0.1in]
\textcolor{blue}{Real grid homology - hat version}\\
$(-1,-1)^{2}\oplus (0,0)\oplus (1,0)^{2}$\\[0.1in]
\textcolor{blue}{Real grid homology - minus version}\\
$U^{\infty}_{(-1,-1)}\oplus U^{2}_{(1,0)}\oplus U_{(1,0)}$
} \tabularnewline
  \hline
  \parbox[c]{\linewidth}{\centering $6_1$\\{\tiny Also twisted${}_4$}} & \centering \includegraphics[width=0.25\textwidth]{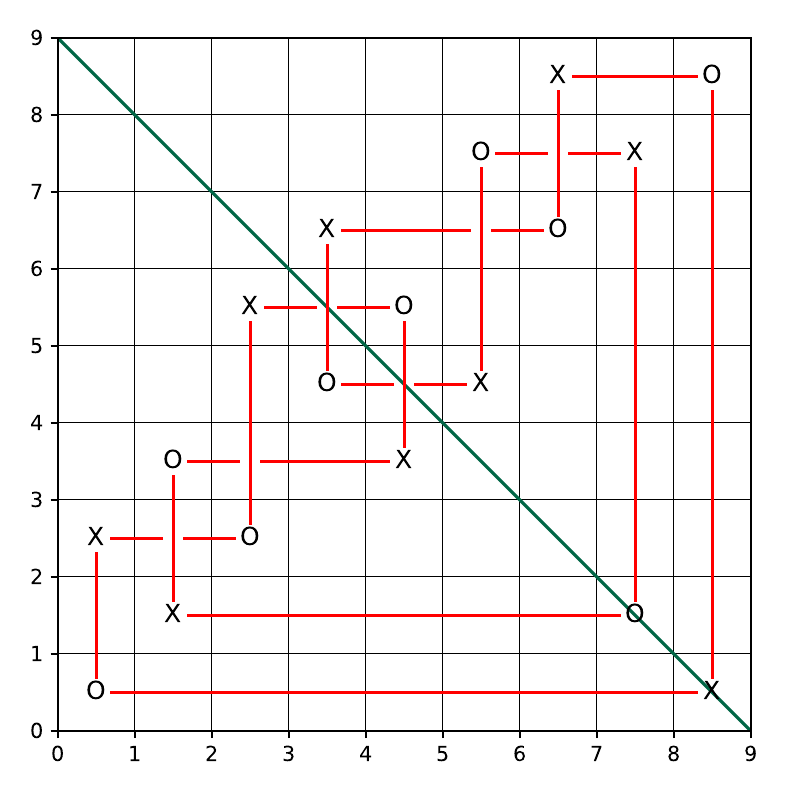} & \parbox[c]{\linewidth}{\centering\tiny
\textcolor{blue}{Polynomial Invariant}\\
$-2t^{-1} + 1 + 2t$\\[0.1in]
\textcolor{blue}{Real grid homology - hat version}\\
$(-1,-1)^{2}\oplus (0,0)\oplus (1,0)^{2}$\\[0.1in]
\textcolor{blue}{Real grid homology - minus version}\\
$U^{\infty}_{(-1,-1)}\oplus U_{(1,0)}\oplus U^{2}_{(1,0)}$
} \tabularnewline
  \hline
  \parbox[c]{\linewidth}{\centering $(6_1)'$\\{\tiny Second symmetry}} & \centering \includegraphics[width=0.25\textwidth]{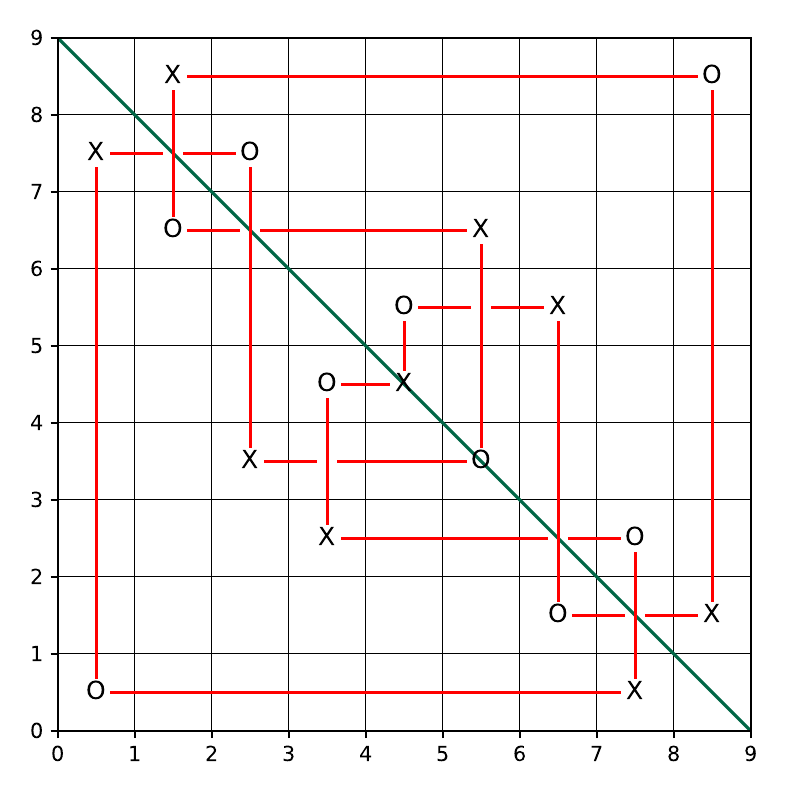} & \parbox[c]{\linewidth}{\centering\tiny
\textcolor{blue}{Polynomial Invariant}\\
$1$\\[0.1in]
\textcolor{blue}{Real grid homology - hat version}\\
$(0,0)$\\[0.1in]
\textcolor{blue}{Real grid homology - minus version}\\
$U^{\infty}_{(0,0)}$
} \tabularnewline
  \hline
  \centering $6_2$ & \centering \includegraphics[width=0.25\textwidth]{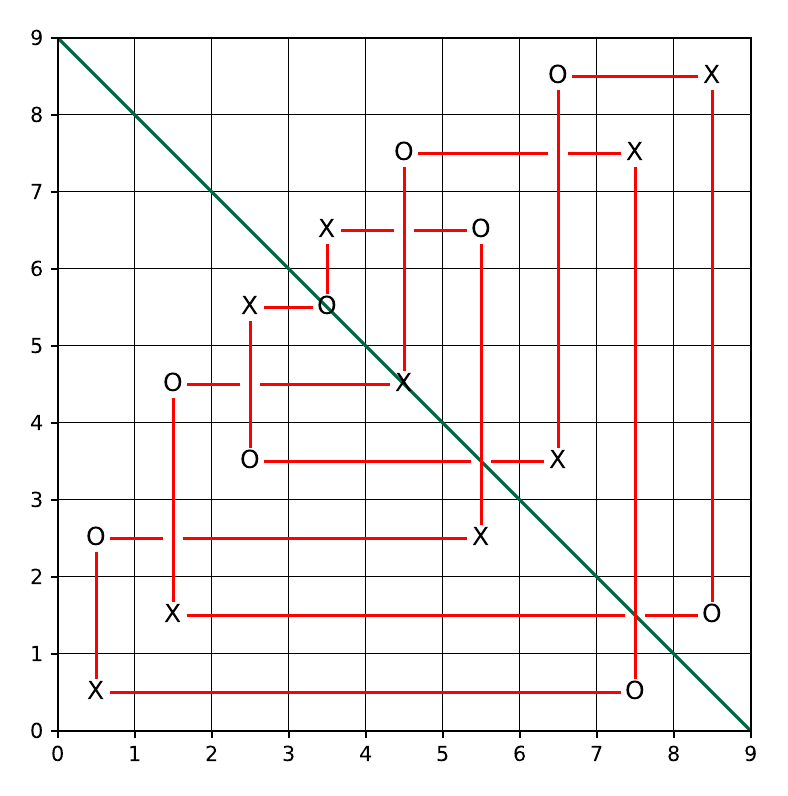} & \parbox[c]{\linewidth}{\centering\tiny
\textcolor{blue}{Polynomial Invariant}\\
$-t^{-2} + t^{-1} + 3 - t - t^{2}$\\[0.1in]
\textcolor{blue}{Real grid homology - hat version}\\
$(-2,-1)\oplus (-1,0)\oplus (0,0)^{3}\oplus (1,1)\oplus (2,1)$\\[0.1in]
\textcolor{blue}{Real grid homology - minus version}\\
$U^{\infty}_{(0,0)}\oplus U_{(0,0)}\oplus U_{(2,1)}\oplus U^{2}_{(0,0)}$
} \tabularnewline
  \hline
  \parbox[c]{\linewidth}{\centering $(6_2)'$\\{\tiny Second symmetry}} & \centering \includegraphics[width=0.25\textwidth]{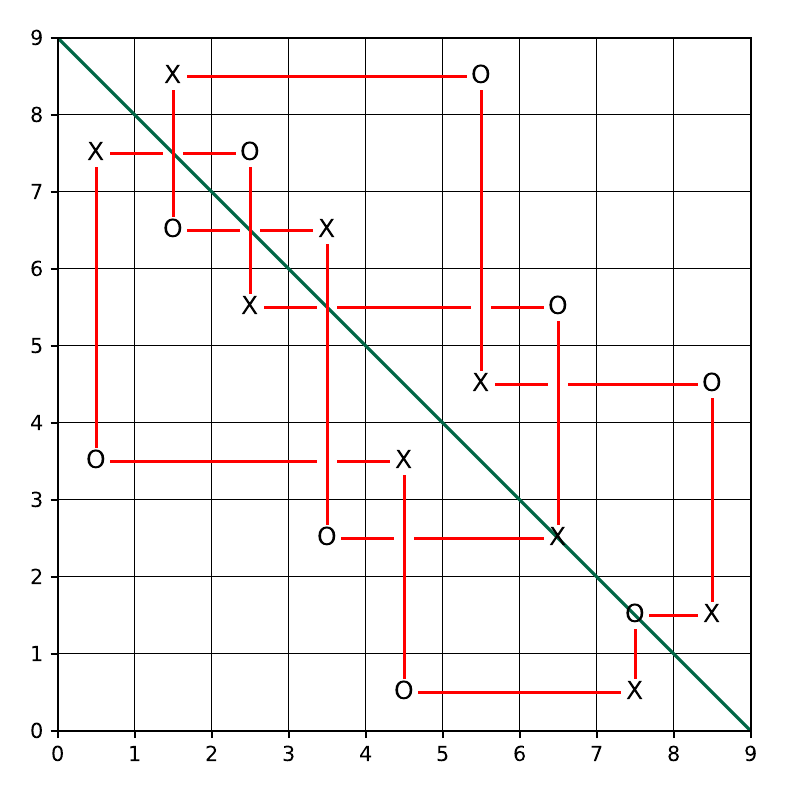} & \parbox[c]{\linewidth}{\centering\tiny
\textcolor{blue}{Polynomial Invariant}\\
$t^{-2} - t^{-1} - 1 + t + t^{2}$\\[0.1in]
\textcolor{blue}{Real grid homology - hat version}\\
$(-2,-2)\oplus (-1,-1)\oplus (0,-1)\oplus (1,0)\oplus (2,0)$\\[0.1in]
\textcolor{blue}{Real grid homology - minus version}\\
$U^{\infty}_{(-2,-2)}\oplus U_{(0,-1)}\oplus U_{(2,0)}$
} \tabularnewline
  \hline
  \centering $6_3$ & \centering \includegraphics[width=0.25\textwidth]{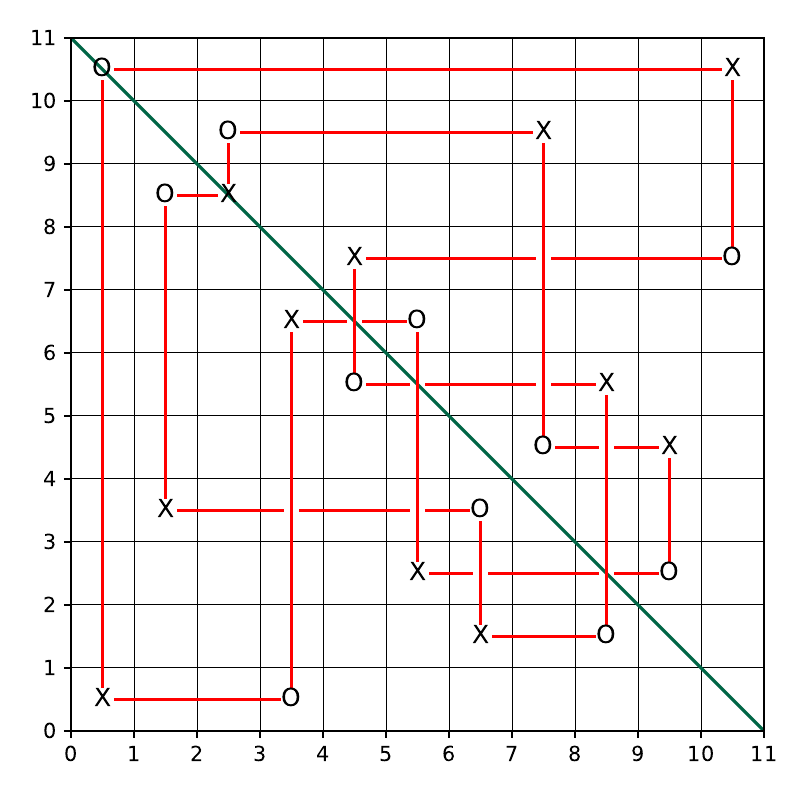} & \parbox[c]{\linewidth}{\centering\tiny
\textcolor{blue}{Polynomial Invariant}\\
$-t^{-2} - t^{-1} + 3 + t - t^{2}$\\[0.1in]
\textcolor{blue}{Real grid homology - hat version}\\
$(-2,-1)\oplus (-1,-1)\oplus (0,0)^{3}\oplus (1,0)\oplus (2,1)$\\[0.1in]
\textcolor{blue}{Real grid homology - minus version}\\
$U^{\infty}_{(0,0)}\oplus U^{2}_{(2,1)}\oplus U_{(-1,-1)}\oplus U_{(1,0)}$
} \tabularnewline
  \hline
  \centering $7_1$ & \centering \includegraphics[width=0.25\textwidth]{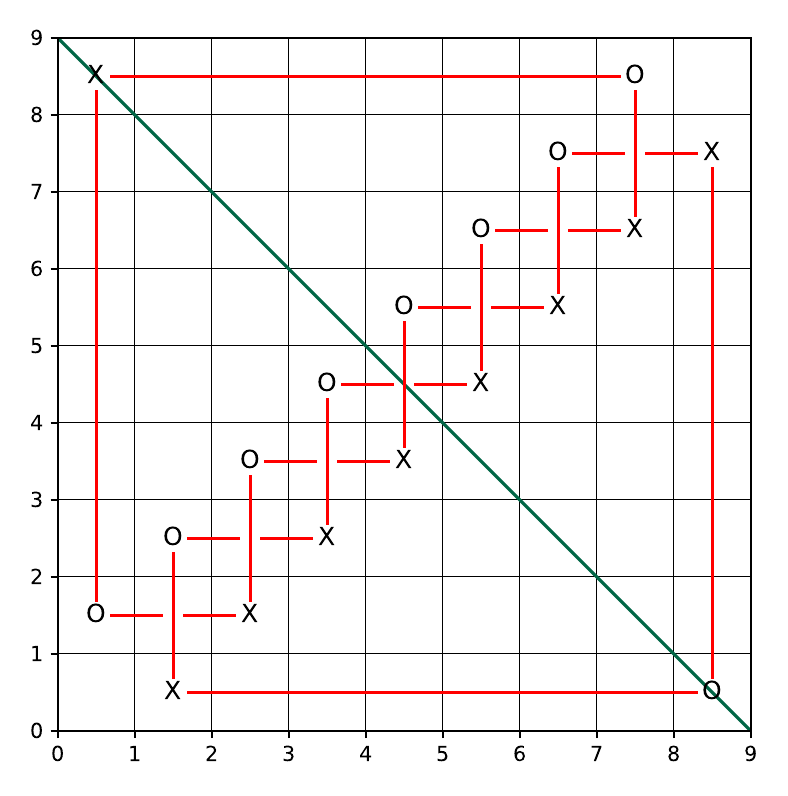} & \parbox[c]{\linewidth}{\centering\tiny
\textcolor{blue}{Polynomial Invariant}\\
$t^{-3} + t^{-2} - t^{-1} - 1 + t + t^{2} - t^{3}$\\[0.1in]
\textcolor{blue}{Real grid homology - hat version}\\
$(-3,0)\oplus (-2,0)\oplus (-1,1)\oplus (0,1)\oplus (1,2)\oplus (2,2)\oplus (3,3)$\\[0.1in]
\textcolor{blue}{Real grid homology - minus version}\\
$U^{\infty}_{(3,3)}\oplus U_{(-2,0)}\oplus U_{(0,1)}\oplus U_{(2,2)}$
} \tabularnewline
  \hline
  \parbox[c]{\linewidth}{\centering $7_2$\\{\tiny Also twisted${}_5$}} & \centering \includegraphics[width=0.25\textwidth]{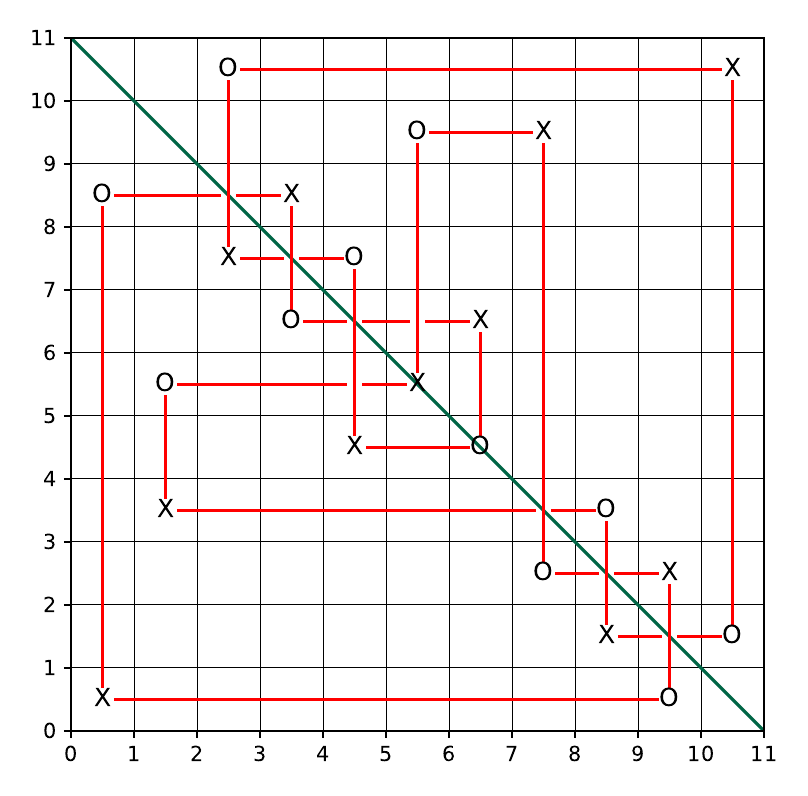} & \parbox[c]{\linewidth}{\centering\tiny
\textcolor{blue}{Polynomial Invariant}\\
$t^{-1} + 1 - t$\\[0.1in]
\textcolor{blue}{Real grid homology - hat version}\\
$(-1,0)\oplus (0,0)\oplus (1,1)$\\[0.1in]
\textcolor{blue}{Real grid homology - minus version}\\
$U^{\infty}_{(1,1)}\oplus U_{(0,0)}$
} \tabularnewline
  \hline
  \parbox[c]{\linewidth}{\centering $(7_2)'$\\{\tiny Second symmetry}} & \centering \includegraphics[width=0.25\textwidth]{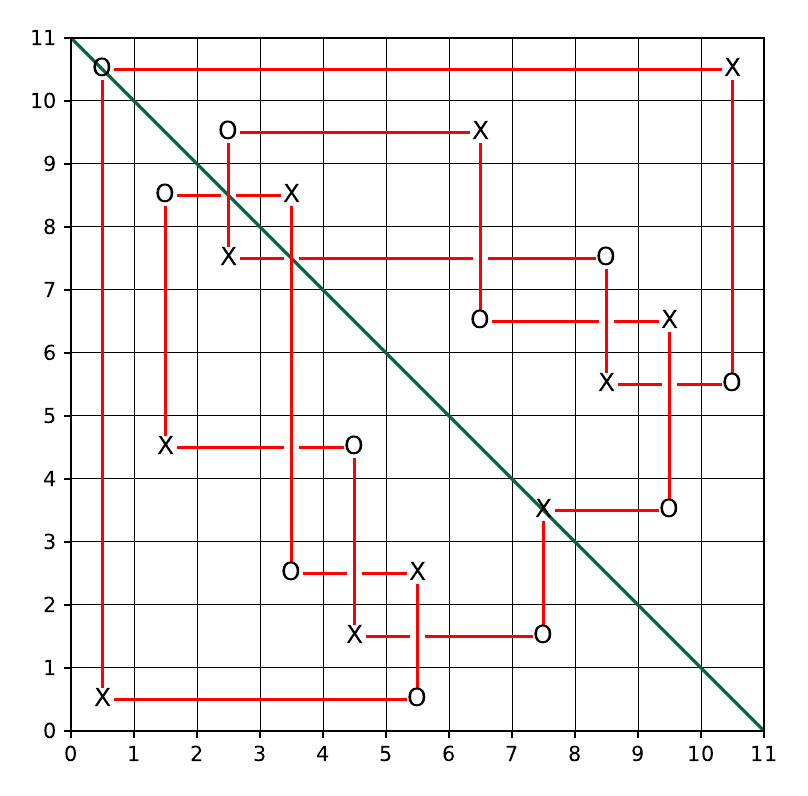} & \parbox[c]{\linewidth}{\centering\tiny
\textcolor{blue}{Polynomial Invariant}\\
$3t^{-1} + 1 - 3t$\\[0.1in]
\textcolor{blue}{Real grid homology - hat version}\\
$(-1,0)^{3}\oplus (0,0)\oplus (1,1)^{3}$\\[0.1in]
\textcolor{blue}{Real grid homology - minus version}\\
$U^{\infty}_{(1,1)}\oplus \bigg(U^{2}_{(1,1)}\bigg)^{2}\oplus U_{(0,0)}$
} \tabularnewline
  \hline
  \centering $7_3$ & \centering \includegraphics[width=0.25\textwidth]{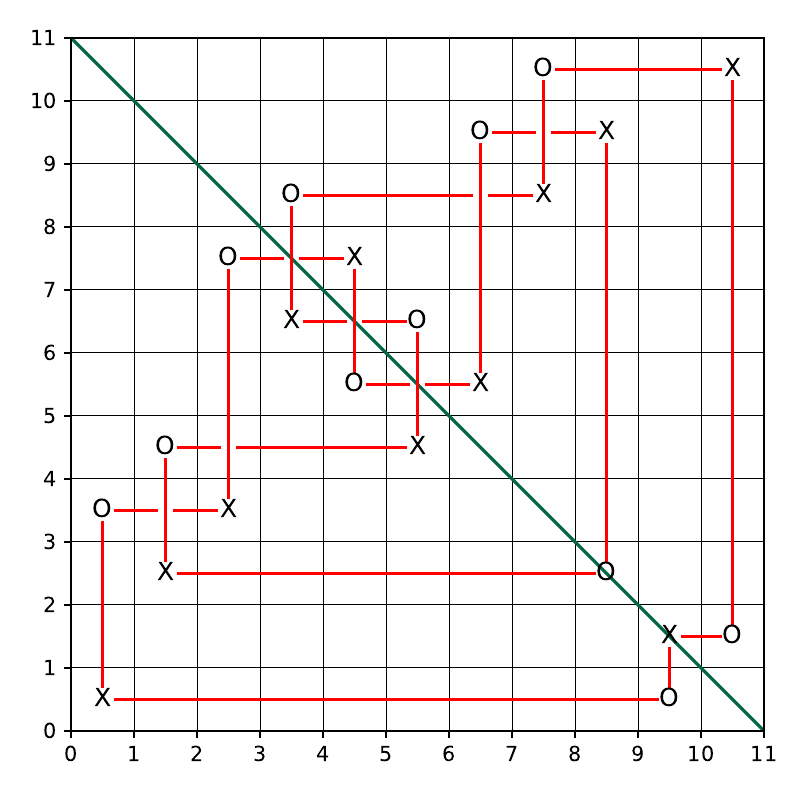} & \parbox[c]{\linewidth}{\centering\tiny
\textcolor{blue}{Polynomial Invariant}\\
$t^{-1} + 1 - t$\\[0.1in]
\textcolor{blue}{Real grid homology - hat version}\\
$(-1,0)\oplus (0,0)\oplus (1,1)$\\[0.1in]
\textcolor{blue}{Real grid homology - minus version}\\
$U^{\infty}_{(1,1)}\oplus U_{(0,0)}$
} \tabularnewline
  \hline
  \parbox[c]{\linewidth}{\centering $(7_3)'$\\{\tiny Second symmetry}} & \centering \includegraphics[width=0.25\textwidth]{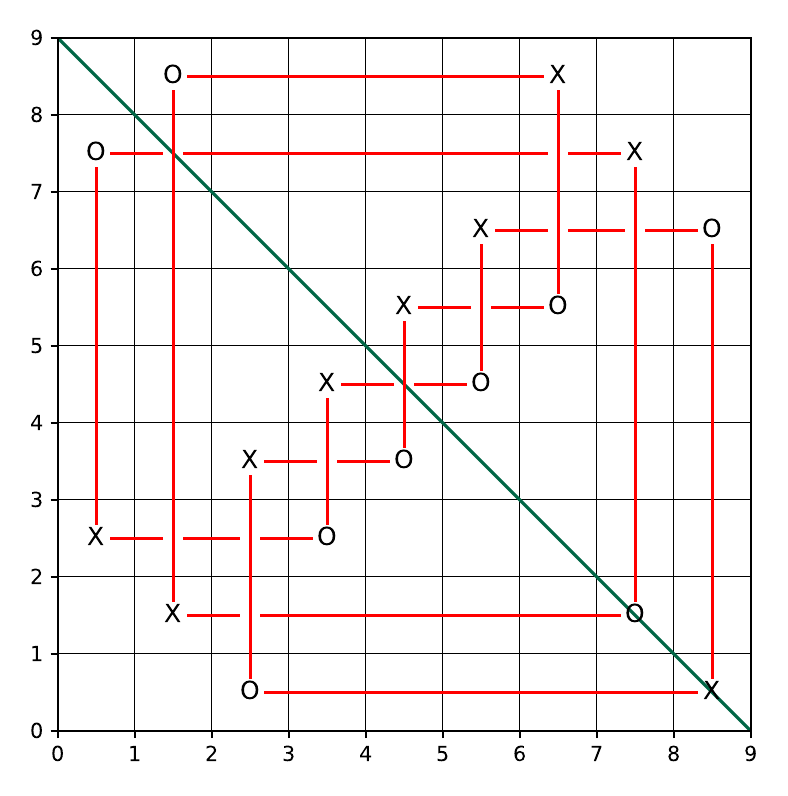} & \parbox[c]{\linewidth}{\centering\tiny
\textcolor{blue}{Polynomial Invariant}\\
$2t^{-2} + t^{-1} - 3 - t + 2t^{2}$\\[0.1in]
\textcolor{blue}{Real grid homology - hat version}\\
$(-2,0)^{2}\oplus (-1,0)\oplus (0,1)^{3}\oplus (1,1)\oplus (2,2)^{2}$\\[0.1in]
\textcolor{blue}{Real grid homology - minus version}\\
$U^{\infty}_{(2,2)}\oplus U^{2}_{(0,1)}\oplus U^{2}_{(2,2)}\oplus U_{(-1,0)}\oplus U_{(1,1)}$
} \tabularnewline
  \hline
  \centering $7_4$ & \centering \includegraphics[width=0.25\textwidth]{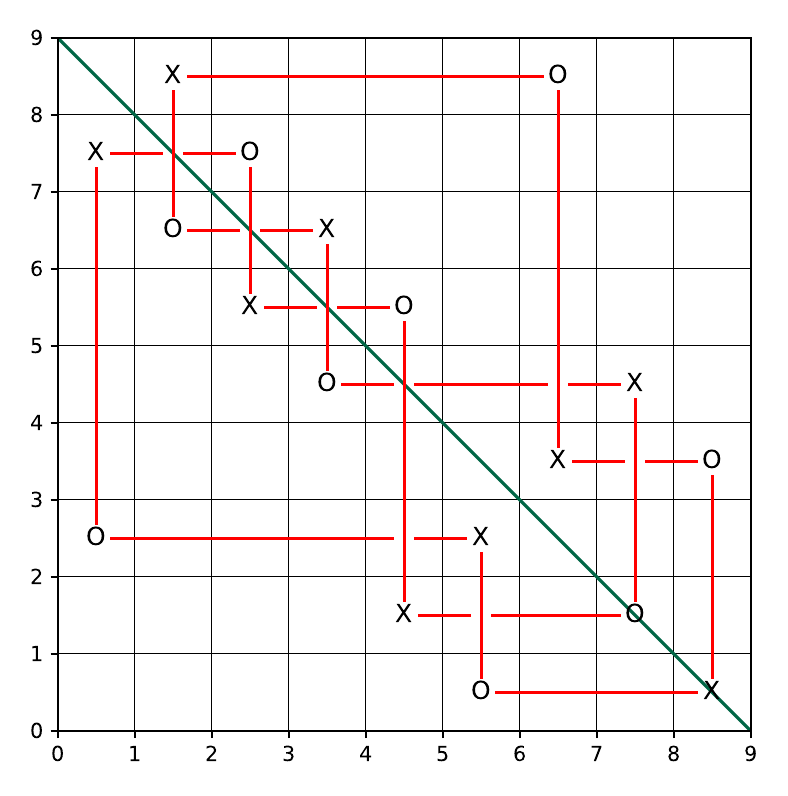} & \parbox[c]{\linewidth}{\centering\tiny
\textcolor{blue}{Polynomial Invariant}\\
$1$\\[0.1in]
\textcolor{blue}{Real grid homology - hat version}\\
$(0,0)$\\[0.1in]
\textcolor{blue}{Real grid homology - minus version}\\
$U^{\infty}_{(0,0)}$
} \tabularnewline
  \hline
  \parbox[c]{\linewidth}{\centering $(7_4)'$\\{\tiny Second symmetry}} & \centering \includegraphics[width=0.25\textwidth]{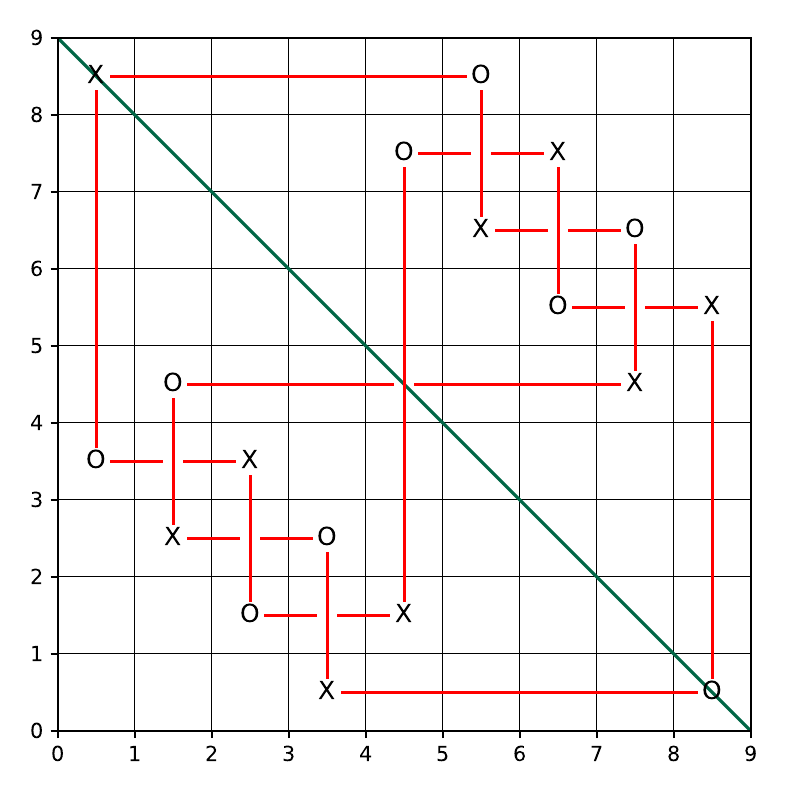} & \parbox[c]{\linewidth}{\centering\tiny
\textcolor{blue}{Polynomial Invariant}\\
$2t^{-1} + 1 - 2t$\\[0.1in]
\textcolor{blue}{Real grid homology - hat version}\\
$(-1,0)^{2}\oplus (0,0)\oplus (1,1)^{2}$\\[0.1in]
\textcolor{blue}{Real grid homology - minus version}\\
$U^{\infty}_{(1,1)}\oplus U^{2}_{(1,1)}\oplus U_{(0,0)}$
} \tabularnewline
  \hline
  \centering $7_5$ & \centering \includegraphics[width=0.25\textwidth]{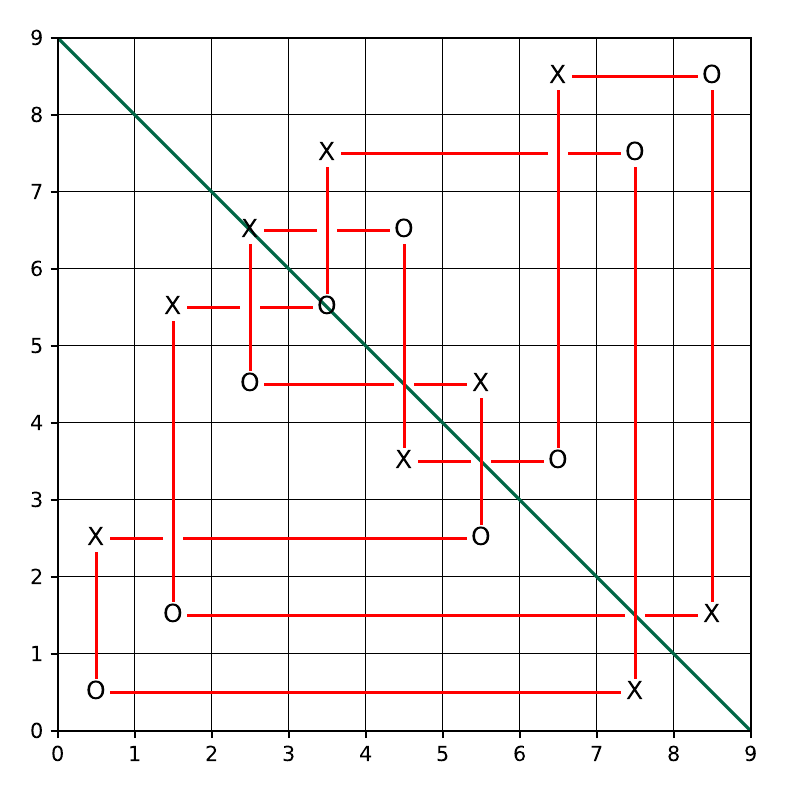} & \parbox[c]{\linewidth}{\centering\tiny
\textcolor{blue}{Polynomial Invariant}\\
$2t^{-1} + 1 - 2t$\\[0.1in]
\textcolor{blue}{Real grid homology - hat version}\\
$(-1,0)^{2}\oplus (0,0)\oplus (1,1)^{2}$\\[0.1in]
\textcolor{blue}{Real grid homology - minus version}\\
$U^{\infty}_{(1,1)}\oplus U^{2}_{(1,1)}\oplus U_{(0,0)}$
} \tabularnewline
  \hline
  \parbox[c]{\linewidth}{\centering $(7_5)'$\\{\tiny Second symmetry}} & \centering \includegraphics[width=0.25\textwidth]{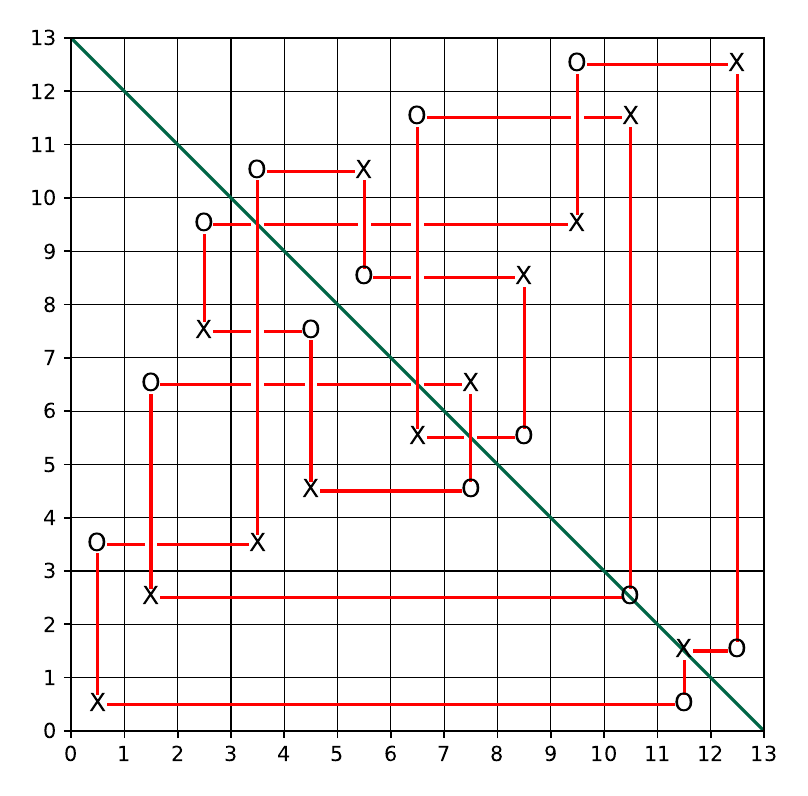} & \parbox[c]{\linewidth}{\centering\tiny
\textcolor{blue}{Polynomial Invariant}\\
$2t^{-2} + 2t^{-1} - 3 - 2t + 2t^{2}$\\[0.1in]
\textcolor{blue}{Real grid homology - hat version}\\
$(-2,0)^{2}\oplus (-1,0)^{2}\oplus (0,1)^{3}\oplus (1,1)^{2}\oplus (2,2)^{2}$
} \tabularnewline
  \hline
  \centering $7_6$ & \centering \includegraphics[width=0.25\textwidth]{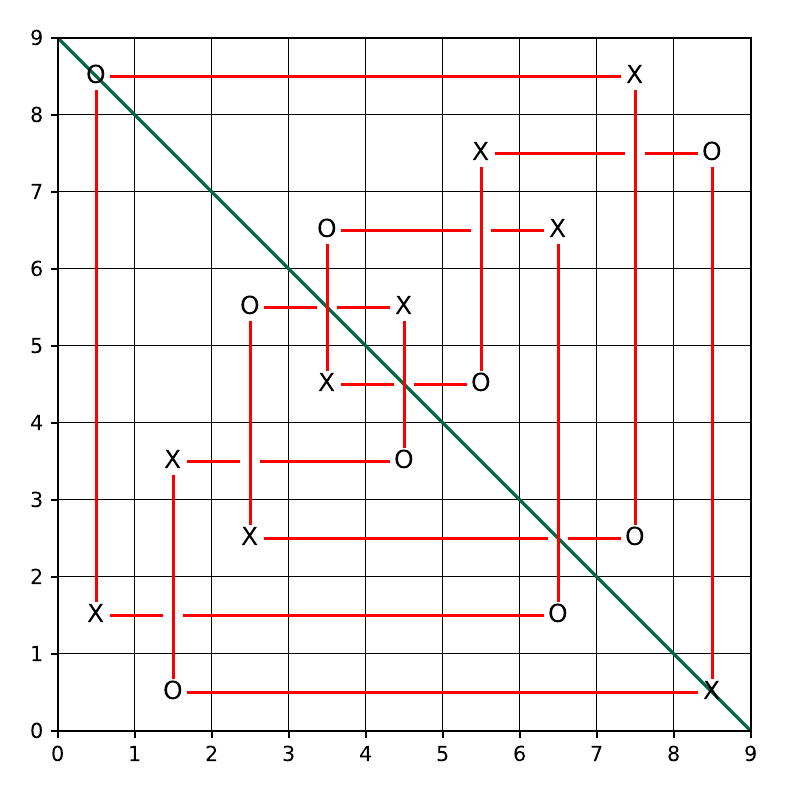} & \parbox[c]{\linewidth}{\centering\tiny
\textcolor{blue}{Polynomial Invariant}\\
$-t^{-2} - t^{-1} + 3 + t - t^{2}$\\[0.1in]
\textcolor{blue}{Real grid homology - hat version}\\
$(-2,-1)\oplus (-1,-1)\oplus (0,0)^{3}\oplus (1,0)\oplus (2,1)$\\[0.1in]
\textcolor{blue}{Real grid homology - minus version}\\
$U^{\infty}_{(0,0)}\oplus U^{2}_{(2,1)}\oplus U_{(-1,-1)}\oplus U_{(1,0)}$
} \tabularnewline
  \hline
  \parbox[c]{\linewidth}{\centering $(7_6)'$\\{\tiny Second symmetry}} & \centering \includegraphics[width=0.25\textwidth]{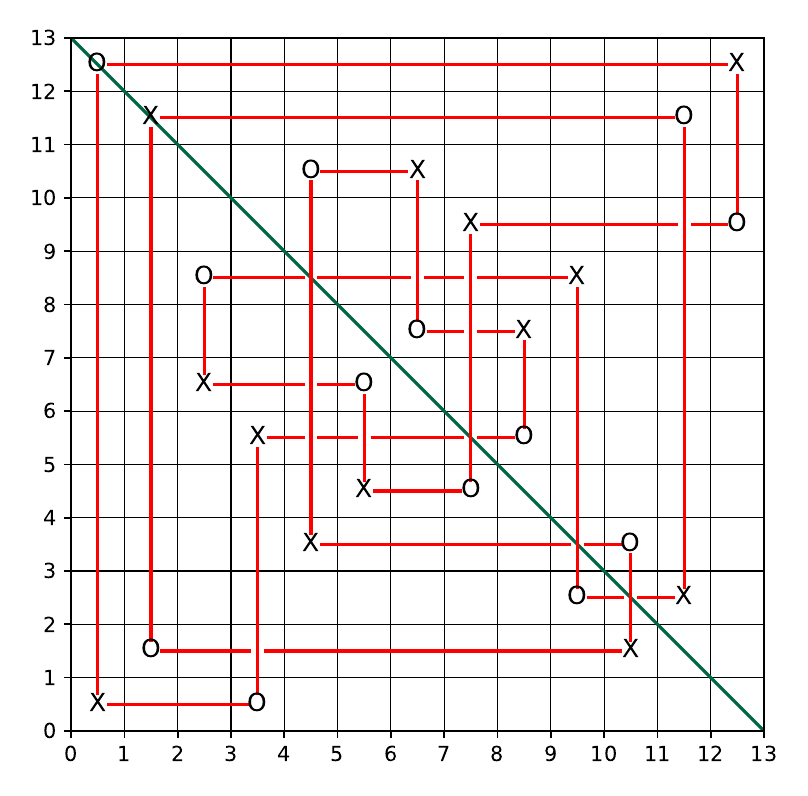} & \parbox[c]{\linewidth}{\centering\tiny
\textcolor{blue}{Polynomial Invariant}\\
$t^{-2} - t^{-1} - 1 + t + t^{2}$\\[0.1in]
\textcolor{blue}{Real grid homology - hat version}\\
$(-2,-2)\oplus (-1,-1)\oplus (0,-1)\oplus (1,0)\oplus (2,0)$
} \tabularnewline
  \hline
  \centering $7_7$ & \centering \includegraphics[width=0.25\textwidth]{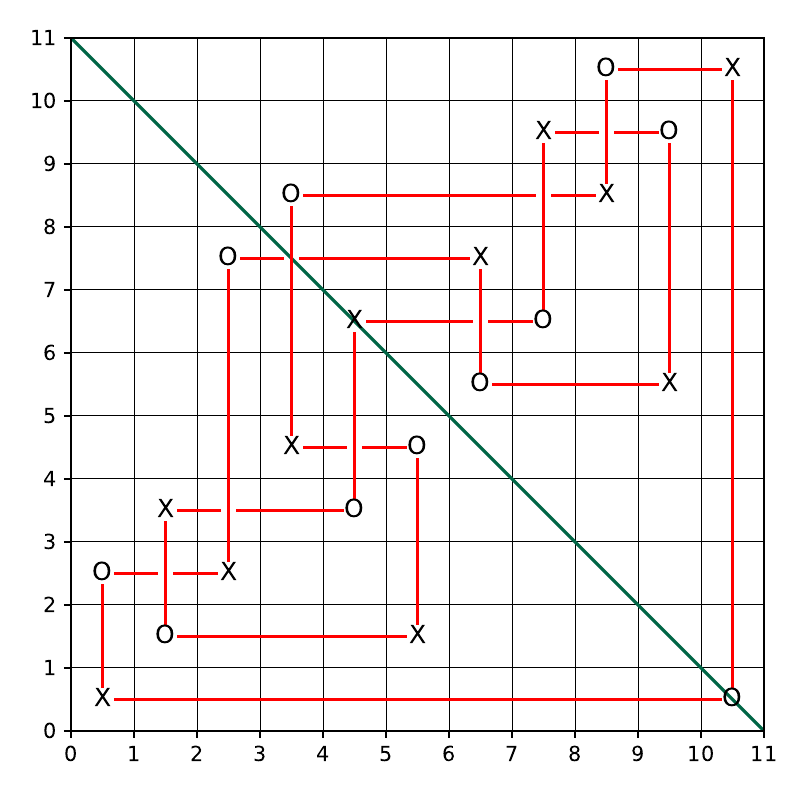} & \parbox[c]{\linewidth}{\centering\tiny
\textcolor{blue}{Polynomial Invariant}\\
$-t^{-2} - t^{-1} + 3 + t - t^{2}$\\[0.1in]
\textcolor{blue}{Real grid homology - hat version}\\
$(-2,-1)\oplus (-1,-1)\oplus (0,0)^{3}\oplus (1,0)\oplus (2,1)$\\[0.1in]
\textcolor{blue}{Real grid homology - minus version}\\
$U^{\infty}_{(0,0)}\oplus U^{2}_{(2,1)}\oplus U_{(-1,-1)}\oplus U_{(1,0)}$
} \tabularnewline
  \hline
  \parbox[c]{\linewidth}{\centering $(7_7)'$\\{\tiny Second symmetry}} & \centering \includegraphics[width=0.25\textwidth]{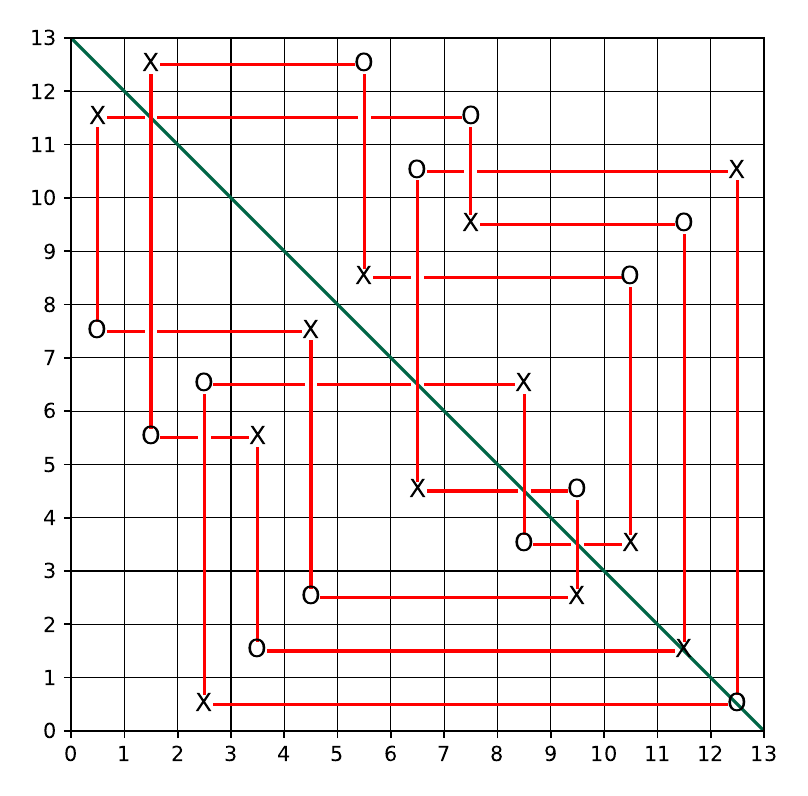} & \parbox[c]{\linewidth}{\centering\tiny
\textcolor{blue}{Polynomial Invariant}\\
$-t^{-2} - t^{-1} + 3 + t - t^{2}$\\[0.1in]
\textcolor{blue}{Real grid homology - hat version}\\
$(-2,-1)\oplus (-1,-1)\oplus (0,0)^{3}\oplus (1,0)\oplus (2,1)$
} \tabularnewline
  \hline
  \parbox[c]{\linewidth}{\centering twsited${}_6$\\{\tiny Another symmetry of $8_1$}} & \centering \includegraphics[width=0.25\textwidth]{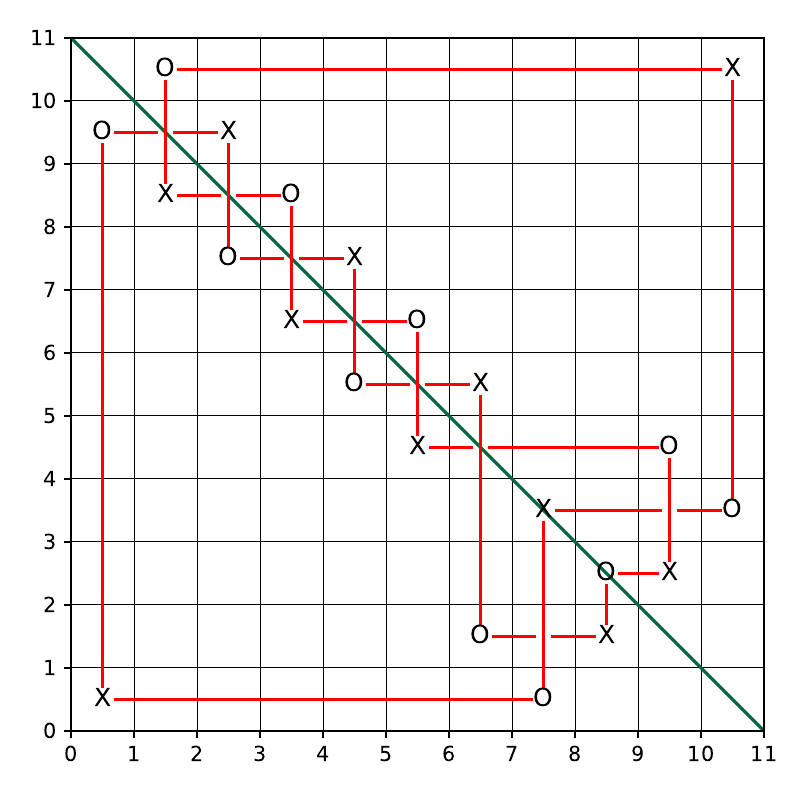} & \parbox[c]{\linewidth}{\centering\tiny
\textcolor{blue}{Polynomial Invariant}\\
$-t^{-1} + 1 + t$\\[0.1in]
\textcolor{blue}{Real grid homology - hat version}\\
$(-1,-1)\oplus (0,0)\oplus (1,0)$\\[0.1in]
\textcolor{blue}{Real grid homology - minus version}\\
$U^{\infty}_{(-1,-1)}\oplus U_{(1,0)}$
} \tabularnewline
  \hline
  \centering $8_1$ & \centering \includegraphics[width=0.25\textwidth]{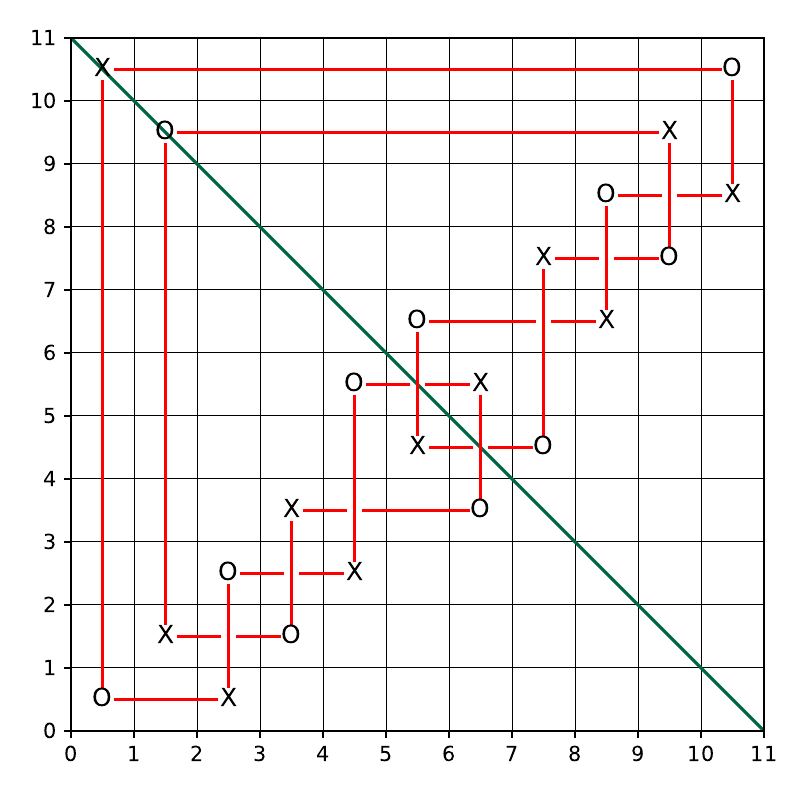} & \parbox[c]{\linewidth}{\centering\tiny
\textcolor{blue}{Polynomial Invariant}\\
$-3t^{-1} + 1 + 3t$\\[0.1in]
\textcolor{blue}{Real grid homology - hat version}\\
$(-1,-1)^{3}\oplus (0,0)\oplus (1,0)^{3}$\\[0.1in]
\textcolor{blue}{Real grid homology - minus version}\\
$U^{\infty}_{(-1,-1)}\oplus \bigg(U^{2}_{(1,0)}\bigg)^{2}\oplus U_{(1,0)}$
} \tabularnewline
  \hline
  \centering twisted${}_7$ & \centering \includegraphics[width=0.25\textwidth]{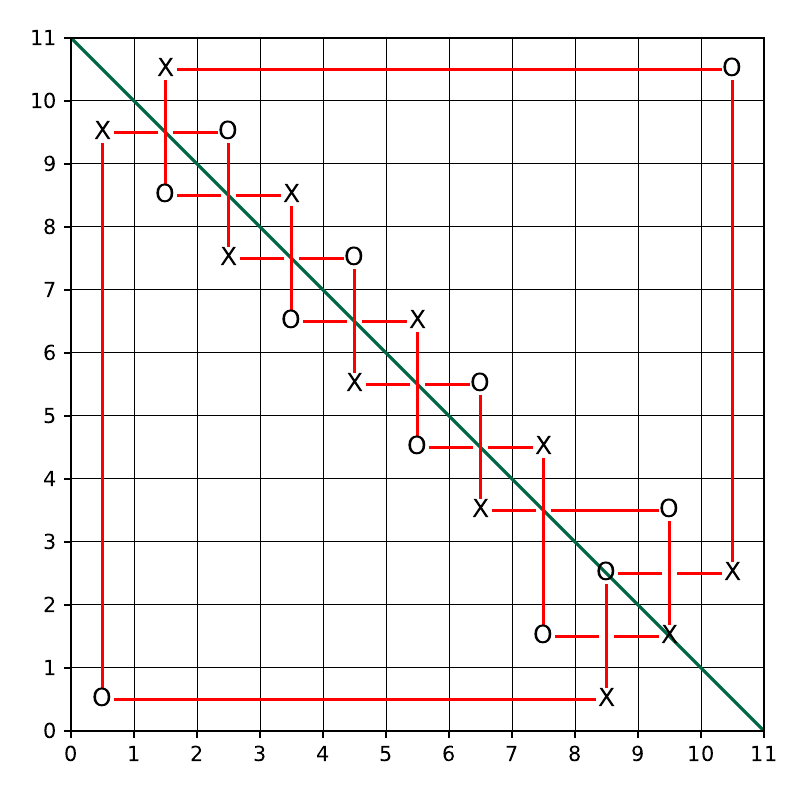} & \parbox[c]{\linewidth}{\centering\tiny
\textcolor{blue}{Polynomial Invariant}\\
$1$\\[0.1in]
\textcolor{blue}{Real grid homology - hat version}\\
$(0,0)$\\[0.1in]
\textcolor{blue}{Real grid homology - minus version}\\
$U^{\infty}_{(0,0)}$
} \tabularnewline
  \hline
  \parbox[c]{\linewidth}{\centering $($twsited${}_7)'$\\{\tiny Second symmetry}} & \centering \includegraphics[width=0.25\textwidth]{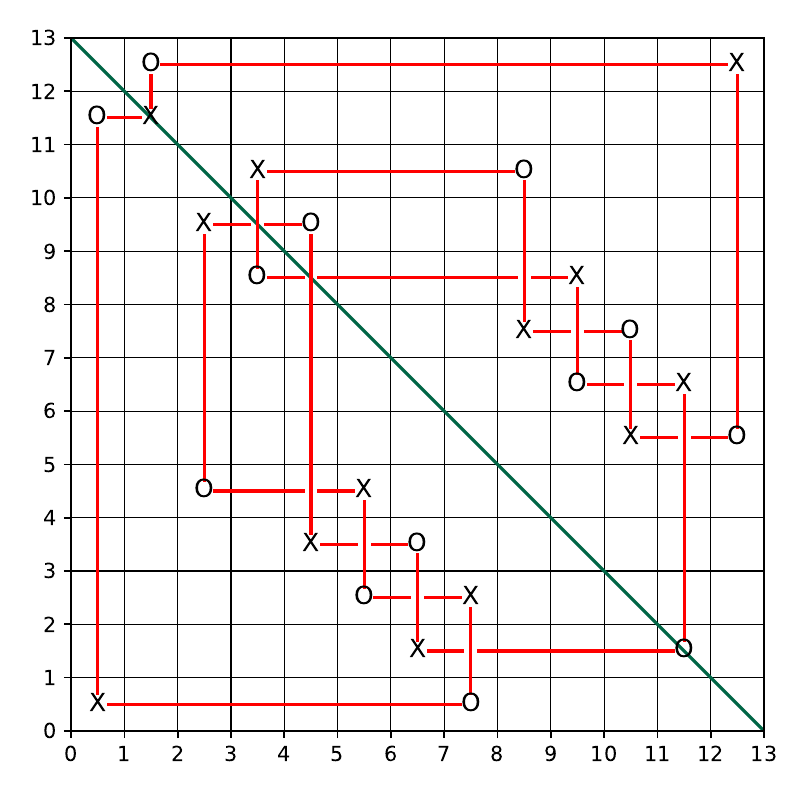} & \parbox[c]{\linewidth}{\centering\tiny
\textcolor{blue}{Polynomial Invariant}\\
$4t^{-1} + 1 - 4t$\\[0.1in]
\textcolor{blue}{Real grid homology - hat version}\\
$(-1,0)^{4}\oplus (0,0)\oplus (1,1)^{4}$
} \tabularnewline
  \hline
  \centering twsited${}_8$ & \centering \includegraphics[width=0.25\textwidth]{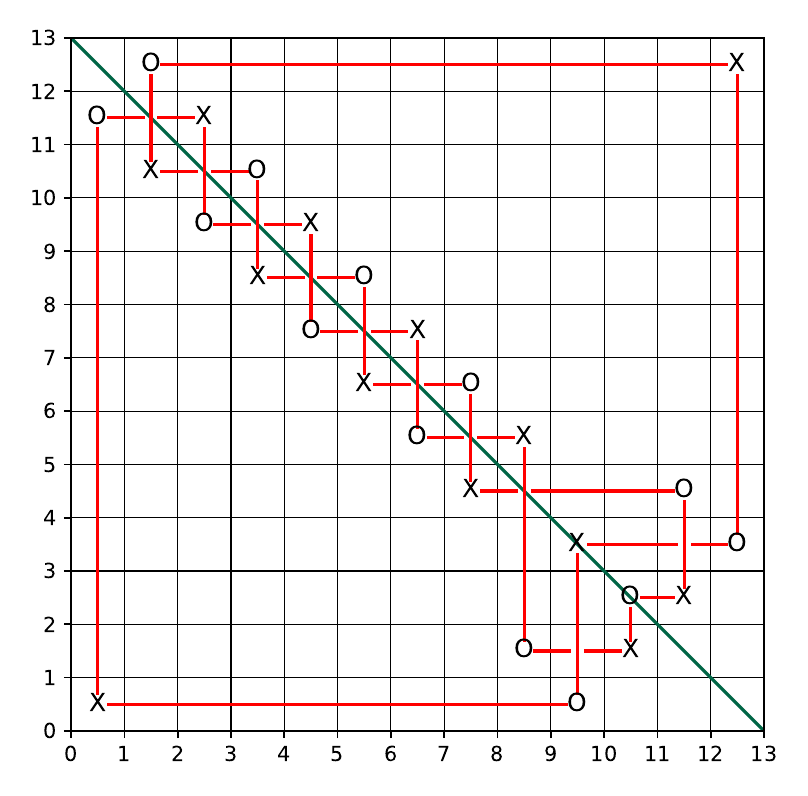} & \parbox[c]{\linewidth}{\centering\tiny
\textcolor{blue}{Polynomial Invariant}\\
$1$\\[0.1in]
\textcolor{blue}{Real grid homology - hat version}\\
$(0,0)$
} \tabularnewline
  \hline
  \parbox[c]{\linewidth}{\centering $($twsited${}_8)'$\\{\tiny Second symmetry}} & \centering \includegraphics[width=0.25\textwidth]{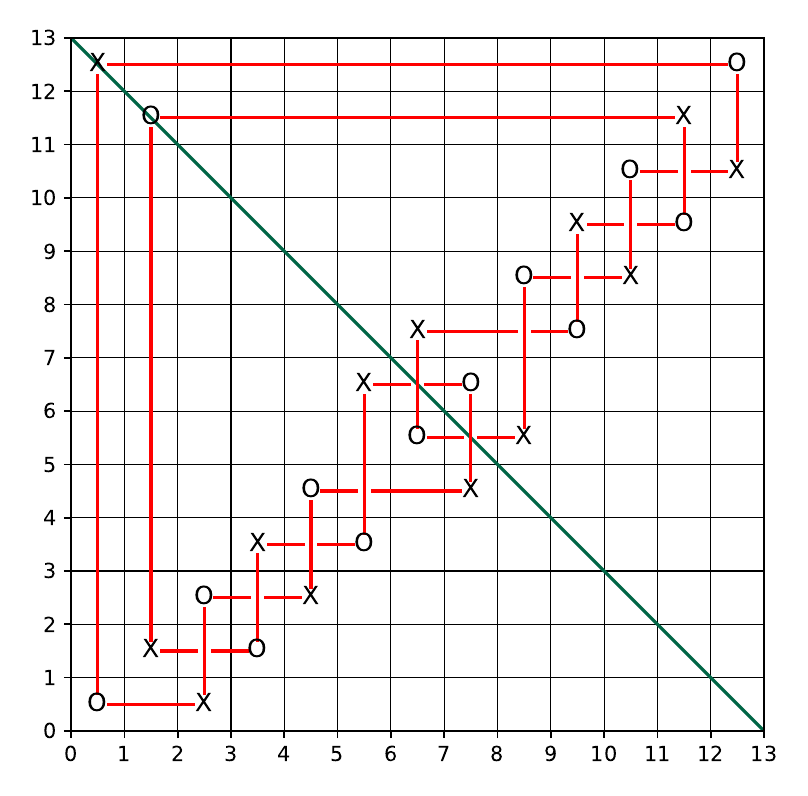} & \parbox[c]{\linewidth}{\centering\tiny
\textcolor{blue}{Polynomial Invariant}\\
$-4t^{-1} + 1 + 4t$\\[0.1in]
\textcolor{blue}{Real grid homology - hat version}\\
$(-1,-1)^{4}\oplus (0,0)\oplus (1,0)^{4}$
} \tabularnewline
  \hline
  \centering $8_2$ & \centering \includegraphics[width=0.25\textwidth]{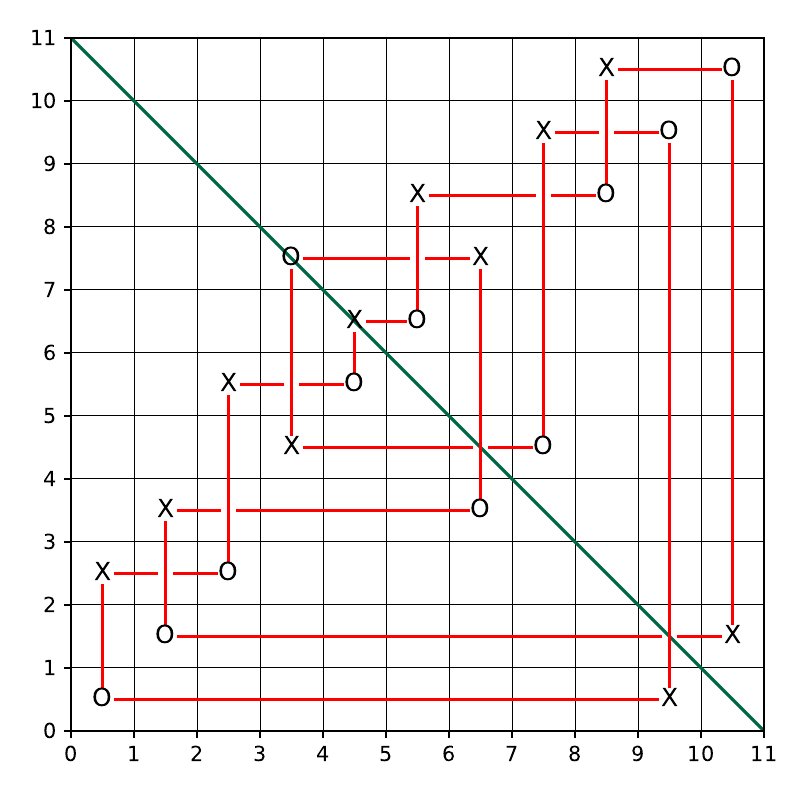} & \parbox[c]{\linewidth}{\centering\tiny
\textcolor{blue}{Polynomial Invariant}\\
$-t^{-3} + t^{-2} + 3t^{-1} - 1 - 3t + t^{2} + t^{3}$\\[0.1in]
\textcolor{blue}{Real grid homology - hat version}\\
$(-3,-1)\oplus (-2,0)\oplus (-1,0)^{3}\oplus (0,1)\oplus (1,1)^{3}\oplus (2,2)\oplus (3,2)$\\[0.1in]
\textcolor{blue}{Real grid homology - minus version}\\
$U^{\infty}_{(1,1)}\oplus U_{(-1,0)}\oplus U_{(1,1)}\oplus U_{(3,2)}\oplus U^{2}_{(-1,0)}\oplus U^{2}_{(1,1)}$
} \tabularnewline
  \hline
  \centering $8_3$ & \centering \includegraphics[width=0.25\textwidth]{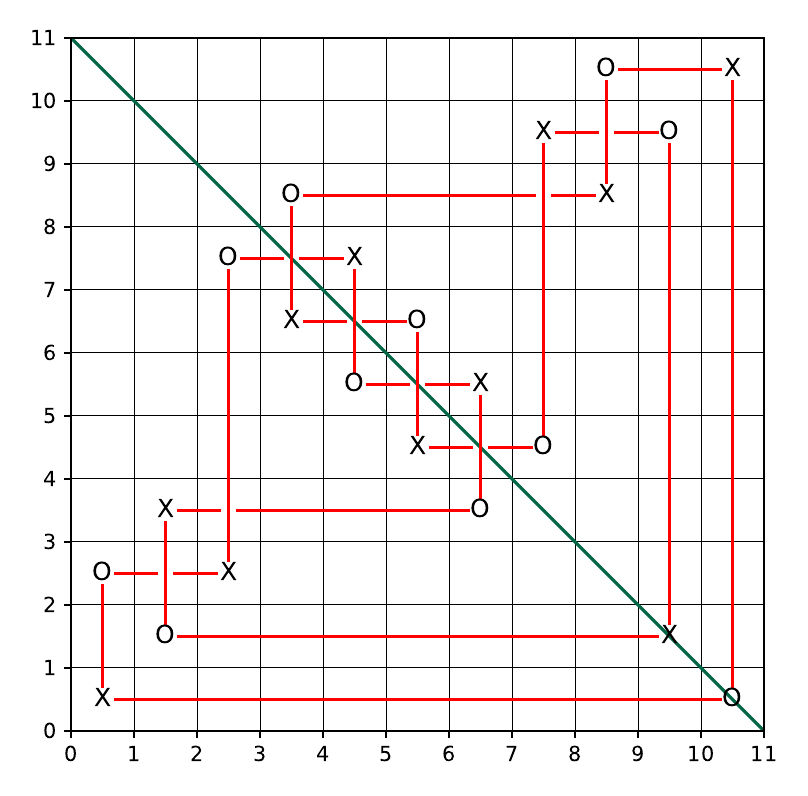} & \parbox[c]{\linewidth}{\centering\tiny
\textcolor{blue}{Polynomial Invariant}\\
$1$\\[0.1in]
\textcolor{blue}{Real grid homology - hat version}\\
$(0,0)$\\[0.1in]
\textcolor{blue}{Real grid homology - minus version}\\
$U^{\infty}_{(0,0)}$
} \tabularnewline
  \hline
  \centering $8_8$ & \centering \includegraphics[width=0.25\textwidth]{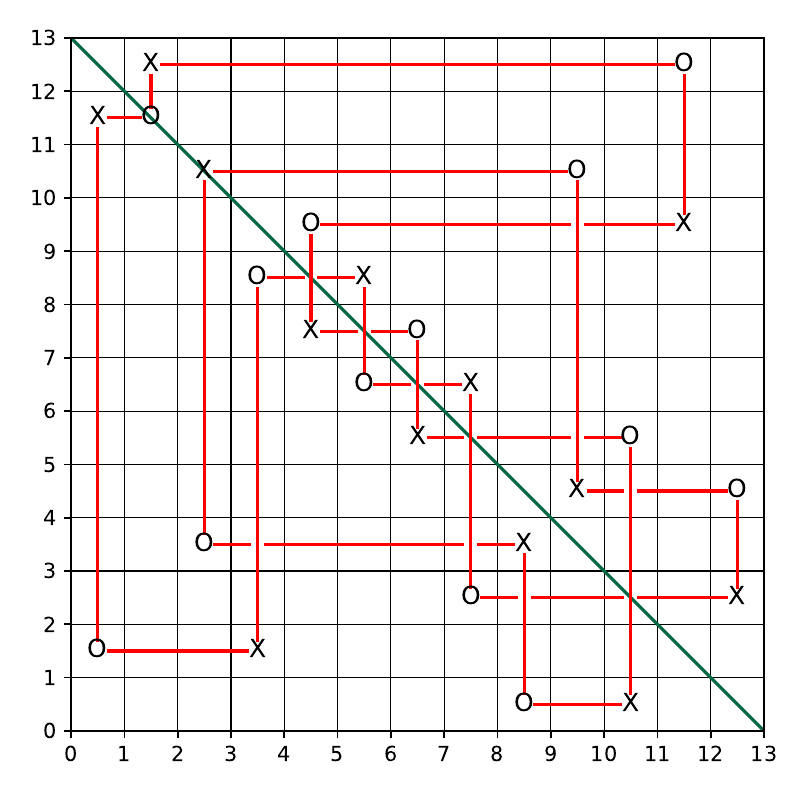} & \parbox[c]{\linewidth}{\centering\tiny
\textcolor{blue}{Polynomial Invariant}\\
$1$\\[0.1in]
\textcolor{blue}{Real grid homology - hat version}\\
$(0,0)$
} \tabularnewline
  \hline
  \parbox[c]{\linewidth}{\centering $8_8$\\{\tiny Second symmetry}} & \centering \includegraphics[width=0.25\textwidth]{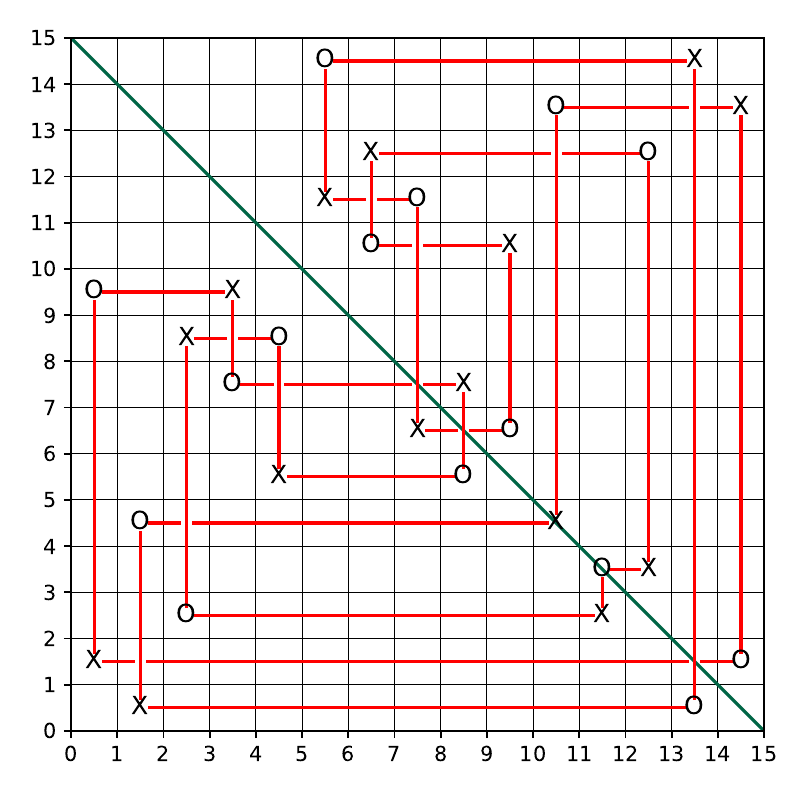} & \parbox[c]{\linewidth}{\centering\tiny
\textcolor{blue}{Polynomial Invariant}\\
$-2t^{-2} + 2t^{-1} + 5 - 2t - 2t^{2}$
} \tabularnewline
  \hline
  \centering $8_9$ & \centering \includegraphics[width=0.25\textwidth]{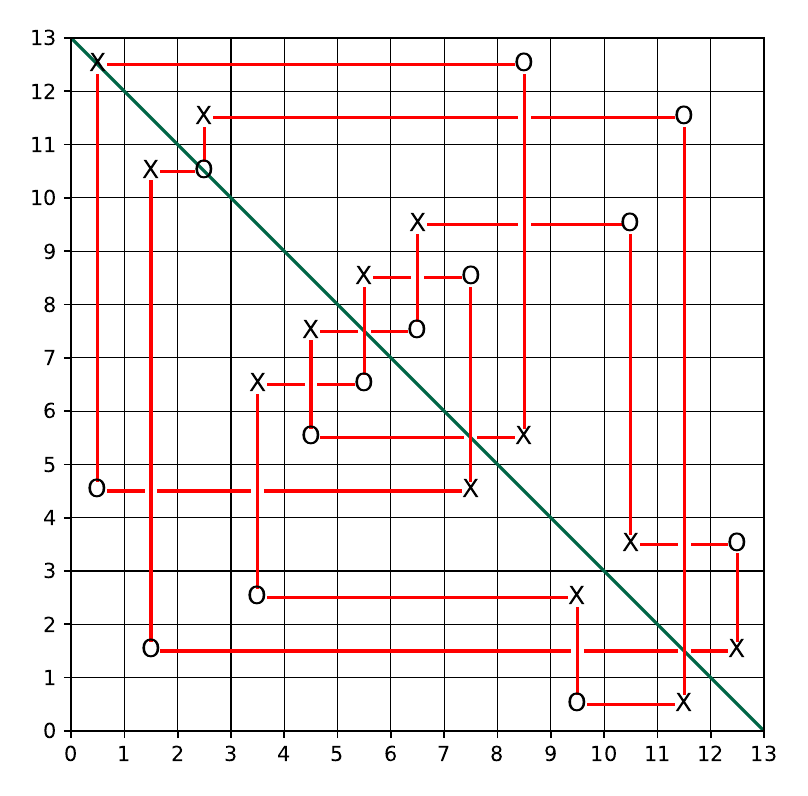} & \parbox[c]{\linewidth}{\centering\tiny
\textcolor{blue}{Polynomial Invariant}\\
$t^{-3} - t^{-2} - 3t^{-1} + 3 + 3t - t^{2} - t^{3}$\\[0.1in]
\textcolor{blue}{Real grid homology - hat version}\\
$(-3,-2)\oplus (-2,-1)\oplus (-1,-1)^{3}\oplus (0,0)^{3}\oplus (1,0)^{3}\oplus (2,1)\oplus (3,1)$
} \tabularnewline
  \hline
  \centering $8_{12}$ & \centering \includegraphics[width=0.25\textwidth]{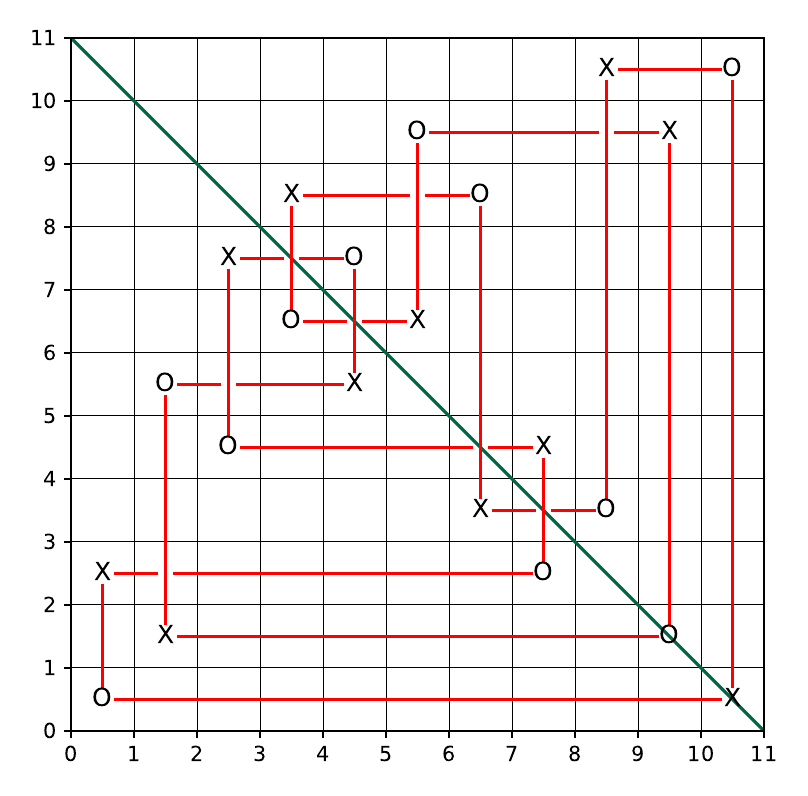} & \parbox[c]{\linewidth}{\centering\tiny
\textcolor{blue}{Polynomial Invariant}\\
$t^{-2} - t^{-1} - 1 + t + t^{2}$\\[0.1in]
\textcolor{blue}{Real grid homology - hat version}\\
$(-2,-2)\oplus (-1,-1)\oplus (0,-1)\oplus (1,0)\oplus (2,0)$\\[0.1in]
\textcolor{blue}{Real grid homology - minus version}\\
$U^{\infty}_{(-2,-2)}\oplus U_{(0,-1)}\oplus U_{(2,0)}$
} \tabularnewline
  \hline
  \centering $8_{18}$ & \centering \includegraphics[width=0.25\textwidth]{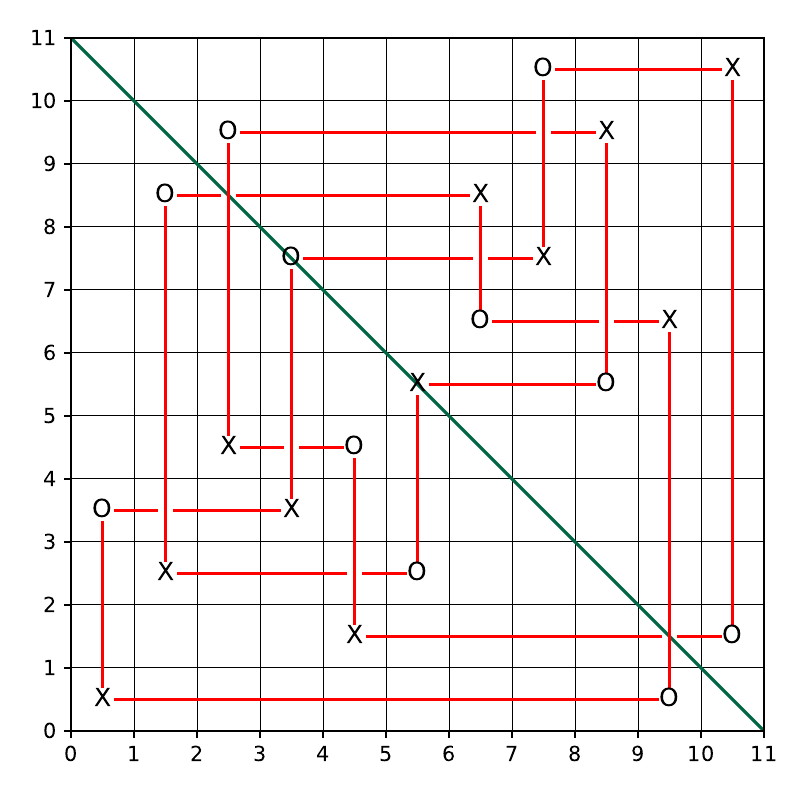} & \parbox[c]{\linewidth}{\centering\tiny
\textcolor{blue}{Polynomial Invariant}\\
$t^{-3} - t^{-2} - 2t^{-1} + 3 + 2t - t^{2} - t^{3}$\\[0.1in]
\textcolor{blue}{Real grid homology - hat version}\\
$(-3,-2)\oplus (-2,-1)\oplus (-1,-1)^{2}\oplus (0,0)^{3}\oplus (1,0)^{2}\oplus (2,1)\oplus (3,1)$\\[0.1in]
\textcolor{blue}{Real grid homology - minus version}\\
$U^{\infty}_{(-1,-1)}\oplus \bigg(U_{(1,0)}\bigg)^{2}\oplus U_{(3,1)}\oplus U^{2}_{(-1,-1)}\oplus U^{2}_{(0,0)}$
} \tabularnewline
  \hline
  \centering $8_{19}$ & \centering \includegraphics[width=0.25\textwidth]{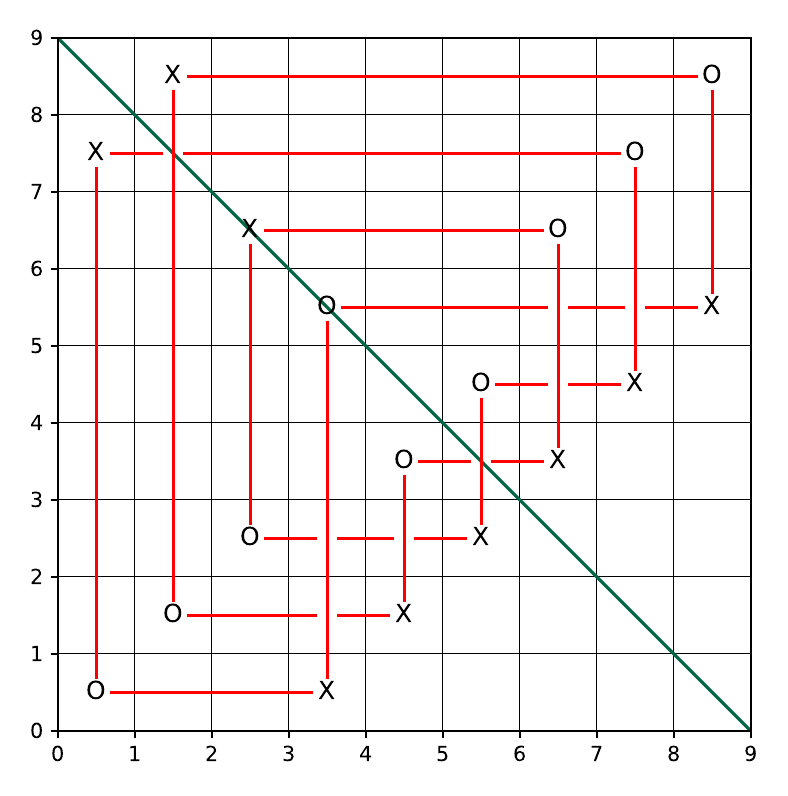} & \parbox[c]{\linewidth}{\centering\tiny
\textcolor{blue}{Polynomial Invariant}\\
$t^{-3} + t^{-2} - 1 + t^{2} - t^{3}$\\[0.1in]
\textcolor{blue}{Real grid homology - hat version}\\
$(-3,0)\oplus (-2,0)\oplus (-1,1)^{0}\oplus (0,1)\oplus (1,2)^{0}\oplus (2,2)\oplus (3,3)$\\[0.1in]
\textcolor{blue}{Real grid homology - minus version}\\
$U^{\infty}_{(3,3)}\oplus U^{2}_{(2,2)}\oplus U_{(-2,0)}$
} \tabularnewline
  \hline
  \centering $8_{20}$ & \centering \includegraphics[width=0.25\textwidth]{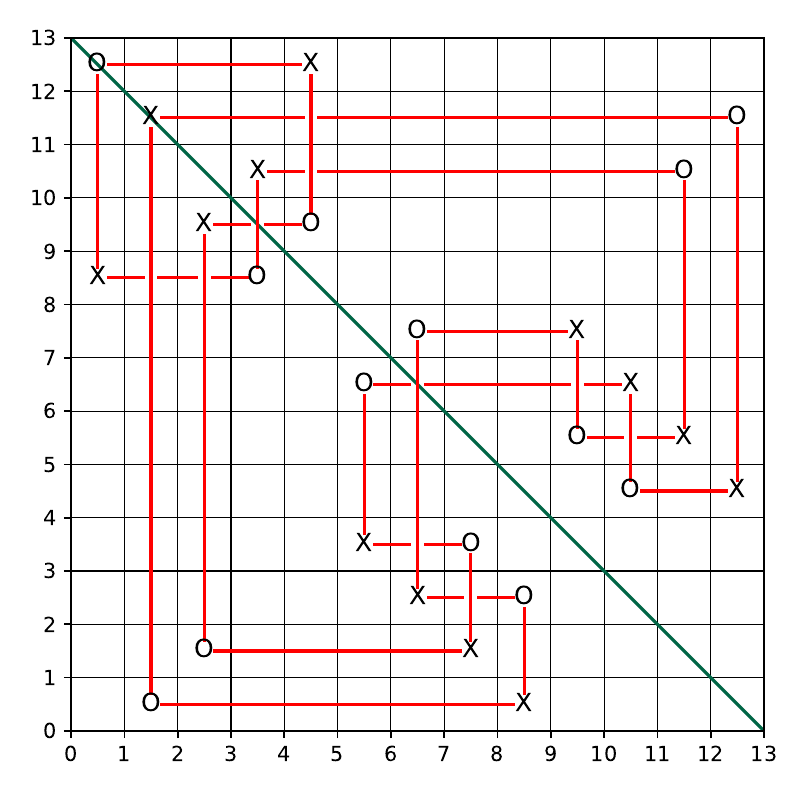} & \parbox[c]{\linewidth}{\centering\tiny
\textcolor{blue}{Polynomial Invariant}\\
$-t^{-2} + 3 - t^{2}$\\[0.1in]
\textcolor{blue}{Real grid homology - hat version}\\
$(-2,-1)\oplus (-1,-1)\oplus (-1,0)\oplus (0,0)^{3}\oplus (2,1)$
} \tabularnewline
  \hline
  \centering $8_{21}$ & \centering \includegraphics[width=0.25\textwidth]{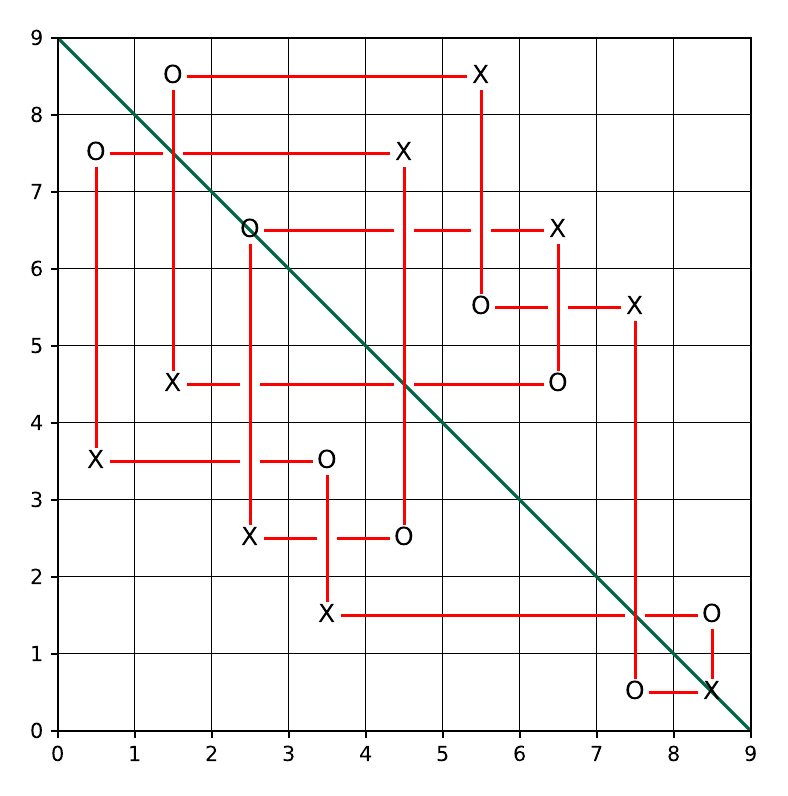} & \parbox[c]{\linewidth}{\centering\tiny
\textcolor{blue}{Polynomial Invariant}\\
$t^{-2} - 2t^{-1} - 1 + 2t + t^{2}$\\[0.1in]
\textcolor{blue}{Real grid homology - hat version}\\
$(-2,-2)\oplus (-1,-1)^{2}\oplus (0,-1)\oplus (1,0)^{2}\oplus (2,0)$\\[0.1in]
\textcolor{blue}{Real grid homology - minus version}\\
$U^{\infty}_{(-2,-2)}\oplus U_{(0,-1)}\oplus U_{(2,0)}\oplus U^{2}_{(1,0)}$
} \tabularnewline
  \hline
  \centering $9_3$ & \centering \includegraphics[width=0.25\textwidth]{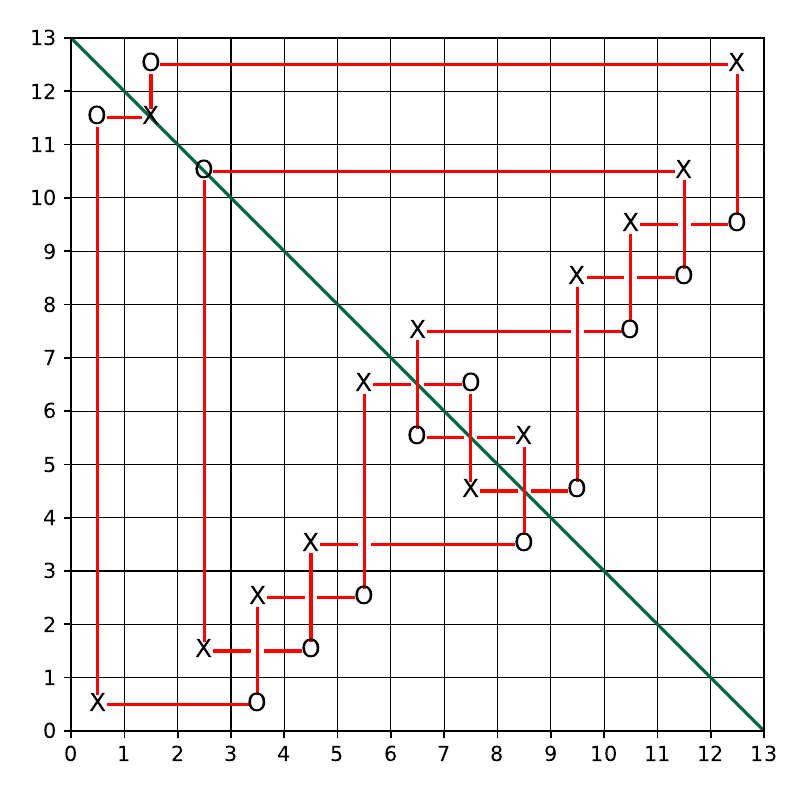} & \parbox[c]{\linewidth}{\centering\tiny
\textcolor{blue}{Polynomial Invariant}\\
$t^{-2} + t^{-1} - 1 - t + t^{2}$\\[0.1in]
\textcolor{blue}{Real grid homology - hat version}\\
$(-2,0)\oplus (-1,0)\oplus (0,1)\oplus (1,1)\oplus (2,2)$
} \tabularnewline
  \hline
  \centering $9_6$ & \centering \includegraphics[width=0.25\textwidth]{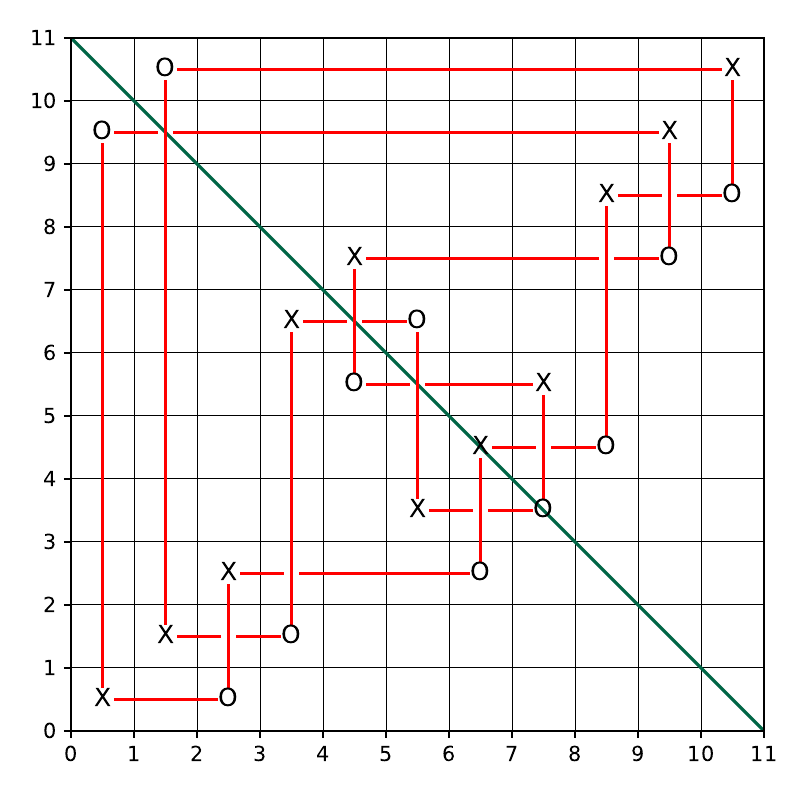} & \parbox[c]{\linewidth}{\centering\tiny
\textcolor{blue}{Polynomial Invariant}\\
$2t^{-2} + t^{-1} - 3 - t + 2t^{2}$\\[0.1in]
\textcolor{blue}{Real grid homology - hat version}\\
$(-2,0)^{2}\oplus (-1,0)\oplus (0,1)^{3}\oplus (1,1)\oplus (2,2)^{2}$\\[0.1in]
\textcolor{blue}{Real grid homology - minus version}\\
$U^{\infty}_{(2,2)}\oplus U^{2}_{(0,1)}\oplus U^{2}_{(2,2)}\oplus U_{(-1,0)}\oplus U_{(1,1)}$
} \tabularnewline
  \hline
  \centering $9_9$ & \centering \includegraphics[width=0.25\textwidth]{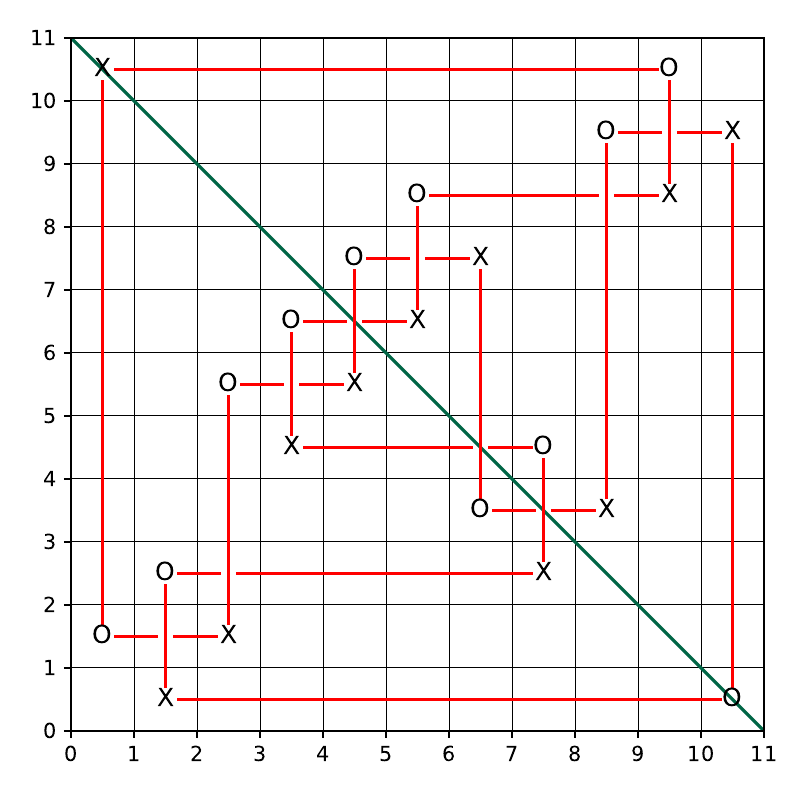} & \parbox[c]{\linewidth}{\centering\tiny
\textcolor{blue}{Polynomial Invariant}\\
$2t^{-2} + 2t^{-1} - 3 - 2t + 2t^{2}$\\[0.1in]
\textcolor{blue}{Real grid homology - hat version}\\
$(-2,0)^{2}\oplus (-1,0)^{2}\oplus (0,1)^{3}\oplus (1,1)^{2}\oplus (2,2)^{2}$\\[0.1in]
\textcolor{blue}{Real grid homology - minus version}\\
$U^{\infty}_{(2,2)}\oplus U^{2}_{(2,2)}\oplus \bigg(U_{(-1,0)}\bigg)^{2}\oplus \bigg(U_{(1,1)}\bigg)^{2}$
} \tabularnewline
  \hline
  \centering $9_{28}$ & \centering \includegraphics[width=0.25\textwidth]{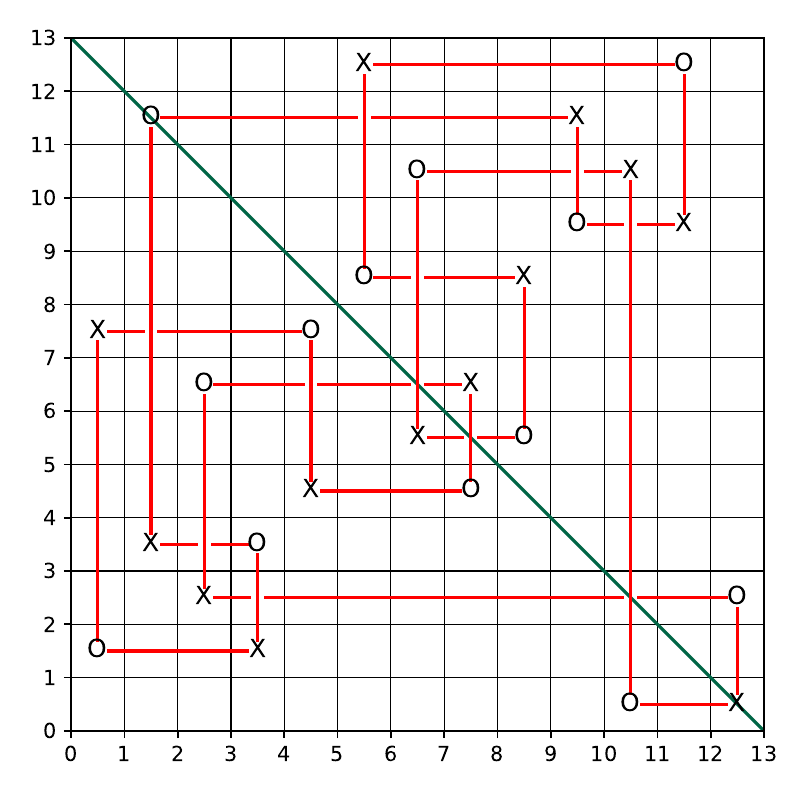} & \parbox[c]{\linewidth}{\centering\tiny
\textcolor{blue}{Polynomial Invariant}\\
$t^{-3} + t^{-2} - 4t^{-1} - 1 + 4t + t^{2} - t^{3}$\\[0.1in]
\textcolor{blue}{Real grid homology - hat version}\\
$(-3,-2)\oplus (-2,-2)\oplus (-1,-1)^{4}\oplus (0,-1)\oplus (1,0)^{4}\oplus (2,0)\oplus (3,1)$
} \tabularnewline
  \hline
  \centering $9_{40}$ & \centering \includegraphics[width=0.25\textwidth]{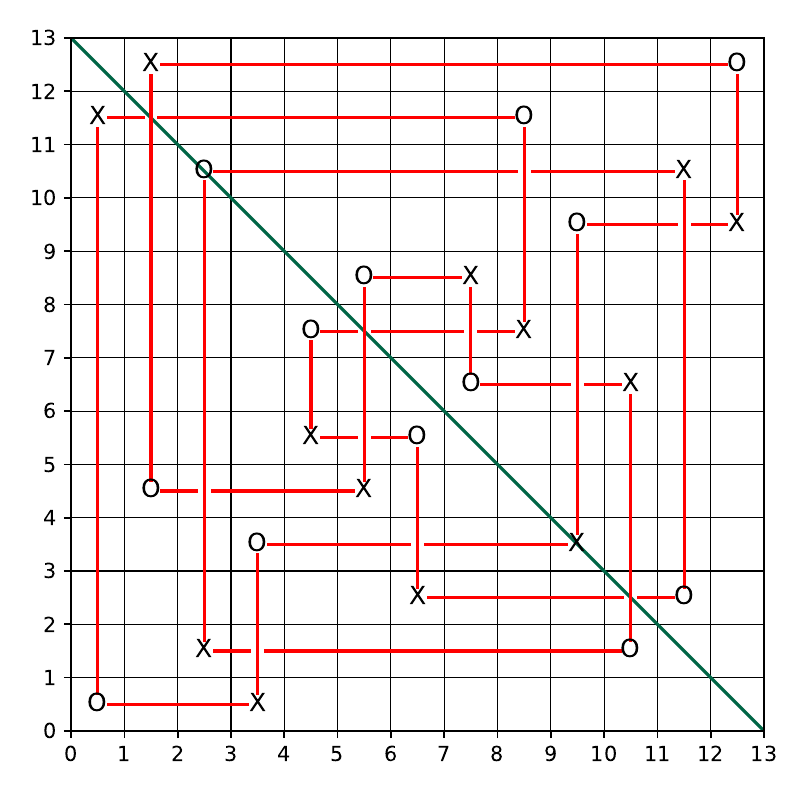} & \parbox[c]{\linewidth}{\centering\tiny
\textcolor{blue}{Polynomial Invariant}\\
$-t^{-3} - t^{-2} + 2t^{-1} + 3 - 2t - t^{2} + t^{3}$\\[0.1in]
\textcolor{blue}{Real grid homology - hat version}\\
$(-3,-1)\oplus (-2,-1)\oplus (-1,0)^{2}\oplus (0,0)^{3}\oplus (1,1)^{2}\oplus (2,1)\oplus (3,2)$
} \tabularnewline
  \hline
  \centering $9_{41}$ & \centering \includegraphics[width=0.25\textwidth]{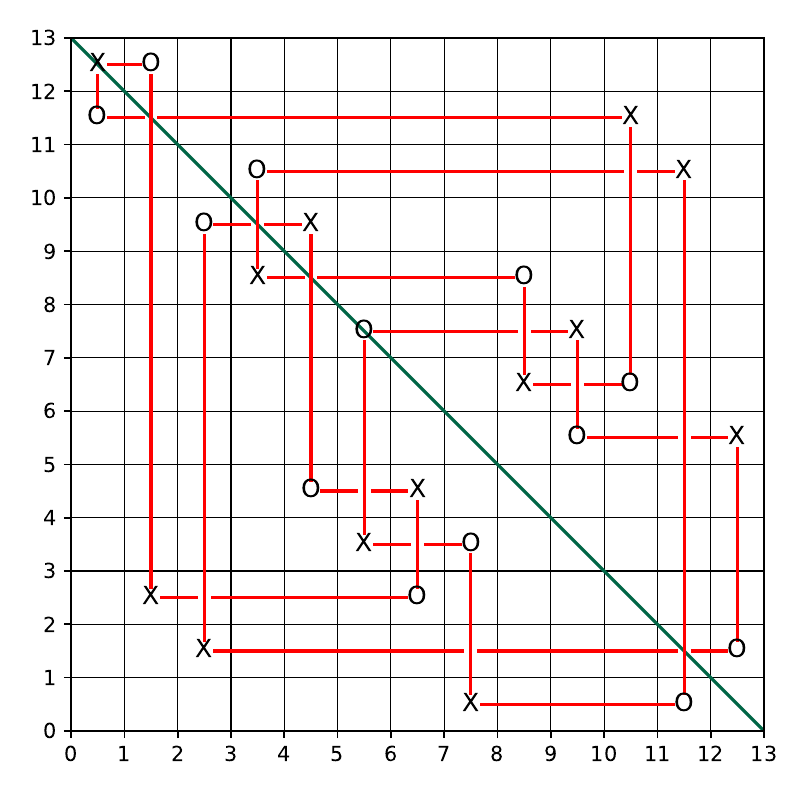} & \parbox[c]{\linewidth}{\centering\tiny
\textcolor{blue}{Polynomial Invariant}\\
$-t^{-2} + 3 - t^{2}$\\[0.1in]
\textcolor{blue}{Real grid homology - hat version}\\
$(-2,-1)\oplus (-1,-1)\oplus (-1,0)\oplus (0,0)^{3}\oplus (2,1)$
} \tabularnewline
  \hline
  \centering $9_{42}$ & \centering \includegraphics[width=0.25\textwidth]{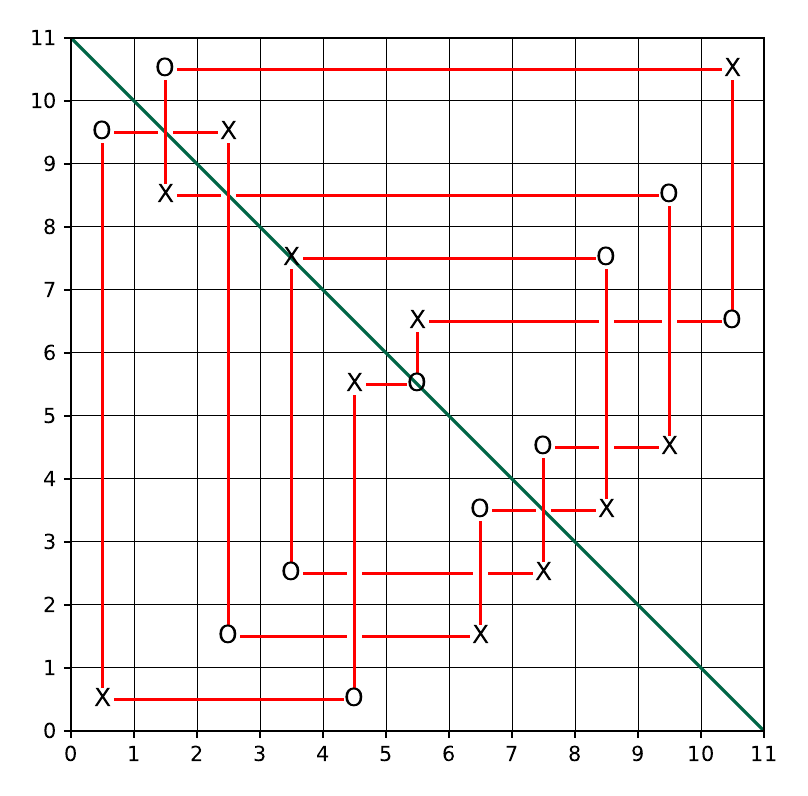} & \parbox[c]{\linewidth}{\centering\tiny
\textcolor{blue}{Polynomial Invariant}\\
$-t^{-2} + 3 - t^{2}$\\[0.1in]
\textcolor{blue}{Real grid homology - hat version}\\
$(-2,-1)\oplus (0,0)^{3}\oplus (1,0)\oplus (1,1)\oplus (2,1)$\\[0.1in]
\textcolor{blue}{Real grid homology - minus version}\\
$U^{\infty}_{(0,0)}\oplus U_{(1,0)}\oplus U_{(2,1)}\oplus U^{2}_{(0,0)}$
} \tabularnewline
  \hline
  \centering $9_{46}$ & \centering \includegraphics[width=0.25\textwidth]{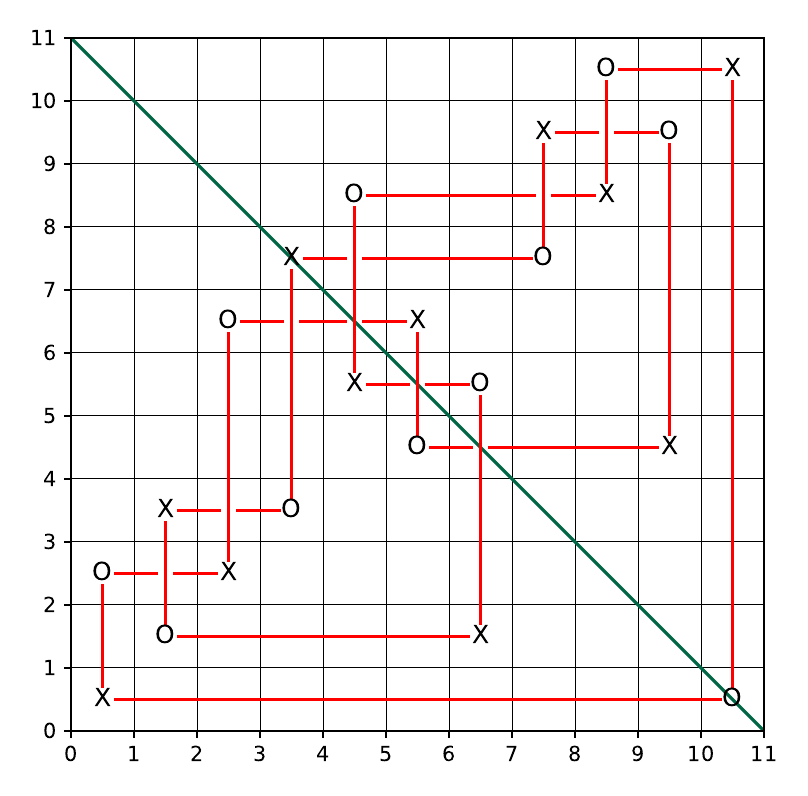} & \parbox[c]{\linewidth}{\centering\tiny
\textcolor{blue}{Polynomial Invariant}\\
$-2t^{-1} + 1 + 2t$\\[0.1in]
\textcolor{blue}{Real grid homology - hat version}\\
$(-1,-1)^{2}\oplus (0,0)\oplus (1,0)^{2}$\\[0.1in]
\textcolor{blue}{Real grid homology - minus version}\\
$U^{\infty}_{(-1,-1)}\oplus U^{2}_{(1,0)}\oplus U_{(1,0)}$
} \tabularnewline
  \hline
  \centering $10_{35}$ & \centering \includegraphics[width=0.25\textwidth]{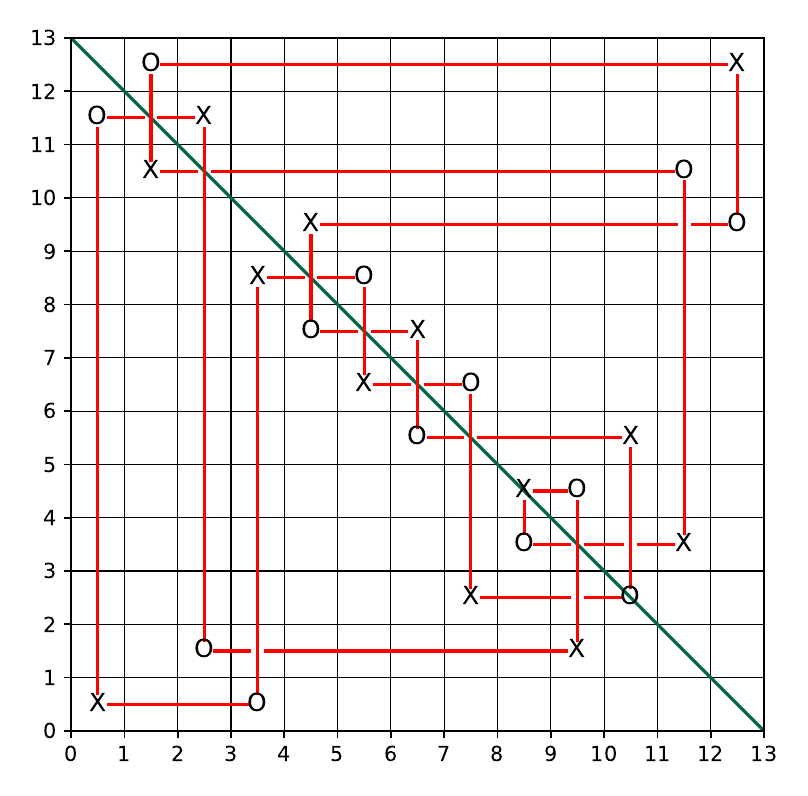} & \parbox[c]{\linewidth}{\centering\tiny
\textcolor{blue}{Polynomial Invariant}\\
$-2t^{-1} + 1 + 2t$\\[0.1in]
\textcolor{blue}{Real grid homology - hat version}\\
$(-1,-1)^{2}\oplus (0,0)\oplus (1,0)^{2}$
} \tabularnewline
  \hline
  \centering $10_{74}$ & \centering \includegraphics[width=0.25\textwidth]{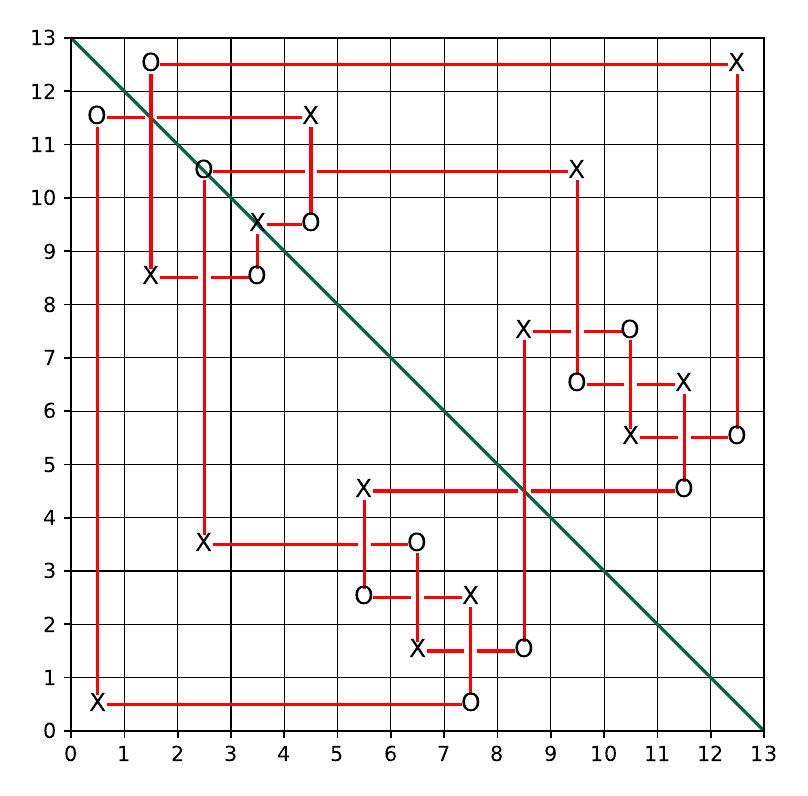} & \parbox[c]{\linewidth}{\centering\tiny
\textcolor{blue}{Polynomial Invariant}\\
$-2t^{-2} + 2t^{-1} + 5 - 2t - 2t^{2}$\\[0.1in]
\textcolor{blue}{Real grid homology - hat version}\\
$(-2,-1)^{2}\oplus (-1,0)^{2}\oplus (0,0)^{5}\oplus (1,1)^{2}\oplus (2,1)^{2}$
} \tabularnewline
  \hline
  \centering $10_{75}$ & \centering \includegraphics[width=0.25\textwidth]{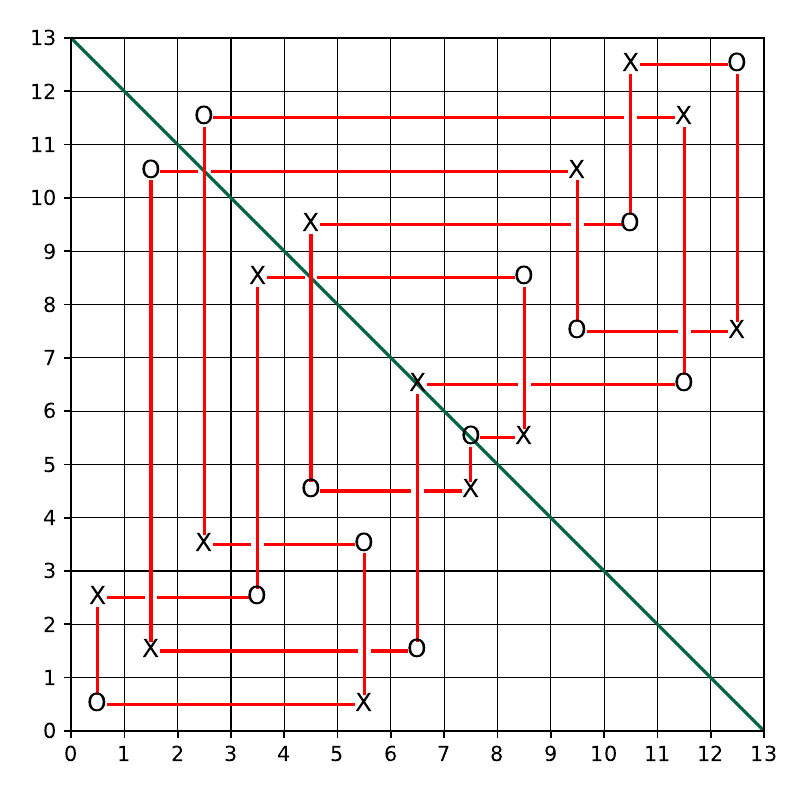} & \parbox[c]{\linewidth}{\centering\tiny
\textcolor{blue}{Polynomial Invariant}\\
$t^{-3} - t^{-2} - 5t^{-1} + 3 + 5t - t^{2} - t^{3}$\\[0.1in]
\textcolor{blue}{Real grid homology - hat version}\\
$(-3,-2)\oplus (-2,-1)\oplus (-1,-1)^{5}\oplus (0,0)^{3}\oplus (1,0)^{5}\oplus (2,1)\oplus (3,1)$
} \tabularnewline
  \hline
  \centering $10_{99}$ & \centering \includegraphics[width=0.25\textwidth]{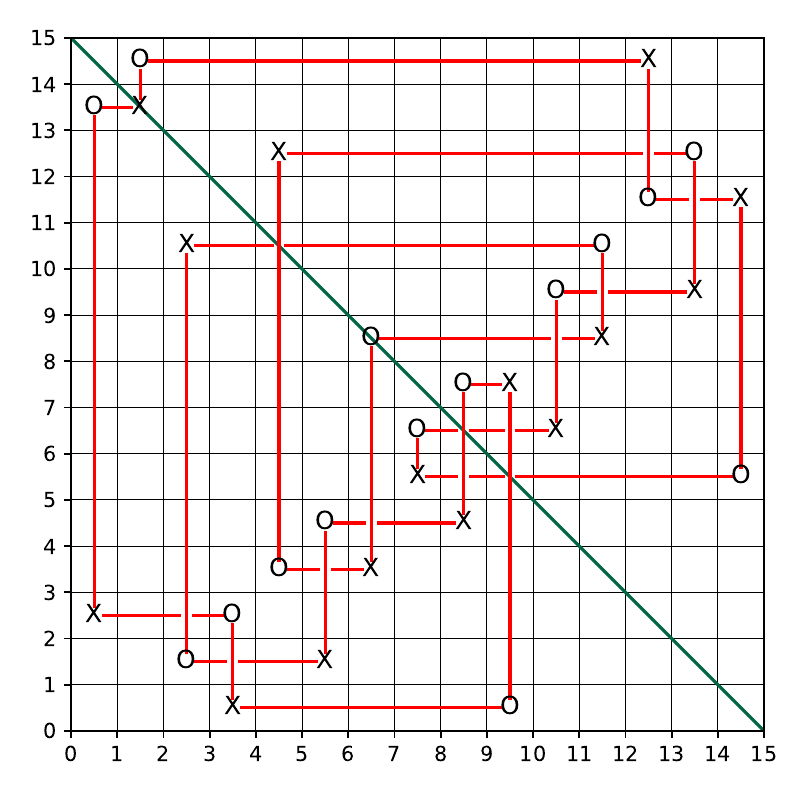} & \parbox[c]{\linewidth}{\centering\tiny
\textcolor{blue}{Polynomial Invariant}\\
$2t^{-2} + t^{-1} - 3 - t + 2t^{2}$
} \tabularnewline
  \hline
  \centering $10_{129}$ & \centering \includegraphics[width=0.25\textwidth]{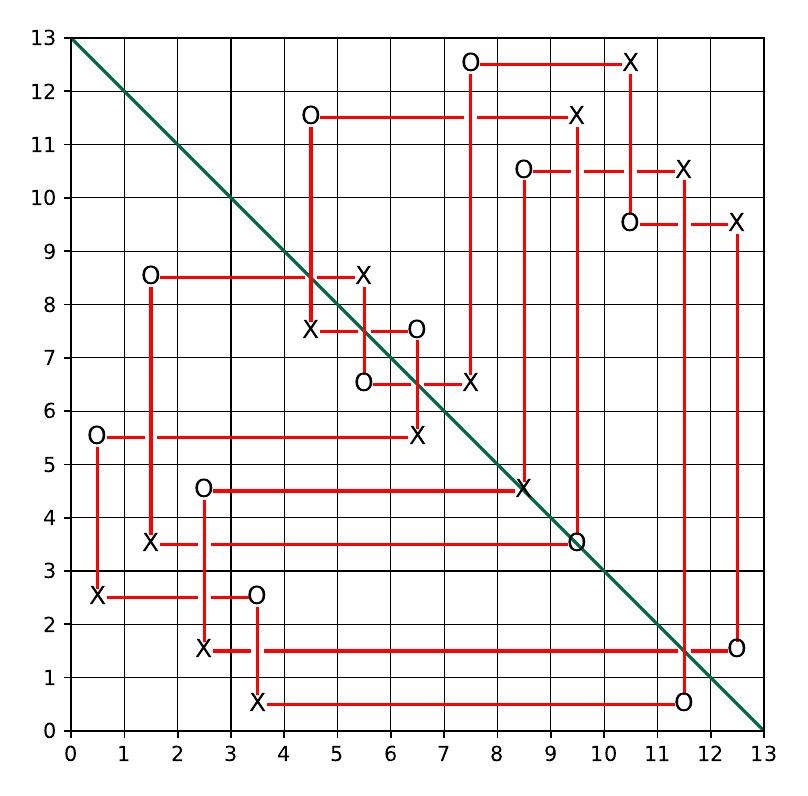} & \parbox[c]{\linewidth}{\centering\tiny
\textcolor{blue}{Polynomial Invariant}\\
$-2t^{-1} + 1 + 2t$\\[0.1in]
\textcolor{blue}{Real grid homology - hat version}\\
$(-1,-3)^{2}\oplus (0,-2)\oplus (1,-2)^{2}$
} \tabularnewline
  \hline
  \centering $11_{a281}$ & \centering \includegraphics[width=0.25\textwidth]{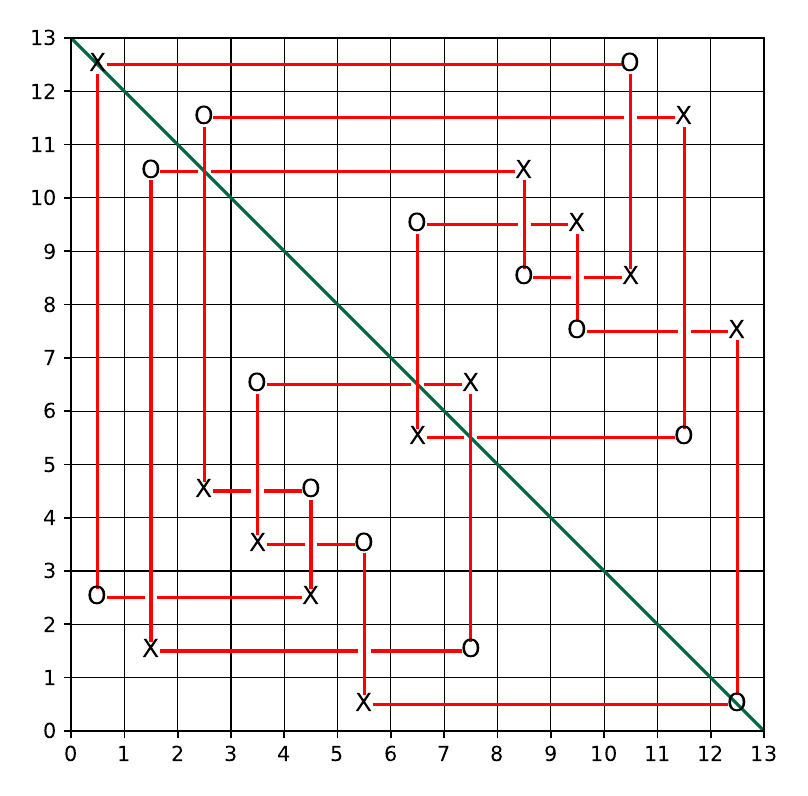} & \parbox[c]{\linewidth}{\centering\tiny
\textcolor{blue}{Polynomial Invariant}\\
$-t^{-4} + 2t^{-2} - t^{-1} - 1 + t + 2t^{2} - t^{4}$\\[0.1in]
\textcolor{blue}{Real grid homology - hat version}\\
$(-4,-3)\oplus (-2,-2)^{2}\oplus (-1,-1)\oplus (0,-1)\oplus (1,0)\oplus (2,0)^{2}\oplus (3,0)\oplus (3,1)\oplus (4,1)$
} \tabularnewline
  \hline
  \centering $11_{a181}$ & \centering \includegraphics[width=0.25\textwidth]{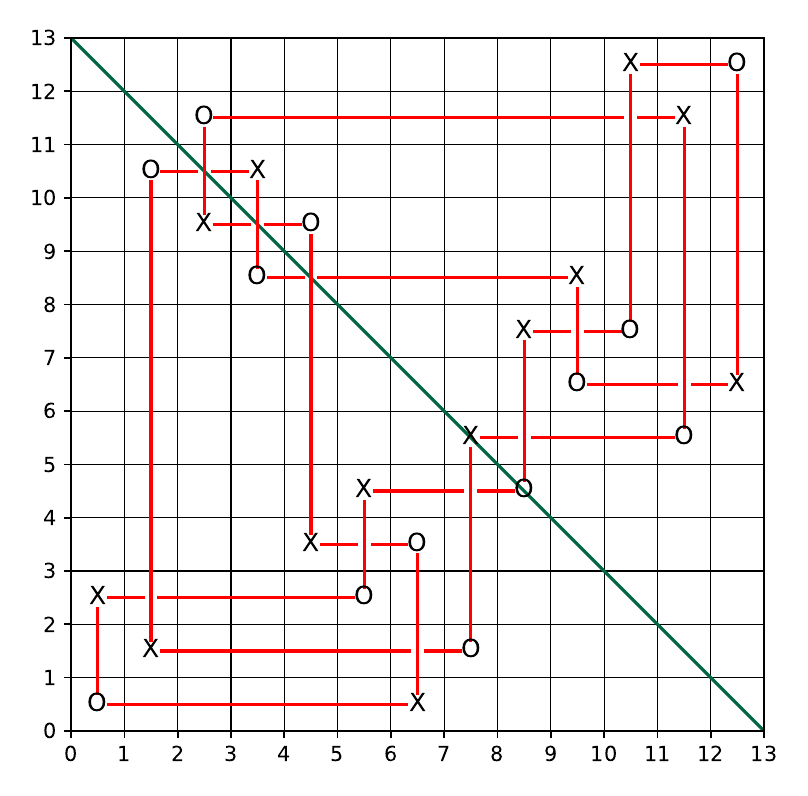} & \parbox[c]{\linewidth}{\centering\tiny
\textcolor{blue}{Polynomial Invariant}\\
$-t^{-2} - t^{-1} + 3 + t - t^{2}$\\[0.1in]
\textcolor{blue}{Real grid homology - hat version}\\
$(-2,-1)\oplus (-1,-1)\oplus (0,0)^{3}\oplus (1,0)\oplus (2,1)$
} \tabularnewline
  \hline
  \centering $3_1 \# 3_1$ & \centering \includegraphics[width=0.25\textwidth]{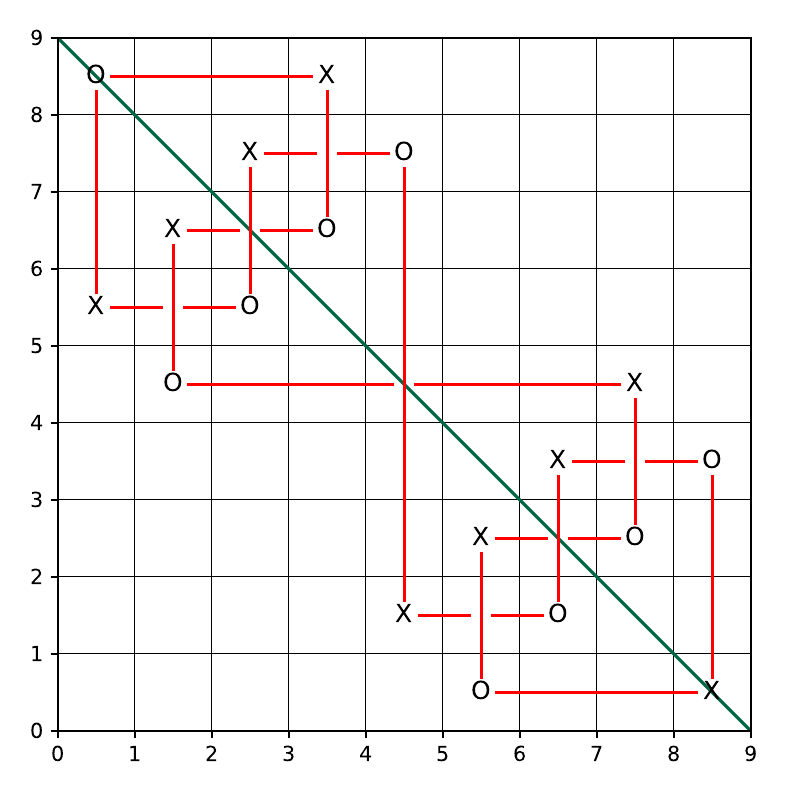} & \parbox[c]{\linewidth}{\centering\tiny
\textcolor{blue}{Polynomial Invariant}\\
$t^{-2} + 2t^{-1} - 1 - 2t + t^{2}$\\[0.1in]
\textcolor{blue}{Real grid homology - hat version}\\
$(-2,0)\oplus (-1,0)^{2}\oplus (0,1)\oplus (1,1)^{2}\oplus (2,2)$\\[0.1in]
\textcolor{blue}{Real grid homology - minus version}\\
$U^{\infty}_{(2,2)}\oplus U_{(-1,0)}\oplus U_{(1,1)}\oplus U^{2}_{(1,1)}$
} \tabularnewline
  \hline
  \centering $(3_1 \# 3_1)_{\text{swap}}$ & \centering \includegraphics[width=0.25\textwidth]{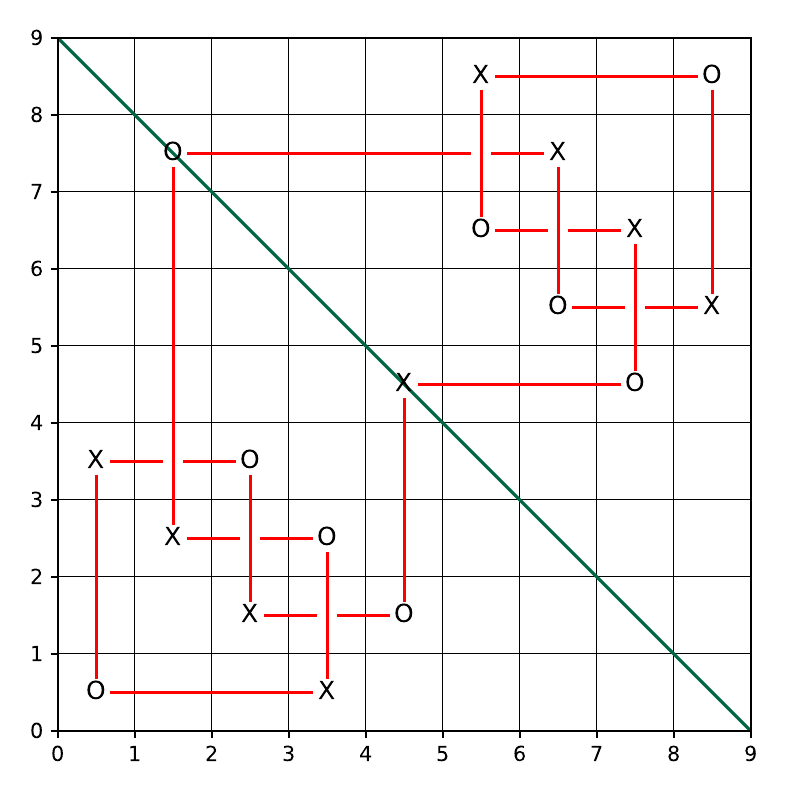} & \parbox[c]{\linewidth}{\centering\tiny
\textcolor{blue}{Polynomial Invariant}\\
$t^{-2} - 1 + t^{2}$\\[0.1in]
\textcolor{blue}{Real grid homology - hat version}\\
$(-2,-2)\oplus (0,-1)\oplus (2,0)$\\[0.1in]
\textcolor{blue}{Real grid homology - minus version}\\
$U^{\infty}_{(-2,-2)}\oplus U^{2}_{(2,0)}$
} \tabularnewline
  \hline
  \centering $3_1 \# 3_1^*$ & \centering \includegraphics[width=0.25\textwidth]{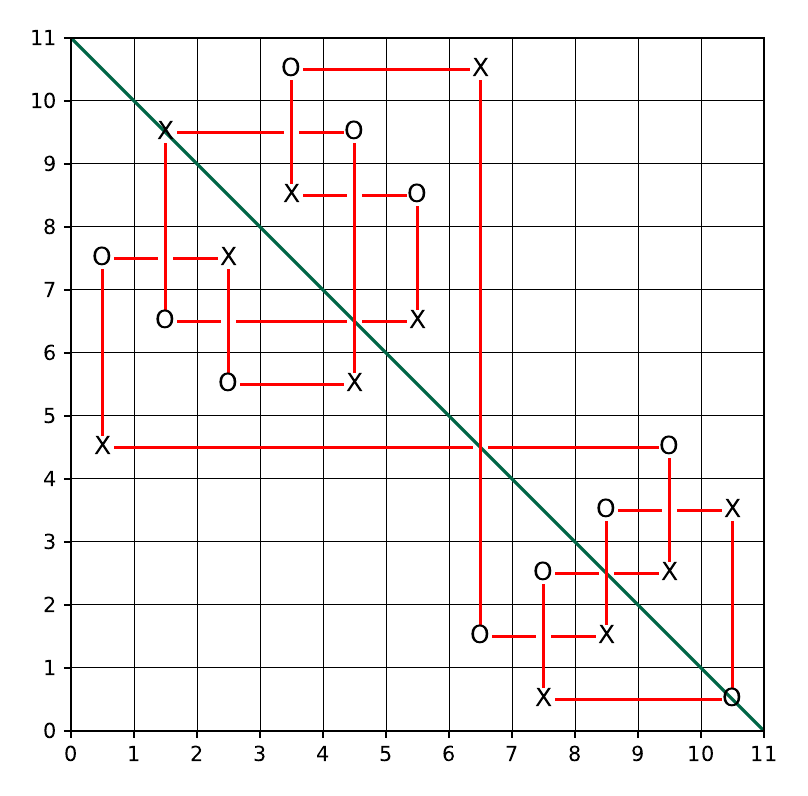} & \parbox[c]{\linewidth}{\centering\tiny
\textcolor{blue}{Polynomial Invariant}\\
$-t^{-2} + 3 - t^{2}$\\[0.1in]
\textcolor{blue}{Real grid homology - hat version}\\
$(-2,-1)\oplus (-1,-1)\oplus (-1,0)\oplus (0,0)^{3}\oplus (2,1)$\\[0.1in]
\textcolor{blue}{Real grid homology - minus version}\\
$U^{\infty}_{(0,0)}\oplus U^{2}_{(2,1)}\oplus U_{(-1,-1)}\oplus U_{(0,0)}$
} \tabularnewline
  \hline
  \centering $3_1 \# 4_1$ & \centering \includegraphics[width=0.25\textwidth]{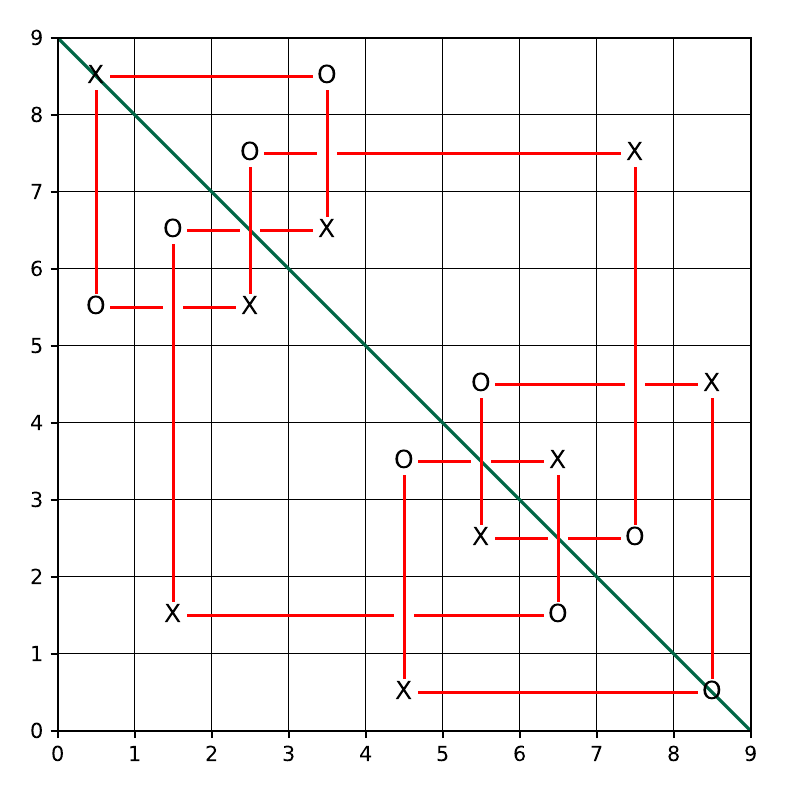} & \parbox[c]{\linewidth}{\centering\tiny
\textcolor{blue}{Polynomial Invariant}\\
$-t^{-2} + 3 - t^{2}$\\[0.1in]
\textcolor{blue}{Real grid homology - hat version}\\
$(-2,-1)\oplus (0,0)^{3}\oplus (1,0)\oplus (1,1)\oplus (2,1)$\\[0.1in]
\textcolor{blue}{Real grid homology - minus version}\\
$U^{\infty}_{(0,0)}\oplus U_{(1,0)}\oplus U_{(2,1)}\oplus U^{2}_{(0,0)}$
} \tabularnewline
  \hline
  \centering $(3_1 \# 4_1)'$ & \centering \includegraphics[width=0.25\textwidth]{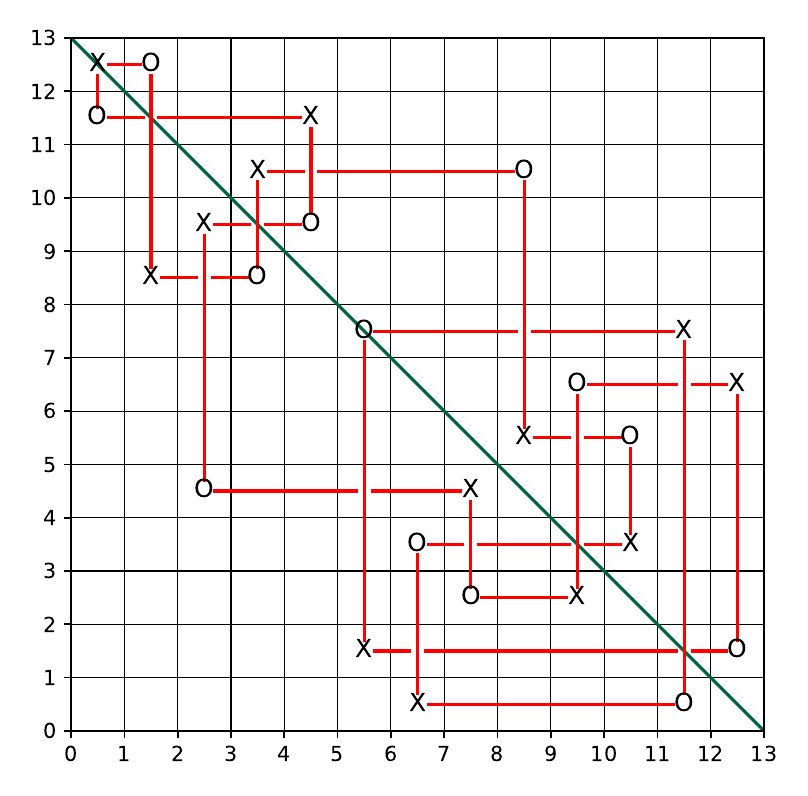} & \parbox[c]{\linewidth}{\centering\tiny
\textcolor{blue}{Polynomial Invariant}\\
$t^{-2} + 2t^{-1} - 1 - 2t + t^{2}$\\[0.1in]
\textcolor{blue}{Real grid homology - hat version}\\
$(-2,0)\oplus (-1,0)^{2}\oplus (0,1)\oplus (1,1)^{2}\oplus (2,2)$
} \tabularnewline
  \hline
  \centering $(3_1 \# 4_1)_{\text{rev}}$ & \centering \includegraphics[width=0.25\textwidth]{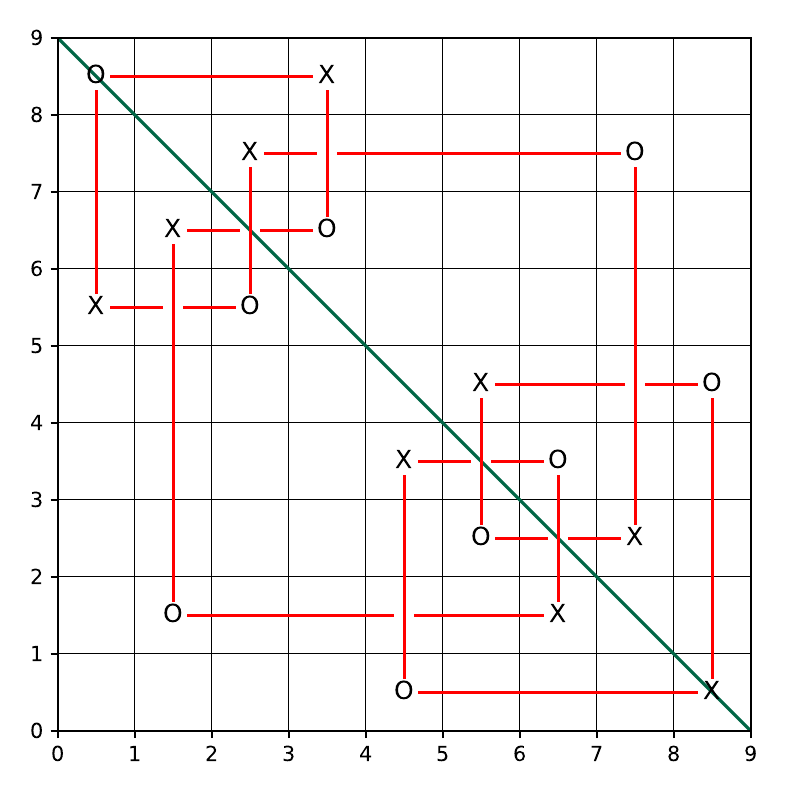} & \parbox[c]{\linewidth}{\centering\tiny
\textcolor{blue}{Polynomial Invariant}\\
$-t^{-2} + 3 - t^{2}$\\[0.1in]
\textcolor{blue}{Real grid homology - hat version}\\
$(-2,-1)\oplus (-1,-1)\oplus (-1,0)\oplus (0,0)^{3}\oplus (2,1)$\\[0.1in]
\textcolor{blue}{Real grid homology - minus version}\\
$U^{\infty}_{(0,0)}\oplus U^{2}_{(2,1)}\oplus U_{(-1,-1)}\oplus U_{(0,0)}$
} \tabularnewline
  \hline
  \centering $3_1 \# 5_1$ & \centering \includegraphics[width=0.25\textwidth]{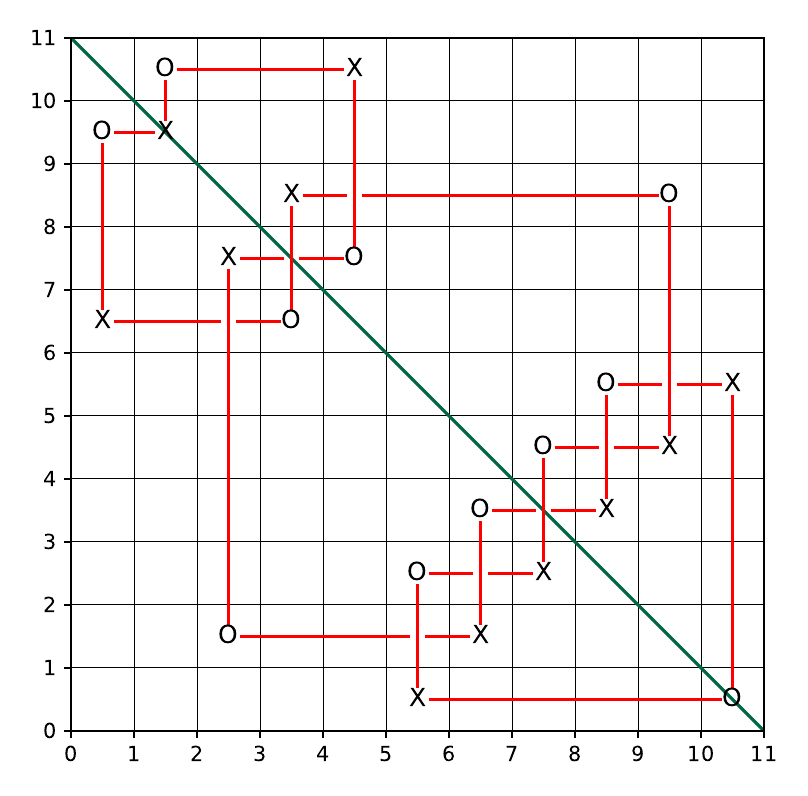} & \parbox[c]{\linewidth}{\centering\tiny
\textcolor{blue}{Polynomial Invariant}\\
$t^{-3} + 2t^{-2} - t^{-1} - 3 + t + 2t^{2} - t^{3}$\\[0.1in]
\textcolor{blue}{Real grid homology - hat version}\\
$(-3,0)\oplus (-2,0)^{2}\oplus (-1,1)\oplus (0,1)^{3}\oplus (1,2)\oplus (2,2)^{2}\oplus (3,3)$\\[0.1in]
\textcolor{blue}{Real grid homology - minus version}\\
$U^{\infty}_{(3,3)}\oplus U^{2}_{(0,1)}\oplus U^{2}_{(2,2)}\oplus U_{(-2,0)}\oplus U_{(0,1)}\oplus U_{(2,2)}$
} \tabularnewline
  \hline
  \centering $3_1 \# 5_2$ & \centering \includegraphics[width=0.25\textwidth]{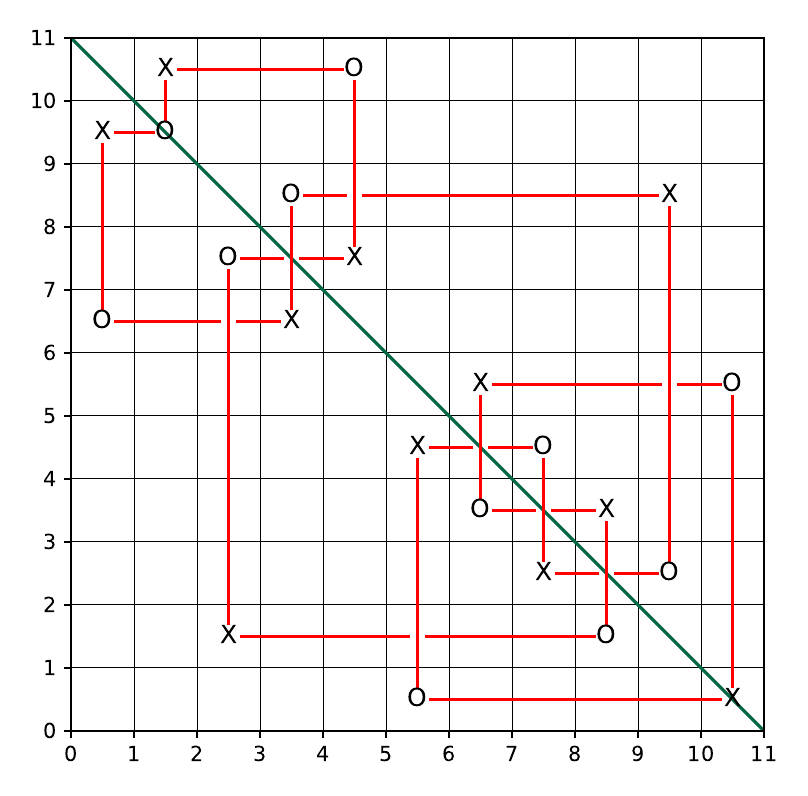} & \parbox[c]{\linewidth}{\centering\tiny
\textcolor{blue}{Polynomial Invariant}\\
$t^{-1} + 1 - t$\\[0.1in]
\textcolor{blue}{Real grid homology - hat version}\\
$(-1,0)\oplus (0,0)\oplus (1,1)$\\[0.1in]
\textcolor{blue}{Real grid homology - minus version}\\
$U^{\infty}_{(1,1)}\oplus U_{(0,0)}$
} \tabularnewline
  \hline
  \centering $3_1 \# 6_1$ & \centering \includegraphics[width=0.25\textwidth]{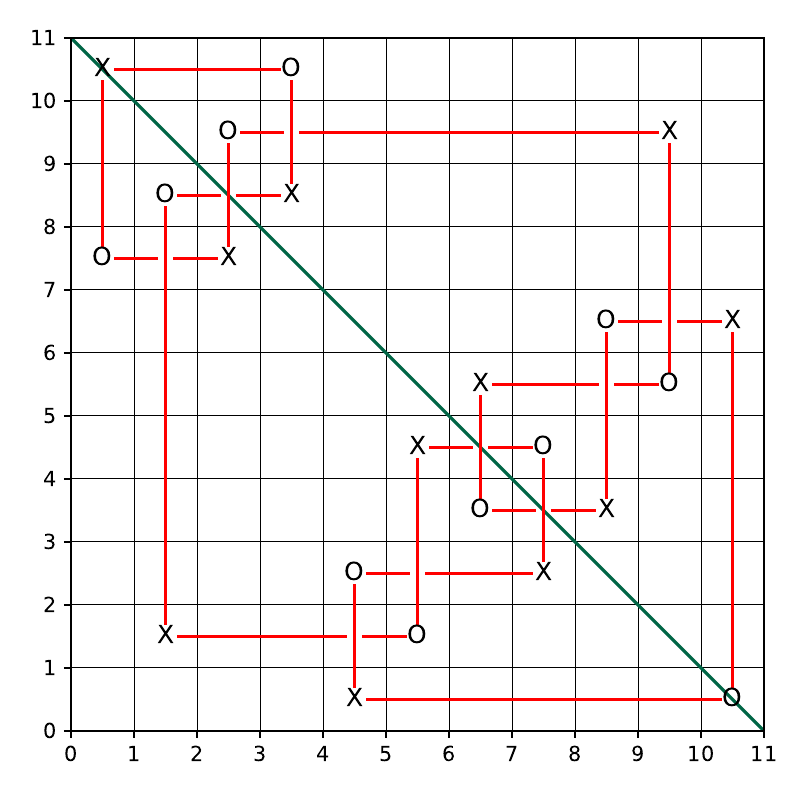} & \parbox[c]{\linewidth}{\centering\tiny
\textcolor{blue}{Polynomial Invariant}\\
$-2t^{-2} - t^{-1} + 5 + t - 2t^{2}$\\[0.1in]
\textcolor{blue}{Real grid homology - hat version}\\
$(-2,-1)^{2}\oplus (-1,-1)\oplus (0,0)^{5}\oplus (1,0)^{2}\oplus (1,1)\oplus (2,1)^{2}$\\[0.1in]
\textcolor{blue}{Real grid homology - minus version}\\
$U^{\infty}_{(0,0)}\oplus U_{(-1,-1)}\oplus \bigg(U_{(1,0)}\bigg)^{2}\oplus U_{(2,1)}\oplus U^{2}_{(0,0)}\oplus U^{2}_{(2,1)}$
} \tabularnewline
  \hline
  \centering $(3_1 \# 6_1)_{\text{rev}}$ & \centering \includegraphics[width=0.25\textwidth]{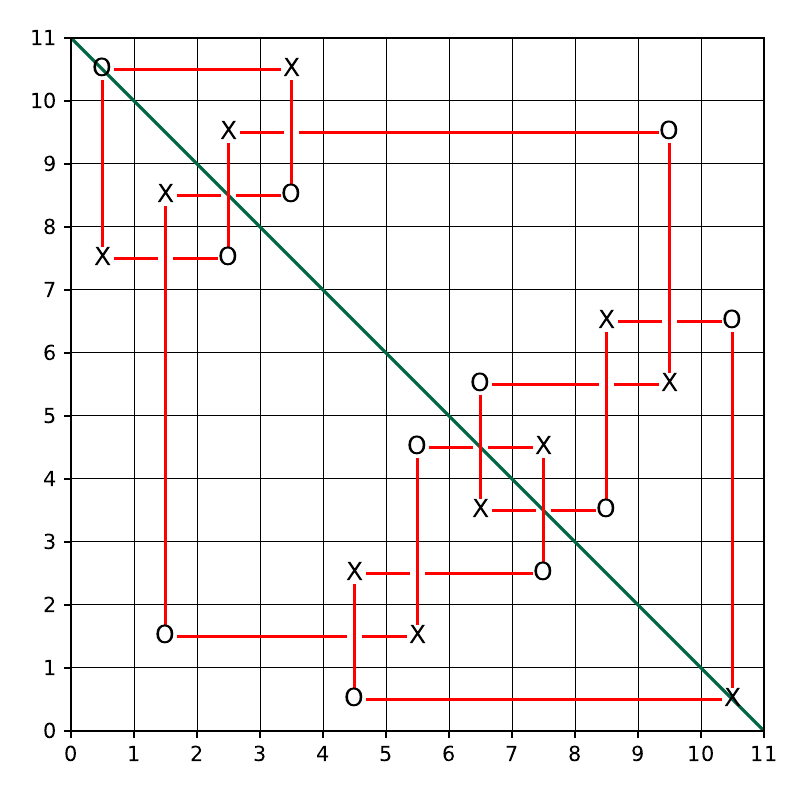} & \parbox[c]{\linewidth}{\centering\tiny
\textcolor{blue}{Polynomial Invariant}\\
$-2t^{-2} - t^{-1} + 5 + t - 2t^{2}$\\[0.1in]
\textcolor{blue}{Real grid homology - hat version}\\
$(-2,-1)^{2}\oplus (-1,-1)^{2}\oplus (-1,0)\oplus (0,0)^{5}\oplus (1,0)\oplus (2,1)^{2}$\\[0.1in]
\textcolor{blue}{Real grid homology - minus version}\\
$U^{\infty}_{(0,0)}\oplus \bigg(U^{2}_{(2,1)}\bigg)^{2}\oplus \bigg(U_{(-1,-1)}\bigg)^{2}\oplus U_{(0,0)}\oplus U_{(1,0)}$
} \tabularnewline
  \hline
  \centering $3_1 \# 6_2$ & \centering \includegraphics[width=0.25\textwidth]{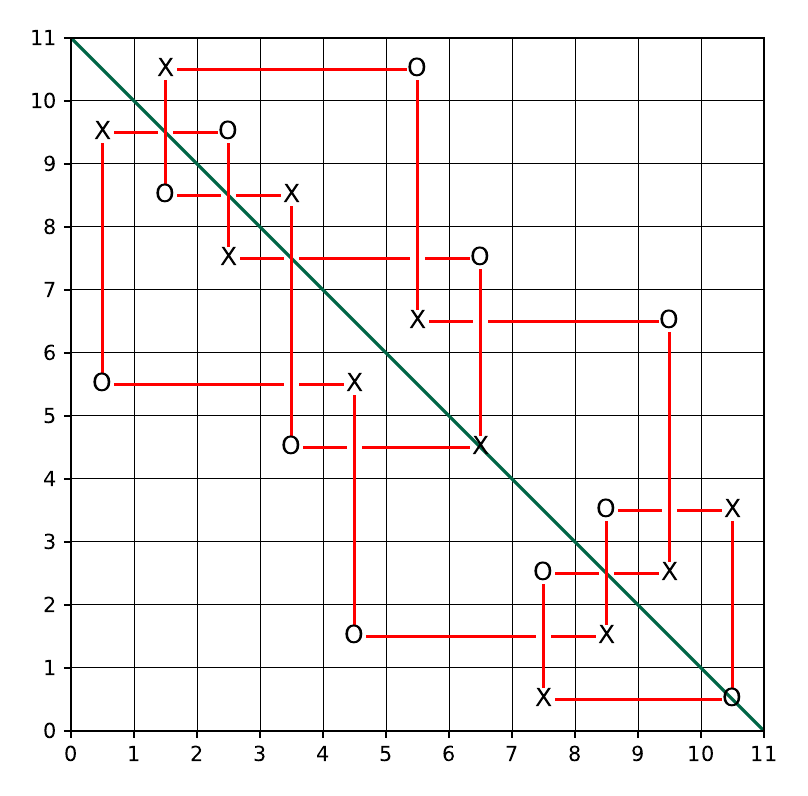} & \parbox[c]{\linewidth}{\centering\tiny
\textcolor{blue}{Polynomial Invariant}\\
$t^{-3} - 3t^{-1} + 1 + 3t - t^{3}$\\[0.1in]
\textcolor{blue}{Real grid homology - hat version}\\
$(-3,-2)\oplus (-2,-2)\oplus (-2,-1)\oplus (-1,-1)^{3}\oplus (0,0)\oplus (1,0)^{3}\oplus (3,1)$\\[0.1in]
\textcolor{blue}{Real grid homology - minus version}\\
$U^{\infty}_{(-1,-1)}\oplus U^{2}_{(1,0)}\oplus U^{2}_{(3,1)}\oplus U_{(-2,-2)}\oplus U_{(-1,-1)}\oplus U_{(1,0)}$
} \tabularnewline
  \hline
  \centering $(3_1 \# 6_2)'$ & \centering \includegraphics[width=0.25\textwidth]{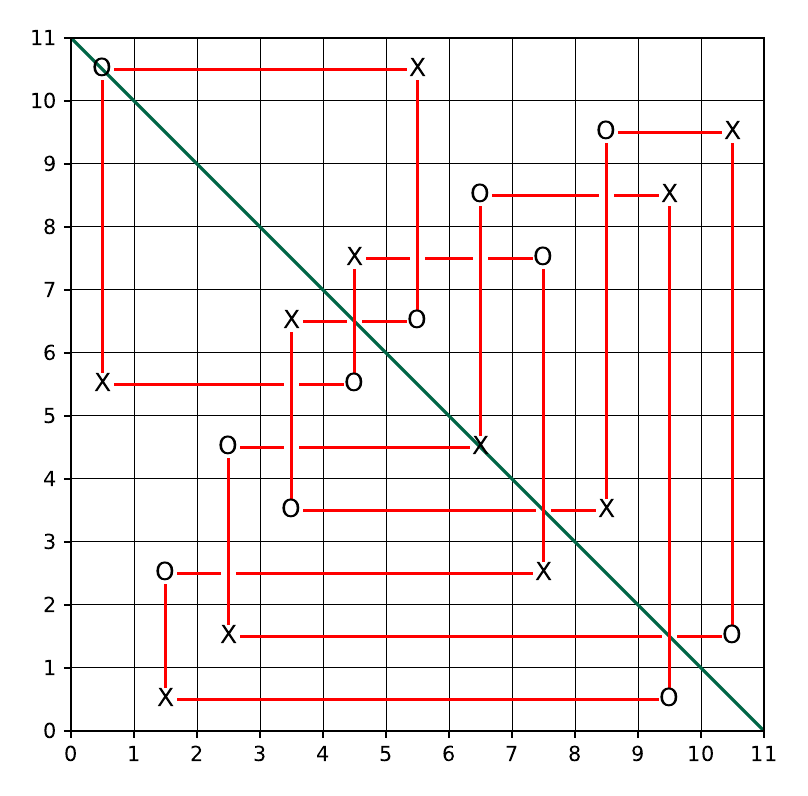} & \parbox[c]{\linewidth}{\centering\tiny
\textcolor{blue}{Polynomial Invariant}\\
$-t^{-3} + 5t^{-1} + 1 - 5t + t^{3}$\\[0.1in]
\textcolor{blue}{Real grid homology - hat version}\\
$(-3,-1)\oplus (-2,-1)\oplus (-2,0)\oplus (-1,0)^{5}\oplus (0,0)^{2}\oplus (0,1)\oplus (1,1)^{5}\oplus (3,2)$\\[0.1in]
\textcolor{blue}{Real grid homology - minus version}\\
$U^{\infty}_{(1,1)}\oplus \bigg(U^{2}_{(1,1)}\bigg)^{2}\oplus U^{2}_{(3,2)}\oplus U_{(-2,-1)}\oplus U_{(-1,0)}\oplus \bigg(U_{(0,0)}\bigg)^{2}\oplus U_{(1,1)}$
} \tabularnewline
  \hline
  \centering $(3_1 \# 6_2)'_{\text{rev}}$ & \centering \includegraphics[width=0.25\textwidth]{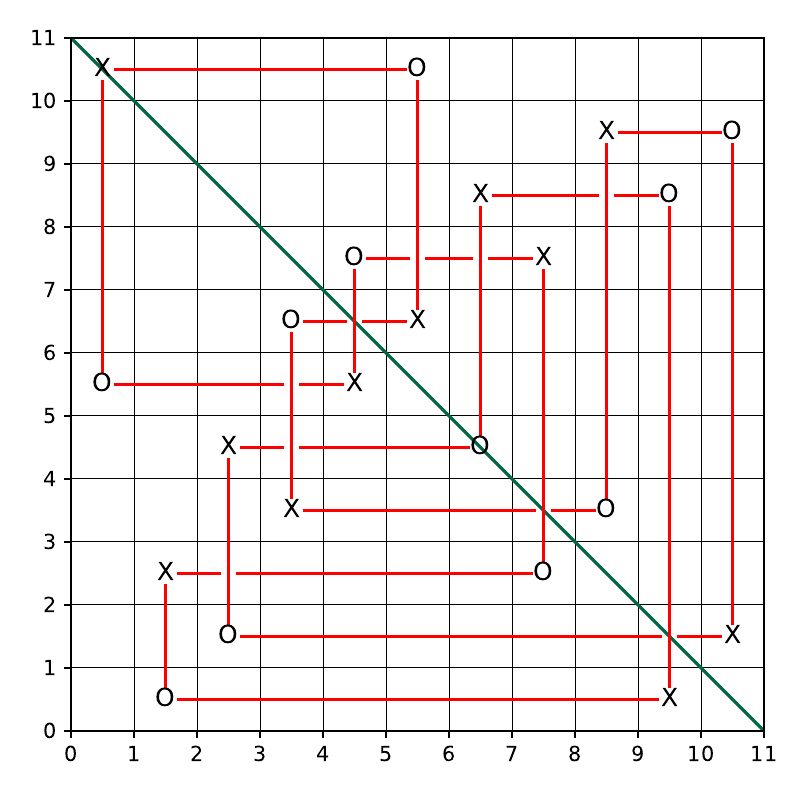} & \parbox[c]{\linewidth}{\centering\tiny
\textcolor{blue}{Polynomial Invariant}\\
$-t^{-3} + 5t^{-1} + 1 - 5t + t^{3}$\\[0.1in]
\textcolor{blue}{Real grid homology - hat version}\\
$(-3,-1)\oplus (-1,0)^{5}\oplus (0,0)^{2}\oplus (0,1)\oplus (1,1)^{5}\oplus (2,1)\oplus (2,2)\oplus (3,2)$\\[0.1in]
\textcolor{blue}{Real grid homology - minus version}\\
$U^{\infty}_{(1,1)}\oplus \bigg(U_{(0,0)}\bigg)^{2}\oplus U_{(1,1)}\oplus U_{(2,1)}\oplus U_{(3,2)}\oplus U^{2}_{(-1,0)}\oplus \bigg(U^{2}_{(1,1)}\bigg)^{2}$
} \tabularnewline
  \hline
  \centering $(3_1 \# 6_2)_{\text{rev}}$ & \centering \includegraphics[width=0.25\textwidth]{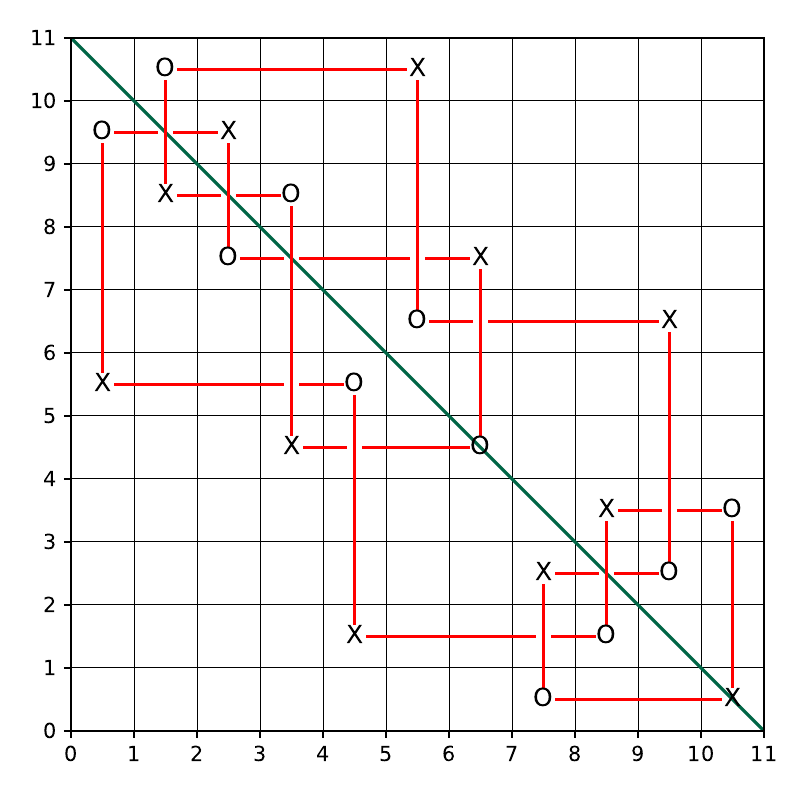} & \parbox[c]{\linewidth}{\centering\tiny
\textcolor{blue}{Polynomial Invariant}\\
$t^{-3} - 3t^{-1} + 1 + 3t - t^{3}$\\[0.1in]
\textcolor{blue}{Real grid homology - hat version}\\
$(-3,-2)\oplus (-1,-1)^{3}\oplus (0,0)\oplus (1,0)^{3}\oplus (2,0)\oplus (2,1)\oplus (3,1)$\\[0.1in]
\textcolor{blue}{Real grid homology - minus version}\\
$U^{\infty}_{(-1,-1)}\oplus U_{(1,0)}\oplus U_{(2,0)}\oplus U_{(3,1)}\oplus U^{2}_{(-1,-1)}\oplus U^{2}_{(1,0)}$
} \tabularnewline
  \hline
  \centering $(4_1 \# 4_1)_{\text{swap}}$ & \centering \includegraphics[width=0.25\textwidth]{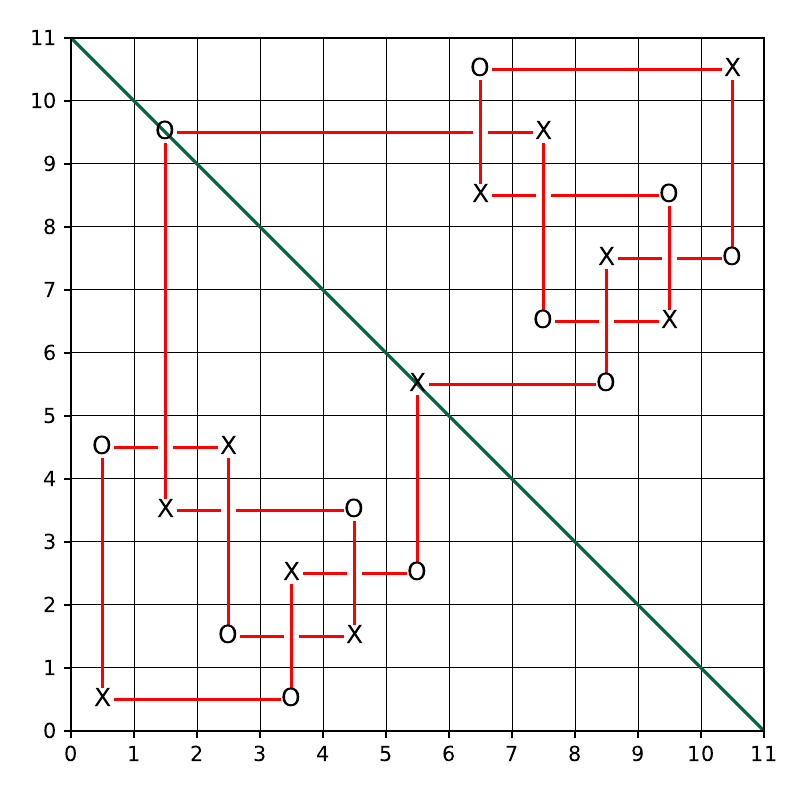} & \parbox[c]{\linewidth}{\centering\tiny
\textcolor{blue}{Polynomial Invariant}\\
$-t^{-2} + 3 - t^{2}$\\[0.1in]
\textcolor{blue}{Real grid homology - hat version}\\
$(-2,-1)\oplus (0,0)^{3}\oplus (2,1)$\\[0.1in]
\textcolor{blue}{Real grid homology - minus version}\\
$U^{\infty}_{(0,0)}\oplus U^{2}_{(0,0)}\oplus U^{2}_{(2,1)}$
} \tabularnewline
  \hline
  \centering $4_1 \# 4_1$ & \centering \includegraphics[width=0.25\textwidth]{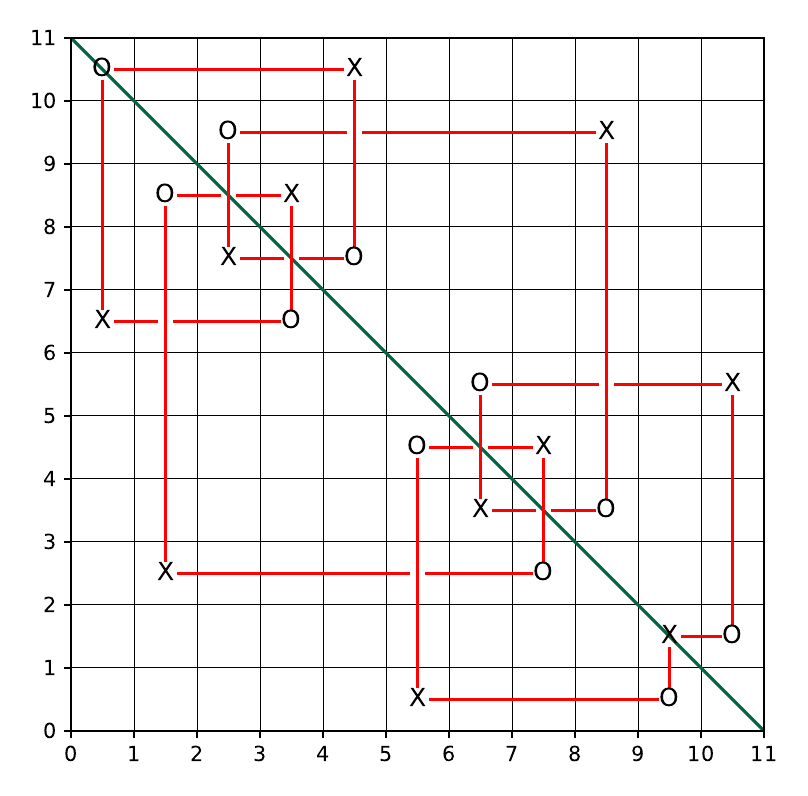} & \parbox[c]{\linewidth}{\centering\tiny
\textcolor{blue}{Polynomial Invariant}\\
$t^{-2} - 2t^{-1} - 1 + 2t + t^{2}$\\[0.1in]
\textcolor{blue}{Real grid homology - hat version}\\
$(-2,-2)\oplus (-1,-1)^{2}\oplus (0,-1)\oplus (1,0)^{2}\oplus (2,0)$\\[0.1in]
\textcolor{blue}{Real grid homology - minus version}\\
$U^{\infty}_{(-2,-2)}\oplus U_{(0,-1)}\oplus U_{(2,0)}\oplus U^{2}_{(1,0)}$
} \tabularnewline
  \hline
  \centering $4_1 \# 5_1$ & \centering \includegraphics[width=0.25\textwidth]{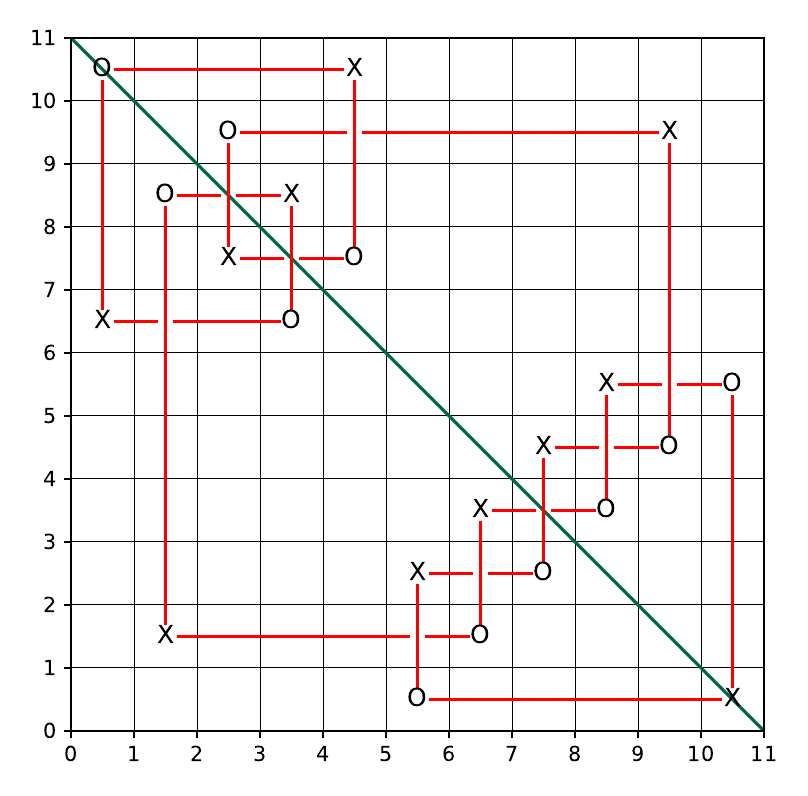} & \parbox[c]{\linewidth}{\centering\tiny
\textcolor{blue}{Polynomial Invariant}\\
$-t^{-3} + 3t^{-1} + 1 - 3t + t^{3}$\\[0.1in]
\textcolor{blue}{Real grid homology - hat version}\\
$(-3,-1)\oplus (-1,0)^{3}\oplus (0,0)\oplus (1,1)^{3}\oplus (2,1)\oplus (2,2)\oplus (3,2)$\\[0.1in]
\textcolor{blue}{Real grid homology - minus version}\\
$U^{\infty}_{(1,1)}\oplus U_{(0,0)}\oplus U_{(2,1)}\oplus U_{(3,2)}\oplus U^{2}_{(-1,0)}\oplus U^{2}_{(1,1)}$
} \tabularnewline
  \hline
  \centering $(4_1 \# 5_1)_{\text{rev}}$ & \centering \includegraphics[width=0.25\textwidth]{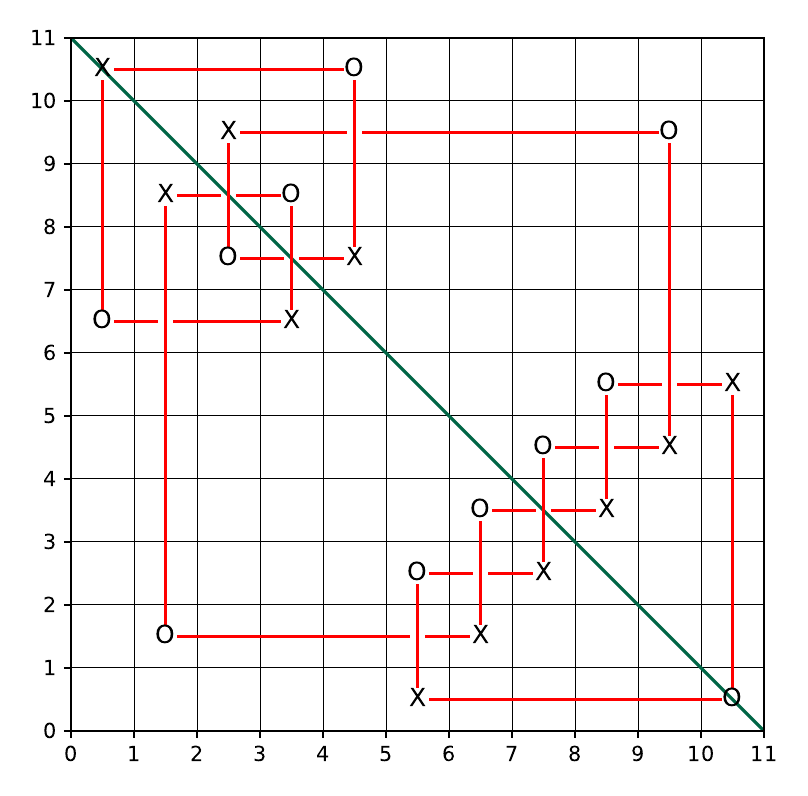} & \parbox[c]{\linewidth}{\centering\tiny
\textcolor{blue}{Polynomial Invariant}\\
$-t^{-3} + 3t^{-1} + 1 - 3t + t^{3}$\\[0.1in]
\textcolor{blue}{Real grid homology - hat version}\\
$(-3,-1)\oplus (-2,-1)\oplus (-2,0)\oplus (-1,0)^{3}\oplus (0,0)\oplus (1,1)^{3}\oplus (3,2)$\\[0.1in]
\textcolor{blue}{Real grid homology - minus version}\\
$U^{\infty}_{(1,1)}\oplus U^{2}_{(1,1)}\oplus U^{2}_{(3,2)}\oplus U_{(-2,-1)}\oplus U_{(-1,0)}\oplus U_{(0,0)}$
} \tabularnewline
  \hline
  \centering $5_1 \# 5_1$ & \centering \includegraphics[width=0.25\textwidth]{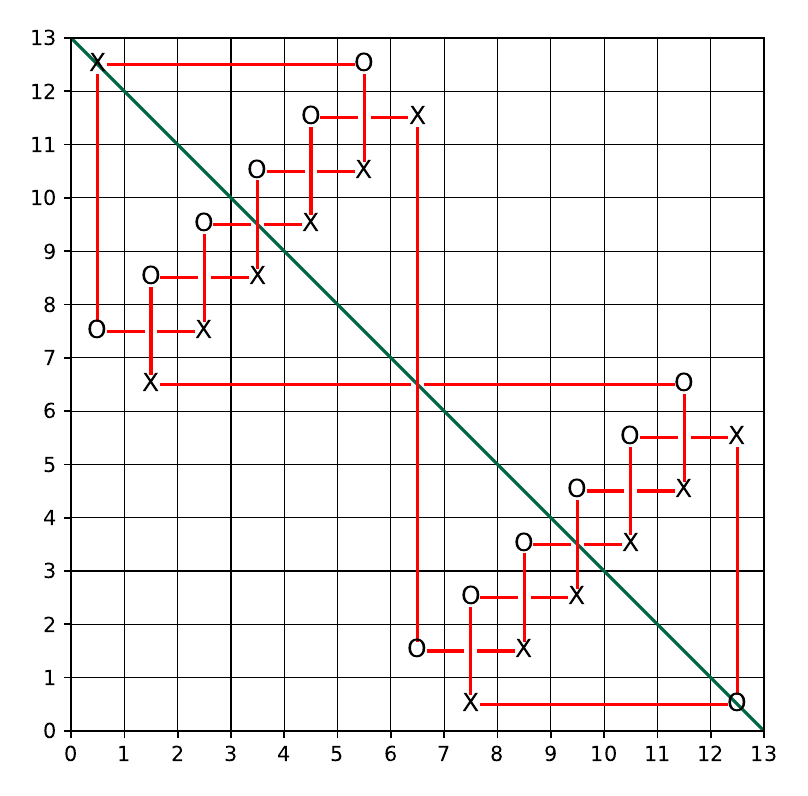} & \parbox[c]{\linewidth}{\centering\tiny
\textcolor{blue}{Polynomial Invariant}\\
$t^{-4} + 2t^{-3} - t^{-2} - 4t^{-1} + 1 + 4t - t^{2} - 2t^{3} + t^{4}$\\[0.1in]
\textcolor{blue}{Real grid homology - hat version}\\
$(-4,0)\oplus (-3,0)^{2}\oplus (-2,1)\oplus (-1,1)^{4}\oplus (0,2)\oplus (1,2)^{4}\oplus (2,3)\oplus (3,3)^{2}\oplus (4,4)$
} \tabularnewline
  \hline
  \centering $(5_1 \# 5_1)_{\text{swap}}$ & \centering \includegraphics[width=0.25\textwidth]{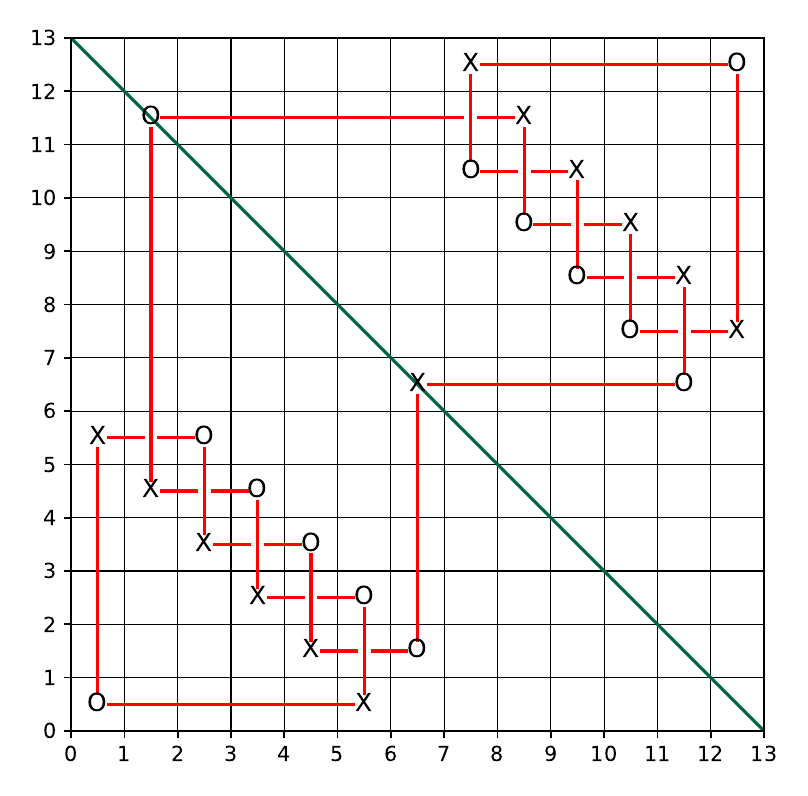} & \parbox[c]{\linewidth}{\centering\tiny
\textcolor{blue}{Polynomial Invariant}\\
$t^{-4} - t^{-2} + 1 - t^{2} + t^{4}$\\[0.1in]
\textcolor{blue}{Real grid homology - hat version}\\
$(-4,-4)\oplus (-2,-3)\oplus (0,-2)\oplus (2,-1)\oplus (4,0)$
} \tabularnewline
  \hline
  \centering $(5_2 \# 5_2)_{\text{swap}}$ & \centering \includegraphics[width=0.25\textwidth]{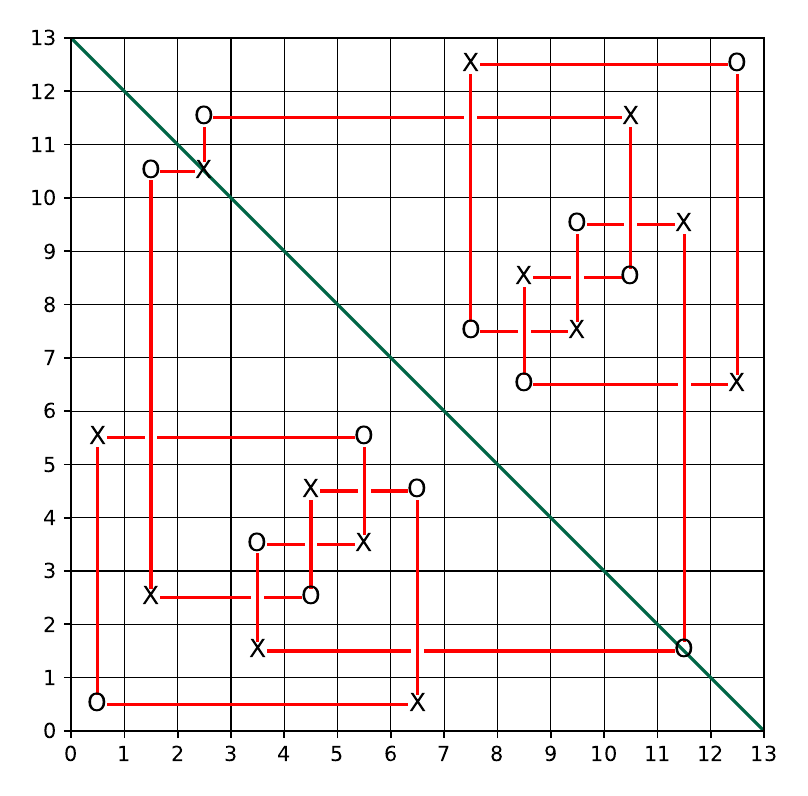} & \parbox[c]{\linewidth}{\centering\tiny
\textcolor{blue}{Polynomial Invariant}\\
$2t^{-2} - 3 + 2t^{2}$\\[0.1in]
\textcolor{blue}{Real grid homology - hat version}\\
$(-2,-2)^{2}\oplus (0,-1)^{3}\oplus (2,0)^{2}$
} \tabularnewline
  \hline
  \centering $(3_1 \# 3_1) \# (3_1 \# 3_1)$ & \centering \includegraphics[width=0.25\textwidth]{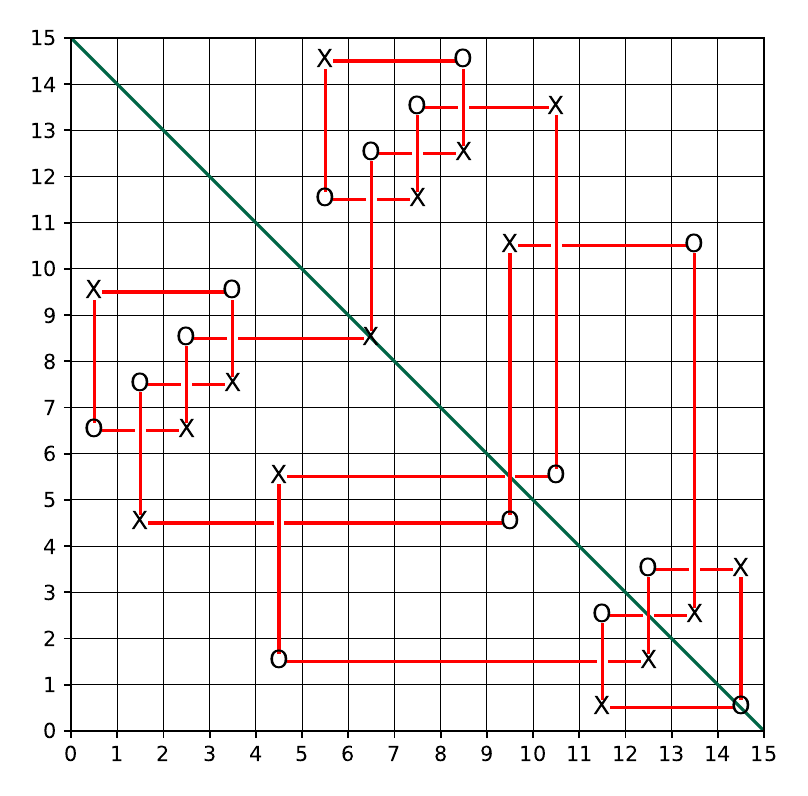} & \parbox[c]{\linewidth}{\centering\tiny
\textcolor{blue}{Polynomial Invariant}\\
$t^{-4} + 2t^{-3} - 2t^{-2} - 4t^{-1} + 3 + 4t - 2t^{2} - 2t^{3} + t^{4}$
} \tabularnewline
  \hline
\end{longtable}


\bibliographystyle{plain}
\bibliography{bibliography}

@article{Konno2024,
author = {Konno, Hokuto and Miyazawa, Jin and Taniguchi, Masaki},
title = {{Involutions, links, and Floer cohomologies}},
journal = {Journal of Topology},
volume = {17},
number = {2},
pages = {e12340},
doi = {https://doi.org/10.1112/topo.12340},
url = {https://londmathsoc.onlinelibrary.wiley.com/doi/abs/10.1112/topo.12340},
eprint = {https://londmathsoc.onlinelibrary.wiley.com/doi/pdf/10.1112/topo.12340},
year = {2024}
}

@article{SarkarWang2010,
	ISSN = {0003486X},
	URL = {http://www.jstor.org/stable/20752237},
	author = {Sucharit Sarkar and Jiajun Wang},
	journal = {Annals of Mathematics},
	number = {2},
	pages = {1213--1236},
	publisher = {Annals of Mathematics},
	title = {{An algorithm for computing some Heegaard Floer homologies}},
	urldate = {2024-12-21},
	volume = {171},
	year = {2010}
}

@article{li2022monopolefloerhomologyreal,
      title={Monopole {F}loer Homology and Real Structures}, 
      author={Jiakai Li},
      year={2022},
      journal={arXiv preprint arXiv:2211.10768},
}

@misc{li2024realmonopolesspectralsequence,
	title={Real monopoles and a spectral sequence from Khovanov homology}, 
	author={Jiakai Li},
	year={2024},
	journal={arXiv preprint arXiv:2406.00152},
	eprint={2406.00152},
	archivePrefix={arXiv},
	primaryClass={math.GT},
	url={https://arxiv.org/abs/2406.00152}, 
}

@article{Lobb2021ArefinementofKhovanovhomology,
author = {Lobb, Andrew and Watson, Liam},
year = {2021},
month = {07},
pages = {1861-1917},
title = {{A refinement of Khovanov homology}},
volume = {25},
journal = {Geometry \& Topology},
doi = {10.2140/gt.2021.25.1861}
}

@article{hirasawa2023equivariant,
  title={The equivariant genera of marked strongly invertible knots associated with $2$-bridge knots},
  author={Hirasawa, Mikami and Hiura, Ryota and Sakuma, Makoto},
  journal={arXiv preprint arXiv:2312.06156},
  year={2023}
}

@incollection {hirasawainvariantseifertsurface,
    AUTHOR = {Hirasawa, Mikami and Hiura, Ryota and Sakuma, Makoto},
     TITLE = {Invariant {S}eifert surfaces for strongly invertible knots},
 BOOKTITLE = {Essays in geometry---dedicated to {N}orbert {A}'{C}ampo},
    SERIES = {IRMA Lect. Math. Theor. Phys.},
    VOLUME = {34},
     PAGES = {325--349},
 PUBLISHER = {EMS Press, Berlin},
      YEAR = {[2023] \copyright 2023},
      ISBN = {978-3-98547-024-2; 978-3-98547-524-7},
   MRCLASS = {57K10 (57M60)},
  MRNUMBER = {4631272},
MRREVIEWER = {Masakazu\ Teragaito},
}

@article{Holomorphicdisksandknotinvariants,
	title = {Holomorphic disks and knot invariants},
	volume = {186},
	issn = {0001-8708},
	url = {http://www.sciencedirect.com/science/article/pii/S0001870803002330},
	doi = {10.1016/j.aim.2003.05.001},
	language = {en},
	number = {1},
	urldate = {2020-06-30},
	journal = {Advances in Mathematics},
	author = {Ozsv\'{a}th, Peter and Szab\'{o}, Zolt\'{a}n},
	month = aug,
	year = {2004},
	pages = {58--116}
}

@article{ozsvath2004holomorphic,
  title={Holomorphic disks and topological invariants for closed three-manifolds},
  author={Ozsv{\'a}th, Peter and Szab{\'o}, Zolt{\'a}n},
  journal={Annals of Mathematics},
  pages={1027--1158},
  year={2004},
  publisher={JSTOR}
}

@article{Rasmussen,
	title = {Floer homology and knot complements},
	url = {http://arxiv.org/abs/math/0306378},
	urldate = {2020-06-30},
	journal = {arXiv:math/0306378},
	author = {Rasmussen, Jacob},
	month = jun,
	year = {2003},
	note = {arXiv: math/0306378}
}

@article{zemke2019link,
  title={Link cobordisms and functoriality in link {F}loer homology},
  author={Zemke, Ian},
  journal={Journal of Topology},
  volume={12},
  number={1},
  pages={94--220},
  year={2019},
  publisher={Wiley Online Library}
}

@article{Kang_2026,
	title={{Cables of the figure-eight knot via real Frøyshov invariants}},
	volume={30},
	ISSN={1465-3060},
	url={http://dx.doi.org/10.2140/gt.2026.30.781},
	DOI={10.2140/gt.2026.30.781},
	number={2},
	journal={Geometry \& Topology},
	publisher={Mathematical Sciences Publishers},
	author={Kang, Sungkyung and Park, JungHwan and Taniguchi, Masaki},
	year={2026},
	month=Mar, pages={781–833} 
}

@article{DKMPS,
	URL = {https://doi.org/10.1007/s00222-024-01286-w},
	author = {Irving Dai and Sungkyung Kang and Abhishek Mallick and  JungHwan Park and  Matthew Stoffregen},
	journal = {Inventiones mathematicae},
	pages = {371--390},
	publisher = {Springer},
	title = {The $(2,1)$-cable of the figure-eight knot is not smoothly slice},
	volume = {238},
	year = {2024}
}

@article {MOST2007grid,
    AUTHOR = {Manolescu, Ciprian and Ozsv\'{a}th, Peter and Szab\'{o},
              Zolt\'{a}n and Thurston, Dylan},
     TITLE = {On combinatorial link {F}loer homology},
   JOURNAL = {Geom. Topol.},
  FJOURNAL = {Geometry \& Topology},
    VOLUME = {11},
      YEAR = {2007},
     PAGES = {2339--2412},
      ISSN = {1465-3060,1364-0380},
   MRCLASS = {57R58 (57M25 57M27)},
  MRNUMBER = {2372850},
MRREVIEWER = {Thomas\ E.\ Mark},
       DOI = {10.2140/gt.2007.11.2339},
       URL = {https://doi.org/10.2140/gt.2007.11.2339},
}

@article {MOS2009knot,
    AUTHOR = {Manolescu, Ciprian and Ozsv\'{a}th, Peter and Sarkar,
              Sucharit},
     TITLE = {A combinatorial description of knot {F}loer homology},
   JOURNAL = {Ann. of Math. (2)},
  FJOURNAL = {Annals of Mathematics. Second Series},
    VOLUME = {169},
      YEAR = {2009},
    NUMBER = {2},
     PAGES = {633--660},
      ISSN = {0003-486X,1939-8980},
   MRCLASS = {57R58 (57M27)},
  MRNUMBER = {2480614},
MRREVIEWER = {Burak\ Ozbagci},
       DOI = {10.4007/annals.2009.169.633},
       URL = {https://doi.org/10.4007/annals.2009.169.633},
}

@book {OSS2015grid,
    AUTHOR = {Ozsv\'{a}th, Peter S. and Stipsicz, Andr\'{a}s I. and
              Szab\'{o}, Zolt\'{a}n},
     TITLE = {Grid homology for knots and links},
    SERIES = {Mathematical Surveys and Monographs},
    VOLUME = {208},
 PUBLISHER = {American Mathematical Society, Providence, RI},
      YEAR = {2015},
     PAGES = {x+410},
      ISBN = {978-1-4704-1737-6},
   MRCLASS = {57M27 (53D10 57M25 57R17 57R58)},
  MRNUMBER = {3381987},
MRREVIEWER = {Paolo\ Ghiggini},
       DOI = {10.1090/surv/208},
       URL = {https://doi.org/10.1090/surv/208},
}

@article {Sarkar2011tau,
    AUTHOR = {Sarkar, Sucharit},
     TITLE = {Grid diagrams and the {O}zsv\'{a}th-{S}zab\'{o} tau-invariant},
   JOURNAL = {Math. Res. Lett.},
  FJOURNAL = {Mathematical Research Letters},
    VOLUME = {18},
      YEAR = {2011},
    NUMBER = {6},
     PAGES = {1239--1257},
      ISSN = {1073-2780,1945-001X},
   MRCLASS = {57M27 (57M25 57N70)},
  MRNUMBER = {2915478},
MRREVIEWER = {Burak\ Ozbagci},
       DOI = {10.4310/MRL.2011.v18.n6.a13},
       URL = {https://doi.org/10.4310/MRL.2011.v18.n6.a13},
}

@article{Alishahi_2020,
   title={{Knot Floer homology and the unknotting number}},
   volume={24},
   ISSN={1465-3060},
   url={http://dx.doi.org/10.2140/gt.2020.24.2435},
   DOI={10.2140/gt.2020.24.2435},
   number={5},
   journal={Geometry \& Topology},
   publisher={Mathematical Sciences Publishers},
   author={Alishahi, Akram and Eftekhary, Eaman},
   year={2020},
   month=dec, pages={2435–2469} }

@article{hendricks2025noterealheegaardfloer,
      title={{A note on real Heegaard Floer homology and localization}}, 
      author={Kristen Hendricks},
      year={2025},
      journal={arXiv preprint arXiv:2508.03897},
      eprint={2508.03897},
      archivePrefix={arXiv},
      primaryClass={math.GT},
      url={https://arxiv.org/abs/2508.03897}, 
}

@article{miyazawa2023gaugetheoreticinvariantembedded,
      title={A gauge theoretic invariant of embedded surfaces in $4$-manifolds and exotic $P^2$-knots}, 
      author={Jin Miyazawa},
      year={2023},
      journal={arXiv preprint arXiv:2312.0204},
      eprint={2312.02041},
      archivePrefix={arXiv},
      primaryClass={math.GT},
      url={https://arxiv.org/abs/2312.02041}, 
}

@article{miyazawa2025satelliteformularealseibergwitten,
      title={A satellite formula for real Seiberg-Witten Floer homotopy types}, 
      author={Jin Miyazawa and JungHwan Park and Masaki Taniguchi},
      year={2025},
      journal={arXiv preprint arXiv:2504.03270},
      eprint={2504.03270},
      archivePrefix={arXiv},
      primaryClass={math.GT},
      url={https://arxiv.org/abs/2504.03270}, 
}

@article{boyle2025equivariantunknottingnumbersstrongly,
      title={Equivariant unknotting numbers of strongly invertible knots}, 
      author={Keegan Boyle and Wenzhao Chen},
      year={2025},
      journal={arXiv preprint arXiv:2412.09797},
      archivePrefix={arXiv},
      primaryClass={math.GT},
      url={https://arxiv.org/abs/2412.09797}, 
}

@article{OSHFKand4ballgenus,
author = {Ozsvath, Peter and Szabó, Zoltán},
year = {2003},
month = {02},
pages = {1364-380},
title = {{Knot Floer homology and the four-ball genus}},
volume = {7},
journal = {Geometry \& Topology},
doi = {10.2140/gt.2003.7.615}
}

@article{DMSequivariantknotandHFK,
author = {Dai, Irving and Mallick, Abhishek and Stoffregen, Matthew},
year = {2023},
month = {09},
pages = {1167-1236},
title = {{Equivariant knots and knot Floer homology}},
volume = {16},
journal = {Journal of Topology},
doi = {10.1112/topo.12312}
}

@article{Sarkar_2015movingbasept,
   title={{Moving basepoints and the induced automorphisms of link Floer homology}},
   volume={15},
   ISSN={1472-2747},
   url={http://dx.doi.org/10.2140/agt.2015.15.2479},
   DOI={10.2140/agt.2015.15.2479},
   number={5},
   journal={Algebraic \& Geometric Topology},
   publisher={Mathematical Sciences Publishers},
   author={Sarkar, Sucharit},
   year={2015},
   month=nov, pages={2479–2515} }

@article{YXHFLR2,
      title={{More versions of Real Link Floer homology}}, 
      author={Yonghan Xiao},
      year={2026},
      journal={In preparation}
}

@article{LOrealbordered,
	title={{Real bordered Floer homology}}, 
	author={Robert Lipshitz and Peter Ozsváth},
	year={2026},
	journal={arXiv preprint arXiv:2604.20565}, 
}

@article{GM_real_naturality,
	title={{Naturality in real Heegaard Floer theory}}, 
	author={Gary Guth and Ciprian Manolescu},
	journal={arXiv preprint arXiv:2608.00256},
	year={2026},
	eprint={2608.00256},
	archivePrefix={arXiv},
	primaryClass={math.GT},
	url={https://arxiv.org/abs/2608.00256}, 
}

@article{BGX,
	title={{Real sutured Heegaard Floer homology}}, 
	author={Fraser Binns and Gary Guth and Yonghan Xiao},
	year={2026},
	journal={arXiv preprint arXiv:2607.11082},
	eprint={2607.11082},
	archivePrefix={arXiv},
	primaryClass={math.GT},
	url={https://arxiv.org/abs/2607.11082}, 
}

@article{collari2026polynomialinvariantstronglyinvolutive,
	title={A Polynomial Invariant of Strongly Involutive Links}, 
	author={Carlo Collari and Paolo Lisca},
	journal={arXiv preprint arXiv:2606.21471},
	year={2026},
	eprint={2606.21471},
	archivePrefix={arXiv},
	primaryClass={math.GT},
	url={https://arxiv.org/abs/2606.21471}, 
}

@article{BDMS26,
	title={Reidemeister and movie moves for involutive links}, 
	author={Maciej Borodzik and Irving Dai and Abhishek Mallick and Matthew Stoffregen},
	year={2026},
	journal={arXiv preprint arXiv:2604.26369},
	eprint={2604.26369},
	archivePrefix={arXiv},
	primaryClass={math.GT},
	url={https://arxiv.org/abs/2604.26369}, 
}

@article{BGCOMPUTATIONSOFHEEGAARD-FLOERKNOTHOMOLOGY,
author = {Baldwin, John A. and Gillam, William D.},
title = {{Computations of Heegaard-Floer knot homology}},
journal = {Journal of Knot Theory and Its Ramifications},
volume = {21},
number = {08},
pages = {1250075},
year = {2012},
doi = {10.1142/S0218216512500757},
URL = {https://doi.org/10.1142/S0218216512500757},
eprint = { https://doi.org/10.1142/S0218216512500757},
}

@InProceedings{Involutionsandisotopiesoflensspaces,
author="Hodgson, Craig
and Rubinstein, J. H.",
editor="Rolfsen, Dale",
title="Involutions and isotopies of lens spaces",
booktitle="Knot Theory and Manifolds",
year="1985",
publisher="Springer Berlin Heidelberg",
address="Berlin, Heidelberg",
pages="60--96",
isbn="978-3-540-39616-1"
}

@article{Realopenbooksandrealcontactstructures,
url = {https://doi.org/10.1515/advgeom-2015-0031},
title = {Real open books and real contact structures},
author = {Ferit Öztürk and Nermin Salepci},
pages = {415--431},
volume = {15},
number = {4},
journal = {Advances in Geometry},
doi = {doi:10.1515/advgeom-2015-0031},
year = {2015},
lastchecked = {2025-08-07}
}

@article{Sakuma,
author = {Sakuma, Makoto},
year = {1986},
month = {01},
pages = {},
title = {On strongly invertible knots},
journal = {Algebraic and topological theories (Kinosaki, 1984)}
}

@article{conway2025simplyslicingknots,
      title={Simply slicing knots}, 
      author={Anthony Conway and Patrick Orson and Mark Pencovitch},
      year={2025},
      journal={arXiv preprint arXiv:2507.00431},
}

@article{HMinvolutive2017,
author = {Kristen Hendricks and Ciprian Manolescu},
title = {{Involutive Heegaard Floer homology}},
volume = {166},
journal = {Duke Mathematical Journal},
number = {7},
publisher = {Duke University Press},
pages = {1211 -- 1299},
year = {2017},
doi = {10.1215/00127094-3793141},
URL = {https://doi.org/10.1215/00127094-3793141}
}

@article{borodzik2025khovanovhomologyequivariantsurfaces,
      title={Khovanov homology and equivariant surfaces}, 
      author={Maciej Borodzik and Irving Dai and Abhishek Mallick and Matthew Stoffregen},
      year={2025},
      journal={arXiv preprint arXiv:2507.13642},
      eprint={2507.13642},
      archivePrefix={arXiv},
      primaryClass={math.GT},
      url={https://arxiv.org/abs/2507.13642}, 
}

@article{Zemke2019absolutegradinginHFL,
author = {Zemke, Ian},
year = {2019},
month = {03},
pages = {207-323},
title = {{Link cobordisms and absolute gradings on link Floer homology}},
volume = {10},
journal = {Quantum Topology},
doi = {10.4171/QT/124}
}

@article{Zemkequasistabandbasepointmoving,
author = {Ian Zemke},
title = {{Quasistabilization and basepoint moving maps in link Floer homology}},
volume = {17},
journal = {Algebraic \& Geometric Topology},
number = {6},
publisher = {MSP},
pages = {3461 -- 3518},
year = {2017},
doi = {10.2140/agt.2017.17.3461},
URL = {https://doi.org/10.2140/agt.2017.17.3461}
}

@article{Edmonds1983GroupAO,
  title={Group actions on fibered three-manifolds},
  author={Allan L. Edmonds and Charles Livingston},
  journal={Commentarii Mathematici Helvetici},
  year={1983},
  volume={58},
  pages={529-542},
  url={https://api.semanticscholar.org/CorpusID:51573484}
}

@article{Murasugi1971OnPK,
  title={On periodic knots},
  author={Kunio Murasugi},
  journal={Commentarii Mathematici Helvetici},
  year={1971},
  volume={46},
  pages={162-174},
  url={https://api.semanticscholar.org/CorpusID:120483606}
}

@article{Trotter1961PeriodicAO,
  title={Periodic automorphisms of groups and knots},
  author={Hale F. Trotter},
  journal={Duke Mathematical Journal},
  year={1961},
  volume={28},
  pages={553-557},
  url={https://api.semanticscholar.org/CorpusID:122115092}
}

@article{Sano_2025,
   title={{Involutive Khovanov homology and equivariant knots}},
   volume={25},
   ISSN={1472-2747},
   url={http://dx.doi.org/10.2140/agt.2025.25.5059},
   DOI={10.2140/agt.2025.25.5059},
   number={8},
   journal={Algebraic \& Geometric Topology},
   publisher={Mathematical Sciences Publishers},
   author={Sano, Taketo},
   year={2025},
   month=nov, pages={5059–5111} }

@article{Boylechennegativeamphichiral,
author = {Boyle, Keegan and Chen, Wenzhao},
year = {2023},
month = {08},
pages = {581-622},
title = {{Negative amphichiral knots and the half-Conway polynomial}},
volume = {40},
journal = {Revista Matemática Iberoamericana},
doi = {10.4171/RMI/1442}
}

@article{BIequiv4genusofsiandperidicknots,
author = {Boyle, Keegan and Issa, Ahmad},
title = {Equivariant 4-genera of strongly invertible and periodic knots},
journal = {Journal of Topology},
volume = {15},
number = {3},
pages = {1635-1674},
doi = {https://doi.org/10.1112/topo.12254},
url = {https://londmathsoc.onlinelibrary.wiley.com/doi/abs/10.1112/topo.12254},
eprint = {https://londmathsoc.onlinelibrary.wiley.com/doi/pdf/10.1112/topo.12254},
year = {2022}
}

@article{MillerPowellequivslicegenera,
author = {Miller, Allison N. and Powell, Mark},
title = {Strongly invertible knots, equivariant slice genera, and an equivariant algebraic concordance group},
journal = {Journal of the London Mathematical Society},
volume = {107},
number = {6},
pages = {2025-2053},
doi = {https://doi.org/10.1112/jlms.12732},
url = {https://londmathsoc.onlinelibrary.wiley.com/doi/abs/10.1112/jlms.12732},
eprint = {https://londmathsoc.onlinelibrary.wiley.com/doi/pdf/10.1112/jlms.12732},
year = {2023},
}

@inbook{Montesinos+1975+227+260,
url = {https://doi.org/10.1515/9781400881512-015},
title = {{Surgery on Links and Double branched covers of $S^3$}},
booktitle = {Knots, Groups and 3-Manifolds (AM-84), Volume 84},
author = {Jose M. Montesinos},
publisher = {Princeton University Press},
address = {Princeton},
pages = {227--260},
doi = {doi:10.1515/9781400881512-015},
isbn = {9781400881512},
year = {1975},
lastchecked = {2025-11-20},
}

@article{boyle2025involutionss4,
      title={{Involutions on $S^4$}}, 
      author={Keegan Boyle and Wenzhao Chen and Anthony Conway},
      year={2025},
      journal={arXiv preprint arXiv:2512.22724},
}

@article{guth2025real,
	title={Real {H}eegaard {F}loer Homology},
	author={Guth, Gary and Manolescu, Ciprian},
	journal={arXiv preprint arXiv:2504.09034},
	year={2025},
}

@article{dai2026exoticdiskssingularinstanton,
	title={Exotic disks and singular instanton Floer homology}, 
	author={Irving Dai and Abhishek Mallick and Masaki Taniguchi},
	year={2026},
	journal={arXiv preprint arXiv:2606.05819},
	eprint={2606.05819},
	archivePrefix={arXiv},
	primaryClass={math.GT},
	url={https://arxiv.org/abs/2606.05819}, 
}

@misc{ZhenkunLioythonprgram,
	author = {Zhenkun Li},
	note = {\url{https://github.com/Zhenkun-Li/RealGridHomology}},
	title = {}
}

@article{DiPrisaequivalgconcordance,
	author = {Di Prisa, Alessio},
	title = {Equivariant algebraic concordance of strongly invertible knots},
	journal = {Journal of Topology},
	volume = {17},
	number = {4},
	pages = {e70006},
	doi = {https://doi.org/10.1112/topo.70006},
	url = {https://londmathsoc.onlinelibrary.wiley.com/doi/abs/10.1112/topo.70006},
	eprint = {https://londmathsoc.onlinelibrary.wiley.com/doi/pdf/10.1112/topo.70006},
	year = {2024}
}

@article{DiprisaTheequivariantconcordancegroupisnotabelian,
	author = {Di Prisa, Alessio},
	title = {The equivariant concordance group is not abelian},
	journal = {Bulletin of the London Mathematical Society},
	volume = {55},
	number = {1},
	pages = {502-507},
	doi = {https://doi.org/10.1112/blms.12741},
	url = {https://londmathsoc.onlinelibrary.wiley.com/doi/abs/10.1112/blms.12741},
	eprint = {https://londmathsoc.onlinelibrary.wiley.com/doi/pdf/10.1112/blms.12741},
	year = {2023}
}

@article{MallickHFKandsurgeryonequivaiantknot,
	author = {Mallick, Abhishek},
	title = {Knot Floer homology and surgery on equivariant knots},
	journal = {Journal of Topology},
	volume = {17},
	number = {4},
	pages = {e70001},
	doi = {https://doi.org/10.1112/topo.70001},
	url = {https://londmathsoc.onlinelibrary.wiley.com/doi/abs/10.1112/topo.70001},
	eprint = {https://londmathsoc.onlinelibrary.wiley.com/doi/pdf/10.1112/topo.70001},
	year = {2024}
}

@article{JZstabilizationdistancefromHFL,
	author = {Juhász, András and Zemke, Ian},
	title = {{Stabilization distance bounds from link Floer homology}},
	journal = {Journal of Topology},
	volume = {17},
	number = {2},
	pages = {e12338},
	doi = {https://doi.org/10.1112/topo.12338},
	url = {https://londmathsoc.onlinelibrary.wiley.com/doi/abs/10.1112/topo.12338},
	eprint = {https://londmathsoc.onlinelibrary.wiley.com/doi/pdf/10.1112/topo.12338},
	year = {2024}
}

@article{Myers1981HomologyW,
	title={Homology \$3\$-spheres which admit no PL involutions.},
	author={John Myers},
	journal={Pacific Journal of Mathematics},
	year={1981},
	volume={94},
	pages={379-384},
	url={https://api.semanticscholar.org/CorpusID:125011828}
}

@misc{knotinfo,
	Author = {Livingston, Charles and Moore, Allison H.},
	howpublished = {URL: \url{knotinfo.org}},
	Month = {Current Month},
	Title = {KnotInfo: Table of Knot Invariants},
	Year = {Current Year},
}
\end{document}